%% file: InertOrd.tex
\documentclass[english]{smfbook}
 
\usepackage{dsfont,appendix}
\usepackage{graphicx}
\usepackage[usenames,dvipsnames,svgnames,table]{xcolor}
\usepackage[utf8]{inputenc}
\usepackage[T1]{fontenc}
\usepackage{color}
\usepackage{mathrsfs}  
\usepackage{amsmath,amsxtra,amsthm,amssymb,xr,enumerate}
\usepackage[all]{xy}
\usepackage[bbgreekl]{mathbbol}
\usepackage{bbm}
\usepackage{enumitem}
\usepackage{chngcntr}
\usepackage{tikz}
\usepackage{array}
\def\checkmark{\tikz\fill[scale=0.4](0,.35) -- (.25,0) -- (1,.7) -- (.25,.15) -- cycle;} 

\def\halfcheckmark{\tikz\draw[scale=0.4,fill=black](0,.35) -- (.25,0) -- (1,.7) -- (.25,.15) -- cycle (0.75,0.2) -- (0.77,0.2)  -- (0.6,0.7) -- cycle;}

\DeclareFontFamily{U}{mathx}{\hyphenchar\font45}
\DeclareFontShape{U}{mathx}{m}{n}{
      <5> <6> <7> <8> <9> <10>
      <10.95> <12> <14.4> <17.28> <20.74> <24.88>
      mathx10
      }{}
\DeclareSymbolFont{mathx}{U}{mathx}{m}{n}
\DeclareMathSymbol{\bigtimes}{1}{mathx}{"91}

\DeclareMathSymbol{\invques}{\mathord}{operators}{`>}
\DeclareUnicodeCharacter{00BF}{\tmquestiondown}
\DeclareRobustCommand{\tmquestiondown}{%
  \ifmmode\invques\else\textquestiondown\fi
}
\usepackage{verbatim}

\counterwithout{footnote}{chapter}

\numberwithin{section}{chapter}
\numberwithin{equation}{chapter}

\makeatletter
\newcommand{\mylabel}[2]{#2\def\@currentlabel{#2}\label{#1}}
\makeatother

\newtheorem{theorem}{Theorem}[section]
\newtheorem{lemma}[theorem]{Lemma}
\newtheorem{conj}[theorem]{Conjecture}
\newtheorem{proposition}[theorem]{Proposition}
\newtheorem{corollary}[theorem]{Corollary}
\newtheorem{defn}[theorem]{Definition}
\newtheorem{example}[theorem]{Example}

\newtheorem{remark}[theorem]{Remark}

\newtheorem{convention}[theorem]{Convention}

\newcommand{\lb}{[[}
\newcommand{\rb}{]]}
\newcommand{\pr}{\mathrm{pr}}

\newcommand{\Gal}{\operatorname{Gal}}
\newcommand{\Fil}{\operatorname{Fil}}

\newcommand{\DD}{\mathbb{D}}
\newcommand{\BB}{\mathbb{B}}
\newcommand{\cB}{\mathcal{B}}
\newcommand{\NN}{\mathbb{N}}
\newcommand{\QQ}{\mathbb{Q}}
\newcommand{\Qp}{\mathbb{Q}_p}
\newcommand{\Zp}{\mathbb{Z}_p}
\newcommand{\ZZ}{\mathbb{Z}}
\newcommand{\et}{\textup{\'et}}
\newcommand{\cBF}{\mathcal{BF}}
\renewcommand{\AA}{\mathbb{A}}
\newcommand{\FF}{\mathbb{F}}
\newcommand{\FFF}{\mathcal{F}}
\newcommand{\fF}{\mathfrak{F}}
\newcommand{\g}{\mathbf{g}}

\newcommand{\ord}{\mathrm{ord}}

\newcommand{\vp}{\varphi}
\newcommand{\cL}{\mathcal{L}}
\newcommand{\cH}{\mathcal{H}}
\newcommand{\cO}{\mathcal{O}}

\newcommand{\HIw}{H^1_{\mathrm{Iw}}}

\newcommand{\GL}{\mathrm{GL}}

\newcommand{\Brig}{\BB_{{\rm rig},L}^+}

\newcommand{\col}{\mathrm{Col}}
\newcommand{\image}{\mathrm{Im}}
\newcommand{\cyc}{\textup{cyc}}

\newcommand{\fC}{\mathfrak{C}}

\newcommand{\loc}{\mathrm{loc}}

\newcommand{\ff}{\mathfrak{f}}
\newcommand{\fL}{\mathfrak{L}}
\newcommand{\fm}{\mathfrak{M}}
\newcommand{\fI}{\mathfrak{I}}

\newcommand{\Hom}{\mathrm{Hom}}
\newcommand{\Sel}{\mathrm{Sel}}
\newcommand{\Char}{\mathrm{char}}

\newcommand{\ac}{\textup{ac}}
\newcommand{\LL}{\Lambda}
\newcommand{\TT}{\mathbb{T}}

\newcommand{\f}{\textup{\bf f}}

\newcommand{\Gr}{\textup{Gr}}

\newcommand{\lra}{\longrightarrow}

\newcommand{\res}{\textup{res}}
\newcommand{\TSym}{\textup{TSym}}
\newcommand{\etaf}{v_{f,\beta}^*}
\newcommand{\ofs}{\omega_{f^\star}}
\newcommand{\ogs}{\omega_{g^\star}}

\newcommand{\BF}{\textup{BF}}
\newcommand{\cP}{\mathcal{P}}
\newcommand{\cF}{\mathcal{F}}
\newcommand{\cN}{\mathcal{N}}

\newcommand{\Dcris}{\mathbb{D}_{\rm cris}}

\newcommand{\fM}{\m}

\newcommand{\Tw}{\mathrm{Tw}}

\newcommand{\olinerho}{\overline{\rho}}
\newcommand{\sC}{\mathscr{C}}

\definecolor{Green}{rgb}{0.0, 0.5, 0.0}

\newcommand{\p}{\mathfrak{p}}

\newcommand{\m}{\mathfrak{m}}

\newcommand{\cG}{\mathcal{G}}

\newcommand{\cR}{\mathcal{R}}

\newcommand{\Cp}{\mathbb{C}_p}

\newcommand{\x}{\mathrm{\bf x}}

\newcommand{\bz}{\mathbf{z}}

\usepackage{hyperref}

 \definecolor{pAlgae}{RGB}{87,115,135}
\definecolor{airforceblue}{rgb}{0.36, 0.54, 0.66}
	\definecolor{bondiblue}{rgb}{0.0, 0.58, 0.71}
\definecolor{britishracinggreen}{rgb}{0.0, 0.26, 0.15}
\definecolor{camouflagegreen}{rgb}{0.47, 0.53, 0.42}
\definecolor{darkcyan}{rgb}{0.0, 0.55, 0.55}

\hypersetup{
    colorlinks = true,
    linkcolor=blue,
     filecolor=blue,
     citecolor = darkcyan,      
     urlcolor=cyan,
    linkbordercolor = {white},
}

\makeindex

\begin{document}
\frontmatter

\title{Iwasawa theory of twists of elliptic modular forms over imaginary quadratic fields at inert primes}

\alttitle{La théorie d'Iwasawa des twists des formes modulaires elliptiques sur les corps quadratiques imaginaires aux nombres premiers inertes}

\begin{abstract}
Our primary goal in this manuscript is to study the Iwasawa theory for semi-ordinary families of automorphic forms on $\mathrm{GL}_2\times\mathrm{Res}_{K/\mathbb{Q}}\mathrm{GL}_1$, where $K$ is an imaginary quadratic field where the prime $p$ is inert. We prove divisibility results towards Iwasawa main conjectures in this context, utilizing the optimized signed factorization procedure for Perrin-Riou functionals and Beilinson--Flach elements for a family of Rankin--Selberg products of $p$-ordinary forms with a fixed $p$-non-ordinary modular form. The optimality enables an effective control on the $\mu$-invariants of Selmer groups and $p$-adic $L$-functions as the modular forms vary in families, which is crucial for our patching argument to establish one divisibility in an Iwasawa main conjecture in three variables. 
\end{abstract}

\begin{altabstract}
Le but principal du présent manuscrit est d'étudier la théorie d'Iwasawa pour les familles semi-ordinaires de formes automorphes sur $\mathrm{GL}_2\times\mathrm{Res}_{K/\mathbb{Q}}\mathrm{GL}_1$, où $K$ est un corps quadratique imaginaire dans lequel le nombre premier $p$ est inerte. Nous démontrons des résultats de divisibilité en vue des conjectures principales d'Iwasawa dans ce contexte, en utilisant la procédure de factorisation signée optimisée pour les fonctionnelles de Perrin-Riou et les éléments de Beilinson--Flach pour une famille de produits de Rankin--Selberg de formes $p$-ordinaires avec une forme modulaire $p$-non-ordinaire fixée. L'optimalité permet un contrôle effectif sur les $\mu$-invariants des groupes de Selmer et des fonctions $L$ $p$-adiques lorsque les formes modulaires varient en familles, ce qui est crucial pour notre argument de recollement visant à établir une divisibilité dans une conjecture principale d'Iwasawa à trois variables.
\end{altabstract}

\author{K\^az\i m B\"uy\"ukboduk}
\address{UCD School of Mathematics and Statistics\\ University College Dublin\\ Ireland}
\email{kazim.buyukboduk@ucd.ie}

\author{Antonio Lei}
\address{Department of Mathematics and Statistics\\University of Ottawa\\
150 Louis-Pasteur Pvt\\
Ottawa, ON\\
Canada K1N 6N5}
\email{antonio.lei@uottawa.ca}

\altkeywords{Théorie d'Iwasawa, produits de Rankin--Selberg, nombres premiers semi-ordinaires, systèmes d'Euler localement restreints, éléments de Beilsinson--Flach}
\maketitle
\setcounter{page}{4}

\tableofcontents

\mainmatter

\include{InertOrdMainChapter1}

\include{InertOrdMainChapter2}

\include{InertOrdMainChapter3}

\include{InertOrdMainChapter4}
\include{InertOrdMainChapter5}
\appendix

\include{InertOrdMainAppendixA}

\include{InertOrdMainAppendixB}

\include{InertOrdMainAppendixC}
\backmatter
\bibliographystyle{amsalpha}
\bibliography{references}

\include{index}
\printindex
\end{document}

%% file: InertOrdMainChapter1.tex
\chapter{Introduction}
We fix forever a prime $p\ge 5$. Let $f\in S_{k_f+2}(\Gamma_1(N_f),\varepsilon_f)$ be a cuspidal eigenform  which is not of CM type, where the level $N_f$ is coprime to $p$. We assume that $f$ admits a $p$-ordinary stabilization $f_{\alpha}$ with $U_p$-eigenvalue $\alpha_f$. We fix an imaginary quadratic field $K$\index{$K$: imaginary quadratic field} with discriminant $D_K$ which is coprime to $N_f$ and where $p$ remains inert. We let $c$ denote any lift of a generator of $\Gal(K/\QQ)$ to $G_\QQ$. We also fix a ray class character $\chi$ of $K$ of order coprime to $p$ (which we call the branch character, following Hida) and of conductor coprime to $D_KN_f$ (but not necessarily to $p$). We assume that $\chi\neq \chi^c$; i.e. we do not treat the ``Eisenstein'' case.

Our eventual goal is to prove  divisibility results on the main conjectures for $p$-semi-ordinary families of automorphic motives on $\GL_2\times {\rm Res}_{K/\QQ}\GL_1$ \footnote{In fact, we treat somewhat a more general class of semi-ordinary Rankin--Selberg convolutions $f_\alpha\otimes g_\mu$ on  $\GL_2\times\GL_2$, where $f_\alpha$ is above and $g$ is a non-ordinary cuspidal eigenform.}. In more explicit (but still very rough) terms, we will prove divisibilities in the Iwasawa main conjectures for families of Rankin--Selberg convolutions $f_{\alpha} \times \theta(\chi\psi)$, where $\psi$ varies among algebraic Hecke characters of $K$ with $p$-power conductor and $\theta(\chi\psi)$ is the theta-series of $\chi\psi$. Our results in this work can be considered as a piece of evidence towards the variational versions of Bloch--Kato conjectures for the relevant class of motives.

 \section{Previous Works}
 \label{sec_previous_works}
There are many earlier results concerning the arithmetic and Iwasawa theoretic aspects of Rankin--Selberg convolutions $f_{\alpha} \times \theta(\chi\psi)$ and their families. We record a significant portion of those that our manuscript extends, highlighting the key differences with the present work both in terms of the assumptions as well as the techniques involved. Table~\ref{table_summary} below serves as a summary of this comparison. We note that the labels in the first row of this table indicate the section numbers where the corresponding results are summarized.

The reader will notice that we have left out the discussion concerning the cyclotomic Iwasawa theory in \S\ref{sec_previous_works}, concentrating either on the anticyclotomic case (c.f. \S\ref{subsec_1_1_1_2021_09_07}) and Iwasawa theory for the $\ZZ_p^2$-extension $K_\infty$ of $K$ (c.f. \S\ref{subsec_1_1_2_2021_09_07} and \S\ref{subsec_1_1_3_2021_09_07}). This is for the following reason. When $\chi=\chi^c$ (so that $\chi$ is a Hecke character of $\QQ$; a scenario that our paper will not address at all), the cyclotomic Iwasawa theory for $f_{\alpha} \times \theta(\chi)$ reduces to that for $f_{\alpha}\otimes\chi$ and $f_{\alpha}\otimes\chi\epsilon_K$, where $\epsilon_K$ is the quadratic character associated to $K/\QQ$. This can be studied by Kato's Euler system, both when $f_\alpha$ is ordinary or otherwise. In general, one can descend to the cyclotomic tower from $K_\infty/K$. In the anticyclotomic case, there is a dichotomy that depends on the root number in the functional equation of the complex $L$-function $L(f/K,s)$ (customarily labelled as the \textit{definite case} when the root number is $+1$ and the \textit{indefinite case} when it is $-1$) and it is a subtle problem to determine how descent from $K_\infty/K$ to the anticyclotomic tower plays out in either of these cases.

In \S\ref{subsec_1_1_1_2021_09_07}, we will review prior works in the anticyclotomic case, which employ arguments based primarily on an anticyclotomic ``Heegner-type'' Euler system and have no immediate bearing to the full $\ZZ_p^2$-extension. 

With the exception of \S\ref{subsec_1_1_2_2021_09_07}, where we review the work of Skinner--Urban, all results we summarize in \S\ref{sec_previous_works} towards the $\ZZ_p^2$-extension involve Beilinson--Flach elements. One often needs to work with a number of hypotheses on the branch character $\chi$, and the organization of our overview reflects the varying set of conditions on $\chi$. The results (both in the present work and those we discuss in \S\ref{subsec_1_1_2_2021_09_07} and \S\ref{subsec_1_1_3_2021_09_07}) towards anticyclotomic Iwasawa theory are deduced from those for the $\ZZ_p^2$-extension through a descent argument. In this manuscript, this comes about unraveling Nekov\'a\v{r}'s general machinery developed in \cite[\S11]{nekovar06}, c.f. \S\ref{subsubsec_anticyclo_1}--\S\ref{subsubsec_anticyclo_3} in the main body of the manuscript.

We would like to emphasize that one novelty of the present paper is that we overcome certain technical hurdles particular to the case where $p$ is inert in $K$. It is worthwhile to note that the only scenario in the earlier works where $p$ is allowed to be inert in $K/\QQ$ is the one discussed in \S\ref{subsubsec_1_1_1_1_2021_09_07}.

\subsection{Anticyclotomic Iwasawa Theory}
\label{subsec_1_1_1_2021_09_07}
The first set of treatments of the anticyclotomic Iwasawa theory of Rankin--Selberg convolutions $f_{\alpha} \times \theta(\chi\psi)$ rely  on Heegner-type Euler systems. We have organized our exposition of these results as two separate threads: In \S\ref{subsubsec_1_1_1_1_2021_09_07} and \S\ref{subsubsec_1_1_1_2_2021_09_07}, we consider the case when $\chi=\mathds{1}=\varepsilon_f$, whereas in \S\ref{subsubsec_1_1_1_3_2021_09_07}, we allow $\chi$ to be a general ring class character over $K$. 

\subsubsection{The definite case: Heegner points and level raising congruences ($\chi=\mathds{1}$)}
\label{subsubsec_1_1_1_1_2021_09_07}
 Bertolini--Darmon in \cite{BertoliniDarmon2005} studied the ``definite'' anticyclotomic Iwasawa main conjectures for the base change $f_{/K}$ of a modular form $f\in S_{2}(\Gamma_0(N_f))$ to $K$, where the prime $p$ may split or remain inert in $K/\QQ$ and $f$ has ordinary reduction at $p$. The family in question consists of the self-dual twists of the Rankin--Selberg motives associated to $f\times \theta(\psi)$ as $\psi$ varies among anticyclotomic Hecke characters of $K$ with $p$-power conductor.  In particular, they showed that the Selmer group of $f_{/K}$ over the anticyclotomic $\Zp$-extension is cotorsion over the Iwasawa algebra and showed that its characteristic ideal divides the corresponding anticyclotomic $p$-adic $L$-function attached to $f_{/K}$, showing one half of the anticyclotomic Iwasawa main conjecture. In the definite case, because of the absence of Heegner hypothesis, Heegner points are not available for $f_{/K}$. Instead, one may find a modular form $g$ of level $N\ell$  which is congruent to $f$ modulo a power of $p$ for some auxiliary prime $\ell$ employing Ribet's results on level raising \cite{ribet90}. Furthermore, one may choose such a $g$ for which the Heegner hypothesis holds. One may then compare the Selmer groups of $f$ and $g$ and employ the Heegner-point Euler system attached to $g$ to show (allowing $g$ to vary) one inclusion of the anticyclotomic Iwasawa main conjecture for $f$.

The work of Chida and Hsieh~\cite{chidahsiehanticyclomainconjformodformscomposito} generalized the work of  Bertolini--Darmon alluded to above and studied the ``definite'' anticyclotomic Iwasawa main conjectures for the base change to $K$ of a modular form $f\in S_{k_f+2}(\Gamma_0(N_f))$ of  even weight that is small compared to $p$. In this work, the prime $p$ once again may split or remain inert in $K/\QQ$. As in \cite{BertoliniDarmon2005}, they proved one inclusion of the anticyclotomic Iwasawa main conjecture. While the idea of proof is similar, Chida and Hsieh have removed several hypotheses from \cite{BertoliniDarmon2005} (for example, $f$ being $p$-isolated and the maximality of the image of the residual representation of $f$).

In \cite{darmoniovita}, Darmon--Iovita studied a generalization of \cite{BertoliniDarmon2005} in the case where $f$ is of weight two with $a_p(f)=0$ and $p$ splits in $K$. They defined plus and minus anticyclotomic $p$-adic $L$-functions as well as the plus and minus Selmer groups, paralleling the works of Pollack \cite{pollack03} and Kobayashi \cite{kobayashi03} in the cyclotomic setting. They showed that these Selmer groups are cotorsion over the corresponding Iwasawa algebra and their characteristic ideals divide the signed $p$-adic $L$-functions. The proofs, as in \cite{BertoliniDarmon2005}, are based on a Heegner-point Euler system argument blended with level-raising congruences, where they proceed via a successive choice of modular forms that are congruent to $f$ modulo ever increasing powers of $p$ and that satisfy the Heegner hypothesis. After the present manuscript was circulated, the authors together with A. Burungale in  \cite{BBL1} generalized the results of \cite{darmoniovita} to two new settings: Firstly, they relaxed the hypothesis in op. cit. that $a_p(f)=0$. Secondly, they allow $p$ be inert in $K/\QQ$ so long as $a_p(f)=0$ and the Hecke field of $f$ can be embedded in $\Qp$.

\subsubsection{Indefinite case: Heegner points ($\chi=\mathds{1}$)} 
\label{subsubsec_1_1_1_2_2021_09_07}
The indefinite case refers to the scenario where the sign of the functional equation of the Rankin--Selberg $L$-series $L(f/K,\psi,s)$ equals $-1$ for all anticyclotomic characters $\psi$ at the central critical value $s=1$ (note that we are implicitly assuming  this in the current subsection, unless stated otherwise), forcing the $L$-function to vanish at $s=1$. When $f$ gives rise to an elliptic curve $E_{/\QQ}$ via the Eichler--Shimura construction (i.e., if the Hecke field $K_f$ of $f$ equals $\QQ$) and $E$ has good ordinary reduction at $p$, Bertolini in \cite{Bertolini1995Compositio} proved, as a reflection of this systematic vanishing of the $L$-functions, that the (discrete) Bloch--Kato Selmer group of the eigenform $f$ over the anticyclotomic $\Zp$-extension is of corank one over the corresponding Iwasawa algebra. When $E$ has good supersingular reduction at $p$ with $a_p(E)=0$, Kim \cite{kim07} and Longo--Vigni \cite{longovigni} proved a similar result for Kobayashi-type plus and minus Selmer groups. These results have implications on the growth of the Mordell--Weil ranks of $E$ inside the anticyclotomic $\Zp$-extension of $K$.

In the ordinary setting, Howard in \cite{howardcompositio1} proved, refining the result of Bertolini that we have alluded to above, one inclusion in  Perrin-Riou's Heegner point (``indefinite'' anticyclotomic) main conjecture for $f\in S_{2}(\Gamma_0(N_f))$ whose Hecke field $K_f$ equals $\QQ$, under the additional assumption that the $p$-adic $G_K$-representation attached to $f$ surjects onto $\GL_2(\Zp)$. More precisely, he showed that the characteristic ideal of the maximal cotorsion quotient of the anticyclotomic (discrete) Selmer group is related to the index of the Heegner points inside the global Iwasawa cohomology. This result was extended to more general $f\in S_{2}(\Gamma_0(N_f))$ (to cover the cases when $K_f\neq\QQ$) in \cite{howardduke} (in fact, Howard in \cite{howardduke} treats the much more general scenario when $f$ in question is a Hilbert modular form of parallel weight $2$ over a totally real field). 

Later in \cite{howard2007}, Howard initiated the study of ``indefinite'' anticyclotomic main conjectures in nearly-ordinary families. The nearly-ordinary family studied in op. cit. interpolates the self-dual twists of the Rankin--Selberg motives associated to $f_{\alpha}\times \theta(\psi)$ as $\psi$ varies among anticyclotomic Hecke characters of $K$ with $p$-power conductor and $f_{\alpha}$ in a Hida family. In particular, he defined the so-called big Heegner points, which interpolate Heegner cycles attached to $f_\alpha$ as it it varies over a Hida family (this fact was established  later, by Castella~\cite{CastellapadicvariationofHeegnerpoints} and Ota~\cite{Ota2020}), and showed that they satisfy Kolyvagin's Euler-system relations. Works of Fouquet~\cite{Fouquet2013} and the first named author~\cite{KbbBigHeegner} utilize the big Heegner points introduced in \cite{howard2007} as an input to obtain results towards ``indefinite'' anticyclotomic main conjectures in the multivariate setting of \cite{howard2007}. In all these works, there is a priori no restriction on the local behaviour of $p$ in the extension $K/\QQ$.  However,  these works have consequences towards main conjectures \emph{without} $p$-adic $L$-functions and one needs to assume further that $p$ is split in $K/\QQ$ in order to deduce statements involving $p$-adic $L$-functions (see also \cite[Theorem E]{disegniuniversalHeegcycle} where this matter is addressed).

In the aforementioned results where the Kolyvagin-system machinery is employed, the residual representation of $f$ at $p$ is assumed to be irreducible. In a recent work by Castella--Grossi--Lee--Skinner \cite{CGLS}, this hypothesis has been weakened in the case where $f$ is of weight two and $p$ splits in $K$.

In the indefinite case, since $L(f/K,1)=0$, it is interesting to study the derivative of $L(f/K,s)$ at $s=1$. In the seminal work of Gross--Zagier \cite{GZ86} (and subsequent generalizations \cite{GZ87,zhang97,zhangexporankinLseriesvolume}), it has been  shown that $L'(f/K,1)$ is related to the height of the Heegner point/cycle attached to $f_{/K}$.
Suppose that the weight of $f$ is 2, $p$ is a split prime in $K$ and that $f$ is ordinary at $p$. In \cite{perrinriou87}, Perrin-Riou proved a $p$-adic analogue of these results, showing that the derivative of the $p$-adic Rankin--Selberg $L$-function of $f_{/K}$ at the trivial character is related to the (cyclotomic) $p$-adic height of the Heegner points associated to $f_{/K}$. When $f$ is of (even) weight greater than 2, Nekov\'a\v{r} in \cite{nekovar95} generalized Perrin-Riou's result, where the first order derivative of the relevant $p$-adic $L$-function is computed in terms of the (cyclotomic) $p$-adic heights of Heegner cycles.

Howard in \cite{howard05} proved an ``Iwasawa-theoretic'' version of Perrin-Riou's result, confirming a conjecture of Mazur and Rubin in ~\cite{MR_ICM2002}. In more precise terms, he showed that 
\begin{itemize}
    \item[a)] the cyclotomic $p$-adic pairings in \cite{perrinriou87} can be interpolated along the anticyclotomic pairing, to give rise to a $\LL(\Gamma_\ac)$-adic height pairing (where $\Gamma_\ac$ is the Galois group of the anticyclotomic $\Zp$-extension of $K$ and $\LL(\Gamma_\ac)=\Zp[[\Gamma_\ac]]$);
    \item[b)] the partial derivative of Perrin-Riou's two-variable $p$-adic $L$-function attached to $f_{/K}$ along the cyclotomic direction can be computed in terms of the $\LL(\Gamma_\ac)$-adic height of the tower of Heegner points along the anticyclotomic tower. 
\end{itemize}
 More recently, Castella in~\cite{castellaJLMS} (see also the related work of present authors in~\cite{BLForum}, where there is no restriction on the weight of $f$) gave a new proof of Howard's result using Beilinson--Flach elements defined in \cite{LLZ1,LLZ2,KLZ1,KLZ2}. This approach will be outlined in \S\ref{subsubsec_1_1_3_1_2021_09_07} with further details. We also remark that Disegni in \cite{disegni17,disegniuniversalHeegcycle} established analogues of the results of Howard in \cite{howard05} in greater generality.
 
 We close \S\ref{subsubsec_1_1_1_2_2021_09_07} by noting that when $f$ is non-ordinary at $p$, the utility of Heegner points and Heegner cycles in the  study of indefinite anticyclotomic Iwasawa theory was rather limited, until the introduction of generalized Heegner cycles by Bertolini--Darmon--Prasanna in~\cite{bertolinidarmonprasanna13}. The applications of generalized Heegner cycles to indefinite anticyclotomic Iwasawa theory will be summarized in \S\ref{subsubsec_1_1_1_3_2021_09_07}. Before moving ahead in this direction, we note that Kobayashi in~\cite{kobayashi13, kobayashi2014_GZ} has proved a $p$-adic Gross--Zagier formula at non-ordinary primes when the eigenform $f$ has weight $2$ and since then, also announced in \cite{kobayashi_higherweight_nonord_GZ} a $p$-adic Gross--Zagier formula at non-ordinary primes when the slope of the $p$-stabilized eigenform is ``small'' relative to the weight. Building on this result, the first named author and Pollack--Sasaki proved in~\cite{BPSI} a $p$-adic Gross--Zagier formula at non-ordinary primes for all weights (without any restriction on the slope, including the non-$\theta$-critical critical-slope case).

\subsubsection{Generalized Heegner Cycles ($\chi$ is general and $p$ is split)}
\label{subsubsec_1_1_1_3_2021_09_07}
In their seminal work \cite{bertolinidarmonprasanna13}, Bertolini--Darmon--Prasanna introduced generalized Heegner cycles that come attached to the central critical twists of Rankin--Selberg motives of the form $f\times \theta(\chi\psi)$, where $\psi$ varies among anticyclotomic Hecke characters of $K$ with $p$-power conductor.% We recall that the condition that $\psi$ be anticyclotomic amounts to requirement that $\psi^c=\psi^{-1}$.  [Perhaps should mention this earlier if we want to make such a reminder?]
To wit,  we remark that in the scenario when $\psi$ has finite order, generalized Heegner cycles essentially reduce to their classical  counterparts,  namely  Heegner cycles. In \cite{bertolinidarmonprasanna13}, Bertolini--Darmon--Prasanna proved a spectacular variant of the $p$-adic Gross--Zagier formula, where they expressed  the values of their anticyclotomic $p$-adic $L$-function at critical points that lie outside the range of interpolation in terms of the generalized Heegner cycles. In \cite{castella13} (respectively, in \cite{JLZ}; see also \cite{BL-GHC}), this result has been extended so as to allow variation of $f$ when $f_\alpha$ runs in  a slope-zero family (respectively, positive  slope family).

Castella and Hsieh  showed in \cite{CastellaHsiehGHC} that the generalized Heegner cycles give rise to an anticyclotomic Euler system for the central critical twist of the motive $f_{/K}\otimes \chi\psi$ over $K$, assuming that (besides a number of minor technical conditions) 
\begin{itemize}
    \item $(p)=\p\p^c$ is split in $K/\QQ$,
    \item $N$ is a product of primes that split in $K/\QQ$ (classical Heegner hypothesis),
    \item $\varepsilon_f=\mathds{1}$ and $\chi^c=\chi^{-1}$.
\end{itemize}
We note that the central critical value of interest is $L(f_{/K}\otimes\chi\psi,1+\frac{k_f}{2})$ and the running assumptions on $\epsilon_f$, $\chi$ and $\psi$ guarantee that the relevant global root number $\epsilon(f_{/K}\otimes\chi\psi)$ equals $\pm 1$. Thanks to the running Heegner hypothesis, we have $\epsilon(f_{/K})=-1$. It then follows that $$\epsilon(f_{/K}\otimes\chi\psi)=\begin{cases} -1 & \hbox{$\infty$-type of  $\psi$ is $(-a,a)$ with $|a|\leq\frac{k_f}{2}$,}\\
+1 & \hbox{$\infty$-type of  $\psi$ is $(-a,a)$ with $|a|>\frac{k_f}{2}$.}
\end{cases}$$ 

A generic non-vanishing result of Hsieh~\cite[Theorem 3.7]{hsiehnonvanishing} shows that 
$$L(f_{/K}\otimes\chi\psi,1+{k_f}/{2})\neq 0$$ 
for all but finitely many choices of $\psi$ with infinity type $(-a,a)$ verifying the inequality $|a|>\frac{k_f}{2}$. In this case, the Bloch--Kato conjecture predicts that the relevant Selmer group is finite and the proof of this prediction is one of the main results of \cite{CastellaHsiehGHC}.

In the complementary case when $\psi$ has infinity type $(-a,a)$ with $|a|\leq\frac{k_f}{2}$, the Bloch--Kato conjecture predicts that the relevant Selmer group has infinite order. Castella--Hsieh proved in \cite[Theorem B]{CastellaHsiehGHC}  a partial result in this direction, assuming the non-vanishing of a generalized Heegner cycle.

Although Castella--Hsieh do not make claims towards Iwasawa theoretic statements, their key ingredient in op. cit. is the  anticyclotomic (Iwasawa theoretic) Euler system that they constructed out of the generalized Heegner cycles, where the assumption that $v_p(\alpha)=0$ is crucially used in their argument. In \cite{kobayashiGHC, kobayashiota,kobayashiota2}, Kobayashi and Ota developed what they call the theory of integral Perrin-Riou twists, which enables them to twist systems of generalized Heegner cycles (which are not norm-compatible) by anticyclotomic characters, to gain control on the subtle behaviour of the denominators of these systems and in turn allow to extend the results of \cite{CastellaHsiehGHC} to the case where $v_p(\alpha)>0$. We also note that \cite{kobayashiota,kobayashiota2}, unlike its predecessor \cite{CastellaHsiehGHC}, in fact offers applications towards anticyclotomic Iwasawa main conjectures that concern the arithmetic properties of the Bertolini--Darmon--Prasanna $p$-adic $L$-functions (to which we shall refer as the ``BDP-type main conjectures'').

\subsection{Iwasawa theory for $K_\infty/K$ ($\chi=\mathds{1}$, $p$ splits and $v_p(\alpha)=0$)}
\label{subsec_1_1_2_2021_09_07}
The discussion in \S\ref{subsec_1_1_2_2021_09_07} will review the results of \cite{skinnerurbanmainconj} (in conjunction with \cite{kato04}) and of \cite{castellawanGL2families2018, xinwanwanrankinselberg}, which are complementary to one another in the sense that the former treats the definite case (namely, when the global root number verifies $\epsilon(f_{/K})=1$) whereas the latter concerns the indefinite case. These results allow variation in Hida families and we note that the property as to whether the members of a given Hida family fall within the definite case or otherwise is constant in the family.
\subsubsection{Eisenstein congruences} 
Skinner--Urban in \cite{skinnerurbanmainconj}
\label{subsubsec_1_1_2_1_2021_09_07}
proved the Iwasawa main conjecture for the base change $f_{/K}$ of a modular form $f$ of any weight, along the $\ZZ_p^2$-tower $K_\infty$ of $K$, assuming that 
\begin{itemize}
    \item $p$ splits in $K/\QQ$,
    \item $\varepsilon_f=\mathds{1}$ and $k_f\equiv 0 \mod p-1$,
    \item $f$ is residually non-Eisenstein and $p$-distinguished (in the sense that the residual representation $\overline{\rho}_f$, where $\rho_f$ is Deligne's $p$-adic representation attached to $f$, is irreducible and its restriction  to the decomposition group at $p$ is non-scalar),
    \item writing $N=N^+N^-$, where the prime divisors of $N^+$ are split in $K/\QQ$, whereas those dividing $N^-$ are inert, $N^-$ is a square-free product of an odd number of primes,
    \item $\overline{\rho}_f$ is ramified at every prime $\ell$ dividing  $N^-$,
    \item The image of $\rho_f$ contains a conjugate of ${\rm SL}_2(\ZZ_p)$.
\end{itemize}
 The fourth condition above ensures that we are in the definite setting.
 This hypothesis is needed in \cite{skinnerurbanmainconj} as Vatsal's main results in \cite{Vatsal2003} on non-vanishing of  $L$-functions associated to modular forms, which requires this hypothesis on $N^-$, is one of the ingredients of the proofs in \cite{skinnerurbanmainconj}.  
 
 The main result \cite[Theorem 3.31]{skinnerurbanmainconj} of Skinner and Urban in fact allows a treatment with variation of eigenforms in families that are naturally parametrized by $3$ formal variables. The families in loc. cit. concern with interpolation of Rankin--Selberg motives associated to $f_{\alpha}\times \theta(\psi)$ as $\psi$ varies among Hecke characters of $K$ with $p$-power conductors and $f_{\alpha}$ in a Hida family. In fact, Skinner and Urban deduced the Iwasawa main conjectures for a single eigenform (verifying the list of conditions above) over the cyclotomic $\Zp$-extension of $\QQ$, the anticyclotomic $\Zp$-extension of $K$ and  the $\Zp^2$-extension $K_\infty/K$ from their result that concerns $3$-parameter families, via Iwasawa descent. 
 
 It is important to note that the fundamental breakthrough in \cite{skinnerurbanmainconj} is their proof of a one-sided containment in the Iwasawa main conjectures in $3$-variables (c.f. Theorem 3.26 in op. cit.), which asserts that the characteristic ideal of the appropriate Selmer group is contained in the ideal generated by the relevant $p$-adic $L$-function. This containment is opposite to that one would hope to deduce from an Euler system, and we shall call it an ``Eisenstein-ideal containment'' to ease our discussion. The choice of this nomenclature is due to the fact that the ``Eisenstein-ideal containment'' in op. cit. is deduced as a consequence of their important generalization of the Eisenstein-ideal method employed by Mazur and Wiles~\cite{mazurwiles, wiles90} in their proofs of the Iwasawa main conjectures for $\GL_1$. In \cite{skinnerurbanmainconj}, the Eisenstein ideal in question accounts for congruences between cusp forms and Eisenstein series on the group ${\rm GU}(2,2)_{/\QQ}$. We also remark that the strategy to prove ``Eisenstein-ideal containments'' in Iwasawa main conjectures have been utilized in different contexts (most notably, in \cite{HidaTilouine1994,MazurTilouine1990}; see also \cite{HarrisLiSkinner} for an extremely general but conjectural framework), exploiting congruences between stable and
other types of `special' (endoscopic) automorphic forms.

We do \emph{not} pursue in the present manuscript ``Eisenstein-ideal containments'' in Iwasawa main conjectures. Instead,  our main results concern the opposite containments, which arise from the Euler-system machinery.
 
To deduce from the one-sided containment in \cite[Theorem 3.26]{skinnerurbanmainconj} towards the full main conjectures in $3$-variables, the sought after equality (Theorem 3.31 in op. cit.), Skinner and Urban relied on the work of Kato~\cite{kato04}, where he proved the opposite containment in the cyclotomic main conjecture (in a single variable) for an eigenform. This step involves a careful descent procedure and it is based on a general criterion recorded as Lemma 3.2 in \cite{skinnerurbanmainconj}. It crucially requires that the pseudo-null submodules of the relevant dual Selmer groups are trivial (which are established by  Greenberg's criteria). 

In the present work, whenever we  employ an argument involving Iwasawa descent, we prefer to work with Selmer complexes. There are  two main reasons for this choice: The first is the built-in descent formalism  for Selmer complexes, and the second is  Nekov\'a\v{r}'s very general results in \cite{nekovar06} concerning their projective dimensions of the associated (extended) Selmer groups (which, among other things, allows  one to control the pseudo-null submodules of these extended Selmer groups).

 \subsubsection{Heegner points, Beilinson--Flach elements and indefinite anticyclotomic Iwasawa theory} 
 \label{subsubsec_1_1_2_2_2021_09_07}
In~\cite{castellawanGL2families2018}, Castella and Wan complemented the results of Skinner and Urban that we have reviewed in \S\ref{subsubsec_1_1_2_1_2021_09_07}, to cover the indefinite set up. More precisely, they prove the Iwasawa main conjecture in $3$-variables for the base change $\f_{/K}$ of a Hida family of eigenforms $\f$ along the $\ZZ_p^2$-tower $K_\infty$ of $K$ (up to $\mu$-invariants), assuming that 
\begin{itemize}
    \item $p$ splits in $K/\QQ$,
    \item the tame nebentype $\varepsilon_\f$ of the Hida family $\f$ is trivial and the tame level $N$ is square-free,
    \item $\f$ is residually non-Eisenstein and $p$-distinguished,
    \item writing $N=N^+N^-$ as before, $N^-$ is a product of an even number of primes,
    \item the residual representation $\overline{\rho}_\f$ is ramified at every prime $\ell$ dividing  $N^-$\,.
\end{itemize}
The fourth condition above is known as the Shimura--Heegner hypothesis and places one in the indefinite setting, ensuring that the sign $\epsilon(f_{/K})$ of the functional equation for the Hecke $L$-function $L(f_{/K},s)$ at the central critical point equals $-1$. 

The argument in \cite{castellawanGL2families2018} has three main steps:
\begin{enumerate}
    \item[\mylabel{CW2020_1}{{\bf (1)}}] Building on the explicit reciprocity laws for the Beilinson--Flach elements proved by Kings--Loeffler--Zerbes in \cite{KLZ2}, Castella and Wan recast the $3$-variable main conjecture as a ``BDP-type main conjecture'', which concerns a Hida--Rankin $p$-adic $L$-function that admits the Bertolini--Darmon--Prasanna $p$-adic $L$-functions as its specializations. This ``BDP-type main conjecture'' is akin to the main conjecture considered in \cite{kobayashiota,kobayashiota2} in the non-ordinary setting (c.f. \S\ref{subsubsec_1_1_1_3_2021_09_07} above). We note that the explicit reciprocity laws in \cite{KLZ2} are not quite as general as needed in  \cite{castellawanGL2families2018} and the forthcoming work of Burungale--Skinner--Tian is expected to provide the required extension. We note that this technical difficulty is also present in \cite{BL_SplitOrd2020} where the present authors treat the case when $\chi\neq \chi^c$ (but it may be that $\overline{\chi}=\overline{\chi}^c$, in which case the associated theta-series, which is cuspidal owing to assumption that $\chi\neq \chi^c$, does admit an Eisenstein congruence), which we will review in \S\ref{subsec_1_1_4_2021_09_07}. 
   \item[\mylabel{CW2020_2}{{\bf (2)}}] The ``Eisenstein-ideal containment'' in this BDP-type main conjecture has been established by Wan in \cite{xinwanwanrankinselberg}. This result is the indefinite analogue of the containment proved by Skinner--Urban in the definite set-up (c.f. \S\ref{subsubsec_1_1_2_1_2021_09_07}) and its proof follows their strategy, but dwelling  on congruences between Klingen--Eisenstein series and cusp forms on the group ${\rm GU}(3,1)$, rather than ${\rm GU}(2,2)$.
  \item[\mylabel{CW2020_3}{{\bf (3)}}] Via the explicit reciprocity laws for generalized Heegner cycles proved in \cite{CastellaHsiehGHC} (which interpolate the celebrated formulae of Bertolini--Darmon--Prasanna in families), the anticyclotomic restriction of the ``BDP-type main conjecture'' can be related to Howard's big Heegner poin main conjecture~\cite[Conj. 3.3.1]{howard2007}. One  readily has containments in the opposite direction to those proved in \cite{xinwanwanrankinselberg} (and reviewed in item \ref{CW2020_2} above) thanks to \cite{Fouquet2013, KbbBigHeegner}. 
\end{enumerate}
One may finally employ the Iwasawa descent argument \cite[Lemma 3.2]{skinnerurbanmainconj} to combine these steps to deduce the required equality in Iwasawa main conjectures.
%%%%%%%%%%
%%%%%%%%%%
%%%%%%%%%%
%%%%%%%%%%

\subsection{Iwasawa theory for $K_\infty/K$ ($\chi\not\equiv \chi^c$ and $p$ is split)}
\label{subsec_1_1_3_2021_09_07}
We now discuss prior results pertaining to the ``residually non-Eisenstein'' scenario, where we assume that $\chi\not\equiv \chi^c$. All these results dwell primarily on the Beilinson--Flach Euler system, especially those concerning the full $\ZZ_p^2$-extension $K_\infty$ of $K$. In the special case where $\varepsilon_f=\mathds{1}$ and $\chi^c=\chi^{-1}$ is anticyclotomic, the specializations of these results over $K_\infty$ to the anticyclotomic tower has an overlap with those in \S\ref{subsubsec_1_1_1_3_2021_09_07}. 

We underline the fact that in all earlier works in this direction, it is required that the prime $p$ be split in $K/\QQ$. Throughout \S\ref{subsec_1_1_3_2021_09_07}, we retain our notation from the earlier portions of \S\ref{sec_previous_works}.

\subsubsection{Beilinson--Flach elements:  $v_p(\alpha)=0$} 
\label{subsubsec_1_1_3_1_2021_09_07}
The Beilinson--Flach elements of \cite{LLZ1, KLZ2} that come attached to the motive of $f\otimes\theta(\chi\psi)$ over $\QQ$ (and $p$-adic families of such) is, a priori, an Euler system over $\QQ$. This means that the Beilinson--Flach classes are defined as cohomology classes over abelian extensions of $\QQ$. To derive the strongest applications towards the Iwasawa theory over the imaginary quadratic field $K$, it would be desirable to construct an Euler system associated to the motive of $f\otimes \chi\psi$ over $K$, with cohomology classes defined over abelian extensions of $K$. This goal was achieved in \cite{LLZ2} assuming that \begin{itemize}
    \item $k_f=0$ (i.e., the eigenform $f$ is of weight $2$),
    \item $p=\p\p^c$ splits in $K/\QQ$,
    \item $\chi$ verifies the $p$-distinguished condition that $$v_p(\chi(\p)-\chi(\p^c))=0\,,$$
where $v_p$ is the normalized $p$-adic valuation,
\end{itemize}  
through a patching argument, which involves a comparison of the CM branches of Hida families with the same branch character but different tame levels. In \cite{BLForum}, the present authors extended this result (starting off with the Rankin--Iwasawa classes of Kings--Loeffler--Zerbes in \cite{KLZ2}) to cover the case of higher weight forms $f$, but still assuming that $p$ splits and $\chi$ is $p$-distinguished.

The patching argument in \cite{LLZ2, BLForum} relies crucially on the existence of $p$-adic families of CM forms. This in turn translates to the requirement that $p$ be split in $K/\QQ$, as otherwise the eigenforms that have CM by $K$ will be non-ordinary and according to \cite{CITJHC70}, there are no CM families of positive slope. This fundamental difficulty in the scenario when $p$ is inert is the main motivation for the present work.    

Let us resume our discussion in the setting when $p$ is split and $\chi$ is $p$-distinguished. Armed with an Euler system over $K$, Castella~\cite{castellaJLMS} (when $k_f=0$) and the present authors~\cite{BLForum} (general $k_f$) proved results for the family that consists of Rankin--Selberg motives associated to $f\times \theta(\chi\psi)$ where $\psi$ varies among Hecke characters of $K$ with $p$-power conductor. This amounts to the proof of a divisibility in the main conjecture for $f_{/K}\otimes\chi$ over the $\ZZ_p^2$-extension $K_\infty/K$.

In \cite{castellaJLMS, BLForum}, the authors descend to the anticyclotomic tower, treating the definite and indefinite cases simultaneously. The descent argument in both articles are modeled after the work of Agboola--Howard and Arnold in \cite{agboolahowardordinary, arnoldhigherweightanticyclo} that concern the case where $f$ itself is a CM form. In these works, the authors start off with the elliptic unit Euler system and Rubin's proof in \cite{rubinmainconj} of the main conjectures for ${{\rm GL}_1}_{/K}$. The key difference between the definite and indefinite scenarios turns out to be the $p$-semi-local position of the elliptic units relative to the $\p$-ordinary filtration; this is established using the Coates--Wiles explicit reciprocity laws for elliptic units. This beautiful observation forms the backbone of the descent argument in \cite{agboolahowardordinary, arnoldhigherweightanticyclo}. In \cite{castellaJLMS, BLForum}, elliptic units are replaced by the Beilinson--Flach elements and one determines the $p$-local position of the Beilinson--Flach elements relative to the $p$-ordinary filtration (whereby showing in the indefinite case that they fall within the Greenberg submodule determined by the $p$-ordinary filtration, whereas they are in general position in the definite case) relying on the reciprocity laws of Kings--Loeffler--Zerbes~\cite{KLZ2}. 

In \cite[\S5.3.1]{BL_SplitOrd2020}, we have simplified this descent procedure relying on the work of Nekov\'a\v{r}~\cite{nekovar06}. When $p$ is inert (as it is in the present work), this alteration is crucial for our purposes, since the descent to the anticyclotomic tower in \cite{castellaJLMS, BLForum}  fundamentally relies on the fact $p$ splits in $K/\QQ$.

We end our discussion \S\ref{subsubsec_1_1_3_1_2021_09_07} noting that the work of Wan~\cite{xinwanwanrankinselberg}, in some cases, complements the results of \cite{castellaJLMS,BLForum} to obtain a proof of Iwasawa main conjectures up to $\mu$-invariants.  

\subsubsection{Beilinson--Flach elements:  $v_p(\alpha)>0$}
\label{subsubsec_1_1_3_2_2021_09_07} Suppose that $f$ is non-ordinary at $p$, by which we mean that $\alpha$ is a non-unit, and $v_p(\alpha)<k_f+1$. The cyclotomic Iwasawa Theory of $f$ is very different from the ordinary counterpart. For example, the behaviour of the classical Bloch--Kato Selmer groups over the cyclotomic $\Zp$-extension is a manifestation of this difference: The direct limit of the discrete Bloch--Kato Selmer groups along the cyclotomic $\Zp$-extension of $\QQ$ is no longer cotorsion over the corresponding Iwasawa algebra. When $f$ corresponds to an elliptic curve $E_{/\QQ}$, Kobayashi \cite{kobayashi03} defined the so-called signed Selmer groups, which are cotorsion subgroups of the usual Selmer group. He further formulated a main conjecture relating these Selmer groups to Pollack's signed $p$-adic $L$-functions defined in \cite{pollack03}. One inclusion of the main conjecture has been shown to hold utilizing Beilinson--Kato's Euler system given in \cite{kato04}. Later, the hypothesis that $a_p(E)=0$ has been removed by Sprung \cite{sprung09}, allowing the treatment of small primes $p$. For the higher weight case, similar results have been obtained in \cite{lei09,LLZ0,leiloefflerzerbes11}, where the signed objects are defined using the theory of Wach modules in $p$-adic Hodge Theory. We note that the elimination of the condition that $a_p(f)=0$ (following Sprung's ideas) is particularly significant for the signed Iwasawa theory of higher weight forms.

Still under the hypothesis that $\chi$ is $p$-distinguished and that $p$ splits in the imaginary quadratic field $K$, the present authors \cite{BFSuper} studied the Iwasawa theory of $f\times \theta(\chi \psi)$ as  $\psi$ varies over Hecke characters whose $p$-adic Galois avatars factor through the $\Zp^2$-extension of $K$ and the anticyclotomic $\Zp$-extension of $K$. Note in particular that $\theta({\chi\psi})$ is always $p$-ordinary. Based on the work of Loeffler--Zerbes \cite{LZ0} on Wach modules over $\Zp^2$-extensions, we developed the theory of two-variable Coleman maps for modular forms at a non-ordinary prime. This allowed us to define two-variable signed Selmer groups and signed $p$-adic $L$-functions similar to the one-variable cyclotomic case studied by Kobayashi, Pollack et al. This generalizes the doubly signed  objects for elliptic curves that were previously studied by Kim  \cite{kimdoublysigned,kim11}, Loeffler  \cite{loeffler18}, Sprung \cite{sprung16} and the second-named author \cite{lei14}. We proved in \cite{BFSuper} one inclusion of the two-variable Iwasawa main conjecture using signed Euler systems of Beilinson--Flach elements. The starting point of our construction of these (integral) signed Euler systems was the works of Loeffler--Zerbes \cite{LZ1} and Kings--Loeffler--Zerbes \cite{KLZ2}. Recall that in \cite{LZ1}, the authors constructed families of Beilinson--Flach elements associated to the Rankin--Selberg convolution of two families of modular forms, where the families are parametrized by sufficiently small wide-open discs in the respective eigencurves. For our purposes in \cite{BFSuper}, where our sights were set to develop an integral theory, we slightly refined the construction of Beilinson--Flach elements drawing from \cite{KLZ2}, so that we can work over the full CM branch of the relevant Hida family (rather than a wide-open disc in the associated rigid-analytic space). We showed that the resulting classes are compatible with variation in both cyclotomic extensions and the tame conductor of $\psi$, where the latter property is crucial in the construction of norm-compatible classes over the ray class extensions of $K$. These classes, as in \cite{LZ1}, are still unbounded and therefore they are not apt for applications with the Euler-system machinery. Inspired by earlier works on one-variable cyclotomic signed Iwasawa theory \cite{kobayashi03,pollack03,sprung09,lei09,lei14,lei17,LLZ0,leiloefflerzerbes11,LLZ3}, we showed that these new two-variable classes can be decomposed into bounded classes by taking appropriate linear combinations, resulting in a pair of two-variable bounded Euler systems for $f_{/K}\otimes \chi$. This allowed us to apply the Euler-system machinery to deduce one inclusion of the signed main conjecture under a standard big image hypothesis (see Theorem~1.3 of op. cit. for the precise statement).

In the definite case, we may specialize to the anticyclotomic line to obtain one inclusion of the (one-variable) anticyclotomic counterpart of the signed Iwasawa main conjectures (see Theorem~1.5 of op. cit.). In the indefinite case, similar to the ordinary case, assuming $a_p(f)=0$, we showed that the signed anticyclotomic Selmer groups  have corank one over the corresponding Iwasawa algebra (see Theorem 1.8 of op. cit.). More generally, we showed that the anticyclotomic Pottharst-style Selmer groups of $f\times \theta(\chi)$ have rank one over an affinoid disc in the anticyclotomic weight space and proved that their torsion submodules are related to the derivatives of the corresponding two-variable $p$-adic $L$-functions in the cyclotomic direction, and that one inclusion in the BDP-type main conjecture holds (labeled Theorems~1.12 and 4.19 respectively in op. cit.), generalizing the results on the ordinary case  discussed in \S\ref{subsubsec_1_1_1_2_2021_09_07} and \S\ref{subsubsec_1_1_1_3_2021_09_07} above to the non-ordinary case. When the weight of $f$ is two, similar problems were treated in \cite{sprung16,CCSS,castellawan1607,wan16}.

In the present manuscript, we study the case where $f$ is ordinary at $p$ and $p$ is inert in $K$. In particular, the CM form $\theta(\chi\psi)$ is now non-ordinary at $p$. 
This essentially swaps the roles played by $f$ and $\theta(\chi\psi)$ in \cite{BFSuper} since the reduction types of the two modular forms are interchanged. {However, there are some diverging aspects that we would like emphasize here:
\begin{itemize}
    \item In the setting of \cite{BFSuper}, both families involved in the construction of Euler systems (namely, the Coleman family through our $p$-stabilized eigenform $f_\alpha$ 
    and the CM family that parametrizes the $p$-ordinary $p$-stabilizations of $\theta(\chi\psi)$) have bounded slope. In the setting of the present manuscript, this is no longer the case: Even though the family through $f^\alpha$ has slope-zero, the slopes of (either one of) the $p$-stabilizations of the $p$-non-ordinary CM forms $\theta(\chi\psi)$ grows linearly with weight. In other words, the locus of the universal deformation space that is determined by the collection $\{\theta(\chi\psi)\}_{\psi}$ as $\psi$ varies does not lift to a connected component of the eigencurve.
    \item Unlike \cite{BFSuper}, where we had constructed a pair of 2-variable Coleman maps for $f_\alpha\times \theta(\chi\psi)$, we define four one-variable cyclotomic Coleman maps for $f_\alpha\times \theta(\chi\psi)$ here. We further show in the present manuscript that two of these  maps can be deformed into two-variable maps as $f_\alpha$ varies in a Hida family $\f$. The extra  variable coming from the weight space of the Hida family  allows us to vary in our main result the Hecke character $\psi$ over those which are crystalline at $p$. This variation is crucial in our main objectives, which is to study of Iwasawa main conjectures for $\f_{/K}\otimes\chi$ over the $\Zp$-power extensions of $K$. To this end, it is of utmost importance that we keep track of delicate integrality properties of the Coleman maps as both $f_\alpha$ and $\psi$ vary.
    \item Our results in the present manuscript establish main conjectures for a $3$-parameter family. It is unclear to us how the methods of \cite{BFSuper} would lead to a similar result concerning the $3$-parameter family involved in op. cit. in an obvious way, which would amount to an interpolation of the main results in \cite{BFSuper} as the non-ordinary form $f^\alpha$ varies over the eigencurve.
    In the present work, it is in fact crucial that we work towards a main conjecture in $3$-variables in the first place: Variation in the Hida family $\f$ allows us to vary our results in $\psi$ (to ensure that the weight of $\theta(\chi\psi)$ is being dominated), and these results for each individual $\psi$ can then be packaged together within the universal deformation space of $\theta(\chi)$, over a locus that is naturally parametrized by the $\ZZ_p^2$-extension of $K$.
    \item In \cite{BFSuper}, the proofs of the results   towards the indefinite anticyclotomic conjecture, the BDP-type main conjectures played a crucial role. We remark that even to formulate a BDP-type main conjecture (e.g. to define the relevant Selmer group), we need to assume that the prime $p$ splits in $K/\QQ$. In other words, when $p$ is inert in $K$ (as is the case in the present work), the methods of \cite{BFSuper} to address indefinite anticyclotomic conjectures fall apart. 
\end{itemize}
}
\subsection{Iwasawa theory for the $\ZZ_p^2$-extension of $K$ ($\chi\neq \chi^c$, $v_p(\alpha)=0$ and $p$ is split)} 
\label{subsec_1_1_4_2021_09_07}
In \cite{BL_SplitOrd2020}, the present authors relaxed the $p$-distinguished condition imposed in \cite{BLForum}. Instead, we assume the weaker condition that $\chi\ne \chi^c$. Without the $p$-distinguished condition, the methods of \cite{LLZ2,BLForum} do not permit a construction of a full-fledged two-variable Euler system over $K$ associated to $f_{/K}\otimes\chi$ out of Beilinson--Flach elements. In \cite{BL_SplitOrd2020}, we were therefore forced to work with the Beilinson--Flach Euler system over $\QQ$ associated to $f_\alpha\times \theta(\chi\psi)$, where $\psi$ is a Hecke character of $K$ that is $p$-crystalline. We note that there are two potential difficulties in this setting:
\begin{itemize}
    \item The image of the associated Galois representations are much smaller as compared to the generic case (of the Rankin--Selberg convolutions of two non-CM and non-conjugate dual forms). This problem is also relevant to the present manuscript and we have proved in Appendix~\ref{appendix_big_images} that there is always an element in the image of the Galois representation of the sort required by the Euler system machinery.
    \item More importantly, in the setting of \cite{BL_SplitOrd2020}, where it may be possible that $\overline{\chi}=\overline{\chi}^c$, the residual representations attached to the Rankin--Selberg products $f_\alpha\times \theta(\chi\psi)$ are reducible (as $f_\alpha$ varies in a non-CM Hida family and $\theta(\chi\psi)$ varies in a CM Hida family). The bulk of \cite{BL_SplitOrd2020} is dedicated to resolve this technical obstacle, where we developed a locally restricted Euler system machinery in the residually reducible scenario. Such a consideration is not required in the present work. In this respect, among others that we highlight in the remainder of \S\ref{subsec_1_1_4_2021_09_07}, the methods of \cite{BL_SplitOrd2020} are significantly different from the present work.
\end{itemize}

In \cite{BL_SplitOrd2020}, there are two threads that we consider, that are a priori independent:  These two cases are determined by the condition that the weight of $f_\alpha$ is greater than that of $\theta(\chi\psi)$, or otherwise. In each of these two cases, there is a natural Selmer condition resulting in an  extended Selmer group over $K$, in the sense of Nekov\'a\v{r}. Using the locally restricted Euler-system machinery (that we extended in op. cit. to cover the residually reducible scenario) with the Beilinson--Flach elements as an input, we proved that the sizes of both Selmer groups are bounded by the indices of the Beilinson--Flach elements inside the corresponding local cohomology group with coefficients in the appropriate subquotients of the local Galois representation (c.f. Theorem~1.12  of op. cit.).

Our next task in \cite{BL_SplitOrd2020} was to show that the aforementioned bounds can be ``patched together'' to yield one inclusion in the two-variable and three-variable Iwasawa main conjectures (as described in Theorem~1.14 in op. cit.), as $f_\alpha$ varies in a Hida family $\f$ and $\psi$ varies over a set of Hecke characters, whose associated Galois characters form a dense subset of the set of characters of the Galois group of the $\Zp^2$-extension of $K$. Relying on Nekov\'a\v{r}'s descent formalism in both the definite and indefinite cases, we were able to prove one inclusion in the one- and two-variable anticyclotomic Iwasawa main conjectures of  $f_{/K}\otimes \chi$ and  $\f_{/K}\otimes\chi$ (see Theorems 1.16 and 1.17 of op. cit.). These results simultaneously generalize previous works we have summarized in \S\S\ref{subsec_1_1_1_2021_09_07}-\ref{subsec_1_1_3_2021_09_07} above.

As we have noted above, there are two different threads one can follow, asking that the $p$-adic $L$-functions and the Selmer groups in question verify an interpolation property (in terms of the $L$-values and the Bloch--Kato Selmer groups, respectively) over the locus of the weight space where specializations of $\f$ dominates over that of the CM family, or vice-versa. It turns out that the Iwasawa main conjectures concerning these two separate threads are equivalent to one another. In \cite{BL_SplitOrd2020}, we frequently made use of this equivalence to maneuver from one main conjecture to another (see also our discussion in the final paragraph of \S\ref{subsubsec_1_1_3_2_2021_09_07} related to this point, in relation to our results in \cite{BFSuper} towards the indefinite anticyclotomic Iwasawa theory). This duology is absent from the present work, highlighting another important difference between the current work and \cite{BL_SplitOrd2020}:  Since there is no connected component of the eigencurve that interpolates (the $p$-stabilizations of) the family $\{\theta(\chi\psi)\}_{\psi}$ of CM forms of positive slope, the associated Galois representation does not automatically admit a triangulation. This is one of the underlying challenges in the scenario when $p$ is inert in the imaginary quadratic field:  We presently do not have a candidate for a Selmer group which interpolates the Bloch--Kato Selmer groups for $f_\alpha\times \theta(\chi\psi)$ as $f_\alpha$ and $\psi$ vary with the restriction that the weight of $f_\alpha$ is smaller than that of $\theta(\chi\psi)$.

We remark that the main Iwasawa theoretic results in \cite{BL_SplitOrd2020} (c.f. Theorems 6.21, 6.37 and 6.38 in op. cit.) have an error term denoted by $\mathscr{C}(u)$, which arises from the  comparison of the natural lattices attached to $f_{/K}\otimes\chi\psi$ and $f\times\theta(\chi\psi)$ via Shapiro's lemma (see in particular the discussion in \S2.2 of op. cit.). It turns out that the error term $\mathscr{C}(u)$ is trivial if $\overline{\chi}\neq \overline{\chi}^c$, c.f. \cite[Proposition 2.19]{BL_SplitOrd2020}. This error term is not present in the current work since our running assumptions guarantee that $\overline{\chi}\neq \overline{\chi}^c$.

In the transition from the upper bounds of Selmer groups to inclusions of main conjectures in \cite{BL_SplitOrd2020}, the variation in the family $\f$ plays a crucial role, since it allows us to vary the weight of $\theta(\chi\psi)$ in a sufficiently large subset of the weight space. One of the key ingredients of the patching argument presented in \cite{BL_SplitOrd2020} is the control of the error term $t$ in Theorem~1.12 op. cit., independently of the choice of $f_\alpha$ and $\psi$ (with $t=0$ if $\overline{\chi}\ne\overline{\chi}^c$, i.e. when the Galois representation attached to $f_\alpha\times \theta(\chi\psi)$ is residually absolutely irreducible). This control on $t$ was achieved thanks to the locally restricted Euler system we have developed in \cite[Appendix A]{BL_SplitOrd2020} in the residually reducible scenarios.

While we employ a patching argument in both \cite{BL_SplitOrd2020} and the present manuscript, one of the most important diverging points of these two works is the reduction type of the CM form $\theta(\chi\psi)$, which is a consequence of the splitting behaviour of the prime $p$ in the imaginary quadratic field $K$. The fact that $p$ is assumed to be inert in $K$ in the current manuscript means that $\theta(\chi\psi)$ is non-ordinary at $p$ (as a matter of fact, things are even worse: the slopes of the corresponding $p$-stabilized eigenforms are unbounded as $\psi$ varies), which requires us to develop a construction of  integral Coleman maps with optimal integrality. We may interpolate these maps as $f_\alpha$ varies in a Hida family and $\psi$ varies over a set of Hecke characters, whose associated Galois characters form a dense subset of the set of characters of the Galois group of the $\Zp^2$-extension of $K$.

In the present work, the challenge concerning the Euler system argument is of different nature as compared to \cite{BL_SplitOrd2020}, and the difficulty is akin to our previous work \cite{BFSuper} described in \S\ref{subsubsec_1_1_3_2_2021_09_07} above: In this paper, we work with a family of semi-ordinary Rankin--Selberg products (of $p$-ordinary eigenforms $f_\alpha$ which vary in a Hida family and the non-ordinary CM form $\theta(\chi\psi)$ and as a result, the Beilinson--Flach elements in this setting have denominators. To run the Euler system machinery obtain bounds on the Selmer groups, one needs an (integral) Euler system, as optimal as possible so that one can patch these bounds as $\psi$ varies. Note that this is significantly different from \cite{BFSuper} in two ways: 
\begin{itemize}
    \item In op. cit., the non-CM form $f_\alpha$ was $p$-ordinary, whereas the CM forms varied in a $p$-ordinary Hida family. This allowed us to resort to the ideas of \cite{LLZ2} and construct an Euler system for $f_\alpha{}_{/K}\otimes \chi$ over $K$, by patching Hecke algebras of different levels. This is not possible in the current setting, since the CM forms in question cannot be interpolated in the eigencurve.
    \item In \cite{BFSuper}, for the reasons outlined in the previous paragraph, we did not have to vary the non-ordinary factor in a family to obtain results on the Iwasawa theory of  $f_\alpha{}_{/K}\otimes \chi$ over the $\Zp^2$-extension of $K$. 
\end{itemize}
The optimal construction of the integral $2$-parameter family of (signed) Beilinson--Flach elements (where one of the variables parametrizes the cyclotomic variation and the other parametetrizes the variation in $f_\alpha$ in a  Hida family) decomposing those non-integral classes given by the work of Loeffler and Zerbes~\cite{LZ1}, is one of the key novel constructions in this manuscript.  The integrality of the Coleman maps in the weight variable of the Hida family we mentioned above plays an indispensable role in our construction of integral  Euler systems (see \S\ref{sec:strategy} below for a more detailed discussion). We note that in \cite{BL_SplitOrd2020}, where both families in question were $p$-ordinary and as a result, the Euler system of Kings--Loeffler--Zerbes~\cite{KLZ2} is readily integral, and such considerations were not necessary. 
\vfill

\pagebreak

\begin{table}[h!]
\centering
\resizebox{\textwidth}{!}{
\begin{tabular}{|| m{8em}|| c c c c c c c c | m{4em} ||} 
 \hline
 &  &  &  &  &  &  &  &  &\\ 
{\bf Assumptions} & \ref{subsubsec_1_1_1_1_2021_09_07}& \ref{subsubsec_1_1_1_2_2021_09_07} & \ref{subsubsec_1_1_1_3_2021_09_07} 
  & \ref{subsubsec_1_1_2_1_2021_09_07}
  & \ref{subsubsec_1_1_2_2_2021_09_07}
  & \ref{subsubsec_1_1_3_1_2021_09_07}
  &\ref{subsubsec_1_1_3_2_2021_09_07}
  & \ref{subsec_1_1_4_2021_09_07}
  &  {\bf This paper}
  \\ [3.5ex] 
 \hline\hline
 &  &  &  &  &  &  &  & & \\
 {$p=\p$ is inert} &\checkmark  & $\bigtimes$ & $\bigtimes$ & $\bigtimes$ & $\bigtimes$ & $\bigtimes$ & $\bigtimes$ & $\bigtimes$ & \quad\,\checkmark \\[2.5ex]  
$p=\p\p^c$ is split & \checkmark & \checkmark & \checkmark & \checkmark & \checkmark & \checkmark & \checkmark & \checkmark & \quad\,$\bigtimes$\\[2.5ex]
$v_p(\alpha)=0$ & \checkmark & \checkmark & \checkmark & \checkmark & \checkmark & \checkmark & $\bigtimes$ & \checkmark & \quad\,\checkmark\\[2.5ex]
$v_p(\alpha)>0$  &\halfcheckmark & $\bigtimes$ &  \checkmark & $\bigtimes$ & $\bigtimes$ & $\bigtimes$ &  \checkmark & $\bigtimes$ &\quad\,$\bigtimes$\\[2.5ex]
$\chi=\mathds{1}$ &  \checkmark &  \checkmark &  $\bigtimes$ &  \checkmark &  \checkmark & $\bigtimes$ & $\bigtimes$ & $\bigtimes$ & \quad\,$\bigtimes$\\ [2.5ex]
$\chi\neq \chi^c$ (non-Eis.) & $\bigtimes$ & $\bigtimes$ & \checkmark & $\bigtimes$ & $\bigtimes$ & $\bigtimes$ & $\bigtimes$ & \checkmark &\quad\,\checkmark\\[2.5ex]
$\chi(\p)\not\equiv \chi(\p^c)$ \,\,\,\,\, ($p$-distinguished) & $\bigtimes$  & $\bigtimes$  & \checkmark & $\bigtimes$  & $\bigtimes$  & \checkmark & \checkmark & \checkmark & \quad\,$\bigtimes$ \\
[3.5ex]
 \hline
 \hline
  &  &  &  &  &  &  &  &&\\
 {\bf Results} &  &  &  &  &  &  &  &  &\\ 
 [2.5ex] 
 \hline
 \hline
 &  &  &  &  &  &  &  & &\\
 Bloch-Kato\quad 
 Conjectures & \checkmark & \checkmark & \checkmark & \checkmark & \checkmark & \checkmark & \checkmark & \checkmark & \quad\,\checkmark
 \\ [2.5ex] 
 \hline
 &  &  &  &  &  &  &  & &\\
 Definite
 anticyc.
 Iwasawa Theory & \checkmark & $\bigtimes$  & \checkmark  & \checkmark & \checkmark & \checkmark & \checkmark & \checkmark & \quad\,\checkmark
 \\ [2.5ex] 
 \hline
 &  &  &  &  &  &  &  & &\\
 Indefinite anticyc.
 Iwasawa Theory  & $\bigtimes$  & \checkmark & \checkmark  & $\bigtimes$  & \checkmark & \checkmark & \checkmark & \checkmark & \quad\,\checkmark
  \\ [2.5ex] 
 \hline
 &  &  &  &  &  &  &  & &\\
 Iwasawa Th. for the $\ZZ_p^2$-extension & $\bigtimes$  & $\bigtimes$  & $\bigtimes$  & \checkmark & \checkmark & \checkmark & \checkmark & \checkmark & \quad\,\checkmark \\ [2.5ex] 
 \hline
 &  &  &  &  &  &  &  & &\\
 Variation in $f_\alpha$ & $\bigtimes$  & \checkmark & $\bigtimes$  & \checkmark  & \checkmark & $\bigtimes$  & $\bigtimes$  & \checkmark & \quad\,\checkmark
 \\ [2.5ex] 
 \hline
 \hline
\end{tabular}
}
\caption{Summary of \S\ref{sec_previous_works} (Previous Works)}
\label{table_summary}
\end{table}
Note that the discussion in \S\ref{subsubsec_1_1_1_1_2021_09_07} that pertains to the case $v_p(\alpha)>0$ only covers the scenario where $a_p(f)=0$ or $p$ is split in $K$. That is why we put a half check mark ``\halfcheckmark'' (rather than the full check mark ``\checkmark'') as the entry in the first column of Table~\ref{table_summary} above corresponding to the case  $v_p(\alpha)> 0$.

\section{This manuscript: The strategy}\label{sec:strategy}

In this manuscript, we prove results towards a variety of Iwasawa main conjectures (\emph{with} $p$-adic $L$-functions) in the setting when the prime $p$ remains inert in $K/\QQ$, where there has been very limited progress as compared to when $p$ splits in $K/\QQ$: To the best of our knowledge, the works of Bertolini--Darmon and Chida--Hsieh that we have recalled in \S\ref{subsubsec_1_1_1_1_2021_09_07} are the only prior results in this direction (and they concern the ``definite'' anticyclotomic Iwasawa theory in the scenario when $\chi=\mathds{1}$). 

Before presenting our results in more precise form and reviewing our strategy, we recall the vital role  the hypotheses that $p$  splits in $K/\QQ$ has played in the works we have recorded in \S\ref{subsubsec_1_1_3_2_2021_09_07} and \S\ref{subsec_1_1_4_2021_09_07}, where analogous results to those in the current work have been obtained. We invite the reader to consult Table~\ref{table_summary}, where we summarize earlier results that concern the Bloch--Kato conjectures and Iwasawa theory for $f_{/K}\otimes \chi$, indicating the specific set of hypotheses required in each work. We recall that in this table, the labels in the first row indicate the section numbers where the corresponding result is summarized.

%When the prime $p$ splits in $K/\QQ$ and the Hecke character $\chi$ is $p$-distinguished, one may attach to the Rankin--Selberg product $f_{/K}\otimes\chi$ a full-fledged Euler system over all ray class extensions of $K$. This construction dwells on a ``patching'' argument that was first introduced and employed in \cite{LLZ2} when $f$ has weight $2$ and later extended in \cite{BLForum} to treat the case when $f$ is of higher weight. The hypotheses that $p$ splits in $K/\QQ$ and the Hecke character $\chi$ is $p$-distinguished are required to identify the irreducible component of the eigencurve that contains the $p$-ordinary stabilization of $\theta(\chi)$. Moreover, this irreducible component necessarily has CM by $K$ thanks to the $p$-distinguished hypothesis, by a result of Bella\"iche and Dimitrov, which interpolates the $p$-ordinary $p$-stabilizations of the theta-series of Hecke characters of $K$ with $p$-power conductor.

The challenge in the scenario when $p$ is inert in $K$ stems from the fact that the eigencurve cannot have components that have CM by the imaginary quadratic $K$; c.f.~\cite[Corollary 3.6]{CITJHC70}. Morally, this is due to the fact that slope of any non-ordinary CM form of weight $k$ is at least $(k-1)/2$, so the refinements of the corresponding theta-series cannot be contained in a finite slope family. We now explain (in very rough terms) how we circumvent this fundamental difficulty:
\begin{enumerate}[align=parleft, labelsep=.2cm,]
\item[\mylabel{item_intro_Step1}{{\bf Step~1}})]\ \  \ Extending \cite{BFSuper,BLLV}, we introduce Perrin-Riou functionals, the signed Coleman maps (in single cyclotomic variable) and signed Beilinson--Flach elements associated to Rankin--Selberg convolutions of the form $f_\alpha\otimes g$, where $f_\alpha$ has slope zero, whereas $g\in S_{k_g+2}(\Gamma_1(N_g),\varepsilon_g)$ is non-ordinary at $p$ with $p\nmid N_g$. We show that these constructions in fact interpolate as $f_\alpha$ varies in a Hida family $\f$. Our results in this direction are summarized in \S\ref{subsubsec_intro_optimizedPRandBF}. These constructions allow us to prove one divisibility in the Iwasawa--Greenberg main conjectures for $f_\alpha\otimes g$ as well as $\f\otimes g$ over the cyclotomic $\Zp$-extension of $\QQ$ (see Theorems~\ref{thm_cyclo_main_conj_fotimesg} and \ref{thm_cyclo_main_conj_ffotimesg} in the main body of the manuscript). 
\item[\mylabel{item_intro_Step2}{{\bf Step~2}})]\ \  \ \ We apply our constructions in \ref{item_intro_Step1} with the particular choice when $g$ is the theta-series of a Hecke character $\psi$ of our fixed imaginary quadratic field $K$. Employing the locally restricted Euler system machinery, we first prove results towards cyclotomic Iwasawa main conjectures. Moreover, our constructions of the Perrin-Riou functionals and signed Beilinson--Flach elements are integrally optimal and as such, they permit us to vary $f_\alpha$ in the Hida family $\f$ and to prove a result (Theorem~\ref{thm_cyclo_main_inert_ff_intro} below) towards two-variable main conjectures (where one of the variables parametrize $\f$ and the other accounts for the cyclotomic variation)  using a general patching criterion we establish in Appendix~\ref{Appendix_Regular_Rings_Divisibility}. 
\item[\mylabel{item_intro_Step3}{{\bf Step~3}})]\ \  \   We then vary $\psi$ to prove Theorem~\ref{thm_3var_main_inert_ff_intro}, which is a result towards  main conjectures in three variables (where one of the variables parametrize $\f$ and the other two the $\ZZ_p^2$-extension of $K$). We once again rely on the patching criterion we have noted in \ref{item_intro_Step2}. The integral optimality of the signed objects from \ref{item_intro_Step1} also plays a crucial role in this portion. We  remark that our results towards Iwasawa main conjectures for $f_\alpha\otimes g$ in \ref{item_intro_Step1} are restricted to the case where  $k_f>k_g$. Consequently, for a fixed $f_\alpha$, our results in \ref{item_intro_Step2} allow only the treatment of crystalline Hecke characters of fixed tame conductor and whose infinity types belong to a fixed finite set. Since  there are only finitely many such characters, this is not enough to obtain results on Iwasawa main conjectures over the $\ZZ_p^2$-extension of $K$. This problem is overcome by  varying $f_\alpha$ in a Hida family, which in turn permits us to vary $\psi$ in an infinite family. This is the main reason for our emphasis on the signed-splitting procedure for Beilinson--Flach elements for Hida families.
\item[\mylabel{item_intro_Step4}{{\bf Step~4}})]\,\,\,\,\,  We use Nekov\'a\v{r}'s (very general) descent formalism for his Selmer complexes to descend our results on the $\Zp^2$-extension of $K$ to the anticyclotomic $\Zp$-tower. This portion of our results is recorded as Theorem~\ref{thm_anticyc_main_inert_ff_definite_intro} (in the definite case) and Theorem~\ref{thm_anticyc_main_inert_ff_indefinite_intro} (in the indefinite case).
\end{enumerate}

Before introducing the necessary notation to state our main results in \S\ref{subsec_intro_summary_of_results}, we also note a separate thread of exciting progress in the Iwasawa theory over imaginary quadratic fields where the prime $p$ remains inert: Andreatta--Iovita in~\cite{AndreattaIovitaBDP} (also Kriz in \cite{Kriz2018}, using a different method) gave a construction of a Bertolini--Darmon--Prasanna-style (anticyclotomic) $p$-adic $L$-function, which interpolates the \emph{central critical values} of the Rankin--Selberg $L$-functions $L(f\times \theta(\psi),s)$ as the Hecke character $\psi$ varies. These $p$-adic $L$-functions are genuinely different from those we work with in the present manuscript, in that the $p$-adic $L$-functions of Andreatta--Iovita and Kriz interpolate along those $\psi$ where the (non-ordinary) theta-series $\theta(\psi)$ has higher weight than $f$  {(as opposed to those we work with in the present work, which concern the scenario where the $p$-ordinary form $f$ has higher weight)}. At present, we do not know how to formulate Iwasawa main conjectures for these $p$-adic $L$-functions, due to the absence of an appropriate triangulation on local cohomology at $p$ (which, as we have indicated above, is related with the non-existence of CM families with positive finite slope).

\section{The setting}
\label{subsec_setting} 
Throughout, we fix an odd prime $p$ and embeddings $\iota_\infty:\overline{\QQ}\hookrightarrow \mathbb{C}$\index{Embeddings $\iota_\infty,\iota_p$} and $\iota_p:\overline{\QQ}\hookrightarrow \mathbb{C}_p$. Let  $f$ and $g$ be two normalized, cuspidal eigen-newforms, of weights $k_f + 2, k_g + 2$, levels $N_f, N_g$, and nebentypes $\varepsilon_f$, $\varepsilon_g$ respectively. We assume $p \nmid N_f N_g$. In \cite{BLForum, BFSuper}, we have studied the Iwasawa theory for the Rankin--Selberg product of $f$ and a $p$-ordinary $g$  (with respect to the embeddings we fixed) as $g$ varies in a CM Hida family. 

In this manuscript, we assume that  $f$ is ordinary at $p$ and vary the ordinary $p$-stabilization of $f$ in a Hida family $\f$, but $g$ shall be non-ordinary at $p$. We give finer results than those given in \cite{BLForum, BFSuper}. We note that in~\cite{BLForum}, both $f$ and $g$ are assumed to be $p$-ordinary (see \S\ref{subsec_1_1_4_2021_09_07} for more details). So, our results in the present work do not overlap with those in op. cit. As discussed \S\ref{subsubsec_1_1_3_2_2021_09_07}, we studied  results towards Iwasawa main conjectures involving $p$-adic $L$-functions and Selmer groups for the Rankin--Selberg product of a non-$p$-ordinary eigenform and a CM Hida family in \cite{BFSuper}. While the results in the present work resemble those in op. cit., we are interested in objects for which the non-$p$-ordinary eigenform $g$ is dominated\footnote{By ``$g$ dominated by members of the family $\f$'', we mean that the $p$-adic $L$-function (resp., the Selmer group) involved interpolates the critical values of the Rankin--Selberg $L$-functions (resp., the Bloch--Kato Selmer groups) associated to $\f(\kappa)\otimes g$ where the classical specialization $\f(\kappa)$ has higher weight than $g$. Similarly, when the objects interpolate those where $g$ has higher weight than $f(\kappa)$, we say that $g$ dominates the members of $\f$.} by the members of the Hida family $\f$, whereas the main objects of interest in \cite{BFSuper} are those for which the non-$p$-ordinary eigenform dominates members of the Hida family.

Let $L/\Qp$ be a finite extension containing the coefficients of $f$ and $g$, the $N_fN_g$-th roots of unity as well as the roots of the Hecke polynomials of $f$ and $g$ at $p$. We shall write \index{$U_p$ eigenvalues $\alpha_?,\beta_?$} $\alpha_f$, $\beta_f$, $\alpha_g$ and $\beta_g$ for these roots; we assume throughout that $\alpha_f$ is a $p$-adic unit and $\alpha_g \ne \beta_g$. For $h\in\{f,g\}$ and $\lambda\in\{\alpha,\beta\}$, we write $h_\lambda$ for the $p$-stabilization of $h$ satisfying $U_ph_\lambda=\lambda_h h_\lambda$. Let $\f$ \index{Hida Theory! Non-CM Hida family $\f$} denote the Hida family passing through $f_\alpha$. The corresponding branch of the Hecke algebra is denoted by $\Lambda_\f$\index{Hida Theory! Non-CM branch $\LL_\f$}.

For each $h\in\{f,g\}$, let $V_h$ denote Deligne's $L$-linear continuous  cohomological $G_\QQ$-representation attached to $h$, so that its restriction to $G_{\QQ_p}$ has Hodge--Tate weights $0$ and $-1-k_h$ (with the convention that the Hodge--Tate weight of the cyclotomic character is $1$). 
Let $\cO$ denote the ring of integers of $L$.  {We fix a Galois-stable $\cO$-lattice $R_h$ inside $V_h$ so that its linear dual $R_h^*=\Hom(R_h,\cO)$ coincides with the Galois representation realized in the cohomology of the  the modular curve $Y_1(N_h)$ with coefficients in the integral sheaf $\TSym^{k_h}\mathscr{H}_{\Zp}(1)$ (see \cite[\S2.3]{KLZ2} and Definition~\ref{def_eigenvectors} below).}\index{Galois representations! Deligne's representations $R_h, R_h^*$}
We study the Iwasawa theory of $T_{f,g}:= R_f^*\otimes R_g^*=\Hom(R_f\otimes R_g,\cO)$ over the cyclotomic tower $\QQ(\mu_{p^\infty})$, as well as when $f$ and $g$ vary in a Hida family and CM forms over a fixed imaginary quadratic field where $p$ is inert respectively.

\section{Conjectures and Results}
\subsection{Main conjectures: The set up and statements}
Let us fix once and for all an imaginary quadratic field $K$\index{$K$: imaginary quadratic field} where the prime $p$ is inert. Let $D_K$ denote its discriminant. We also fix a ray class character $\chi$ of $K$ with conductor dividing $\ff p^\infty$ (where we assume that $\ff$ is coprime to $pD_K$) and order coprime to $p$. We will also work with an algebraic Hecke character\index{Hecke characters of $K$! $\psi$} $\psi$ with infinity type\footnote{The convention for the infinity types in the current manuscript is the same as the one utilized in \cite{bertolinidarmonprasanna13}. Namely, the infinity type of a Hecke character $\xi$ equals $(m,n)$ if it sends  $(\alpha)$ to $\alpha^m\overline\alpha^n$,  where $\alpha\in K$ is coprime to the conductor of $\xi$.} $(0,k_g+1)$ and conductor $\ff$\index{Hecke characters of $K$! $\ff$: conductor of $\psi$}, where $k_g$ is a positive integer. The character $\chi$ \index{Hecke characters of $K$! $\chi$ (branch character)} will be fixed throughout (and will take on the role of a branch character, in the sense of Hida) whereas $\psi$ will be allowed to vary. For any ray class character $\eta$ of $K$, we shall set $\eta^c:=\eta\circ c$,\index{Hecke characters of $K$! $\eta^c$} where $c$\index{$K$: imaginary quadratic field! $c$: generator of $\Gal(K/\QQ)$} is the generator of $\Gal(K/\QQ)$.

Let\index{$K$: imaginary quadratic field! $\Gamma_K,\Gamma_\cyc, \Gamma_\ac$} $\Gamma_\cyc=\Gal(\QQ(\mu_{p^\infty})/\QQ)\cong\Delta\times\Gamma_1$, where $\Delta\cong\ZZ/(p-1)\ZZ$ and $\Gamma_1\cong\Zp$\index{$K$: imaginary quadratic field! $\Gamma_1,\Delta$}. Let $K_\infty,K_\cyc$ and $K_\ac$ be the $\Zp^2$-extension, the cyclotomic $\Zp$-extension and the anticyclotomic $\Zp$-extension of $K$ respectively\index{$K$: imaginary quadratic field! $K_\infty,K_\cyc,K_\ac$}. We write $\Gamma_K=\Gal(K_\infty/K)$, $\Gamma_\ac=\Gal(K_\ac/K)$. We also identify $\Gamma_1$ and $\Gamma_\cyc$ with $\Gal(K_\cyc/K)$ and $\Gal(K(\mu_{p^\infty})/K)$ respectively.

For any profinite abelian group group $G$, we let\index{Completed group rings! $\LL(G), \LL_{\cO}(G)$} $\LL(G):=\ZZ_p[[G]]$ denote the completed group ring of $G$ with coefficients in $\ZZ_p$. We% let $\cO$ denote the ring of integers of a finite extension of $\QQ_p$ (which we shall enlarge as necessary throughout the article, e.g. to contain the values of $\chi$ and $\psi$) and
 put $\LL_{\cO}(G):=\LL(G)\otimes_{\ZZ_p}\cO$. We will mainly concern ourselves with the cases when $G=\Gamma_K, \Gamma_\cyc, \Gamma_1$ or $\Gamma_\ac$. Let us denote by \index{$K$: imaginary quadratic field! Universal characters $\Psi_?$}
$$\Psi:G_K\twoheadrightarrow\Gamma_K\xrightarrow{\gamma\mapsto \gamma^{-1}}\Gamma_K \hookrightarrow \LL_\cO(\Gamma_K)^\times$$ 
$$\Psi_{?}:G_K\twoheadrightarrow\Gamma_?\xrightarrow{\gamma\mapsto \gamma^{-1}}\Gamma_? \hookrightarrow \LL_\cO(\Gamma_1)^\times$$
the tautological characters, where $?=1,\ac$. We set $\bbchi:=\chi\Psi$ and similarly define the $\LL_\cO(\Gamma_1)$-valued character $\bbchi_\cyc:=\chi\Psi_1$ and the $\LL_\cO(\Gamma_\ac)$-valued character $\bbchi_\ac:=\chi\Psi_\ac$. \index{$K$: imaginary quadratic field! Universal characters $\bbchi_?$} We shall denote by \index{$K$: imaginary quadratic field! $\LL_\cO(\Gamma_?)^\iota$} $\LL_{\cO}(\Gamma_K)^{\iota}$ the free $\LL_{\cO}(\Gamma_K)$-module of rank one on which $G_K$ acts via $\Psi_K$ (and similarly, we define $\LL_{\cO}(\Gamma_?)^{\iota}$ for $?=1,\ac$). More generally, we put $M^\iota:=M\otimes_{\LL_{\cO}(\Gamma_?)} \LL_{\cO}(\Gamma_?)^{\iota}$\index{$K$: imaginary quadratic field! $M^\iota$ (for a $\LL(\Gamma_?)$-module $M$)} for any $\LL_{\cO}(\Gamma_?)$-module $M$ on which $G_K$ acts.

Throughout, we also fix a branch $\f\in \LL_\f[[q]]$\index{Hida Theory! Non-CM Hida family $\f$} of a primitive non-CM Hida family with tame conductor $N_f$\index{Hida Theory! Tame level $N_f$} and residually non-Eisenstein (in the sense that $\overline{\rho}_\f|_{G_{\QQ}}$ is absolutely irreducible)\index{Hida Theory! Residually non-Eisentein Hida family} and $p$-distinguished (in the sense that $\overline{\rho}_\f|_{G_{\Qp}}$ is not scalar)\index{Hida Theory! $p$-distinguished Hida family}, where $\LL_\f$ is an irreducible component of Hida's universal ordinary Hecke algebra (see Definition~\ref{defn_Hecke_algebra_intro} for further details)\index{Hida Theory! Non-CM branch $\LL_\f$}. We write $\kappa_0$ for the prime ideal of $\LL_\f$ which specializes $\f$ to $f$. We assume that $N_f$ is coprime to $D_K \mathbf{N}\ff$. We denote by $R_{\f}^*$\index{Hida Theory! Hida's Galois representation $R_\f^*$}  Hida's $G_\QQ$-representation attached to $\f$, which is  free of rank 2 over $\LL_\f$.  We will assume without loss of generality that $\LL_\f$ contains $\cO$ and put $\LL_\f(\Gamma_?):=\LL(\Gamma_?)\otimes_{\ZZ_p}\LL_\f$\index{Completed group rings! $\LL_\f(\Gamma_?)$}.

For any $\chi$ and $\psi$ as above, we define the $G_K$-representations\index{Galois representations! $\TT_{\f,\psi}^{\rm cyc}$}\index{Galois representations! $\TT_{\f,\psi}^{?}$, $?=K,\ac$}
$$\TT_{\f,\psi}^{\rm cyc}:=R_{\f}^*\otimes\psi\,\widehat{\otimes}\,\LL_{\cO}(\Gamma_1)^\iota \qquad \hbox{and } \qquad \TT_{\f,\chi}^{?}:=R_{\f}^*\otimes\chi\,\widehat{\otimes}\,\LL_{\cO}(\Gamma_?)^\iota$$ 
where $?=K,\ac$. We let $\widetilde{H}^2_{\rm f}(G_{K,\Sigma},\TT_{\f,\psi}^{\cyc};\Delta_{\Gr})$\index{Selmer complexes! Extended Greenberg Selmer group $\widetilde{H}^2_{\rm f}(G_{K,\Sigma},\TT_{\f,\psi}^{?};\Delta_{\Gr})$} denote the extended Greenberg Selmer group attached to the representation $\TT_{\f,\psi}^{\rm cyc}$ (c.f. Definition~\ref{defn_Selmer_complex_inert_ord}) and similarly define $\widetilde{H}^2_{\rm f}(G_{K,\Sigma},\TT_{\f,\chi}^{K};\Delta_{\Gr})$. Finally, we let 
$ L_p^{\rm RS}(\f_{/K}\otimes\psi)$ denote the Rankin--Selberg $p$-adic $L$-function given as in Definition~\ref{defn_geo_for_trivial_char}.\index{$p$-adic $L$-functions! $L_p^{\rm RS}(\f_{/K}\otimes\psi)$}

We are finally ready to state the first of a series of three conjectures, which concerns the cyclotomic Iwasawa theory of the family $\f\otimes \psi$.

\begin{conj}
\label{conjecture_cyclo_main_inert_ff_intro}
Suppose that $\overline{\rho}_\f$ is absolutely irreducible.  Then,
$$ L_p^{\rm RS}(\f_{/K}\otimes\psi)\cdot \LL_{\f}(\Gamma_1)\otimes_{\ZZ_p}\QQ_p=\Char_{\LL_\cO(\Gamma_1)}\left(\widetilde{H}^2_{\rm f}(G_{K,\Sigma},\TT_{\f,\psi}^{\cyc};\Delta_{\Gr})^\iota\right)\otimes_{\ZZ_p}\QQ_p\,.$$
\end{conj}

In Appendix~\ref{appendix_sec_padicRankinSelbergHida}, we introduce the ``semi-universal'' Rankin--Selberg $p$-adic $L$-function\index{$p$-adic $L$-functions! $L_p^{\rm RS}(\f_{/K}\otimes\bbchi)$ (``semi-universal'' Rankin--Selberg $p$-adic $L$-function)}
$$L_p^{\rm RS}(\f_{/K}\otimes\bbchi)\in \LL_\f(\Gamma_K)$$ in three variables, extending the work of Loeffler~\cite{Loeffler2020universalpadic} ever so slightly. With that, we can state the second conjecture, which concerns the $3$-parameter family $\f\otimes\bbchi$:

\begin{conj}[Main conjectures over $\LL_\f(\Gamma_K)$]
\label{conj_3var_main_inert_ff_intro}
Suppose $\chi$ is a ray class character as above and assume that $\overline{\rho}_\f$ is absolutely irreducible. Then
\[
    L_p^{\rm RS}(\f_{/K}\otimes\bbchi)\cdot \LL_\f(\Gamma_K)\otimes_{\ZZ_p}\QQ_p\,= \,\Char_{\LL_\f(\Gamma_K)}\left(\widetilde{H}^2_{\rm f}(G_{K,\Sigma},\TT_{\f,\chi}^{K};\Delta_{\Gr})^\iota\right)\otimes_{\ZZ_p}\QQ_p\,.
\]
\end{conj}

To state our third conjecture, which concerns the anticyclotomic Iwasawa theory of (families of) Rankin--Selberg products, we consider the restriction 
$$L_p^{\dagger}(\f_{/K}\otimes\bbchi_\ac)\in \LL_\f(\Gamma_\ac)$$ 
\index{$p$-adic $L$-functions! $L_p^{\dagger}(\f_{/K}\otimes\bbchi_\ac)$ (central critical Rankin--Selberg $p$-adic $L$-function)}of the twist $L_p^{\dagger}(\f_{/K}\otimes\bbchi)\in \LL_\f(\Gamma_K)$ of Loeffler's $p$-adic $L$-function to the anticyclotomic tower (c.f. Definition~\ref{defn_twff_map}(ii)). We also denote by $\TT_{\f,\bbchi_\ac}^\dagger$\index{Galois representations! $\TT_{\f,\bbchi_\ac}^\dagger$} the central critical twist of $\TT_{\f,\chi}^\ac$, which we introduce in Definition~\ref{defn_twff_map}(i). 

We assume in this portion that $\chi$ is a ring class character (so that $\chi^c=\chi^{-1}$) and that the nebentype character $\varepsilon_f$ of the Hida family $\f$ is trivial. As it is customary in the study of anticyclotomic Iwasawa theory, we shall write $N_f=N^-N^+$, where $N^+$ (resp. $N^-$) is a product of primes which are split (resp. inert) in $K/\QQ$.\index{$K$: imaginary quadratic field! $N^+,N^-$}

We denote by $\partial_\cyc\, L_p^{\dagger}(\f_{/K}\otimes\bbchi) \in \LL_\f(\Gamma_\ac)$\index{$p$-adic $L$-functions! $\partial_\cyc\, L_p^{\dagger}(\f_{/K}\otimes\bbchi)$} the restriction (to the anticyclotomic tower) of the derivative of $L_p^{\dagger}(\f_{/K}\otimes\bbchi)$ in the cyclotomic direction (c.f. Definition~\ref{defn_derivatives}). We set $$r(\f,\bbchi_\ac):={\rm rank}_{\LL_{\f}(\Gamma_\ac)}\,\widetilde{H}^2_{\rm f}(G_{K,\Sigma},\TT_{\f,\bbchi_\ac}^{\dagger};\Delta_{\Gr})$$\index{Selmer complexes! $r(\f,\bbchi_\ac)$ (generic algebraic rank)}
and call it the generic algebraic rank. Finally, we let ${\rm Reg}_{\f,\bbchi_\ac}\subset \LL_\f(\Gamma_\ac)$\index{Selmer complexes! ${\rm Reg}_{\f,\bbchi_\ac}$ ($\LL_\f(\Gamma_\ac)$-adic regulator)} denote the $\LL_\f(\Gamma_\ac)$-adic regulator given as in Definition~\ref{defn_height_regulators_Nek}. 

\begin{conj}[Anticyclotomic Main conjectures]
\label{conj_anticyclo_main_inert_ff_intro}
Suppose $\chi$ is a ring class character and $\overline{\rho}_\f$ is absolutely irreducible. Then,
\begin{equation*}
\label{conj_anticyclo_main_inert_ff_intro_0}
    L_p^{\dagger}(\f_{/K}\otimes\bbchi_\ac)\cdot\LL_\f(\Gamma_\ac)\otimes_{\ZZ_p}\QQ_p\,= \,\Char_{\LL_\f(\Gamma_\ac)}\left(\widetilde{H}^2_{\rm f}(G_{K,\Sigma},\TT_{\f,\bbchi_\ac}^{\dagger};\Delta_{\Gr})^\iota\right)\otimes_{\ZZ_p}\QQ_p\,.
\end{equation*}
\item[i)] \emph{(Definite case)} Suppose $N^-$ is a square-free product of odd number of primes. Then $L_p^{\dagger}(\f_{/K}\otimes\bbchi_\ac)\neq 0$; in particular, the $\LL_\f(\Gamma_\ac)$-module $\widetilde{H}^2_{\rm f}(G_{K,\Sigma},\TT_{\f,\bbchi_\ac}^{\dagger};\Delta_{\Gr})$ is torsion.
\item[ii)] \emph{(Indefinite case)} Suppose $N^-$ is a square-free product of even number of primes. Then $L_p^{\dagger}(\f_{/K}\otimes\bbchi_\ac)=0$ and 
$$r(\f,\bbchi_\ac)=1={\rm ord}_{\gamma_+-1}\, L_p^{\dagger}(\f_{/K}\otimes\bbchi)\,.$$
Moreover, 
\begin{align*} 
\partial_\cyc\, L_p^{\dagger}(\f_{/K}\otimes\bbchi)\cdot\LL_\f(\Gamma_\ac)&\,\otimes_{\ZZ_p}\QQ_p\,=\\
& \,{\rm Reg}_{\f,\bbchi_\ac}\cdot\Char_{\LL_\f(\Gamma_\ac)}\left(\widetilde{H}^2_{\rm f}(G_{K,\Sigma},\TT_{\f,\bbchi_\ac}^{\dagger};\Delta_{\Gr})_{\rm tor}^\iota\right)\otimes_{\ZZ_p}\QQ_p\,.
\end{align*}
\end{conj}

Despite the wealth of results in the setting when $p$ is split in $K/\QQ$, the evidence for these conjectures when $p$ is inert have been extremely scant and  are limited to the ``definite'' anticyclotomic Iwasawa theory of Rankin--Selberg products of the form $f_{/K}\otimes \chi$, where $\chi$ has finite order (c.f., \cite{BertoliniDarmon2005,chidahsiehanticyclomainconjformodformscomposito}). As far as we are aware, there are no previous results in the indefinite case when $p$ remains inert, nor  results which allow variation in $f$, nor in the case when $\psi$ has infinite order. To the best of our knowledge, the only results on the variation of certain Iwasawa invariants as $f$ varies are due to Emerton--Pollack--Weston \cite{EPW}, Kim \cite{kim17} and Castella--Kim--Longo \cite{CKL}. More speicially, in the definite case, if $p$ splits in $K$ and $f$ is $p$-ordinary, then one may relate  the $\mu$- and $\lambda$-invariants of both the $p$-adic $L$-function and the Selmer group of $f$ over the anticyclotomic $\Zp$-extension of $K$ as $f$ varies in a Hida family. However, the aforementioned results do not have an immediate bearing on the one-sided divisibilities in the Iwasawa main conjectures that we establish here.

 Our treatment in this manuscript, which builds on fundamentally different techniques from those utilized in op. cit., will hopefully shed some light to these cases of main conjectures, which previously seemed inaccessible. We illustrate our results towards these conjectures in Theorems~\ref{thm_cyclo_main_inert_ff_intro} (cyclotomic case), \ref{thm_3var_main_inert_ff_intro} ($3$-variable case), \ref{thm_anticyc_main_inert_ff_definite_intro} (definite anticyclotomic  case) and \ref{thm_anticyc_main_inert_ff_indefinite_intro} (indefinite anticyclotomic case) below.

%Before we move on to discuss our results in the present article, we underline the fundamental difference between the two cases when $p$ is inert vs split in $K/\QQ$. In the split case, one makes crucial use of Hida families that has CM by $K$ to interpolate the Beilinson--Flach elements to obtain a collection of cohomology classes over all ray class extensions of $K$; c.f. \cite{LLZ2, BLForum, BFSuper}. In contrast, the eigencurve cannot have components that have CM by $K$ when $p$ is inert; c.f.~\cite[Corollary 3.6]{CITJHC70}. 

\section{A summary of our results}
\label{subsec_intro_summary_of_results}
We will discuss our results in two parts: In \S\ref{subsubsec_intro_optimizedPRandBF}, we summarize our results concerning the (optimized) signed splitting of Beilinson--Flach elements and Perrin-Riou maps. In the second part (\S\ref{subsubsec_results_intro}), we illustrate applications of these results towards Conjectures~\ref{conjecture_cyclo_main_inert_ff_intro}, \ref{conj_3var_main_inert_ff_intro} and \ref{conj_anticyclo_main_inert_ff_intro}.  
\subsection{Optimized Perrin-Riou maps and Beilinson--Flach elements for semi-ordinary families}
\label{subsubsec_intro_optimizedPRandBF}
The first result we record in the introduction concerns the construction of certain Perrin-Riou functionals, which are eigen-projections of two-variate Perrin-Riou maps. For each choice of $\lambda,\mu\in\{\alpha,\beta\}$, we choose a $\vp$-eigenvector of $\Dcris(T_{f,g})\otimes_\cO L$ associated to the eigenvalue $(\lambda_f\mu_g)^{-1}$. Given a finite unramified extension $F/\Qp$, we write 
$$\cL_{F,f,g}^{(\lambda,\mu)}:\HIw(F(\mu_{p^\infty}),T_{f,g})\rightarrow \cH_{\ord_p(\lambda_f\mu_g)}(\Gamma_\cyc)\otimes F$$ 
\index{Perrin-Riou maps! $\cL_{F,f,g}^{(\lambda,\mu)}$} for the Perrin-Riou map on $\HIw(F(\mu_{p^\infty}),T_{f,g})$ projecting to the chosen  $\vp$-eigenvector.

We take additional care with the choices of our $\vp$-eigenvectors to normalize these functionals, which in turn allow us to keep track of the crucial integrality properties of the functionals (with respect to the natural lattices inside the Galois representations we work with). We also note that we are not directly working with the affinoid-valued Perrin-Riou maps of \cite{LZ1}, but rather introduce $\Lambda_\f$-adic maps (as part of our proof of Theorem~\ref{thm:PRHida} below).% which take values in rings that are amenable to the signed factorization procedure.
\begin{theorem}[Theorem~\ref{thm:PRHida}, optimized Perrin-Riou functionals for semi-ordinary families]
\label{thm_PRHida_intro}
 For $\mu\in\{\alpha,\beta\}$ and a finite unramified extension $F/\Qp$,  there exists a $\Lambda_\f(\Gamma_\cyc)$-morphism
\[
\cL_{F,\f,g, \mu}:\HIw(F(\mu_{p^\infty}),T_{\f,g})\lra \Lambda_\f\,\widehat{\otimes}\, \cH_{\ord_p(\mu_g)}(\Gamma_\cyc)\otimes F
\]
\index{Perrin-Riou maps! $\cL_{F,\f,g, \mu}$}
whose specialization at $\kappa_0$ equals, up to a $p$-adic unit,  $$\frac{1}{\lambda_{N_f}(f)\left(1-\frac{\beta_f}{p\alpha_f}\right)\left(1-\frac{\beta_f}{\alpha_f}\right)}\cL_{F,f_\alpha,g}^{(\alpha,\mu)},$$ where  $\cL_{F,f_\alpha,g}^{(\alpha,\mu)}$\index{Atkin--Lehner pseudo-eigenvalue $\lambda_{N_f}(f)$} is  the $p$-stabilization of $\cL_{F,f,g}^{(\alpha,\mu)}$ (see Definition~\ref{defn:stabliziedmaps} below) and  $\lambda_{N_f}(f)$ is the pseudo-eigenvalue of the Atkin--Lehner operator of level $N_f$ (which is given by the identity $W_{N_f}f=\lambda_{N_f}(f)f^*$).  
\end{theorem}

We may then factorize these multivariate Perrin-Riou functionals into doubly-signed $\LL_\f(\Gamma_\cyc)$-adic Coleman maps. 

\begin{theorem}[Theorem~\ref{thm_decompose_PR_SSf}, optimized factorization of Perrin-Riou maps for semi-ordinary families]
\label{thm_decompose_PR_SSf_intro}
Suppose that $g$ satisfies  either the Fontaine-Laffaille condition $p>k_g+1$ or  $a_p(g)=0$. There exist a logarithmic matrix\index{Coleman maps! $Q_g^{-1}M_g$ (Logarithmic Matrix)}
$$Q_g^{-1}M_g\in M_{2\times2}(\cH(\Gamma_1))$$ 
and a pair of $\Lambda_\f(\Gamma_\cyc)$-morphisms
\[
\col_{F,\f,g,\#},\ \col_{F,\f,g,\flat}:\HIw(F(\mu_{p^\infty}),T_{\f,g})\lra \Lambda_\f(\Gamma_\cyc)\otimes\cO_F
\]
that verify the factorization
\[
\begin{pmatrix}
\cL_{F,\f,g, \alpha}\\\\ \cL_{F,\f,g, \beta}
\end{pmatrix}=Q_g^{-1}M_g\begin{pmatrix}
\col_{F,\f,g,\#}\\\\ \col_{F,\f,g,\flat}
\end{pmatrix}.
\] 
\index{Coleman maps! $\col_{F,\f,g,\bullet}$}
\end{theorem}

Note that in the case where $a_p(g)=0$, the matrix $Q_g^{-1}M_g$ is essentially given by Pollack's half logarithms introduced in \cite{pollack03}, which is the reason why we call  $Q_g^{-1}M_g$ a logarithmic matrix. Note that in the main body of the manuscript, we deal with the two hypotheses on $g$ separately. In the case where $a_p(g)=0$, the logarithmic matrix  is denoted by $Q_g^{-1}M_g'$ instead of $Q_g^{-1}M_g$ and the symbols $\#$ and $\flat$ are replaced by $+$ and $-$ respectively. See also Proposition~\ref{prop:PRpm} for a variant of Theorem~\ref{thm_decompose_PR_SSf_intro}, where we concern ourselves only with individual Rankin--Selberg products. We further remark that the factorization result in Theorem~\ref{thm_decompose_PR_SSf_intro} is akin to that in \cite[Theorem 1.1]{BFSuper}, where a CM branch of a Hida family has played the role of $\f$.

We next turn our attention to distribution-valued (non-integral) Beilinson--Flach elements, with the aim to produce their integral counterparts, so as to be able to employ with them the Euler system argument.  Perrin-Riou's philosophy of higher rank Euler systems leads one to predict that these (non-integral) Beilinson--Flach elements also admit a signed factorization (c.f. the discussion in \S\ref{subsubsec_ESrank2} which concerns this point of view), to give rise to signed Beilinson--Flach elements with better integrality properties. We verify this prediction in Theorem~\ref{thm_BF_factorization_for_families_intro} below.  This result can be regarded as a refinement (and an extension of, in that it also allows variation in families) of \cite{BLLV} in this semi-ordinary setting. 

We further note that the non-integral Beilinson--Flach elements are not exactly those constructed in \cite{LZ1}, but rather a slight variant as per the coefficients of the modules where they take values in (c.f. Definition~\ref{defn:BF-Hida} below). This slight alteration is necessary in order to apply our signed factorization procedure.

\begin{theorem}[Theorem~\ref{thm_BF_factorization_for_families}, optimized factorization of Beilinson--Flach elements for semi-ordinary families]
\label{thm_BF_factorization_for_families_intro}Suppose that either $p>k_g+1$ or $a_p(g)=0$. Let $\cN$ be the set of positive square-free integers that are coprime to $6pN_fN_g$. For $m\in\cN$, there exists a pair of cohomology classes\index{Beilinson--Flach elements! $\BF_{\f,g,\#,m},\BF_{\f,g,\flat,m}$}\index{Beilinson--Flach elements! $\BF_{\f,g,\alpha,m},\BF_{\f,g,\beta,m}$}
$$\BF_{\f,g,\#,m},\BF_{\f,g,\flat,m}\in \varpi^{-s(g)}H^1\left(\QQ(\mu_m),R_\f^*\otimes R_g^*\widehat{\otimes}\Lambda_\cO(\Gamma_\cyc)^\iota\right)$$ 
such that
\[
\begin{pmatrix}
\BF_{\f,g,\alpha,m}\\\\ \BF_{\f,g,\beta,m}
\end{pmatrix}
=Q_g^{-1}M_g
\begin{pmatrix}
\BF_{\f,g,\#,m}\\\\ \BF_{\f,g,\flat,m}
\end{pmatrix}.
\]
Here, $s(g)$ is a natural number that depends on $k_g$ but is independent of $m$.\index{Beilinson--Flach elements! $s(g)$}
\end{theorem}

We can bound the exponent $s(g)$ in the denominator of the Beilinson--Flach elements uniformly for the family $\f$ and integers $m\in \cN$. This is crucial for our purposes. We conjecture that the exponents $s(g)$ in the statement of Theorem~\ref{thm_BF_factorization_for_families_intro} are bounded independently of $g$. In Corollary~\ref{cor_s_g_is_uniformly_bounded} below, we verify this conjecture granted the (conjectural) existence of a rank--$2$ Euler system (c.f. Conjecture~\ref{conj_ESrank2}), whose existence is predicted by the Perrin-Riou philosophy; see \cite{pr-es,LZ0}.

We next study the $p$-local properties of the signed Beilinson--Flach elements and prove that they form a locally restricted Euler system. We also analyze the images of the Beilinson--Flach elements under the signed Coleman maps (which give rise to what we shall call doubly-signed Rankin--Selberg $p$-adic $L$-functions), to compare them to the Loeffler--Zerbes geometric $p$-adic $L$-functions. Thanks to the careful choices of the normalizations concerning the eigen-projections of the multivariate Perrin-Riou maps, this comparison we prove is sufficiently precise and allows us to keep track of delicate integrality questions.

\begin{theorem}[Proposition~\ref{prop:values-of-col-BF}]
\label{thm_values-of-col-BF_intro}
Suppose that  either $p>k_g+1$ or $a_p(g)=0$, and let $\bullet\in\{\#,\flat\}$. We have
\[
\col_{\Qp,\f,g,\bullet}\circ\loc_p(\BF_{\f,g,\bullet,1})=0.
\]
Furthermore, 
\[
\col_{\Qp,\f,g,\#}\circ\loc_p(\BF_{\f,g,\flat,1})=-\col_{\f,g,\flat}\circ\loc_p(\BF_{\f,g,\#,1})=D_{g}\delta_{k_g+1}L_{p}^{\rm geo}(\f,g),
\]
where {$L_{p}^{\rm geo}(\f,g)$ is the Loeffler--Zerbes geometric $p$-adic $L$-function introduced in Definition~\ref{defn:geopadicL}},  $\loc_p$ is the localization map at $p$, $D_{g}$ is a unit in $\Lambda_\cO(\Gamma_1)$ and $\delta_{k_g+1}$ is some explicit elements in $\LL_\cO(\Gamma_1)$ (see \S\ref{sec:Iw}).\index{Perrin-Riou maps! $D_g$}\index{Perrin-Riou maps! $\delta_{m}$}\index{$p$-adic $L$-functions! $L_{p}^{\rm geo}(\f,g)$ (Loeffler--Zerbes ``geometric'' $p$-adic $L$-function)}
\end{theorem}

See also Remark~\ref{rk:special} for a variant that concerns individual Rankin--Selberg convolutions, but also covers more ground in that case.

\subsection{Main conjectures: Results}
\label{subsubsec_results_intro}

We will now summarize the applications of our constructions and results in \S\ref{subsubsec_intro_optimizedPRandBF} towards the Conjectures~\ref{conjecture_cyclo_main_inert_ff_intro}, \ref{conj_3var_main_inert_ff_intro} and \ref{conj_anticyclo_main_inert_ff_intro} above. 

\begin{theorem}[Theorem~\ref{thm_cyclo_main_inert_ff}, cyclotomic main conjectures for $\f_{/K}\otimes\psi$]
\label{thm_cyclo_main_inert_ff_intro}
Suppose that $p\geq 7$ and ${\rm SL}_2(\FF_p)\subset\overline{\rho}_\f(G_{\QQ(\mu_{p^\infty})})$ as well as that $k_g\neq p-1$ and $p+1\nmid k_g+1$. Then,
$$ \varpi^{s(\psi)} L_p^{\rm RS}(\f_{/K}\otimes\psi)\,\in \,\Char_{\LL_\cO(\Gamma_1)}\left(\widetilde{H}^2_{\rm f}(G_{K,\Sigma},\TT_{\f,\psi}^{\cyc};\Delta_{\Gr})^\iota\right)$$
where $s(\psi)=s(\theta(\psi))$ is given as in Theorem~\ref{thm_BF_factorization_for_families_intro}.\index{Beilinson--Flach elements! $s(\psi)$}
\end{theorem}

Theorem~\ref{thm_cyclo_main_inert_ff_intro} is in fact the special case of a more general result we prove (Theorem~\ref{thm_cyclo_main_conj_ffotimesg} below), where we pick $g=\theta(\psi)$ the theta-series of the Hecke character $\psi$ of our fixed imaginary quadratic field $K$. We also remark that Theorem~\ref{thm_cyclo_main_inert_ff} relies heavily on Appendix~\ref{appendix_big_images}, where we study the images of Galois representations associated automorphic forms on $\GL2_{/\QQ}\times {\rm Res}_{K/\QQ}\GL_1$.

We deduce Theorem~\ref{thm_cyclo_main_conj_ffotimesg} from Theorem~\ref{thm_cyclo_main_conj_fotimesg}, which concerns individual (as opposed to families of) Rankin--Selberg products $f\otimes g$, via a patching criterion we establish in Appendix~\ref{Appendix_Regular_Rings_Divisibility}. This criterion rests on \cite[Lemma 7]{Chevalley1943} (which characterizes the $\mathfrak{m}$-adic topology of complete local regular rings); we thank T. Ochiai for bringing this lemma of Chevalley to our attention. Also crucial for our patching argument is to have a uniform control over the $\mu$-invariants. This is achieved through our optimized constructions of the signed Beilinson--Flach elements and the optimized factorizations of Perrin-Riou functionals into signed Coleman maps (which we have summarized in \S\ref{subsubsec_intro_optimizedPRandBF}).   

In order to prove Theorem~\ref{thm_cyclo_main_conj_fotimesg}, we utilize a locally restricted Euler system argument. The input is the signed Beilinson--Flach classes we produce via Theorem~\ref{thm_BF_factorization_for_families_intro}. Here, the adjective ``locally restricted'' refers to the $p$-local properties of the signed Beilinson--Flach elements, c.f. Theorem~\ref{thm_values-of-col-BF_intro}. This argument a priori yields bounds on the signed Selmer groups (c.f. Definition~\ref{defn:signedSel}) in terms of signed $p$-adic $L$-functions (c.f. Definition~\ref{defn:padicL}). It turns out that these signed Selmer groups can be identified with the classical Greenberg Selmer groups (see Corollary~\ref{cor_somesignedSelmergroupsareGreenberg} below). Likewise, we compare the signed $p$-adic $L$-functions with the Rankin--Selberg $p$-adic $L$-functions in Remark~\ref{remark_signed_geometric_padicL}(ii). This explicit comparison is also crucial for our patching argument.

The following is our main result towards the Iwasawa main conjectures over $\LL_\f(\Gamma_K)$ (Conjecture~\ref{conj_3var_main_inert_ff_intro}) for an imaginary quadratic field $K$ where $p$ remains inert.

\begin{theorem}[Theorem~\ref{thm_3var_main_inert_ff}]
\label{thm_3var_main_inert_ff_intro}
Suppose $\chi$ is a ray class character as above and assume that $p\geq 7$ as well as that ${\rm SL}_2(\FF_p)\subset \overline{\rho}_\f(G_{\QQ(\mu_{p^\infty})})$. Assume also that the uniform boundedness condition \ref{item_bounded_spsi}  on the variation of $s(\psi)$ holds true\index{Beilinson--Flach elements! ${\mathbf{(Bdd_{s(\psi)})}}$}. We then have the following containment in the Iwasawa main conjecture for the family $\f_{/K}\otimes\bbchi$ of Rankin--Selberg products:
\[
    L_p^{\rm RS}(\f_{/K}\otimes\bbchi)\,\in \,\Char_{\LL_\f(\Gamma_K)}\left(\widetilde{H}^2_{\rm f}(G_{K,\Sigma},\TT_{\f,\chi}^{K};\Delta_{\Gr})^\iota\right)\otimes_{\ZZ_p}\QQ_p\,.
\]
\end{theorem}

As we have remarked in the paragraph following the statement of Theorem~\ref{thm_BF_factorization_for_families}, the condition on the growth of $s(\psi_i)$ can be verified granted the existence of a rank-$2$ Euler system predicted by Conjecture~\ref{conj_ESrank2}.

We deduce Theorem~\ref{thm_3var_main_inert_ff_intro} from Theorem~\ref{thm_cyclo_main_inert_ff_intro} by varying $\psi $ and using the patching criterion (in Appendix~\ref{Appendix_Regular_Rings_Divisibility}) alluded to above. We note that it is crucial that we work with the Hida family $\f$ (as opposed to an individual eigenform $f$); see Remark~\ref{rem_s_g_uniform_bound_required_but_not_enough} where we discuss this technical point.

Let us assume in Theorems~\ref{thm_anticyc_main_inert_ff_definite_intro} and \ref{thm_anticyc_main_inert_ff_indefinite_intro} below that the nebentype $\varepsilon_f$ is trivial and we put $N_f=N^+N^-$ as before, where $N^+$ (resp. $N^-$) is a product of primes which are split (resp. inert) in $K/\QQ$. Utilizing the descent formalism of \cite[\S5.3.1]{BL_SplitOrd2020} (which relies crucially on the work of Nekov\'a\v{r}) applied together with Theorem~\ref{thm_3var_main_inert_ff_intro}, we shall obtain divisibilities in both ``definite'' and ``indefinite'' anticyclotomic Iwasawa main conjectures (see Conjecture~\ref{conj_anticyclo_main_inert_ff_intro}). We underline the importance to  work with Nekov\'a\v{r}'s Selmer complexes (rather than classical Selmer groups): Among other things, it allows us to by-pass any issues that may stem from the potential existence of nonzero pseudo-null submodules of Selmer groups. 

\begin{theorem}[Theorem~\ref{thm_anticyc_main_inert_ff_definite},  anticyclotomic main conjectures in the ``definite'' case]
\label{thm_anticyc_main_inert_ff_definite_intro}
Suppose $\chi$ is a ring class character such that $\widehat{\chi}_{\vert_{G_{\QQ_{p^2}}}}\neq \widehat{\chi^c}_{\vert_{G_{\QQ_{p^2}}}}$ and $N^-$ is a square-free product of odd number of primes. Assume that the following conditions hold true: 
\begin{itemize}
    \item $p\geq 7$ and ${\rm SL}_2(\FF_p)\subset\overline{\rho}_\f(G_{\QQ(\mu_{p^\infty})})$.
    \item \ref{item_bounded_spsi} is valid.
\end{itemize}
Then the $\LL_\f(\Gamma_\ac)$-module $\widetilde{H}^2_{\rm f}(G_{K,\Sigma},\TT_{\f,\bbchi_\ac}^{\dagger};\Delta_{\Gr})$ is torsion and the following containment in the anticyclotomic Iwasawa main conjecture for the family $\f_{/K}\otimes\bbchi_{\rm ac}$ holds:
\begin{equation*}
    L_p^{\dagger}(\f_{/K}\otimes\bbchi_\ac)\,\in \,\Char_{\LL_\f(\Gamma_\ac)}\left(\widetilde{H}^2_{\rm f}(G_{K,\Sigma},\TT_{\f,\bbchi_\ac}^{\dagger};\Delta_{\Gr})^\iota\right)\otimes_{\ZZ_p}\QQ_p\,.
\end{equation*}
\end{theorem}

\begin{theorem}[Theorem~\ref{thm_anticyc_main_inert_ff_indefinite}, anticyclotomic main conjectures in the ``indefinite'' case]
\label{thm_anticyc_main_inert_ff_indefinite_intro}
Suppose $\chi$ is a ring class character such that $\widehat{\chi}_{\vert_{G_{\QQ_p}}}\neq \widehat{\chi^c}_{\vert_{G_{\QQ_p}}}$ and $N^-$ is a square-free product of even number of primes. Assume that the following conditions hold true. 
\begin{itemize}
    \item $p\geq 7$ and ${\rm SL}_2(\FF_p)\subset\overline{\rho}_\f(G_{\QQ(\mu_{p^\infty})})$.
    \item \ref{item_bounded_spsi} is valid.
\end{itemize}
Then: 
\item[i)] $\ord_{(\gamma_+-1)}\, L_p^{\dagger}(\f_{/K}\otimes\bbchi)\geq 1$.

\item[ii)] The following containment \emph{(partial $\LL_\f(\Gamma_\ac)$-adic BSD formula for the family $\f_{/K}\otimes\bbchi_{\rm ac}$)} is valid:
\begin{equation*}
    \partial_\cyc^{r(\f,\bbchi_\ac)} L_p^{\dagger}(\f_{/K}\otimes\bbchi)\,\in \,{\rm Reg}_{\f,\bbchi_\ac}\cdot\Char_{\LL_\f(\Gamma_\ac)}\left(\widetilde{H}^2_{\rm f}(G_{K,\Sigma},\TT_{\f,\bbchi_\ac}^{\dagger};\Delta_{\Gr})_{\rm tor}^\iota\right)\otimes_{\ZZ_p}\QQ_p\,.
\end{equation*}
\end{theorem}

\begin{remark}
We show in Corollary~\ref{cor_s_g_is_uniformly_bounded} that the uniform boundedness condition \ref{item_bounded_spsi} on the variation of the exponents $s(\psi)$ holds true granted the existence of a rank-$2$ Euler system that the Perrin-Riou philosophy predicts. \index{Beilinson--Flach elements! ${\mathbf{(Bdd_{s(\psi)})}}$}
\end{remark}

In the situation of Theorem~\ref{thm_anticyc_main_inert_ff_indefinite_intro}, note that Conjecture~\ref{conj_anticyclo_main_inert_ff_intro}(ii) predicts in addition that 
$${\rm ord}_{(\gamma_+-1)}\,L_p^{\dagger}(\f_{/K}\otimes\bbchi)\stackrel{?}{=}1\stackrel{?}{=}r(\f,\bbchi_\ac)\,.$$ 
In the setting where the prime $p$ splits in $K/\QQ$, the second expected equality follows from \cite[Theorem 3.15]{BLForum}. Still when the prime $p$ splits in $K/\QQ$, one could also utilize \cite[Theorem 3.30]{BLForum} to show that the first expected equality holds if and only if ${\rm Reg}_{\f,\bbchi_\ac}\neq 0$.

Based on the recent work of Andreatta and Iovita~\cite{AndreattaIovitaBDP}, it seems very likely that the inequality $r(\f,\bbchi_\ac)\geq 1$ will be within reach very soon, see Remark~\ref{remark_andreatta_iovita} for a detailed discussion concerning this point. Note that granted this lower bound on the generic algebraic rank, the statement of Theorem~\ref{thm_anticyc_main_inert_ff_indefinite_intro}(ii) can be recast in the following form (c.f. Corollary~\ref{cor_thm_anticyc_main_inert_ff_indefinite_improvement_1}), which is further in line with the $\LL_\f(\Gamma_\ac)$-adic Birch and Swinnerton-Dyer formula predicted by Conjecture~\ref{conj_anticyclo_main_inert_ff_intro}(ii):
$$\partial_\cyc\, L_p^{\dagger}(\f_{/K}\otimes\bbchi)\,\in \,{\rm Reg}_{\f,\bbchi_\ac}\cdot\Char_{\LL_\f(\Gamma_\ac)}\left(\widetilde{H}^2_{\rm f}(G_{K,\Sigma},\TT_{\f,\bbchi_\ac}^{\dagger};\Delta_{\Gr})_{\rm tor}^\iota\right)\otimes_{\ZZ_p}\QQ_p\,.$$

Still in the situation of Theorem~\ref{thm_anticyc_main_inert_ff_indefinite_intro}, we note that a natural extension of Hsieh's generic non-vanishing result~\cite[Theorem C]{hsiehnonvanishing} combined with our results in the present work yields the upper bound $r(\f,\bbchi_\ac)\leq 1$ on the generic algebraic rank. We explain this in details in Remark~\ref{remark_Hsieh_improve_indefinite_ff} below.

\section{Layout}
We begin our discussion in \S\ref{S:Coleman} with an analysis of the integrality properties of the relevant Dieudonn\'e modules (which, among other things, require to keep track of Hida's cohomological congruence numbers). This involves the determination of their bases which are integral with respect to the natural lattices that arise when we realize the Galois representations in question inside the \'etale cohomology of modular curves. Such an analysis is inevitable since the integrality properties of the Beilinson--Flach elements are relative to the same lattices, and since our methods require a good control on the $\mu$-invariants as our Galois representations vary in $p$-adic families.

With the choice of an integral basis, we explain in \S\ref{sec_PR_1var} how to factor cyclotomic one-variable Perrin-Riou functionals (which are the eigen-projections of the Perrin-Riou map, normalized according to a good choice of an integral eigen-basis) to signed Coleman maps. We use this input in \S\ref{subsec_twovarPRandColeman}  to do the same for the two-variable Perrin-Riou functionals (With the second variable coming from a Hida family), which interpolate those constructed in \S\ref{sec_PR_1var} in a precise manner. In \S\ref{subsec_ap_is_0}, we consider the special case when $a_p(g)=0$, which allows us to relax the Fontaine--Laffaille condition on $g$. This alteration is crucial for scenarios where we would like to allow variation in $g$. For our eventual goals towards Conjectures~\ref{conj_3var_main_inert_ff_intro} and \ref{conj_anticyclo_main_inert_ff_intro}, we shall take  $g$ to be the theta-series of a Hecke character of the imaginary quadratic field $K$ where $p$ remains inert (and vary this Hecke character); note that $a_p(g)=0$ for such $g$.

In \S\ref{subsec_review_BF_elements}, we establish the analogous factorization results for the distribution-valued (non-integral) Beilinson--Flach elements into signed Beilinson--Flach elements with better integral properties. In \S\ref{subsec_padicL}, we study the $p$-local properties of the Beilinson--Flach elements (signed or otherwise) and recall their relation with Rankin--Selberg $p$-adic $L$-functions. 

We introduce the doubly-signed Selmer groups in \S\ref{subsec_def_Selmer_classical}, which naturally arise when we employ the locally restricted Euler system machinery with signed Beilinson--Flach elements, and we compare them to their classical counterparts (i.e., the Greenberg Selmer groups). Note that we eventually (in Section~\ref{subsubsec_setting_fKpsi}) introduce and work with yet another group of Selmer groups (i.e., Nekov\'a\v{r}'s extended Selmer groups), which have better base-change properties. In \S\ref{subsec_MC_signed_classical}, we formulate signed main conjectures (which are accessible via the locally restricted Euler system signed Beilinson--Flach classes) and compare them to their classical variants (which are amenable to $p$-adic variation). An important aspect in our comparison is the integrally-optimal determination of the images of the Perrin-Riou functionals. 

In \S\ref{subsec_fotimesg}  (resp., in \S\ref{subsec_ffotimesg}), we discuss the consequences of our constructions in \S\S\ref{S:Coleman}--\ref{sec_BF_and_padicL} towards the (cyclotomic) main conjectures for Rankin--Selberg convolutions $f\otimes g$ (resp., for the family of Rankin--Selberg convolutions $\f\otimes g$). In \S\ref{subsec_IMCimagquadinert} (more particularly, in \S\ref{subsubsec_setting_fKpsi} and \S\ref{subsubsec_cyclo_main_inert}), we recast our results in \S\ref{subsec_fotimesg} and \S\ref{subsec_ffotimesg} in the special case when $g=\theta(\psi)$ is the theta-series of an algebraic Hecke character $\psi$ of the imaginary quadratic field $K$ (where $p$ remains inert), and in terms of Nekov\'a\v{r}'s Selmer complexes. Relying on the main results of \S\ref{subsubsec_cyclo_main_inert} (which concern a fixed choice of $\psi$) and the patching criterion we establish in Appendix~\ref{Appendix_Regular_Rings_Divisibility}, we prove in \S\ref{subsubsec_3var_inert_ord} our main results towards the $\LL_\f(\Gamma_K)$-adic main conjectures in three variables. In \S\S\ref{subsubsec_anticyclo_1}--\ref{subsubsec_anticyclo_3}, we descend to the anticyclotomic tower dwelling on the general descent formalism Nekov\'a\v{r} has established (see also \cite[\S5.3.1]{BL_SplitOrd2020} where his results have been simplified to cover a limited scope of Galois representations that we concern ourselves with). This concludes the main body of the manuscript.

In the first of the three appendices (Appendix~\ref{Appendix_Regular_Rings_Divisibility}), we prove a divisibility criterion in regular rings, which plays a crucial role to patch our results for individual Rankin--Selberg convolutions $f\otimes g$ to obtain results towards main conjectures in three-variables. The criterion we prove is based on a lemma of Chevalley; we thank Tadashi Ochiai for bringing this observation of Chevalley to our attention. In Appendix~\ref{appendix_sec_padicRankinSelbergHida}, we revisit Loeffler's recent work~\cite[\S4]{Loeffler2020universalpadic} where he constructs $p$-adic Rankin--Selberg $L$-functions. We extend it very slightly to cover the case of minimally ramified universal deformations. This input is important even in the formulation of Conjectures~\ref{conj_3var_main_inert_ff_intro} and \ref{conj_anticyclo_main_inert_ff_intro}. In Appendix~\ref{appendix_big_images}, we study the images of the Galois representations we are interested in. We record a number of sufficient conditions to ensure the validity of the image-related hypotheses in our divisibility results towards main conjectures.

\section{Notation}
We conclude our introduction with some further notation we shall rely on in the main body of our manuscript.
\begin{defn}[The weight space and Hida's universal Hecke algebra]
\label{defn_Hecke_algebra_intro}
\item[i)]\index{Hida Theory! $[\,.\,]$} Let us put $[\,\cdot\,]: \ZZ_p^\times \hookrightarrow \LL(\ZZ_p^\times)^\times$ for the natural injection. The universal weight character $\bbkappa$ is defined as the composite map\index{Hida Theory! Universal weight character $\bbkappa$}
$$\bbkappa: G_\QQ\stackrel{\chi_\cyc}{\lra}\ZZ_p^\times\hookrightarrow \LL(\ZZ_p^\times)^\times,$$  
where $\chi_\cyc$ is the $p$-adic cyclotomic character. A ring homomorphism 
$$\kappa: \LL_{\rm wt}:=\LL(\Zp^\times) \lra \cO$$ 
is called an arithmetic specialization of weight $k+2 \in \ZZ$ if the compositum 
$$G_\QQ\stackrel{\bbkappa}{\lra} \LL_{\rm wt}^\times\stackrel{\kappa}{\lra}\cO$$
agrees with $\chi_\cyc^{k}$ on an open subgroup of $G_\QQ$.\index{Hida Theory! Arithmetic specialization}
\item[ii)] Given an eigenform $f$ as above, let us define $$\LL(\ZZ_p^\times)^{(f)}\cong\LL(1+p\ZZ_p)$$ 
as the component that is determined by the weight $k_f+2$, in the sense that the map $\LL(\ZZ_p^\times)^\times \stackrel{\langle k_f \rangle}{\lra} \ZZ_p$ factors through $\LL(\ZZ_p^\times)^{(f)}$. Here, for a given integer $k$, we have written $\langle k \rangle:\LL(\ZZ_p^\times)\rightarrow \ZZ_p$ to denote the group homomorphism induced from the map $[x]\mapsto x^k$.\index{Hida Theory! $\LL(\ZZ_p^\times)^{(f)}$}
\item[iii)] We let $\f=\sum_{n=1}^{\infty} \mathbb{a}_{n}(\f)q^n \in \LL_\f[[q]]$\index{Hida Theory! Non-CM Hida family $\f$} denote the (non-CM) branch of the primitive Hida family of tame conductor $N_f$, which admits $f_{\alpha}$ as a weight $k_f+2$ specialization. Here, $\LL_\f$\index{Hida Theory! Non-CM branch $\LL_\f$} is the branch (i.e., the irreducible component) of the Hida's universal ordinary Hecke algebra determined by $f_{\alpha}$. It is finite flat over $\LL(\ZZ_p^\times)^{(f)}$ and the universal weight character $\bbkappa$ restricts to a character (also denoted by $\bbkappa$)
$$\bbkappa:  G_\QQ\lra \LL_{\rm wt}^\times \twoheadrightarrow \LL(\ZZ_p^\times)^{(f),\times}\lra \LL_\f^\times\,.$$
\item[iv)] We set $\LL_{\f}(\Gamma):=\LL_\f\,\widehat{\otimes}\,\Lambda_\cO(\Gamma)$ for any profinite group $\Gamma$.\index{Completed group rings! $\LL_\f(\Gamma_?)$}
\end{defn}

\subsection{Iwasawa algebras and distribution algebras}\label{sec:Iw}
We fix a topological generator $\gamma$ of $\Gamma_1$ and identify   $\Lambda_\cO(\Gamma_\cyc)$ and  $\Lambda_\cO(\Gamma_1)$ with the power series rings $\cO[\Delta]\lb X\rb$ and $\cO\lb X\rb$ respectively, where $X=\gamma-1$. 

We also consider $\Lambda_\cO(\Gamma_\cyc)$ and $\Lambda_\cO(\Gamma_1)$ as  subrings of  $\cH(\Gamma_\cyc)$ and $\cH(\Gamma_1)$\index{$\cH(\Gamma_?)$} respectively, which are  the rings of power series $F \in L[\Delta]\lb X \rb$ (respectively $F \in L\lb X \rb$) which converge on the open unit disc $|X| < 1$ {in $\Cp$, where $|\ |$ denotes the $p$-adic norm on $\Cp$ normalized by $|p|=p^{-1}$}. 

For any real number $r\ge 0$, we write  $\cH_{r}(\Gamma_\cyc)$ and $\cH_{r}(\Gamma_1)$\index{$\cH_r(\Gamma_?)$}  for the set of power series $F$ in $\cH(\Gamma_\cyc)$ and $\cH(\Gamma_1)$ respectively satisfying $\sup_{t} p^{-tr}\Vert F\Vert_{\rho_t}<\infty$,
where $\rho_t=p^{-1/p^{t-1}(p-1)}$ {and $\Vert F\Vert_{\rho_t}=\sup_{|z|\le \rho_t}|F(z)|$}. It is common to write $F=O(\log_p^r)$ when $F$ satisfies this condition.% When $r=0$, we write $\Lambda_L(\Gamma_\cyc)=\cH_0(\Gamma_\cyc)=\Lambda_\cO(\Gamma_\cyc)\otimes_{\cO}L$ and $\Lambda_{L}(\Gamma_1)=\cH_{0}(\Gamma_1)=\Lambda_\cO(\Gamma_1)\otimes_{\cO}L$.

%We also define $\Lambda_\f(\Gamma_\cyc)$ (resp. $\Lambda_\f(\Gamma_1)$) as $\Lambda_\f\,\widehat{\otimes}\,\Lambda_\cO(\Gamma_\cyc)$ (resp. $\Lambda_\f\,\widehat{\otimes}\,\Lambda_\cO(\Gamma_1)$).

For each integer $n\ge0$, we put $\omega_{n}(X):=(1+X)^{p^{n}}-1$\index{Cyclotomic polynomials and half-logarithms! $\omega_{n}(X)$}. We set $\Phi_0(X) = X$, and $\Phi_n(X) = \omega_n(X)/\omega_{n-1}(X)$ for $n \ge 1$.\index{Cyclotomic polynomials and half-logarithms! $\Phi_n(X)$} Let $\chi_\cyc$ denote the $p$-adic cyclotomic character.\index{$\chi_\cyc$}

We write $\Tw$\index{$\Tw$} for the ring automorphism of $\cH(\Gamma_\cyc)$ defined by $\sigma \mapsto \chi_\cyc(\sigma)\sigma$ for $\sigma\in \Gamma$.
If we set $u:=\chi_\cyc(\gamma)$, then $\Tw$ maps $X$ to $u(1 + X) - 1$. If $m\ge 1$ is an integer, we define
\begin{align*}\index{Cyclotomic polynomials and half-logarithms! $\omega_{n,m}(X)$}
 \omega_{n,m}(X)=\prod_{i=0}^{m-1}\Tw^{-i}\left(\omega_{n}(X)\right);\,\,\,\,\,\,
 \Phi_{n,m}(X)=\prod_{i=0}^{m-1}\Tw^{-i}\left(\Phi_{n}(X)\right)\,.
\end{align*}\index{Cyclotomic polynomials and half-logarithms!  $\Phi_{n,m}(X)$}
We let $\log_p \in \cH(\Gamma_1)$ denote the $p$-adic logarithm. We also define 
\[
\log_{p,m}=\prod_{i=0}^{m-1}\Tw^{-i}\left(\log_p\right).
\]\index{Cyclotomic polynomials and half-logarithms! $\log_{p,m}$}
We also define Pollack's half logarithms
\begin{align*}
\log_{p,m}^+&=\prod_{i=0}^{m-1}\Tw^{-i}\left(\prod_{m=1}^\infty \frac{\Phi_{2m}(X)}{p}\right),  \\  
\log_{p,m}^-&=\prod_{i=0}^{m-1}\Tw^{-i}\left(\prod_{m=1}^\infty \frac{\Phi_{2m-1}(X)}{p}\right).
\end{align*}\index{Cyclotomic polynomials and half-logarithms! $\log_{p,m}^\pm$}
Finally, we set
\[
\delta_m=\prod_{i=0}^{m-1}\Tw^{-i}X.
\]\index{Perrin-Riou maps! $\delta_{m}$}
%%%%
%%%%
\subsection*{Acknowledgements} The authors thank Francesc Castella, Ming-Lun Hsieh, Jackie Lang, David Loeffler and Sarah Zerbes for illuminating discussions during the preparation of the manuscript. We would also like to thank the anonymous referees for their helpful comments and suggestions. The second named author’s research is supported by the NSERC  Discovery Grants Program (RGPIN-2020-04259 and RGPAS-2020-00096).
%%%%%%%%%%%
%%%%%%%%%%%

%-----------------------------------------------------------------------
% End of chap1
%-----------------------------------------------------------------------

%% file: InertOrdMainChapter2.tex
\chapter{Construction of the optimized Coleman maps}\label{S:Coleman}
Throughout this chapter, we fix a pair of eigenforms $f$ and $g$ as in the introduction. Our goal in this chapter is to define the appropriate Coleman maps for $T_{f,g}$. These maps will be used to formulate  Iwasawa main conjectures later on. We will slightly alter the construction of Coleman maps in \cite{BLLV} where the Rankin--Selberg products of two non-ordinary forms were studied. We will show that two of these Coleman maps behave well as $f$ varies in a Hida family, giving rise to $\Lambda_\f(\Gamma_\cyc)$-adic  Coleman maps, which in turn allow us to formulate (two-variable) signed Iwasawa main conjectures which also incorporates variation in the Hida families.

We begin our discussion by fixing some notation.
\section{Dieudonn\'e modules and $\vp$-eigenvectors}
\begin{defn}
 {Let us identify $M_{\rm dR}(h)\otimes_{\ZZ}\ZZ_p$ with $\Dcris(V_h)$ via the comparison isomorphism for each $h\in\{f,g\}$, where $M_{\rm dR}(h)$ is the de Rham realization of the Scholl motive associated to $h$}. Let $h^\star$ denotes the conjugate form of $h$ (denoted by $h^*$  in \cite[Definition~ 6.1.2]{KLZ1}).

We recall (following the notation of \cite[\S2.8]{KLZ2}) that $V_h$ denotes the $h$-isotypic subspace of $H^1_{\textup{\'et}}(Y_1(N_f)_{\overline{\QQ}},{\rm Sym}^k(\mathscr{H}_{\ZZ_p}^\vee))\otimes_{\ZZ_p}L$, whereas $V_h^*$ is the $h$-isotypic quotient of $H^1_{\textup{\'et}}(Y_1(N_f)_{\overline{\QQ}},{\rm Tsym}^k(\mathscr{H}_{\ZZ_p})(1))\otimes_{\ZZ_p}L$  and contains the natural lattice $R_h^*$, which is generated by the image of $H^1_{\textup{\'et}}(Y_1(N_f)_{\overline{\QQ}},{\rm Tsym}^k(\mathscr{H}_{\ZZ_p})(1))\otimes_{\ZZ_p}\cO$. We let $R_{h^\star}\subset V_{h^\star}$ denote the lattice that maps isomorphically to $R_h^*(-1-k_h)$ under the isomorphism $V_{h^\star}\xrightarrow{\sim}V_h^*(-1-k_h)$.\index{Galois representations! Deligne's representations $R_h, R_h^*$}

We write $\omega_h\in \Fil^1\Dcris(V_h)$\index{Galois representations! $\omega_h \in \Fil^1\Dcris(V_h)$} to be the basis vector as defined in \cite[\S6.1]{KLZ1} and write $\omega_{h^\star}\in \Fil^1\Dcris(V_{h^\star})$ for the corresponding element for the conjugate form $h^\star$, which we regard as an element of $\Fil^0\Dcris(V_h^*)$ via the natural isomorphism $V_{h^\star}\xrightarrow{\sim}V_h^*(-1-k_h)$. 
\end{defn}
In order to define our Coleman maps, we will have to fix certain $\vp$-eigenvectors of the Dieudonn\'e modules. Let $h\in\{f,g\}$. We recall that the $\vp$-eigenvalues on $\Dcris(V_h)$ are $\alpha_h$ and $\beta_h$, whereas those on $\Dcris(V_h^*)$ are $\alpha_h^{-1}$ and $\beta_h^{-1}$. Throughout, we assume $\alpha_h\ne \beta_h$. Let $$\langle-,-\rangle: \Dcris(V_h)\otimes \Dcris({V_{h}^*})\lra \Dcris(L)= L$$ denote the natural pairing. We have the following orthogonality relation:
\begin{lemma}\label{lem:orth}
Let $h\in\{f,g\}$ and $\lambda\in\{\alpha,\beta\}$. If $\lambda'$ is the unique element of $\{\alpha,\beta\}\setminus\{\lambda\}$, then 
\[
\langle u,v\rangle=0
\]
for all $u\in \Dcris(V_h)^{\vp=\lambda_h}$ and $v\in \Dcris(V_h^*)^{\vp=(\lambda_h')^{-1}}$.
\end{lemma}
\begin{proof}
  The lemma follows from 
  \begin{align*}
      \langle u,v\rangle &=\vp(\langle u,v\rangle)\\
      &=\langle \vp(u),\vp(v)\rangle\\
      &=\frac{\lambda_h}{\lambda_h'}\langle u,v\rangle.
  \end{align*}
  \end{proof}
  This allows us to give the following definition:

\begin{defn}
\label{def_eigenvectors}
\item[i)]For  $h\in\{f,g\}$ and $\lambda\in\{\alpha,\beta\}$, let $v_{h,\lambda}\in \Dcris(V_h)^{\vp=\lambda_h}$  be an $\cO$-basis of
\[
\Hom_\cO\left(\Dcris(R_h^*)/\Dcris(R_h^*)^{\vp=(\lambda'_h)^{-1}},\cO\right),
\]\index{Galois representations! $v_{g,\lambda}\in \Dcris(V_g)$}where $\lambda'$ is the unique element of $\{\alpha,\beta\}\setminus\{\lambda\}$ and we have identified $\Dcris(V_h)$ with $\Hom_L(\Dcris(V_h^*),\cO)$ via $\langle-,-\rangle$.

\item[ii)] We define $\{v_{h,\alpha}^*,v_{h,\beta}^*\}$ to be the dual basis of $\{v_{h,\alpha},v_{h,\beta}\}$. 
\item[iii)] Given a pair $\lambda,\mu\in\{\alpha,\beta\}$, we set $v_{\lambda,\mu}:=v_{f,\lambda}^*\otimes v_{g,\mu}^*$. \index{Galois representations! $v_{\lambda,\mu}$}
\end{defn}

\begin{remark}\label{rk:e-vecs}
%It is clear that $v_{h,\lambda}^*\in \Dcris(V_h^*)^{\vp=\lambda_h^{-1}}$, which is also an $\cO$-basis of $\Dcris(R_h^*)/\Dcris(R_h^*)^{\vp=(\lambda_h')^{-1}}$. \red{(If we are taking $h=g$ in the remainder, we should do so from the beginning.)}
Suppose that $R_g^*$ satisfies the Fontaine--Laffaille condition $p>k_g+1$. As $g$ is non-ordinary at $p$, it follows from \cite[Lemma~3.1]{LLZ3} that $\omega_{g^\star}$ and $\vp(\omega_{g^\star})$ form an $\cO$-basis of $\Dcris(R_g^*)$. Since $\Dcris(V_g^*)^{\vp=(\lambda'_g)^{-1}}$ is generated by $\omega_{g^\star}-\lambda_g\vp(\omega_{g^\star})$, the $\cO$-module $\Dcris(R_g^*)/\Dcris(R_g^*)^{\vp=(\lambda_g')^{-1}}$ is generated by $\lambda_g\vp(\omega_{g^\star})$. We may (and will) take
\[
v_{g,\lambda}^*=\frac{\lambda_g}{\lambda_g-\lambda_g'}\left(\omega_{g^\star}-\lambda'_g\vp(\omega_{g^\star})\right).
\]
A direct computation using the fact that 
$$\langle\omega_g,\omega_{g^\star}\rangle=\langle\vp(\omega_g),\vp(\omega_{g^\star})\rangle=0$$ 
and the duality $\langle \vp(\omega_g),\omega_{g^\star}\rangle=\langle \omega_g,\vp^{-1}(\omega_{g^\star})\rangle$ shows that
\[
v_{g,\lambda}=\frac{1}{\langle\vp(\omega_g),\omega_{g^\star}\rangle}\left(\vp(\omega_g)-\lambda'_g\omega_g\right) .\]
\end{remark}  

\begin{remark}\label{rk:f-basis}
    
We discuss the choice of $v_{f,\lambda}^*$ under the Fontaine-Laffaille condition  $p>k_f+1$.

\noindent
{\textbf{Case 1: $V_f^*$ is not locally split at $p$.}} We recall from \cite[\S5]{LZWach} that the Wach module $\NN(R_f^*)$ admits an $\AA_L^+$-basis $\{n_1,\varpi^d n_2\}$  with respect to which  the matrix of $\vp$ is of the form
\[
\begin{pmatrix}
\left(\frac{\pi}{\vp(\pi)}\right)^{k_f+1}\varepsilon_f(p)^{-1}\alpha_f&\varpi^d\pi^{k_f+1}x_f\\
0&\alpha_f^{-1}
\end{pmatrix},
\]
where $d\ge0$ is an integer, $\varpi$ is an uniformizer of $L$ and $x_f\in\AA_L^+$ satisfies $t^{k_f+1}x_f=k_f!+O(t^{k_f+2})$.
We  recall here for the convenience of the reader that $\AA_L^+=\cO[[\pi]]$,  which  is equipped with an $\cO$-linear actions by $\vp$ and $\Gamma$ given by $\vp(\pi)=(1+\pi)^p-1$ and $\sigma\cdot \pi=(1+\pi)^{\chi(\sigma)}-1$) for $\sigma\in\Gamma$.

Following \cite[Théorème~III.4.4]{berger04}, we identify $\Dcris(R_f^*)$ with $\NN(R_f^*)/\pi$. Let us write $\nu_1$ and $\nu_2$ for the images of $\varpi^dn_2$ and $k_f!n_1$ in $\Dcris(R_f^*)$ respectively. Then $\{\nu_1,\nu_2\}$ is an $\cO$-basis of $\Dcris(R_f^*)$ since $k_f!$ is a $p$-adic unit. Furthermore, it follows from \cite[Proposition~5.3 and Corollary~5.5]{LZWach} that $\nu_2$ is a $\vp$-eigenvector with eigenvalue $\beta^{-1}$ and that $\nu_1$ generates $\Fil^0\Dcris(R_f^*)$ $\cO$-linearly. The matrix of $\vp$ with respect to the basis $\{\nu_1,\nu_2\}$ is given by
\[
\begin{pmatrix}
\alpha_f^{-1}&0\\
\varpi^d&\beta_f^{-1}
\end{pmatrix}.
\]

\noindent
\textbf{Case 2: $V_f^*$ is locally split at $p$.} In this case, it is clear that $\Dcris(R_f^*)$ admits an $\cO$-basis consisting of $\vp$-eigenvectors $\{\nu_1,\nu_2\}$ such that $\nu_1$ generates $\Fil^0\Dcris(R_f^*)$ $\cO$-linearly and that the matrix of $\vp$ with respect to this basis is equal to
\[
\begin{pmatrix}
\alpha_f^{-1}&0\\
0&\beta_f^{-1}
\end{pmatrix}.
\]

In both cases, we may assume $\nu_1=\omega_{f^\star}=$  and $v_{f,\beta}^*=\nu_2$. As for $v_{f,\alpha}^*$, we may set it to be   either $\nu_1+\frac{\varpi^d\alpha_f\beta_f}{\beta_f-\alpha_f}\nu_2=\frac{\alpha_f}{\alpha_f-\beta_f}(\omega_{f^\star}-\beta_f\vp(\omega_{f^\star}))$ (which has the same form as $v_{g,\lambda}^*$) or $\nu_1=\omega_{f^\star}$ depending on whether $V_f^*$ is locally split at $p$ or not.\end{remark}

For a finite unramified extension $F$ of $\Qp$, we have a Perrin-Riou map
\[
\cL_{F,f,g}:\HIw(F(\mu_{p^\infty}),T)\rightarrow \cH_{k_f+k_g+2}(\Gamma_\cyc)\otimes_{\Zp}\Dcris(T)\otimes_{\Zp}F
\]\index{Perrin-Riou maps! $\cL_{F,f,g}$}
 as defined in  \cite[\S3.1]{leiloefflerzerbes11} and \cite[Appendix B]{LZ0}.
For any $\lambda,\mu\in\{\alpha,\beta\}$, our choice of $\vp$-eigenvectors  in Definition~\ref{def_eigenvectors} gives rise to a map 
\begin{equation}
    \cL_{F,f,g}^{(\lambda,\mu)}:\HIw(F(\mu_{p^\infty}),T)\rightarrow \cH_{\ord_p(\lambda_f\mu_g)}(\Gamma_\cyc)\otimes_{\Zp} F
\label{eq:projectingPR}
\end{equation}\index{Perrin-Riou maps! $\cL_{F,f,g}^{(\lambda,\mu)}$}
on projecting $\cL_{F,f,g}$ to the $(\lambda_f\mu_g)^{-1}$-eigenvector $v_{\lambda,\mu} \in \Dcris(T)\otimes_{\ZZ_p}\Qp$.
When $F=\Qp$, we omit $F$ from the notation and simply write $\cL_{f,g}$ and $\cL^{(\lambda,\mu)}_{f,g}$ in place of $\cL_{\Qp,f,g}$ and $\cL^{(\lambda,\mu)}_{\Qp,f,g}$, respectively. \index{Perrin-Riou maps! $\cL_{f,g}$} \index{Perrin-Riou maps! $ \cL^{(\lambda,\mu)}_{f,g}$}
                                     
\section{Factorization of one-variable Perrin-Riou maps}\label{sec_PR_1var}
In \cite[\S4]{BLLV}, under the hypothesis that both $f$ and $g$ are non-ordinary at $p$, we have constructed a logarithmic matrix attached to an $\cO$-basis $\cB=\{v_i:i=1,\ldots 4\}$ of $\Dcris(T_{f,g})$, where we recall that $T_{f,g}=R_f^*\otimes R_g^*$\index{Galois representations! $T_{f,g}$}, using the theory of Wach modules. This matrix  is in turn used to define the bounded Coleman maps in the context of op. cit. We explain here that one may still carry out a similar construction even when $f$ is ordinary at $p$. The main ingredient in the construction of the logarithmic matrix in loc. cit. is Berger's result \cite[proof of Proposition~V.2.3]{berger04}, which allows one to lift a basis of $\Dcris(T_{f,g})$ to an $\AA_L^+$-basis of the Wach module $\NN(T_{f,g})$\index{Galois representations! $\NN(T_{f,g})$ (Wach module)}. Berger's result applies whenever the following three conditions hold:
\begin{itemize}
     \item[\mylabel{item_FL}{{\bf (FL$_{f,g}$)}}] $T_{f,g}$ verifies the Fontaine--Laffaille condition $p>k_f+k_g+2$;
     \item[\mylabel{item_Slo}{{\bf (Slo)}}] The slope of $\vp$ on $\Dcris(T_{f,g})$ does not attain $-k_f-k_g-2$ and $0$ simultaneously;
     \item[\mylabel{item_Fil}{{\bf (Fil)}}] The basis {$\cB$} of $\Dcris(T_{f,g})$ is  a filtered basis, in the sense that it gives rise to bases for the graded pieces of $\Dcris(T_{f,g})$.
\end{itemize}

 \begin{remark}{We have chosen in \cite[\S4.1]{BLLV} the basis $$\left\{\ofs\otimes\ogs,\vp(\ofs)\otimes\ogs,\ofs\otimes\vp(\ogs),\vp(\ofs)\otimes\vp(\ogs)\right\},$$
 which verifies the hypothesis \ref{item_Fil} under the assumption that $k_f\le k_g$. 
We note that there is a typo in the labelling of $v_2$ and $v_3$ in \cite{BLLV}. There is another typo in the formula of the matrix of $\vp$ on $\Dcris(T_h^*)$ (labelled $A_h$ in loc. cit.), where $\varepsilon_h(p)$ should have been $\varepsilon_h(p)^{-1}$. The formula for $A_0$ and $Q$ should also be  altered accordingly. These typos do not affect subsequent calculations. For completeness, the correct formula for $A_0$ reads
 \[
 A_0=\begin{pmatrix}
 0&0&0&(\varepsilon_f(p)\varepsilon_g(p))^{-1}\\
 0&0&\varepsilon_g(p)^{-1}&-\varepsilon_g(p)^{-1}a_p(f)\\
 0&-\varepsilon_f(p)^{-1}&0&-\varepsilon_f(p)^{-1}a_p(g)\\
 1&a_p(f)&a_p(f)&a_p(f)a_p(g)
 \end{pmatrix}.
 \] When  $k_f\ge k_g$, we may swap the ordering of $\vp(\ofs)\otimes\ogs$ and $\ofs\otimes\vp(\ogs)$. This then gives a basis satisfying the condition \ref{item_Fil}.} 
\end{remark}

In our current setting, the fact that $g$ is non-ordinary at $p$ ensures that \ref{item_Slo} holds even when $f$ is $p$-ordinary. Therefore, granted the Fontaine--Laffaille condition \ref{item_FL}, we are left to determine a basis of $\Dcris(T_{f,g})$ that verifies the condition \ref{item_Fil} in order to apply Berger's result. We recall from Remarks~\ref{rk:e-vecs} and \ref{rk:f-basis} that $\omega_{g^\star}$ and $\vp(\omega_{g^\star})$ form an $\cO$-basis of $\Dcris(R_g^*)$ whereas $\omega_{f^\star}$ and $v_{f,\beta}^*$ form an $\cO$-basis of $\Dcris(R_f^*)$. This leads us to pick the following $\cO$-basis of $\Dcris(T_{f,g})$:

\begin{defn}
We let $\{v_1,v_2,v_3,v_4\}\subset \Dcris(T_{f,g})$ denote the $\cO$-basis given by 
$$v_1=\ofs\otimes\ogs, \,\,\,\,\, v_2=\ofs\otimes \vp(\ogs)\,\,\,\,\, v_3=\etaf\otimes\ogs,\,\,\,\,\, v_4=\etaf\otimes\vp(\ogs).$$
\end{defn}

Note that if  $k_g\le k_f$, this basis satisfies \ref{item_Fil}; otherwise, the basis $\{v_1,v_3,v_2,v_4\}$ does. 
{
\begin{defn}
\item[i)] Suppose $k_g\le k_f$. We let $\{n_1,n_2,n_3,n_4\}$ denote the basis of the Wach-module $\NN(T_{f,g})$ given by \cite[Proof of Proposition~V.2.3]{berger04}, that lifts the ordered basis $\{v_1,v_2,v_3,v_4\}$ of $\Dcris(T_{f,g})$.
\item[ii)] Suppose $k_g> k_f$. We similarly define $\{n_1,n_2,n_3,n_4\}$ as the ordered basis of $\NN(T_{f,g})$ lifting the ordered basis $\{v_1,v_3,v_2,v_4\}$ of $\Dcris(T_{f,g})$.
\item[iii)] Set $q:=\vp(\pi)/\pi\in\AA_L^+$ and $\xi=p/(q-\pi^{p-1})\in (\AA_L^+)^\times$. 
\end{defn}
\begin{lemma}\label{lem:matrices}
Let  $A_{f,g,0}\in\GL_4(\cO)$ so that the matrix $A_{f,g}$ of $\vp_{\vert _{\Dcris(T_{f,g})}}$ with respect to the ordered basis $\{v_1,v_2,v_3,v_4\}$ verifies
\[
A_{f,g}=A_{f,g,0}\cdot\begin{pmatrix}
1&&&\\
&1/p^{k_g+1}&&\\
&&1/p^{k_f+1}&\\
&&&1/p^{k_f+k_g+2}
\end{pmatrix}\,.
\]
Then the matrix of $\vp_{\vert_{\NN(T_{f,g})}}$ with respect to the ordered basis $\{n_1,n_2,n_3,n_4\}$  is given by
\[
P_{f,g}=A_{f,g,0}\cdot \begin{pmatrix}
\xi^{k_f+k_g+2}&&&\\
&\xi^{k_f+1}/q^{k_g+1}&&\\
&&\xi^{k_g+1}/q^{k_f+1}&\\
&&&1/q^{k_f+k_g+2}
\end{pmatrix}.
\]
\end{lemma}
\begin{proof}
We note that  \cite[Proposition~V.2.3]{berger04} applies to the representation $T_{f,g}(-k_f-k_g-2)$. We have the identifications 
\begin{align*}
    \NN(T_{f,g})&\xrightarrow{\times\,\pi^{k_f+ k_g+2}} \NN(T_{f,g}(-k_f-k_g-2)),\\
     \Dcris(T_{f,g})&\xrightarrow{\times\, t^{k_f+ k_g+2} } \Dcris(T_{f,g}(-k_f-k_g-2)),
\end{align*}
where $t=\log(1+\pi)$. Let us put $n_i'=\pi^{ k_f+k_g+2}\cdot n_i$ and  $v_i'= t^{ k_f+k_g+2}\cdot v_i$ for $i=1,\ldots, 4$.
The matrix  of $\vp_{\vert_{\Dcris(T_{f,g}(-k_f-k_g-2))}}$ with respect to  $\{v_i'\}_{
i=1}^4$ equals
\[
A_{f,g,0}\cdot\begin{pmatrix}
p^{k_f+k_g+2}&&&\\
&p^{k_f+1}&&\\
&&p^{k_g+1}&\\
&&&1
\end{pmatrix},
\]
then matrix of $\vp_{\vert_{\NN(T_{f,g}(-k_f-k_g-2))}}$ with respect to $\{n_i'\}_{
i=1}^4$  is given by
\[
A_{f,g,0}\cdot\begin{pmatrix}
(\xi q)^{k_f+k_g+2}&&&\\
&(\xi q)^{k_f+1}&&\\
&&(\xi q)^{k_g+1}&\\
&&&1
\end{pmatrix}.
\]
We  then obtain the matrix of $\vp|_{\NN(T_{f,g})}$ with respect to the basis $\{n_i\}_{
i=1}^4$ using the identity $\vp(\pi)=q\pi$.
\end{proof}}

\begin{defn}
We write $M_{f,g}$ for the resulting logarithmic matrix with respect to the basis $\{n_1,n_2,n_3,n_4\}$, defined as in \cite[\S4.2]{BLLV}. In more explicit terms,\index{Coleman maps! The matrix $M_{f,g}$}
$$M_{f,g}:=\fM^{-1}\left((1+\pi)A_{f,g}\vp(M)\right) \in {\rm Mat}_{4\times 4}(\cH_{k_f+k_g+1}(\Gamma_1)),$$ 
where $M$ is the change of basis matrix as given in (4.2.2) of op. cit. {and $\mathfrak{m}$ is the Mellin transform}.
\end{defn}

Concretely, $M_{f,g}$ is the change of basis matrix satisfying
\[
\begin{pmatrix}
(1+\pi)\vp(n_1)&(1+\pi)\vp(n_2)&(1+\pi)\vp(n_3)&(1+\pi)\vp(n_4)
\end{pmatrix}
=\begin{pmatrix}
v_1&v_2&v_3&v_4
\end{pmatrix}M_{f,g}.
\]

\begin{defn}\label{defn:Mg}
Let $A_g$ be the matrix of $\vp$ with respect to the basis $\{\ogs,\vp(\ogs)\}$, which we  factor as
\[
A_g=A_{g,0}\begin{pmatrix}
1&\\
&1/p^{k_g+1}
\end{pmatrix},
\]
where $A_{g,0}\in \GL_2(\cO)$.  We write $M_g \in {\rm Mat}_{2\times 2}(\cH_{k_g+1,1})$ for the resulting $2\times 2$ logarithmic matrix,  given as in \cite[\S2]{BFSuper} (note that we may define such a matrix as long as $p>k_g+1$).

Similarly, let \[
A_{f}=\begin{pmatrix}
\alpha_f^{-1}&0\\
\delta_f&\beta_f^{-1}
\end{pmatrix}=A_{f,0}\begin{pmatrix}
1&\\
&1/p^{k_f+1}
\end{pmatrix}
\]
be the matrix of $\vp$ with respect to $\{\omega_{f^\star},\etaf\}$, where $\delta_f=\varpi^d$ or $0$ depending on whether $V_f^*$ is locally split at $p$ or not (as given in Remark~\ref{rk:f-basis}), and $A_{f,0}=\begin{pmatrix}
\alpha_f^{-1}&0\\
\delta_f&p^{k_f+1}\beta_f^{-1}
\end{pmatrix}\in\GL_2(\cO)$.
\end{defn}
We now describe the relation between the logarithmic matrices $M_{f,g}$ and $M_g$.

\begin{proposition}\label{prop:semiordlog}
There exist elements $u_f\in \Lambda_\cO(\Gamma_1)^\times$ and $\ell_f\in \cH_{k_f+1}(\Gamma_1)$ such that
\[
M_{f,g}=\left(
\begin{array}{c|c}
    u_f M_g & 0 \\\hline
     * &  \ell_f \Tw_{k_f+1}M_g
\end{array}
\right).
\]
Moreover, $\ell_f=\frac{\log_{p,k_f+1}}{\delta_{k_f+1}}$ up to {multiplication by} a unit of {$\Lambda_\cO(\Gamma_1)$}.
\end{proposition}
\begin{proof}
{Thanks to the choice of the basis $\{v_1,v_2,v_3,v_4\}$, the matrix $A_{f,g}$ is of the form
$\left(\begin{array}{c|c}
     \alpha_f^{-1}A_{g}& 0 \\\hline
   \delta_fA_g & \beta_f^{-1}A_g
\end{array}
\right)$.
Consequently, 
\[
A_{f,g,0}=\left(\begin{array}{c|c}
     \alpha_f^{-1}A_{g,0}& 0 \\\hline
   \delta_fA_{g,0} & p^{k_f+1}\beta_f^{-1}A_{g,0}
\end{array}
\right),
\]
which in turn shows
\begin{align*}
P_{f,g}&=A_{f,g,0}\begin{pmatrix}
\xi^{k_f+k_g+2}\\
&\xi^{k_f+1}/q^{k_g+1}\\
&&\xi^{k_g+1}/q^{k_f+1}\\
&&&1/q^{k_f+k_g+2}
\end{pmatrix}\\\\
&=\left(\begin{array}{c|c}
   \alpha_f^{-1}\xi^{k_f+1} P_g & 0 \\\\\hline\\
    \delta_f\xi^{k_f+1} P_g  &p^{k_f+1}(\beta_fq^{k_f+1})^{-1}P_g
\end{array}\right),
\end{align*}
where $P_g$ is the matrix of {$\vp_{\vert_{\NN(R_g)}}$} with respect to the Wach-module basis of {$\NN(R_g)$} obtained by lifting the basis $\{\omega_g,\vp(\omega_g)\}$ via \cite[Proof of Proposition~V.2.3]{berger04}.
We therefore have
\[
P_{f,g}^{-1}=\left(\begin{array}{c|c}
   \alpha_f\xi^{-k_f-1} P_g^{-1} & 0 \\\\\hline\\
    \delta_f\varepsilon_f(p) q^{k_f+1}P_g^{-1}  &p^{-k_f-1}\beta_fq^{k_f+1}P_g^{-1}
\end{array}\right),
\]
which yields
\[
A_{f,g}^{n+1}\vp^n(P_{f,g}^{-1})\cdots \vp(P_{f,g}^{-1})=\left(\begin{array}{c|c}
   u_{f,n}\Xi_n & 0 \\\\\hline\\
    v_{f,n}\Xi_n  & \beta_f^{-1}\frac{\vp(q^{k_f+1})}{p^{k_f+1}}\cdots\frac{\vp^n(q^{k_f+1})}{p^{k_f+1}}\Xi_n 
\end{array}\right),
\]
where $u_{f,n}$ is a unit, $v_{f,n}$ is some linear combination of $\frac{\vp(q^{k_f+1})}{p^{k_f+1}},\ldots,\frac{\vp^n(q^{k_f+1})}{p^{k_f+1}}$ and $\Xi_n$ is given by the product $A_g^{n+1}\vp^n(P_g^{-1}) \cdots \vp(P_g^{-1}) $. }

{In what follows, we write $F\sim G$ if $F$ and $G$ are two elements of $\cH(\Gamma_1)$ such that $F=uG$ for some $u\in \Lambda_\cO(\Gamma_1)^\times$.} By  \cite[Theorem~5.4]{LLZ0} and \cite[Theorem~2.1]{LLZ3},
\begin{align*}
    &\fM^{-1}\left((1+\pi)\left(\prod_{i=1}^n\vp^i(q^{k_f+1})\right)\Xi_n\right)\\
    \sim& A_g^{n+1}\prod_{i=1}^n\fM^{-1}\left((1+\pi)\vp^{i}(q^{k_f+1})\vp^i(P_g^{-1})\right)\\
\sim &A_g^{n+1}\prod_{i=1}^n\left(\prod_{j=0}^{k_f}\Tw_j\Phi_i\right)\Tw_{k_f+1}\fM^{-1}\left((1+\pi)\vp^i(P_g^{-1})\right)\\
\sim& \left(\prod_{i=1}^n\prod_{j=0}^{k_f}\Tw_j\Phi_i\right)\Tw_{k_f+1}\fM^{-1}\left((1+\pi)A_g^{n+1}\prod_{i=1}^n\vp^i(P_g^{-1})\right)
\end{align*}
since the entries of $\vp^i(P_g^{-1})$ are, up to units in $\AA_L^+$, either $\vp^{i}(q^{k_g+1})$ or constants.  The result follows from \cite[(4.2.5)]{BLLV} on passing to limit in $n$. 
\end{proof}

In order to define Coleman maps using the logarithmic matrix $M_{f,g}$ to decompose the Perrin-Riou maps $\cL_{F,f,g}^{(\lambda,\mu)}$, we  introduce the following change of basis matrices, which will allow us to pass between our chosen eigenvector basis $\{v_{\lambda,\mu}\}$ and the integral basis of $\Dcris(T_{f,g})$ studied above.
\begin{defn}\label{defn:Q}
\item[i)]{Let $Q_{f}=\begin{pmatrix}
1&0\\
\frac{\delta_f\alpha_f\beta_f}{\beta_f-\alpha_f}&1
\end{pmatrix}$} be the change of basis matrix from $\{v_{f,\alpha}^*,v_{f,\beta}^*\}$ to $\{\ofs,\etaf\}$, so that we have $Q_{f}^{-1}A_{f}Q_{f}=\begin{pmatrix}
\alpha_f^{-1}&\\&\beta_{f}^{-1}
\end{pmatrix}$.

\item[ii)] We set $
Q_g=\frac{1}{\alpha_g-\beta_g}\begin{pmatrix}
\alpha_g&-\beta_g\\
-\alpha_g\beta_g&\alpha_g\beta_g
\end{pmatrix}$, which is the change of basis matrix from  $\{v_{g,\alpha}^*,v_{g,\beta}^*\}$ to $\{\ogs,\vp(\ogs)\}$, so that we have $Q_g^{-1}A_gQ_g=\begin{pmatrix}
\alpha_g^{-1}&\\&\beta_g^{-1}
\end{pmatrix}$.

\item[iii)] We denote by $Q_{f,g}$ the change of basis matrix from $\{v_{\lambda,\mu}:\lambda,\mu\in\{\alpha,\beta\}\}$ to $\{v_1,v_2,v_3,v_4\}$, so that we have $$Q_{f,g}^{-1}A_{f,g}Q_{f,g}=\begin{pmatrix}
(\alpha_f\alpha_g)^{-1}\\
&(\alpha_f\beta_g)^{-1}\\
&&(\beta_f\alpha_g)^{-1}\\
&&&(\beta_f\beta_g)^{-1}
\end{pmatrix}.$$
\end{defn}

\begin{remark}\label{rk:semiordQ}
Observe that we have
{\[
Q_{f,g}=\left(\begin{array}{c|c}
   Q_g   & 0 \\\hline
    \frac{\delta_f\alpha_f\beta_f}{\alpha_f-\beta_f}Q_g &Q_g
\end{array}\right).
\]}
Combining this observation with Proposition~\ref{prop:semiordlog}, we deduce that\index{Coleman maps! $Q_{f,g}^{-1}M_{f,g}$ (Logarithmic Matrix)}
\[
Q_{f,g}^{-1}M_{f,g}=\left(\begin{array}{c|c}
   u_fQ_g^{-1}M_g   & 0 \\\hline
   * &*
\end{array}\right).
\]
\end{remark}

\begin{proposition}\label{prop:semiordPR}
If $F$ is a finite unramified extension of $\Qp$, there exists a quadruple of bounded Coleman maps\index{Coleman maps! $\col_{F,f,g}^{(\circ,\bullet)}$}
$$\col_{F,f,g}^{(\omega,\#)},\quad \col_{F,f,g}^{(\omega,\flat)},\quad \col_{F,f,g}^{(\eta,\#)},\quad\col_{F,f,g}^{(\eta,\flat)}\quad:\quad \HIw(F(\mu_{p^\infty}),T_{f,g})\lra \Lambda_\cO(\Gamma_\cyc)\otimes \cO_F,$$ which satisfies
\[\begin{pmatrix}
\cL_{F,f,g}^{(\alpha,\alpha)}\\\\\cL_{F,f,g}^{(\alpha,\beta)}\\\\\cL_{F,f,g}^{(\beta,\alpha)}\\\\\cL_{F,f,g}^{(\beta,\beta)}
\end{pmatrix}
=
Q_{f,g}^{-1}M_{f,g}
\begin{pmatrix}
\col_{F,f,g}^{(\omega,\#)}\\\\\col_{F,f,g}^{(\omega,\flat)}\\\\\col_{F,f,g}^{(\eta,\#)}\\\\\col_{F,f,g}^{(\eta,\flat)}
\end{pmatrix}.
\]
In particular,\index{Coleman maps! $Q_g^{-1}M_g$ (Logarithmic Matrix)}
\[\begin{pmatrix}
\cL_{F,f,g}^{(\alpha,\alpha)}\\\\\cL_{F,f,g}^{(\alpha,\beta)}
\end{pmatrix}
=
u_fQ_{g}^{-1}M_{g}
\begin{pmatrix}
\col_{F,f,g}^{(\omega,\#)}\\\\\col_{F,f,g}^{(\omega,\flat)}
\end{pmatrix}.
\]
\end{proposition}
\begin{proof}
The first assertion follows from \cite[Theorem~3.5]{LLZ0}; see also the discussion in \cite[\S5.1]{BLLV}. The second follows from the first, combined with Remark~\ref{rk:semiordQ}.
\end{proof}

We conclude this subsection with the following definition of the semi-local Coleman maps, which we will later use to define local conditions.
\begin{defn}\label{def:semiloc}
Let $\cN$ be the set of positive square-free integers that are coprime to $6pN_fN_g$. For positive integers $m\in\cN$, we let $\QQ(m)/\QQ$ to be the maximal $p$-extension contained in $\QQ(\mu_m)$. For $\circ\in\{\omega,\eta\}$ and $\bullet\in\{\#,\flat\}$, we define the semi-local Coleman maps $\col^{(\circ,\bullet)}_{f,g,m}=\bigoplus_{v|p}\col_{\QQ(m)_v,f,g}^{(\circ,\bullet)}$. When $m=1$, we will simply write $\col_{f,g}^{(\circ,\bullet)}$ in place of $\col_{f,g,1}^{(\circ,\bullet)}$.\index{Coleman maps! $\col_{f,g}^{(\circ,\bullet)}$}\index{Coleman maps! $\col_{f,g,m}^{(\circ,\bullet)}$ (Semi-local Coleman maps)}
\end{defn}

\section{{Coleman maps for a rank-two quotient}}
In this section, we consider the construction of Coleman maps for a rank-two quotient of $T_{f,g}$ under the following weaker version of the hypothesis \ref{item_FL}.
\begin{itemize}
     \item[\mylabel{item_FLg}{{\bf (FL$_{g}$)}}] $R_g^*$ verifies the Fontaine--Laffaille condition $p>k_g+1$.
\end{itemize}

Since $f$ is ordinary at $p$, we may define the following objects:
\begin{defn}
\item[i)]We write $\cF^-R_f^*$ for the unique rank-one quotient  of $R_f^*$, which is an unramified $G_{\Qp}$-representation.
\item[ii)]Let $T_{f,g}^{-\emptyset}$ denote the tensor product $\cF^-R_f^*\otimes R_g^*$ of $G_{\Qp}$-representations.
\end{defn} The advantage of working with the rank-two quotient $T_{f,g}^{-\emptyset}$ is that we do not have to find an integral basis of $\Dcris(R_f^*)$ in order to define our Coleman maps. This  allows us to relax the Fontaine--Laffaille condition \ref{item_FL} to \ref{item_FLg}. Consequently, we may deform these Coleman maps in Hida families, resulting in two-variable version of these maps. We will carry out the construction of these two-variable maps  in  \S\ref{subsec_twovarPRandColeman}.

\begin{remark}
   \item[i)] The Hodge--Tate weights of the rank-two $G_{\Qp}$-representation $T_{f,g}^{-\emptyset}$ are  $0$ and $k_g+1$.
   \item[ii)]$\vp$ acts on $\Dcris(\cF^-R_f^*)$ as multiplication-by-$\alpha_f^{-1}$ and we may identify $\Dcris(\cF^-R_f^*)$ with $\Dcris(R_f^*)/\Dcris(R_f^*)^{\vp=\beta_f^{-1}}$. In particular, the vector $v_{f,\alpha}^*$ in Definition~\ref{def_eigenvectors} gives an $\cO$-basis of $\Dcris(\cF^-R_f^*)$.
   \item[iii)]  If we write $V_{f,g}^{-\emptyset}=\cF^-V_f^*\otimes V_g^*$, where $\cF^-V_f^*=\cF^-R_f^*\otimes\Qp$, it can be seen that $v_{\alpha,\alpha}$ and $v_{\alpha,\beta}$ form a $\vp$-eigenbasis  of $\Dcris(V_{f,g}^{-\emptyset})$. 
   \item[iv)] For $\mu\in\{\alpha,\beta\}$, the projections of the Perrin-Riou map $\cL_{F,f,g}^{(\alpha,\mu)}$  defined in \eqref{eq:projectingPR} factor through $\HIw(F(\mu_{p^\infty}),T_{f,g}^{-\emptyset})$. By an abuse of notation, we will denote the maps on this group by the same symbols.
\end{remark}

Since $g$ is non-ordinary at $p$, the analogous condition of \ref{item_Slo} holds for $T_{f,g}^{-\emptyset}$, meaning that the slope of $\vp$ on $\Dcris(T_{f,g}^{-\emptyset})$ does not attain $-k_g-1$ and $0$ simultaneously. Furthermore, we have the filtered basis formed by $v_{f,\alpha}^*\otimes \omega_{g^\star}$ and $v_{f,\alpha}^*\otimes\vp( \omega_{g^\star})$. 
Together with \ref{item_FLg}, we may apply \cite[proof of Proposition~V.2.3]{berger04} to lift this basis to a basis of the Wach module $\NN(T_{f,g}^{-\emptyset})$. Consequently, we have the following analogue of Proposition~\ref{prop:semiordPR}: 
\begin{proposition}
For a finite unramified extension $F/\Qp$, there exist Coleman maps $$\col_{F,f,g}^{(\omega,\#)},\quad \col_{f,g}^{(\omega,\flat)}:\quad \HIw(F(\mu_{p^\infty},T_{f,g}^{-\emptyset})\longrightarrow \Lambda_\cO(\Gamma_\cyc)\otimes\cO_F$$ 
such that
\[\begin{pmatrix}
\cL_{F,f,g}^{(\alpha,\alpha)}\\\\\cL_{F,f,g}^{(\alpha,\beta)}
\end{pmatrix}
=
u_fQ_{g}^{-1}M_{g}
\begin{pmatrix}
\col_{F,f,g}^{(\omega,\#)}\\\\\col_{F,f,g}^{(\omega,\flat)}
\end{pmatrix},
\]
where $u_f\in\Lambda_\cO(\Gamma_1)^\times$. 
\end{proposition} 
\begin{remark}
    The Coleman maps defined here coincide with two of the four Coleman maps  defined in Proposition~\ref{prop:semiordPR} after composing with the natural projection map
    \[
    \HIw(F(\mu_{p^\infty},T_{f,g})\longrightarrow\HIw(F(\mu_{p^\infty},T_{f,g}^{-\emptyset}).
    \]
    By an abuse of notation, we shall use the same symbols to denote these maps.
\end{remark}

As in Definition~\ref{def:semiloc}, for positive integers $m\in\cN$ and $\bullet\in\{\#,\flat\}$, we have the semi-local Coleman maps $\col^{(\omega,\bullet)}_{f,g,m}=\bigoplus_{v|p}\col_{\QQ(m)_v,f,g}^{(\omega,\bullet)}$ for $m\in \cN$.

\section{Two-variable Perrin-Riou maps and Coleman maps}
\label{subsec_twovarPRandColeman}
Recall from the introduction that $\f$ denotes  a Hida family passing through $f_\alpha$  (we remark that our Hida families are what \cite{KLZ2} calls branches, as explained in Section 7.5 of op. cit.). 
To the best of the authors' knowledge,  a theory of Wach modules for families, which would allow one to decompose a three-variable Perrin-Riou map (where both $f$ and $g$ vary in families), is currently unavailable. {We shall content ourselves with studying  two-variable $\Lambda_\f$-adic Perrin-Riou maps interpolating $\cL_{F,f,g}^{(\alpha,\alpha)}$ and $\cL_{F,f,g}^{(\alpha,\beta)}$ as $f_\alpha$ varies. We will show that these maps can be decomposed into bounded Coleman maps, deforming $\col_{F,f,g}^{(\omega,\#)}$ and $\col_{F,f,g}^{(\omega,\flat)}$ under \ref{item_FLg}.
}

%\begin{remark}
%One advantage of working with only two of the Perrin-Riou maps $\cL_{F,f,g}^{(\alpha,\alpha)}$ and $\cL_{F,f,g}^{(\alpha,\beta)}$ and two of the Coleman maps $\col_{F,f,g}^{(\omega,\#)}$ and $\col_{F,f,g}^{(\omega,\flat)}$  is that we only need an integral basis for the Dieudonn\'e module of a rank-two quotient of $T_{f,g}$, rather than the whole of $T_{f,g}$. In particular, not only can we weaken  hypothesis \ref{item_FL} to \ref{item_FLg}, but  we may also discard the hypothesis \ref{item_theta}.
%\end{remark}

We now introduce a number of objects attached to the family $\f$. Throughout this subsection, $g$ is a fixed non-$p$-ordinary cuspidal eigenform as in the introduction.

\begin{defn}
\item[i)]The prime ideal of $\Lambda_\f$ that corresponds to the specialization of $\f$ to $f_\alpha$ (as well as the specialization itself) will be denoted by $\kappa_0$.
\item[ii)] We write $\alpha_\f$ for the $U_p$-eigenvalue of $\f$.\index{Hida Theory! $\alpha_\f, \alpha_\kappa$} Given a specialization $\kappa$ of $\Lambda_\f$, we write $\f_\kappa$ and $\alpha_\kappa$ for the respective specializations of $\f$ and $\alpha_\f$ at $\kappa$.

\item[iii)] Let $R_{\f}^*$ be the $\Lambda_\f$-adic $G_\QQ$-representation attached to $\f$ as  given in \cite[Definition~7.2.5]{KLZ2}  and let {$\cF^- R_\f^*$}\index{Galois representations!  $\cF^- R_\f^*$} denote the free rank-one $\Lambda_\f$-module given as in Theorem~7.2.8 of op. cit. 
\item[iv)] We put ${T_{\f,g}}:=R_\f^*\otimes R_g^*$\index{Galois representations! $T_{\f,g}$} and $T_{\f,g}^{-\emptyset} :=\cF^-R_\f^*\otimes R_g^*$\index{Galois representations! $T_{\f,g}^{-\emptyset}$}.
\end{defn}

{\begin{remark}
\label{rem_pralphasupstar}
\item[i)] As explained in \cite[Theorem~7.2.3(iii)]{KLZ2}, the $G_{\Qp}$-representation $\cF^- R_\f^*\vert_{G_{\QQ_p}}$ is unramified.
\item[ii)] The specialization of $R_\f^*$ at $\kappa_0$ coincides with the natural $\cO$-lattice $R_{f_\alpha}^*$ contained in Deligne's representation attached to the $p$-stabilized eigenform $f_\alpha$, which is realized as the $f_\alpha$-isotypic Hecke eigenspace inside $H^1_{\textup{\'et}}(Y_1(pN_f)_{\overline{\QQ}},{\rm TSym}^{k_f}(\mathscr{H}_{\ZZ_p})(1))$. We let $R_{f_\alpha}$\index{Galois representations! $R_{f_\alpha}$} denote the lattice we have introduced in Definition~\ref{def_eigenvectors}(i). This is a Galois-stable lattice contained in the $f_\alpha$-isotypic Hecke eigenspace
$$H^1_{\textup{\'et}}(Y_1(pN_f)_{\overline{\QQ}},{\rm Sym}^{k_f}(\mathscr{H}_{\QQ_p}^\vee))[f_\alpha]\subset H^1_{\textup{\'et}}(Y_1(pN_f)_{\overline{\QQ}},{\rm Sym}^{k_f}(\mathscr{H}_{\QQ_p}^\vee))\,.$$ 
\item[iii)] We recall that the lattice $R_f$ is defined in a similar manner, which is realized in the cohomology of $Y_1(N_f)$. The lattices $R_f$ and $R_{f_\alpha}$ are related via a morphism  $(\Pr^\alpha)^*:R_{f}\rightarrow R_{f_\alpha}$, where $(\Pr^\alpha)^*=(\Pr_1)^*-\frac{\beta_f}{p^{k_f+1}}(\Pr_2)^*$ is the map dual to $(\Pr^\alpha)_*:R_{f_\alpha}^*\rightarrow R_f^*$, which was studied in \cite[\S7.3]{KLZ2}. Note that the map $(\Pr^\alpha)^*$ is necessarily injective, since it induces an isomorphism $R_{f}[1/p]\xrightarrow{\sim} R_{f_\alpha}[1/p]$.\index{Galois representations! $(\Pr^\alpha)^*, (\Pr^\alpha)_*$}
\end{remark}}

We now introduce the following $p$-stabilized version of the Perrin-Riou maps for the representation $R_{f_\alpha}^*\otimes R_g^*$, which can be defined under the hypothesis \ref{item_FLg}.
\begin{defn}\label{defn:stabliziedmaps} 
Let $F/\Qp$ a finite unramified extension.
For {$\mu\in\{\alpha,\beta\}$},  we define the $\Lambda_\cO(\Gamma_\cyc)$-morphism
\[
{\cL_{F,f_\alpha,g}^{(\alpha,\mu)}}:\HIw(F(\mu_{p^\infty}),R_{f_\alpha}^*\otimes R_g^*)\longrightarrow {\cH_{\ord_p(\mu_g)}(\Gamma_\cyc)\otimes_{\Zp}F}
\]
as the compositum 
$$\HIw(F(\mu_{p^\infty}),R_{f_\alpha}^*\otimes R_g^*)\xrightarrow{(\Pr^\alpha)_*} \HIw(F(\mu_{p^\infty}),R_{f}^*\otimes R_g^*) \xrightarrow{\cL_{F,f,g}^{(\alpha,\mu)}} \cH_{\ord_p(\mu_g)}(\Gamma_\cyc)\otimes_{\Zp}F$$ 
where the morphism $\cL_{F,f,g}^{(\alpha,\mu)}$ is given in \eqref{eq:projectingPR}.
%\item[ii)]For $\circ\in\{\omega,\eta\}$ and $\bullet\in\{\#,\flat\}$, We define the bounded Coleman maps
%$$\col_{F,f_\alpha,g}^{(\circ,\bullet)}: \HIw(F(\mu_{p^\infty}),R_{f_\alpha}^*\otimes R_g^*)\rightarrow\Lambda_\cO(\Gamma_\cyc)\otimes \cO_F$$
%via the corresponding morphisms defined in Proposition~\ref{prop:semiordPR} similarly.
\end{defn}

{\begin{lemma}
\label{lemma_prsupstar_explicit}
Let $(\Pr^\alpha)^*: R_f \hookrightarrow R_{f_\alpha}$ be the morphism given  as in Remark~\ref{rem_pralphasupstar}. Then $(\Pr^\alpha)^*$ is an isomorphism if at least one of the following conditions holds true: {\rm i)} $k_f>0$, {\rm ii)} $v_p(\alpha_f-\beta_f/p)=0$, {\rm iii)} $\overline{\rho}_f$ is absolutely irreducible. 
\end{lemma}}

\begin{proof}
{In the situation of (i)--(ii), this is essentially \cite[Prop. 7.3.1]{KLZ2} transposed, whereas the case (iii) is \cite[Prop. 4.3.6]{LLZ2} transposed. We go over the details in the setting of (i)--(ii). We shall apply the (``dual'') argument in the proof of \cite[Prop. 7.3.1]{KLZ2} to show that the cokernel of $(\Pr^\alpha)^*$ is annihilated by $\alpha_f-\frac{\beta_f}{p}$. This will suffice to conclude the proof of our lemma when (i) or (ii) holds true.}

{In fact, we shall prove that the map $({\rm pr}_2)_*\circ (\Pr^\alpha)^*: R_{f} \to R_{f}$ is given by multiplication by $\alpha_f-\frac{\beta_f}{p}$. To see that, we directly compute 
\begin{align*}
    ({\rm pr}_2)_*\circ ({\rm Pr}^\alpha)^*&=({\rm pr}_2)_*\circ {\rm pr}_1^*-\frac{\beta_f}{p^{k+1}}({\rm pr}_2)_*\circ {\rm pr}_2^*\\
    &=a_p(f)-\frac{\beta_f(p+1)}{p}=\alpha_f-\beta_f/p
\end{align*}
where the second equality is because $({\rm pr}_2)_*\circ {\rm pr}_1^*=T_p$, where $T_p$ is the Hecke operator at $p$ and $({\rm pr}_2)_*\circ {\rm pr}_2^*=p^{k_f}(p+1)$; whereas the final equality is because $a_p(f)=\alpha_f+\beta_f$. The proof of our lemma is now complete in the situation of (i)--(ii). The case (iii) can be treated following the proof of  \cite[Prop. 4.3.6]{LLZ2} but ``dualizing'' the argument in op. cit. as above. } 
\end{proof}

The $\Lambda_\f$-adic objects we introduce in Definition~\ref{defn_lambdaadicDmodules} below is akin to the constructions of \cite[\S8.2]{KLZ2}.

\begin{defn}
\label{defn_lambdaadicDmodules}
Let $F$ be a finite unramified extension of $\Qp$.
 \item[i)] We define  $\DD(F,\cF^- R_\f^*)=(\cF^- R_\f^*\otimes_{\Zp}\hat\ZZ_p^{\mathrm ur})^{G_F}$, which is a $\Lambda_\f$-module equipped with an operator  $\vp$ induced by the arithmetic Frobenius action on $\ZZ_p^{\mathrm ur}$ and define the $\AA_L^+$-module $\NN(F,\cF^- R_\f^*)=\DD(F,\cF^- R_\f^*)[[\pi]]$, with $\vp$ sending $\pi$ to $(1+\pi)^p-1$. There is a left inverse $\psi$ of $\vp$ induced by the trace operator on $\Zp[[\pi]]\rightarrow\Zp[[\pi]]$.
 \item[ii)] We define  the corresponding objects for $T_{\f,g}^{-\emptyset} $ on setting  $\DD(F,T_{\f,g}^{-\emptyset} )=\DD(F,\cF^- R_\f^*)\otimes_\cO\Dcris(R_g^*) $ and $\NN(F,T_{\f,g}^{-\emptyset} )=\NN(F,\cF^- R_\f^*)\otimes_{\AA_L^+}\NN(R_g^*)$, respectively.
\end{defn}

\begin{remark}\label{rk:explicit}
Explicitly, $\DD(F,\cF^- R_\f^*)=\Lambda_\f(\alpha_\f^{-1})\otimes_{\Zp}\cO_F$, where $\Lambda_\f(\alpha_\f^{-1})$ is the free $\Lambda_\f$-module of rank one on which $\vp$ acts  via the scalar $\alpha_\f^{-1}$.  Consequently, 
\begin{align*}
    \DD(F,T_{\f,g}^{-\emptyset} )&=\Lambda_\f(\alpha_\f^{-1})\otimes_{\cO} \Dcris(F,R_g^*)=\Lambda_\f(\alpha_\f^{-1})\otimes_{\cO} \Dcris(R_g^*)\otimes_{\Zp}\cO_F;\\
    \NN(F,T_{\f,g}^{-\emptyset} )&=\Lambda_\f(\alpha_\f^{-1})\,\widehat{\otimes}_{\cO}\NN(F,R_g^*)=\Lambda_\f(\alpha_\f^{-1})\,\widehat{\otimes}_{\cO}\NN(R_g^*)\otimes_{\Zp}\cO_F.
\end{align*}
\end{remark}

\begin{lemma}\label{lem:Hida-Wach}
There is an isomorphism of $\Lambda_\f(\Gamma_\cyc)$-modules
\[
\HIw(F(\mu_{p_\infty}),T_{\f,g}^{-\emptyset} )\cong \NN(F,T_{\f,g}^{-\emptyset} )^{\psi=1}.
\]
\end{lemma}
\begin{proof}
Our proof follows closely part of the  proof of  \cite[Theorem~8.2.3]{KLZ2}. Since all the modules above commute with inverse limits, we may replace $\cF^- R_\f^*$ by a finitely generated $\Zp$-module $M$ equipped with an unramified action by $G_{\Qp}$.

Since $\NN(F,M\otimes R_f^*)$ is now the usual Wach module of $M\otimes R_f^*$ over $F$,  \cite[Theorem~A.3]{berger03} tells us that 
\[
\HIw(F(\mu_{p^\infty}),M\otimes R_g^*)\cong\left(\pi^{-1}\NN(F,M\otimes R_g^*)\right)^{\psi=1}.
\]
As the Hodge--Tate weights of $M\otimes R_g^*$ are non-negative and $R_g^*$ is non-ordinary, it follows  that $M\otimes R_g^*$ does not  admit a subquotient isomorphic to the trivial representation. Thus, we may  eliminate $\pi^{-1}$ from $\pi^{-1}\NN(M\otimes R_g^*)$ in the isomorphism above, proving the lemma.
\end{proof}
Equipped with the isomorphism given by Lemma~\ref{lem:Hida-Wach}, we may mimic the strategy of \cite{LZ0} and \cite{KLZ2} to define a two-variable Perrin-Riou map on $\HIw(F(\mu_{p_\infty}),T_{\f,g}^{-\emptyset} )$. This is done in the proof of Theorem~\ref{thm:PRHida} below. Pairing with appropriate $\vp$-eigenvectors, which we introduce below, we may deform the maps $\cL_{F,f_\alpha,g}^{(\alpha,\mu)}$ as $f_\alpha$ varies in a Hida family.
\begin{defn}
\item[i)] Let $\mathscr{C}_{f_\alpha}$\index{Congruence number $\mathscr{C}_{f_\alpha}$} denote the cohomological congruence number associated to the eigenform $f_\alpha$. We refer the readers to \cite[(0.3)]{Hida81Congruences} for its definition; see also \cite[Corollary 4.19]{DDT} for an elaboration in the particular case $k_f=0$.
 \item[ii)]Let $I_\f$ be the congruence ideal\index{Hida Theory! Congruence ideal $I_\f$} of $\f$ and let $H_\f\in I_\f$ denote Hida's congruence divisor. We remark that, since we assume that $\f$ is residually non-Eisenstein and $p$-distinguished, the Hecke algebra of $\f$ is Gorenstein by \cite[Corollary~2 of Theorem~2.1]{wiles95}. Thence, $I_\f$ is indeed a principal ideal thanks to \cite[Theorem~0.1]{Hida88AJM}.\index{Hida Theory! Congruence divisor $H_\f$} 
 \item[iii)] Let $\eta_\f:\DD(\Qp,\cF^- R_\f^*)\rightarrow \frac{1}{H_{\f}}\LL_\f$\index{Galois representations! $\eta_\f, \eta_\kappa$} be the $\Lambda_\f$-morphism given by \cite[Proposition~10.1.1(2)]{KLZ2}. Its specialization under $\kappa$ will be denoted by $\eta_\kappa$, which we identify as an element of $\Dcris(V_{\f_\kappa})$ via duality.
\end{defn}

\begin{lemma}\label{lem:e-vec-cong-no}
Let $\eta_f'\in \Dcris({V_{f}})/\Fil^1$ be the unique vector that verifies
\begin{equation}
    \label{eqn_poincare_1}
\langle\eta_f',{\omega_{f}}\rangle_{ Y_1(N_f)}=G(\varepsilon_f^{-1}),
\end{equation}
where $\langle-,-\rangle_{ Y_1(N_f)}$ is a pairing induced by Poincar\'e duality (see the discussion in \cite{KLZ1} just before the statement of Proposition 6.1.3) and $G(\varepsilon_f^{-1})$ denotes the Gauss sum for $\varepsilon_{f}^{-1}$.  Then,
\[
    v_{f,\alpha}\equiv\sC_{f_\alpha}\eta_f'\mod \Fil^1\Dcris({V_{f}})
\]
up to $p$-adic units.
\end{lemma}
\begin{proof}
This fact is implicitly used in \cite[\S11.6]{KLZ2}; we review its proof for completeness\footnote{We thank D. Loeffler for a helpful exchange concerning this point.}.

 Given an $\cO$-basis $\{v_0,\omega_{f}\}$ of $\Dcris(R_{f})$,  the cohomological congruence number $\sC_f$ is given (up to a $p$-adic unit) via the following chain of identities:
\begin{align}
    \label{eqn_poincare_2}
    \notag\langle v_0,\omega_{f^\star}\rangle_{Y_1(N_f)}\cdot\cO &=\langle v_0,W_{N_f}\omega_{f}\rangle_{Y_1(N_f)}\cdot\cO\\
    &=\mathscr{C}_f\cdot \cO\,.
\end{align}
Here, $W_{N_f}$ is the Atkin--Lehner operator and the first equality holds since $W_{N_f}f=\lambda_{N_f}f^*$ and the Atkin--Lehner pseudo-eigenvalue $\lambda_{N_f}$ is a $p$-adic unit, whereas the second equality is one of the definitions of the cohomological congruence number $\mathscr{C}_f$. Note that $G(\varepsilon_f^{-1})$ is a $p$-adic unit since $p$ does not divide the conductor of $\varepsilon_f^{-1}$.  
 Combining \eqref{eqn_poincare_1} and \eqref{eqn_poincare_2}, it follows that $ v_0\equiv\sC_f\eta_f'\mod \Fil^1\Dcris(V_h)$ up to $p$-adic units.
 
 Since ${\rm span}_{\cO}\{\omega_f\}={\rm Fil}^1\Dcris(R_{f})$, we have reduced to showing that
 $$\Dcris(R_{f})=\Dcris(R_{f})^{\varphi=\alpha}\,+\,{\rm Fil}^1\Dcris(R_{f})\,.$$ 
 Since $({\rm Pr}^\alpha)^*$ induces an isomorphism $R_{f}\xrightarrow{\sim} R_{f_\alpha}$ and $({\rm Pr}^\alpha)^*(\omega_{f})=\omega_{f_\alpha}$ (see \cite[Proposition 10.1.1]{KLZ2}), this is equivalent to the requirement that 
 $$\Dcris(R_{f_\alpha})=\Dcris(R_{f_\alpha})^{\varphi=\alpha}\,+\,{\rm Fil}^1\Dcris(R_{f_\alpha})\,.$$
 In alternative wording, it suffices to prove that the natural map
 \begin{equation}
     \label{eqn_Ohta_surjection}
     {\rm Fil}^1\Dcris(R_{f_\alpha})\lra \Dcris(R_{f_\alpha})/\Dcris(R_{f_\alpha})^{\varphi=\alpha}
 \end{equation}
 is surjective. On taking $k=k_f$ and $r=1$ in \cite[Theorem 9.5.2]{KLZ2} (the higher weight generalization of Ohta's integral Eichler--Shimura map) and passing to the $f_\alpha$-isotypic direct summand in both sides of the lower isomorphism (modulo the Eisenstein subspace) in the diagram in \cite[Theorem 9.5.2]{KLZ2}, we deduce that the map \eqref{eqn_Ohta_surjection} is indeed an isomorphism (since $f_\alpha$ is cuspidal) . This concludes our proof.
\end{proof}
\begin{theorem}\label{thm:PRHida}
 For $\mu\in\{\alpha,\beta\}$ and a finite unramified extension $F/\Qp$,  there exists a $\Lambda_\f(\Gamma_\cyc)$-morphism\index{Perrin-Riou maps! $\cL_{F,\f,g, \mu}$} 
\[
\cL_{F,\f,g, \mu}:\HIw(F(\mu_{p^\infty}),T_{\f,g})\rightarrow \Lambda_\f\,\widehat{\otimes}\, \cH_{\ord_p(\mu_g)}(\Gamma_\cyc)\otimes F
\]
whose specialization at $\kappa_0$ equals, up to a $p$-adic unit,  $$\frac{1}{\lambda_{N_f}(f)\left(1-\frac{\beta_f}{p\alpha_f}\right)\left(1-\frac{\beta_f}{\alpha_f}\right)}\cL_{F,f_\alpha,g}^{(\alpha,\mu)},$$ where  $\cL_{F,f_\alpha,g}^{(\alpha,\mu)}$ is  given as in Definition~\ref{defn:stabliziedmaps} and  $\lambda_{N_f}(f)$ is the pseudo-eigenvalue of the Atkin--Lehner operator of level $N_f$ (which is given by the identity $W_{N_f}f=\lambda_{N_f}(f)f^*$).  In particular, the specialization of $\cL_{F,\f,g, \mu}$ at $\kappa_0$ agrees with $\cL_{F,f_\alpha,g, \mu}$ up to multiplication by a $p$-adic unit, if at least one of the following conditions holds true: {\rm i)} $k_f>0$, {\rm ii)} $v_p(\alpha_f-\beta_f/p)=0$, {\rm iii)} $\overline{\rho}_f$ is absolutely irreducible.
\end{theorem}
\begin{proof}
Let us write $\vp^*\NN(F,T_{\f,g}^{-\emptyset} )$ for the $\AA_L^+$-module generated by the image of $\NN(F,T_{\f,g}^{-\emptyset} )$ under $\vp$ and define $\vp^*\NN(M)$ similarly if $M$ is a $G_{\Qp}$-representation whose Hodge--Tate weights are non-negative. The Perrin-Riou map $\cL_{F,M}$ of  $M$ over $F$ is given as the composite map
\begin{align*}
    \HIw(F(\mu_{p^\infty}),M)\xrightarrow{\sim}\NN(F,M)^{\psi=1}&\xrightarrow{1-\vp}\vp^*\NN(F,M)^{\psi=0}\\
    &\hookrightarrow\Dcris(M)\otimes( \Brig)^{\psi=0}\otimes F\\ 
    &\xrightarrow{1\otimes \mathfrak{m}^{-1}\otimes 1} \Dcris(M)\otimes\cH(\Gamma_\cyc)\otimes F,
\end{align*}
where $\Brig$ denotes the set of power series in $L[[\pi]]$ that converge on the open unit disk.

Since $\vp$ acts invertibly on $\Lambda_\f(\alpha_\f^{-1})$, the description of $\NN(F,T_{\f,g}^{-\emptyset} )$ in Remark~\ref{rk:explicit} tells us that
\[
\vp^*\NN(F,T_{\f,g}^{-\emptyset} )=\Lambda_\f(\alpha_\f^{-1})\, \widehat{\otimes}_{\cO}   \vp^*\NN(F,R_g^*).
\]
Thanks to Lemma~\ref{lem:Hida-Wach}, we may define the Perrin-Riou map $\cL_{F,T_{\f,g}^{-\emptyset} }$ for $T_{\f,g}^{-\emptyset} $ over $F$ as the composite map
\begin{align}
\label{eq:comp}
\begin{aligned}
  \HIw(F(\mu_{p^\infty}),T_{\f,g}^{-\emptyset} )\cong \NN(F,T_{\f,g}^{-\emptyset} )^{\psi=1} &\xrightarrow{1-\vp}\vp^*\NN(F,M_{\f.g})^{\psi=0}\\
  &=\Lambda_\f(\alpha_\f^{-1})\, \widehat{\otimes}_{\Zp}   \vp^*\NN(F,R_g^*)\\
&\hookrightarrow\Lambda_\f(\alpha_\f^{-1})\, \widehat{\otimes}_{\Zp} \Dcris(M)\otimes( \Brig)^{\psi=0}\otimes F\\
&=\Dcris(T_{\f,g}^{-\emptyset} )\otimes( \Brig)^{\psi=0} \otimes F\\ 
&\xrightarrow{1\otimes \mathfrak{m}^{-1}\otimes 1} \DD(\Qp,T_{\f,g}^{-\emptyset} )\,\widehat{\otimes}\,\cH(\Gamma_\cyc)\otimes F.    
\end{aligned}
\end{align}

 We now define $\cL_{F,\f,g,\mu}$ by the compositum of the following arrows
\begin{align*}
\HIw(F(\mu_{p^\infty}),T_{\f,g})\lra\HIw(F(\mu_{p^\infty}),T_{\f,g}^{-\emptyset} )\lra &\cH(\Gamma_\cyc)\,\widehat{\otimes}\,\DD(F,T_{\f,g}^{-\emptyset} )\otimes F\\
&\lra  \Lambda_\f\,\widehat{\otimes}\, \cH_{\ord_p(\mu_g)}(\Gamma_\cyc)\otimes F,
\end{align*}
where the first arrow is induced by the natural projection, the second arrow is given by $\cL_{F,T_{\f,g}^{-\emptyset} }$ and the final arrow is the pairing with {${H_\f}\cdot \eta_\f\otimes v_{g,\mu}$}.

{Given an arbitrary element $z\in \HIw(F(\mu_{p^\infty}),R_{f_\alpha}^*\otimes R_g^*)$, set $z_\alpha:=({\rm Pr}^\alpha)_*(z)\in \HIw(F(\mu_{p^\infty}),R_{f}^*\otimes R_g^*)$. Observe then that
\begin{equation}
    \begin{aligned}
   \cL_{F,f_\alpha,g}^{(\alpha,\mu)}(z)&:= \left\langle (\cL_{F,f,g}(z_\alpha), v_{f,\alpha}\otimes v_{g,\mu} 
    \right\rangle_{(N_f,N_g)}\\
   \label{eqn_PRmap_underPralpha} &= \left\langle ({\rm Pr}^\alpha)_*\circ\cL_{F,f_\alpha,g}(z),  v_{f,\alpha}\otimes v_{g,\mu} 
    \right\rangle_{(N_f,N_g)}\\
 &= \left\langle \cL_{F,f_\alpha,g}(z),  ({\rm Pr}^\alpha)^*(v_{f,\alpha})\otimes v_{g,\mu} 
    \right\rangle_{(N_fp,N_g)}
\end{aligned} 
\end{equation}
  where the first equality follows from the functoriality of the construction of the Perrin-Riou map and the second using the fact that $({\rm Pr}^\alpha)_*$ is the transpose of $({\rm Pr}^\alpha)^*$. We further note that the subscript $(N_f,N_g)$ (resp., $(N_fp,N_g)$) signifies the levels of the pair of modular curves where one realizes $f$ and $g$ (resp., $f_\alpha$ and $g$), on which the Poincar\'e duality induces the indicated pairing.}

{Let $\eta_f'\in \Dcris(V_f)/\Fil^1$ be the vector defined in Lemma~\ref{lem:e-vec-cong-no}.
It follows from \cite[Proposition~10.1.1(2)(b)]{KLZ2}  that the   specialization of $H_\f\cdot \eta_\f$ at $\kappa$ is  given by $$\frac{\sC_{f_\alpha}}{\lambda_{N_f}(f)\left(1-\frac{\beta_f}{p\alpha_f}\right)\left(1-\frac{\beta_f}{\alpha_f}\right)}\cdot ({\Pr}^\alpha)^*(\eta_f')\,.$$
The first assertion concerning the specialization of $\cL_{F,\f,g,\mu}$ at $\kappa_0$ now follows on combining this fact with  Lemma~\ref{lem:e-vec-cong-no} and \eqref{eqn_PRmap_underPralpha}.}

The last assertion of the theorem follows from Lemma~\ref{lemma_prsupstar_explicit}, since the constant $\lambda_{N_f}(f)\left(1-\frac{\beta_f}{p\alpha_f}\right)\left(1-\frac{\beta_f}{\alpha_f}\right)$ is a $p$-adic unit under the stated list of hypotheses.
\end{proof}

The following theorem can be considered as a deformed version of  the second half of Proposition~\ref{prop:semiordPR}.
\begin{theorem}\label{thm_decompPRLamdaf}
Suppose that \ref{item_FLg} holds. There exist $\Lambda_\f(\Gamma_\cyc)$-morphisms\index{Coleman maps! $\col_{F,\f,g,\bullet}$}
\[
\col_{F,\f,g,\#},\col_{F,\f,g,\flat}:\HIw(F(\mu_{p^\infty}),T_{\f,g})\lra \Lambda_\f(\Gamma_\cyc)\otimes_{\Zp} \cO_F
\]
such that\index{Coleman maps! $Q_g^{-1}M_g$ (Logarithmic Matrix)}
\[
\begin{pmatrix}
\cL_{F,\f,g, \alpha}\\ \\ \cL_{F,\f,g, \beta}
\end{pmatrix}=Q_g^{-1}M_g\begin{pmatrix}
\col_{F,\f,g,\#}\\ \\ \col_{F,\f,g,\flat}
\end{pmatrix}.
\]
%Furthermore, for each $\bullet\in\{\#,\flat\}$, the specialization of $\col_{F,\f,g,\bullet}$ at the member $f_\alpha$ of the Hida family $\f$ coincides with the map $u_fC_{k_f+2}\col_{F,f,g}^{(\omega,\bullet)}$ given in Proposition~\ref{prop:semiordPR}.
\end{theorem}
\begin{proof}
  Let $\{n_{g,1},n_{g,2}\}$ be the Wach module basis lifting $\{\omega_{g^\star},\vp(\omega_{g^\star})\}$, given by Berger in the proof of \cite[Proposition V.2.3]{berger04}. Then $\{(1+\pi)\vp(n_{g,1}),(1+\pi)\vp(n_{g,2})\}$ is a $\Lambda_\cO(\Gamma_\cyc)$-basis of $\vp^*\NN(R_g^*)^{\psi=0}$ by \cite[Theorem~3.5]{LLZ0}.  Let us define $\pr_\#$ and $\pr_\flat$ to be the projection maps $\vp^*\NN(R_g^*)^{\psi=0}\rightarrow \Lambda_\cO(\Gamma_\cyc)$ given by this basis. 
  We define for $\bullet\in\{\#,\flat\}$ the map $\col_{F,\f,g,\bullet}$ to be the compositum
 \begin{align*}
     \HIw(F,T_{\f,g})&\lra\Lambda_\f(\alpha_\f^{-1})\,\widehat{\otimes}_{\Zp}\vp^*\NN(R_g^*)^{\psi=0}\otimes_{\Zp}\cO_F\\
     &\qquad\qquad\xrightarrow{{H_\f}\cdot {\eta_\f}\otimes \pr_\bullet\otimes 1}\Lambda_\f\,\widehat{\otimes}_{\cO}\Lambda_\cO(\Gamma_\cyc)\otimes_{\Zp}\cO_F=\Lambda_\f(\Gamma_\cyc)\otimes_{\Zp}\cO_F,
 \end{align*}
 where the first arrow is given by \eqref{eq:comp}. The theorem follows from the fact that 
 \[
 \begin{pmatrix}
 (1+\pi)\vp(n_{g,1})&(1+\pi)\vp(n_{g,n})\end{pmatrix}=
 \begin{pmatrix}
 \omega_{g^\star}&\vp(\omega_{g^\star})
 \end{pmatrix}M_g= \begin{pmatrix}
 v_{g,\alpha}^*&v_{g,\beta}^*
 \end{pmatrix}Q_g^{-1}M_g.
 \]
 \end{proof}

\begin{convention}
Instead of $\cL_{\Qp,\f,g,\mu}$ and $\col_{\Qp,\f,g,\bullet}$, we shall  write $\cL_{\f,g,\mu}$ for $\mu\in\{\alpha,\beta\}$ and $\col_{\f,g,\bullet}$ for $\bullet\in\{\#,\flat\}$, respectively.\index{Perrin-Riou maps! $\cL_{\f,g,\mu}$}\index{Coleman maps!  $\col_{\f,g,\bullet}$}
\end{convention}

\begin{defn}[Semi-local $\Lambda_\f(\Gamma_\cyc)$-adic Coleman maps]
For each $m\in\cN$ and $\bullet\in\{\#,\flat\}$, we define the semi-local Coleman maps on setting
$$\col_{\f,g,\bullet,m}:=\bigoplus_{v|p}\col_{\QQ(m)_v,\f,g,\bullet}.$$ 
When $m=1$, we will simply write $\col_{\f,g,\bullet}$ in place of $\col_{\f,g,\bullet,1}$.\index{Coleman maps!  $\col_{\f,g,\bullet,m}$ (Semi-local Coleman maps for families)}
\end{defn}

\section{The case $a_p(g)=0$}
\label{subsec_ap_is_0}
{In the final portion of our paper, we will study in detail the case where $g$ is the $\theta$-series of a Hecke character of an imaginary quadratic field where $p$ remains inert.} In particular,  $a_p(g)$ will vanish. The goal of this section is to investigate the construction of Coleman maps under this condition.  For the rest of the current section, we replace the hypothesis \ref{item_FL} by the following.
\begin{itemize}
     \item[\mylabel{item_SS}{{\bf (FL$_f$-S$_g$)}}]  $p>k_f+1$ and $a_p(g)=0$.
\end{itemize}

Under \ref{item_SS},  we have a very explicit description of $M_g$ in terms of Pollack's plus and minus logarithms. Furthermore, we may allow $g$ to be outside the Fontaine--Laffaille range.
The corresponding objects in Definition~\ref{defn:Mg} are available under \ref{item_SS}, even if $p\le k_g+1$, thanks to the work of Berger--Li-Zhu \cite{bergerlizhu04}.
\begin{defn}
We let $\{n_+,n_-\}$ denote the basis of $\NN(R_g^*)$ constructed in \cite[\S3.1]{bergerlizhu04} (which we have to twist by $\varepsilon_g^{-1/2}$ as in \S4.2 in op. cit., since the authors constructed a $p$-adic representation with trivial central character in \S3.1 of op. cit.). Let $\{v_+,v_-\}$ be the basis of $\Dcris(T_g^*)$ given by $n_\pm\mod \pi$.
\end{defn}
As in the proof of Lemma~\ref{lem:matrices}, we note that the representation considered in op. cit. is in fact $R_g^*(-k_g-1)$. We have the identifications
\begin{align*}
    \NN(R_{g}^*)&\xrightarrow{\times\,\pi^{k_g+1}} \NN(R_{g}^*(-k_g-1)),\\
     \Dcris(R_{g}^*)&\xrightarrow{\times\,t^{ k_g+1}} \Dcris(R_{g}^*(-k_g-1))\,.
\end{align*}
The matrix of $\vp_{|\Dcris(R_g^*)}$ with respect to the basis $\{v_+,v_-\}$ is given by
\[
\begin{pmatrix}
0&\frac{-\eta}{p^{k_g+1}}\\
\eta&0
\end{pmatrix},
\]
where $\eta=\varepsilon_g(p)^{-1/2}$, whereas the matrix of $\vp_{|\NN(R_g^*)}$ with respect to $\{n_+,n_-\}$ is given by
\[
\begin{pmatrix}
0&\frac{-\eta}{q^{k_g+1}}\\
\eta &0
\end{pmatrix}.
\]
Note that $v_+$ generates $\Fil^0\Dcris(R_g^*)$. Therefore, upon multiplying $n_+$ by a suitable unit in $\cO$, we may arrange that $\omega_{g^\star}=v_+$. 

\begin{defn}
We set $n_{g,1}:=u\cdot n_+$, where $u\in\cO^\times$ is chosen so that $n_{g,1}\equiv \omega_{g^\star}\mod \pi$. Furthermore, we put $n_{g,2}:=u\eta\cdot n_-$ and also define $v_{g,i}:=n_{g,i}\mod \pi$ for $i=1,2$.
\item[i)] We let $A_g'$ and $P_g'$ denote the matrices of $\vp_{|\Dcris(R_g^*)}$ and  $\vp_{|\NN(R_g^*)}$ with respect to the bases $\{v_{g,1},v_{g,2}\}$ and $\{n_{g,1}, n_{g,2}\}$, respectively. More explicitly, we have
\[
A_g'=\begin{pmatrix}
0&\frac{-1}{\varepsilon_g(p)p^{k_g+1}}\\
1&0
\end{pmatrix} \quad\text{and}\quad P_g'=\begin{pmatrix}
0&\frac{-1}{\varepsilon_g(p)q^{k_g+1}}\\
1 &0
\end{pmatrix}.
\]
\item[ii)] We finally define the logarithmic matrix associated to the basis $\{n_{g,1},n_{g,1}\}$:
\[
M'_g=\mathfrak{m}^{-1}\left(\lim_{n\rightarrow\infty}(1+\pi)(A_g')^{n+1}\vp^n(P_g')^{-1}\cdots \vp(P_g')^{-1}\right).
\]
\end{defn}

\begin{remark}\label{rk:matrices-ap=0}
    \item[i)]It follows from the calculations in  \cite[\S5.2.1]{LLZ0} that
    \[
    M_g'=\begin{pmatrix}
    0&\frac{a_g^-}{p^{k_g+1}}\log_{p,k_g+1}^-\\
   a_g^+\log_{p,k_g+1}^+&0
    \end{pmatrix},
    \]
    where $a_g^\pm\in\Lambda_\cO(\Gamma_1)^\times$.
    \item[ii)] Note that $A_g'$ coincides with $A_g$ given in Definition~\ref{defn:Mg}. However, the matrix $P_g'$ differs from $P_g$ given in the proof of Proposition~\ref{prop:semiordlog} slightly (one of the entries would differ by the scalar $\xi^{k_g+1}$). The resulting matrix $M_g'$ is therefore  slightly different from $M_g$. But it can be shown under the hypothesis \ref{item_SS} that the matrix {$M_g$} also has the form $$M_g=\begin{pmatrix}
    0&\frac{b_g^-}{p^{k_g+1}}\log_{p,k_g+1}^-\\
   b_g^+\log_{p,k_g+1}^+&0
    \end{pmatrix},$$ 
    where $b_g^\pm\in \Lambda_\cO(\Gamma_1)^\times$.
\item[iii)] Under the hypothesis \ref{item_SS}, the matrix $Q_g$ simplifies to
\[
\frac{1}{2}\begin{pmatrix}
1&1\\
\alpha_g&-\alpha_g
\end{pmatrix}\,.
\]
 \end{remark}

We now turn our attention to the representation attached to $f$. Even though we have assumed that $f$ itself satisfies the Fontaine--Laffaille condition in \ref{item_SS},  neither \cite[Proof of Proposition~V.2.3]{berger04} nor the construction of Wach modules in \cite{bergerlizhu04} applies to $R_f^*$.  We content ourselves with the following non-canonical Wach-module basis for $\NN(R_f^*)$.

\begin{defn}
Let $\{n_{f,1},n_{f,2}\}$ be an $\AA_L^+$-basis of $\NN(R_f^*)$ lifting the basis $\{\omega_{f^\star},\eta_{f^\star}\}$ as in the proof of \cite[Proposition~3.36]{LLZ0}. We note that $\Dcris(R_f^*)^{\vp=\beta_f^{-1}}$ corresponds to the one-dimensional $G_{\Qp}$-stable subspace in $\Dcris(R_f^*(-k_f-1))$ considered in \S3.6 of op. cit.

On multiplying the basis elements $n_{f,1}$ and $n_{f,2}$ by units if necessary, the matrix of $\vp_{|_{\NN(R_f^*)}}$  in this basis can be chosen so as to satisfy
\[
 P_f'=\begin{pmatrix}
\alpha_f^{-1}&0\\
* &\frac{*}{q^{k_f+1}}
\end{pmatrix}.
\]
\end{defn}

These bases then give rise to the following bases for the representation $T_{f,g}=R_f^*\otimes R_g^*$.

\begin{defn}
We define an $\AA_L^+$-basis $\{n_i': i=1,\cdots,4\}$ of $\NN(T_{f,g})$ by setting
\[
n_1'=n_{f,1}\otimes n_{g,1},\quad n_2'=n_{f,1}\otimes n_{g,1},\quad n_3'=n_{f,2}\otimes n_{g,1},\quad n_4'=n_{f,2}\otimes n_{g,2}.
\]
We {similarly} define the vectors $v_1',v_2',v_3'$ and $v_4'$, to give a basis of $\Dcris(T_{f,g})$.

We let $A_{f,g}'$ and $P_{f,g}'$ denote the matrices of $\vp_{|_{\Dcris(R_{f,g}^*)}}$ and  $\vp_{|_{\NN(R_{f,g}^*)}}$  with respect to these bases and set
\[
M'_{f,g}:=\fm^{-1}\left(\lim_{n\rightarrow\infty}(1+\pi)(A_{f,g}')^{n+1}\vp^n(P_{f,g}')^{-1}\cdots \vp(P_{f,g}')^{-1}\right).
\]
\end{defn}

Similar to Proposition~\ref{prop:semiordlog}, it can be shown that there exist elements $u_f'\in \Lambda_\cO(\Gamma_1)^\times$ and $\ell_f'\in \cH_{k_f+1}(\Gamma_1)$ such that\index{Coleman maps! The matrix $M_{f,g}'$}
\begin{equation}
\label{eqn_Mfgprime}
 M_{f,g}'=\left(
\begin{array}{c|c}
    u_f' M_g' & 0 \\\hline
     * &  \ell_f' \Tw_{k_f+1}M_g'
\end{array}
\right),   
\end{equation}
where $\ell_f'$ is equal to $\frac{\log_{p,k_f+1}}{\delta_{k_f+1}(\gamma_0)}$ up to {multiplication by} a unit in {$\Lambda_\cO(\Gamma_1)$}.

Let $Q_{f,g}$ be the change of basis given as in Definition~\ref{defn:Q}(iii). As in Proposition~\ref{prop:semiordPR}, the basis $\{n'_i:i=1,\cdots,4\}$ determines the following Coleman maps.

\begin{proposition}\label{prop:PRpm}
If $F$ is a finite unramified extension of $\Qp$, there exist bounded Coleman maps $\col_{F,f,g}^{(\omega,+)}$, $\col_{F,f,g}^{(\omega,-)}$, $\col_{F,f,g}^{(\eta,+)}$ and $\col_{F,f,g}^{(\eta,-)}$ 
with source $\HIw(F(\mu_{p^\infty}),T)$ and target $\Lambda_\cO(\Gamma_\cyc)\otimes{\cO_F}$, which satisfy
\index{Coleman maps! $Q_{f,g}^{-1}M_{f,g}'$}
\[\begin{pmatrix}
\cL_{F,f,g}^{(\alpha,\alpha)}\\\\\cL_{F,f,g}^{(\alpha,\beta)}\\\\\cL_{F,f,g}^{(\beta,\alpha)}\\\\\cL_{F,f,g}^{(\beta,\beta)}
\end{pmatrix}
=
Q_{f,g}^{-1}M_{f,g}'
\begin{pmatrix}
\col_{F,f,g}^{(\omega,+)}\\\\\col_{F,f,g}^{(\omega,-)}\\\\\col_{F,f,g}^{(\eta,+)}\\\\\col_{F,f,g}^{(\eta,-)}
\end{pmatrix}.
\]
In particular,\index{Coleman maps! $Q_g^{-1}M_g'$}
\[\begin{pmatrix}
\cL_{F,f,g}^{(\alpha,\alpha)}\\\\\cL_{F,f,g}^{(\alpha,\beta)}
\end{pmatrix}
=
u'_fQ_{g}^{-1}M_{g}'
\begin{pmatrix}
\col_{F,f,g}^{(\omega,+)}\\\\\col_{F,f,g}^{(\omega,-)}
\end{pmatrix}.
\]
\end{proposition}
{
We note that the construction of the $\Lambda_\f(\Gamma_\cyc)$-adic Perrin-Riou maps in Theorem~\ref{thm:PRHida} does not rely on the hypothesis \ref{item_FLg} and can be carried over under  the following weaker version of \ref{item_SS}:
\begin{itemize}
     \item[\mylabel{item_SSf}{{\bf (S$_g$)}}] $g$ satisfies $a_p(g)=0$,
\end{itemize}}
yielding the maps $\cL_{F,\f,g,\alpha}$ and $\cL_{F,\f,g,\beta}$. The proof of Theorem~\ref{thm_decompPRLamdaf} also generalizes easily to yield the following:
\begin{theorem}
\label{thm_decompose_PR_SSf}
Suppose \ref{item_SSf} holds. There exist a pair of $\Lambda_\f(\Gamma_\cyc)$-morphisms
\[
\col_{F,\f,g,\pm}:\HIw(F(\mu_{p^\infty}),T_{\f,g})\lra \Lambda_\f(\Gamma_\cyc)\otimes\cO_F
\]\index{Coleman maps! $\col_{F,\f,g,\bullet}$}
that verify the factorization
\[
\begin{pmatrix}
\cL_{F,\f,g, \alpha}\\\\ \cL_{F,\f,g, \beta}
\end{pmatrix}=Q_g^{-1}M_g'\begin{pmatrix}
\col_{F,\f,g,+}\\\\ \col_{F,\f,g,-}
\end{pmatrix}.
\]
\end{theorem}

%%%%%%%%%%%
%%%%%%%%%%%

%-----------------------------------------------------------------------
% End of chap2
%-----------------------------------------------------------------------

%% file: InertOrdMainChapter3.tex
\chapter{Beilinson--Flach elements and $p$-adic $L$-functions}
\label{sec_BF_and_padicL}
In this chapter, we recall the unbounded Beilinson--Flach elements for the representations $T_{\f,g}=R_\f^*\otimes R_g^*$ and explain how they can be decomposed into bounded elements via the logarithmic matrices defined in Chapter\ \ref{S:Coleman}. Furthermore, we study how these elements are  related to various $p$-adic $L$-functions via the Perrin--Riou maps and Coleman maps given in the previous chapter.

\section{Beilinson--Flach elements}
\label{subsec_review_BF_elements}
We begin this section by  explaining how to obtain $\Lambda_\f$-adic Beilinson--Flach elements for $T_{\f,g}$ utilizing results in \cite{KLZ2,LZ1}. Our arguments follow very closely those presented in \cite[\S3]{BFSuper} (with the role of $f$ and $\g$ in op. cit. played by $g$ and $\f$ respectively).

\begin{defn}
Suppose $\mu\in\{\alpha,\beta\}$ and set $N=\mathrm{lcm}(N_f,N_g)$. Let $M\ge1$ be an integer such that $c$ is coprime to $6pMN$. 
\item[i)] We write\index{Beilinson--Flach elements! $z_{M,j}^{\f,\mu}$}  $z_{M,j}^{\f,\mu}\in H^1(\QQ(\mu_{M}),R_\f^*\otimes R_{g_\mu}^*(-j))$ for the image of the Rankin--Iwasawa class $\mathcal{RI}_{M,pMN,1}$ of \cite[Definiton~3.2.1]{LZ1} under the compositum of the arrows
\begin{align*}
&H^3_{\et}\left(Y(M,pMN)^2,\Lambda(\mathscr{H}_{\Zp}\langle t\rangle)^{[j,j]}(2-j)\right)\\
&\longrightarrow H^1\left(\QQ(\mu_M), H^1_{{\et}}\left(\overline{Y_1(pN)},\Lambda(\mathscr{H}_{\Zp}\langle t\rangle)^{[j]}\right)^{\boxtimes 2}(2-j)\right)\\
&\longrightarrow H^1\left(\QQ(\mu_M),H^1_{{\et}}\left(\overline{Y_1(pN)},\Lambda(\mathscr{H}_{\Zp}\langle t\rangle)(1)\right) \boxtimes H^1_{{\et}}\left(\overline{Y_1(pN)},{\rm TSym}^{k_g}\mathscr{H}_{\Zp}(1)\right) (-j)\right)\\
&\longrightarrow H^1\left(\QQ(\mu_M),  H^1_{\ord}(Np^\infty)^{[k_g]}\otimes R_{g_\mu}^*(-j)\right)\\
&\longrightarrow H^1\left(\QQ(\mu_M),  R_{\f}^{*,[k_g]}\otimes R_{g_\mu}^*(-j)\right),
\end{align*}
which are given as in \cite[(17)]{BFSuper}, composed with the projection 
$$H^1_{\ord}(Np^\infty)^{[k_g]}\lra R_{\f}^{*,[k_g]}\,,$$ where $H^1_{\ord}(Np^\infty)^{[k_g]}$ is defined as in \S3.1 of op. cit. Here, the superscript $[k_g]$ means that we have twisted the representation $R_{\f}^*$ by the $k_g^{\rm{th}}$ power of the weight character. We also note that  we have dropped the left-subscript $c$ from the notation (see  Remark~3.1 of op. cit.).

\item[ii)]When $M=mp^n$ where $n\ge1$, we define\index{Beilinson--Flach elements! $\x_{M,j}^{\f,\mu}$}  $\x_{M,j}^{\f,\mu}\in H^1(\QQ(\mu_{M}),R_f^{*,[k_g]}\otimes W_{g_\mu}^*(-j))$ to be the image of 
\[
\frac{{(U_p')^{-n}\times (U_p')^{-n}}}{(-1)^j j!\binom{k_g}{j}^2}\mathcal{RI}_{M,pMN,1} 
\]
under the same series of maps in (i). 
\end{defn}

\begin{theorem}
\label{thm_interpolate_cyc_BF}
Let $m\in\cN$ and $\mu\in\{\alpha,\beta\}$. There exists a unique element
\[
\cBF_m^{\f,\mu}\in H^1\left(\QQ(\mu_m),R_\f^{*,[k_g]}\otimes R_{g_\mu}^*\widehat{\otimes}\cH_{\ord(\mu_g)}(\Gamma_\cyc)^\iota\right)
\]
such that its image in $H^1\left(\QQ(\mu_{mp^n}),R_\f^{*,[k_g]}\otimes W_{g_\mu}^*(-j)\right)$ equals $\x^{\f,\mu}_{mp^n,j}$ for all $0\le j\le k_g$ and $n\ge1$.
\end{theorem}
\begin{proof}
This follows from the same proof as \cite[Theorem~3.2]{BFSuper}. Note that the requirement in the equation (18) of op. cit. translates to 
$$H^0(\QQ(\mu_{mp^\infty}),R_{\f}^{*,[k_g]}\otimes R_{g_\mu}^*(-j))=0,$$ which follows as a consequence of the fact that $g$ is non-ordinary at $p$, whereas $\f$ is a $p$-ordinary family.
\end{proof}

\begin{defn}\label{defn:BF-Hida}
We define 
$$\BF_{\f,g,\mu,m}\in H^1\left(\QQ(\mu_m),R_\f^{*}\otimes R_{g_\mu}^*\widehat{\otimes}\cH_{\ord(\mu_g)}(\Gamma_\cyc)^\iota\right)$$ 
\index{Beilinson--Flach elements! $\BF_{\f,g,\mu,m}$} 
as the image of $\cBF_m^{\f,\mu}$ given in Theorem~\ref{thm_interpolate_cyc_BF} under the weight-twisting morphism 
\begin{equation*}
 \resizebox{\displaywidth}{!}{
$H^1\left(\QQ(\mu_m),R_\f^{*,[k_g]}\otimes R_{g^\mu}^*\widehat{\otimes}\cH_{\ord(\mu_g)}(\Gamma_\cyc)^\iota\right)\lra H^1\left(\QQ(\mu_m),R_\f^{*}\otimes R_{g_\mu}^*\widehat{\otimes}\cH_{\ord(\mu_g)}(\Gamma_\cyc)^\iota\right).$}
\end{equation*}
\end{defn}

We recall from \cite[Theorem~5.4.1]{BLLV} that the logarithmic matrix  constructed using Berger's Wach module basis when both $f$ and $g$  are $p$-non-ordinary leads to a partial splitting of unbounded Beilinson--Flach elements associated to the Rankin--Selberg products. 
The reason why we fell short of establishing the full splitting Conjecture~5.3.1 in op. cit. is the lack of information on the Beilinson--Flach elements associated to the twists $T_{f,g}(-j)$ for $j>\max(k_f,k_g)$). When $f$ is $p$-ordinary, we may allow $f_\alpha$ to vary in the Hida family $\f$ to by-pass this restriction imposed by the weights and obtain a full decomposition of the unbounded Beilinson--Flach elements $\{\BF_{\f,g,\mu,m}\}_{m\in\cN}$, $\mu\in\{\alpha,\beta\}$ into bounded ones: 

\begin{theorem}
\label{thm_BF_factorization_for_families}
\item[i)]Suppose that \ref{item_FLg} holds. For $m\in\cN$, there exists a pair of cohomology classes 
$$\BF_{\f,g,\#,m},\BF_{\f,g,\flat,m}\in \varpi^{-s(g)}H^1\left(\QQ(\mu_m),R_\f^*\otimes R_g^*\widehat{\otimes}\Lambda_\cO(\Gamma_\cyc)^\iota\right)$$ 
such that\index{Beilinson--Flach elements! $\BF_{\f,g,\#,m},\BF_{\f,g,\flat,m}$} 

\[
\begin{pmatrix}
\BF_{\f,g,\alpha,m}\\\\ \BF_{\f,g,\beta,m}
\end{pmatrix}
=Q_g^{-1}M_g
\begin{pmatrix}
\BF_{\f,g,\#,m}\\\\ \BF_{\f,g,\flat,m}
\end{pmatrix},
\]
where $Q_g$ and $M_g$ are the matrices given in Definitions~\ref{defn:Q} and \ref{defn:Mg}, respectively.
Here, $s(g)$ is a natural number that depends on $k_g$ but is independent of $m$.
\item[ii)]{Suppose that \ref{item_SSf} holds.} Then the statement holds if we replace $M_g$, $\#$ and $\flat$ by $M_g'$, $+$ and $-$ respectively.
\end{theorem}
\begin{proof}
Let us first consider part (i). It is enough to show that  $\cBF_m^{\f,\alpha}$ and $\cBF_m^{\f,\beta}$ verify the analogous factorization. In more precise wording, it suffices to show that there exists a pair 
$$\cBF_m^{\f,\#},\cBF_m^{\f,\flat}\in \varpi^{-s(g)}H^1\left(\QQ(\mu_m),R_\f^{*,[k_g]}\otimes R_g^*\widehat{\otimes}\Lambda_\cO(\Gamma_\cyc)^\iota\right)$$ 
such that
\[
\begin{pmatrix}
\cBF_m^{\f,\alpha}\\\\ \cBF_m^{\f,\beta}
\end{pmatrix}
=Q_g^{-1}M_g
\begin{pmatrix}
\cBF_m^{\f,\#}\\\\ \cBF_m^{\f,\flat}
\end{pmatrix}.
\]
This follows from the same proof as \cite[Theorem~3.7]{BFSuper}, where the Hida family there is assumed to be CM, but the CM condition in fact plays no role in the argument. We overview the main steps of the proof of \cite[Theorem~3.7]{BFSuper}.  For $i=0,\ldots, k_g$ and $\mu\in\{\alpha,\beta\}$,   consider the cocycle 
\[
c_{n,i}^{\f,\mu}\in Z^1\left(G_{\QQ(\mu_{mp^\infty})},R_\f^{*,[k_g]}\otimes R_g^*\right)\otimes \cO[\Gamma/\Gamma^{p^{n-1}}]
\]
lifting $\res\left(\mu_g^n \x_{mp^n,i}^{\f,\mu}\right)$  as in op. cit., which satisfies
\begin{equation}\label{eq:boundedtwists}
    \left|\left|p^{-jn}\sum_{i=0}^j(-1)^i\binom{j}{i}1\otimes\Tw^{-i} c_{n,i}^{\lambda,\g}\right|\right|\le 1.
\end{equation}
Using Lemma~2.3 and Remark~2.4 in op. cit., we obtain the bounded cocycle 
\[
c_{n}^{\f,\mu}\in Z^1\left(G_{\QQ(\mu_{mp^\infty})},R_\f^{*,[k_g]}\otimes R_g^*\right)\otimes \varpi^{-s_0(g)}\Lambda_\cO(\Gamma_\cyc)/\omega_{n-1,k_g+1}
\]
whose image modulo $\Tw^{-i}\omega_{n-1}$ is {$c_{n,i}^{\f,\mu}$}. This in turn gives rise to a pair of cohomology classes
\[
\varpi^{s_0(g)}{\mu^n\x_n^{\f,\mu}}\in H^1(\QQ(\mu_{mp^\infty}),R_\f^{*,[k_g]}\otimes R_g^* )\otimes \Lambda_\cO(\Gamma_\cyc)/\omega_{n-1,k_g+1}\,,\hspace{.6cm}\mu\in \{\alpha,\beta\}\,.
\]

Let $C_{n-1}$ denote the matrix
\[
\mathfrak{m}^{-1}\left((1+\pi)\vp^{n-1}(P_g^{-1})\cdots \vp(P_g^{-1})\right).
\]
The interpolative properties of the Rankin--Iwasawa classes  given by  Corollary~3.4  in op. cit. allow us to apply \cite[Proposition~2.10]{BFSuper} to obtain a pair of bounded classes 
$$\x_{\#,mp^n},\x_{\flat,mp^n}\in H^1\left(\QQ(\mu_{mp^\infty}),R_\f^{*,[k_g]}\otimes R_{g^\mu}^*\otimes \varpi^{-s_1(g)}\Lambda_\cO(\Gamma_\cyc)^{\oplus 2}/\ker h_{n-1}\right)$$ 
(where $h_{n-1}$ denotes the map $\Lambda_\cO(\Gamma_\cyc)^{\oplus 2}\rightarrow \Lambda_\cO(\Gamma_\cyc)^{\oplus 2}$ given by the matrix $C_{n-1}$), satisfying
\[
\begin{pmatrix}
\alpha_g^n\x_{mp^n}^{\f,\alpha}\\\\\beta_g^n\x_{mp^n}^{\f,\beta}
\end{pmatrix}=Q_g^{-1}C_{n-1}\begin{pmatrix}
\x_{mp^n}^{\f,\#}\\\\ \x_{mp^n}^{\f,\flat}
\end{pmatrix},
\]
The result then follows on passing to limit in $n$ and setting $\cBF_m^{\f,\bullet}:=\lim_n \x_{mp^n}^{\f,\bullet}$ for $\bullet\in \{\#,\flat\}$ with $s(g)=s_0(g)+s_1(g)$.

We now consider part (ii). As in the first portion, for each $\mu\in\{\alpha,\beta\}$, we have bounded  {cocycles} ${c_n^{\f,\mu}}$ interpolating $\res\left(\mu_g^n\x^{\f,\mu}_{mp^n,i}\right)$  since the inequality \eqref{eq:boundedtwists} is still valid. Observe in addition that \cite[Corollary~3.4]{BFSuper} also holds without the hypothesis \ref{item_FLg}. Notice that we now have $\alpha_g=-\beta_g$. This tells us that for $1\le m\le n-1$, the cocycle
\[
c_n^{\f,\alpha}\pm c_n^{\f,\beta}  \in Z^1\left(G_{\QQ(\mu_{mp^\infty})},R_\f^{*,[k_g]}\otimes R_g^*\right)\otimes \varpi^{-s_0(g)}\Lambda_\cO(\Gamma_\cyc)/\omega_{n-1,k_g+1}
\]
is divisible by the polynomial $\Phi_{m,k_g+1}$, where the parity of $m$ determines the sign above. This gives rise to a pair of bounded cocycles $c_n^{\f,\pm}$ satisfying 
\begin{equation}\label{eq:pm-cocycles}
    \begin{pmatrix}
1&-1\\ \\ 1&1
\end{pmatrix}
\begin{pmatrix}
c_n^{\f,\alpha}\\ \\ c_n^{\f,\beta}
\end{pmatrix}=
\begin{pmatrix}
\omega_{n,k_g+1}^+c_n^{\f,+}\\ \\ \omega_{n,k_g+1}^-c_n^{\f,-}
\end{pmatrix},
\end{equation}
where $\omega_{n,k_g+1}^\pm=\prod\Phi_{n,k_g+1}$ with the products running over all even (respectively odd) integers between 1 and $n-1$ for the sign $+$ (respectively $-$).

It follows from Remark~\ref{rk:matrices-ap=0} that
\[
Q_g^{-1}M_g'\equiv \frac{1}{\alpha_g^n}\begin{pmatrix}
\sigma_n^+\omega_{n,k_g+1}^+&\sigma_n^-\omega_{n,k_g+1}^+\\\\
-\sigma_n^+\delta_n\omega_{n,k_g+1}^+&\sigma_n^-\omega_{n,k_g+1}^+
\end{pmatrix}\mod \omega_{n-1,k_g+1}\,,
\]
where $\sigma_n^\pm$ are bounded polynomials. Combining this with \eqref{eq:pm-cocycles} gives
\[
\begin{pmatrix}
c_n^{\f,\alpha}\\ \\ c_n^{\f,\beta}
\end{pmatrix}=
Q_g^{-1}M_g'
\begin{pmatrix}
d_n^{\f,+}\\ \\ d_n^{\f,-}
\end{pmatrix},
\]
where $d_n^{\f,\pm}$ are bounded cocycles. This in turn produces a pair of bounded cohomology classes $\x_{mp^n}^{\f,\pm}$ represented by $d_{mp^n}^{\f,\pm}$ which verify
\[
\begin{pmatrix}
\x_{mp^n}^{\f,\alpha}\\ \\\x_{mp^n}^{\f,\beta}
\end{pmatrix}=Q_g^{-1}M_g'\begin{pmatrix}
\x_{mp^n}^{\f,+}\\\\ \x_{mp^n}^{\f,-}
\end{pmatrix}.
\]
The result now follows on setting $\BF_{\f,g,\pm,m}$ as the limit of $\x_{mp^n}^{\f,\pm}$ as $n$ tends to infinity.
\end{proof}

\begin{defn}\label{defn:boundedBF}
Suppose $\bullet\in\{\#,\flat\}$ under the hypothesis \ref{item_FLg} or else $\bullet\in\{+,-\}$ under \ref{item_SSf}. We shall write $\BF^{(\alpha,\bullet)}_{f,g,m}$ for the specialization of the element $\BF_{\f,g,\bullet,m}$\index{Beilinson--Flach elements! $\BF^{(\alpha,\bullet)}_{f,g,m}$} 
under $\kappa_{f_\alpha}: \LL_\f\to \cO$ associated to the member $f_\alpha$ of the Hida family $\f$. Similarly, for $\lambda\in\{\alpha,\beta\}$, we write $\BF^{(\alpha,\lambda)}_{f,g,m}$ for the corresponding specialization of $\BF_{\f,g,\lambda,m}$.
\end{defn}
We remark that we then automatically deduce from Theorem~\ref{thm_BF_factorization_for_families} the factorizations
\begin{equation}
    \label{eq:BFdecomp}
    \begin{pmatrix}
\BF_{f,g,m}^{(\alpha,\alpha)}\\\\ \BF_{f,g,m}^{(\alpha,\beta)}
\end{pmatrix}
=Q_g^{-1}M_g
\begin{pmatrix}
\BF_{f,g,m}^{(\alpha,\#)}\\\\ \BF_{f,g,m}^{(\alpha,\flat)}
\end{pmatrix}, \quad 
\begin{pmatrix}
\BF_{f,g,m}^{(\alpha,\alpha)}\\\\ \BF_{f,g,m}^{(\alpha,\beta)}
\end{pmatrix}
=Q_g^{-1}M_g'
\begin{pmatrix}
\BF_{f,g,m}^{(\alpha,+)}\\
\\ \BF_{f,g,m}^{(\alpha,-)}
\end{pmatrix}
\end{equation}
for every specialization $f_\alpha$ of the Hida family $\f$ and integer $m\in \cN$. %Analogous factorizations hold true for the $\pm$ classes, under the hypothesis \ref{item_SS}.

\section{Euler systems of rank $2$ and uniform integrality}
\label{subsubsec_ESrank2}
The goal of the current section is to analyze the variation of exponent $s(g)$ with the eigenform $g$. Our main result in this vein is Corollary~\ref{cor_s_g_is_uniformly_bounded} below.

This analysis is relevant to our attempt towards the Iwasawa main conjectures for $f_{/K}\otimes\psi$, where $K$ is an imaginary quadratic field where $p$ is inert, $f_{/K}$ is the ``base-change'' of $f$ to $K$ and $\psi$ is a Hecke character of $K$. See \S\ref{subsec_IMCimagquadinert} for a detailed discussion, more particularly Remark~\ref{rem_s_g_uniform_bound_required} and \S\S\ref{subsubsec_3var_inert_ord}--\ref{subsubsec_anticyclo_3} for the relevance of our discussion in \S\ref{subsubsec_ESrank2}.

The following conjecture, which is in line with Perrin-Riou's philosophy~\cite{pr-es}, is the $\LL_\f$-adic version of \cite[Conjecture 3.5.1]{BLLV}. We retain the notation of \S\ref{subsec_twovarPRandColeman} and the present \S\ref{subsec_review_BF_elements}.

\begin{conj}
\label{conj_ESrank2}
Suppose $m\in \cN$. There exists a unique element
\[
\mathbb{BF}^{\f,g}_m\in {\bigwedge}^2\,H^1\left(\QQ(\mu_m),T_{\f,g}\widehat{\otimes}\Lambda_\cO(\Gamma_\cyc)^\iota\right)
\]
such that for $\mu=\alpha,\beta$, we have
$$j_\mu\circ \cL_{\QQ(\mu_m)_p,\f,g,\mu}\left(\mathbb{BF}^{\f,g}_m\right)=\BF_{\f,g,\mu,m}\,.$$
Here, we have regarded the functional $\cL_{\QQ(\mu_m)_p,\f,g,\mu}$ as a map
\begin{align*}
    {\bigwedge}^2\,H^1&\left(\QQ(\mu_m),T_{\f,g}\widehat{\otimes}\Lambda_\cO(\Gamma_\cyc)^\iota\right)\xrightarrow{
    \cL_{\QQ(\mu_m)_p,\f,g,\mu}}  \\
    &\quad\qquad\qquad\qquad H^1\left(\QQ(\mu_m),T_{\f,g}\widehat{\otimes}\Lambda_\cO(\Gamma_\cyc)^\iota\right){\otimes}_{\Lambda_\cO(\Gamma_\cyc)}\cH_{\ord(\mu_g)}(\Gamma_\cyc)
\end{align*}
$$\cL_{\QQ(\mu_m)_p,\f,g,\mu}\left(x\otimes y \right):=\cL_{\QQ(\mu_m)_p,\f,g,\mu}(\res_p(x))\cdot y-\cL_{\QQ(\mu_m)_p,\f,g,\mu}(\res_p(y))\cdot x$$
and $j_\mu$ is stands for the compositum of the arrows 
\begin{align*}
    H^1\left(\QQ(\mu_m),T_{\f,g}\widehat{\otimes}\Lambda_\cO(\Gamma_\cyc)^\iota\right)&\widehat{\otimes}_{\Lambda_\cO(\Gamma_\cyc)}\cH_{\ord(\mu_g)}(\Gamma_\cyc)\lra\\
    & H^1\left(\QQ(\mu_m),T_{\f,g}\widehat{\otimes}\cH_{\ord(\mu_g)}(\Gamma_\cyc)^\iota\right)  \xrightarrow{{\rm id}_\f\times ({\rm pr}^\mu)^*}\\
    &\qquad\qquad\qquad \qquad H^1\left(\QQ(\mu_m),T_{\f,g_\mu}\widehat{\otimes}\cH_{\ord(\mu_g)}(\Gamma_\cyc)^\iota\right)\,.
\end{align*}
\end{conj}

See \cite{BL_ESrank2_Sym, KBOchiai_2} for results in support of this conjecture. 

Assuming the validity of Conjecture~\ref{conj_ESrank2}, one may define the following pair of signed Euler systems without going through the technical difficulties one needs to overcome in the proof of Theorem~\ref{thm_BF_factorization_for_families}.

\begin{defn}
\label{defn_signed_BF_ff_bis}
We assume the truth of Conjecture~\ref{conj_ESrank2}.
\item[i)] Suppose that \ref{item_FLg} holds and $m\in \cN$. For the pair of signed Coleman maps $\col_{\QQ(\mu_m)_p,\f,g,?}$ (with $?=\#,\flat$) given as in Theorem~\ref{thm_decompPRLamdaf}, we define the element 
$$\mathbb{BF}^{\f,g,?}_m:=\col_{\QQ(\mu_m)_p,\f,g,?}\left(\mathbb{BF}^{\f,g}_m\right) \in H^1\left(\QQ(\mu_m),T_{\f,g}\widehat{\otimes}\Lambda_\cO(\Gamma_\cyc)^\iota\right)$$
\item[ii)] Suppose that \ref{item_SSf} holds and $m\in \cN$. For the pair of signed Coleman maps $\col_{\QQ(\mu_m)_p,\f,g,?}$ (with $?=+,-$) given as in Theorem~\ref{thm_decompose_PR_SSf}, we define the element 
$$\mathbb{BF}^{\f,g,?}_m:=\col_{\QQ(\mu_m)_p,\f,g,?}\left(\mathbb{BF}^{\f,g}_m\right) \in H^1\left(\QQ(\mu_m),T_{\f,g}\widehat{\otimes}\Lambda_\cO(\Gamma_\cyc)^\iota\right)\,.$$
\end{defn}

One may use Theorem~\ref{thm_decompPRLamdaf} and  Theorem~\ref{thm_decompose_PR_SSf} to easily prove that the elements $\{\mathbb{BF}^{\f,g,?}_m\}$ verify the following factorization statement (which should be compared to the conclusions of Theorem~\ref{thm_BF_factorization_for_families}):

\begin{proposition}
\label{prop_when_ESrank2_exists_BF_factorization}
We assume the truth of Conjecture~\ref{conj_ESrank2}.
\item[i)]Suppose that \ref{item_FLg} holds. For every $m\in\cN$, 
\[
\begin{pmatrix}
\BF_{\f,g,\alpha,m}\\\\ \BF_{\f,g,\beta,m}
\end{pmatrix}
=Q_g^{-1}M_g
\begin{pmatrix}
\mathbb{BF}^{\f,g,\#}_m\\\\ \mathbb{BF}^{\f,g,\flat}_m
\end{pmatrix}.
\]
\item[ii)]Suppose that \ref{item_SSf} holds. For every $m\in\cN$, we have
\[
\begin{pmatrix}
\BF_{\f,g,\alpha,m}\\\\ \BF_{\f,g,\beta,m}
\end{pmatrix}
=Q_g^{-1}M_g'
\begin{pmatrix}
\mathbb{BF}^{\f,g,+}_m\\\\ \mathbb{BF}^{\f,g,-}_m
\end{pmatrix}.
\]
\end{proposition}

\begin{corollary}
\label{cor_s_g_is_uniformly_bounded}
Assume the truth of Conjecture~\ref{conj_ESrank2}. Then,
$$\BF_{\f,g,?,m}=\mathbb{BF}^{\f,g,?}_m\qquad ?=\#,\flat,+,-\quad\hbox{and } \quad m\in \cN\,.$$
In particular, the exponents $s(g)$ in the statement of Theorem~\ref{thm_BF_factorization_for_families} are bounded independently of $g$.
\end{corollary}

\begin{proof}
This immediate on comparing the identities we have verified in Theorem~\ref{thm_BF_factorization_for_families} (concerning the elements $\BF_{\f,g,?,m}$) with those in Proposition~\ref{prop_when_ESrank2_exists_BF_factorization} (concerning the elements $\mathbb{BF}^{\f,g,?}_m$).
\end{proof}
\section{$p$-adic $L$-functions}
\label{subsec_padicL}

We next  study the link between Beilinson--Flach classes to $p$-adic $L$-functions. By an abuse of notation, we shall identify the Beilinson--Flach classes with their images under the $p$-localization map.
We first prove the preliminary Lemmas \ref{lem:geom} and \ref{lem:geom-Hida}.
\begin{lemma}\label{lem:geom}Suppose $m\in\cN$.
\item[i)]For   $\mu\in\{\alpha,\beta\}$, we have
\[  \cL^{(\alpha,\mu)}_{f,g,m}\left(\BF^{(\alpha,\mu)}_{f,g,m}\right)=0.
\]
\item[ii)]Suppose that $k_f\ge k_g$, then 
\[
\cL^{(\alpha,\alpha)}_{f,g,m}\left(\BF^{(\alpha,\beta)}_{f,g,m}\right)=-\cL^{(\alpha,\beta)}_{f,g,m}\left(\BF^{(\alpha,\alpha)}_{f,g,m}\right)\in\log_{p,k_g+1}\Lambda_\cO(\Gamma_\cyc)\otimes_{\cO}\QQ(m)_p,
\]
where $\QQ(m)_p=\QQ(m)\otimes_{\QQ}\QQ_p$. We note that we have denoted (by a slight abuse of notation) the image of $\BF_{f,g,m}^{\alpha,\mu}$ under the corestriction map 
\[H^1\left(\QQ(\mu_m),T_{f,g}\otimes\cH_{\ord(\mu_g)}(\Gamma_\cyc)^\iota\right)\lra H^1\left(\QQ(m),T_{f,g}\otimes\cH_{\ord(\mu_g)}(\Gamma_\cyc)^\iota\right)\]
by the same symbol.
\end{lemma}
\begin{proof}
Part (i) follows from \cite[Theorem~7.1.2]{LZ1}.

{For part (ii), note that the eigenvectors $v_{f,\lambda}$ and $v_{g,\mu}$ given in Definition~\ref{def_eigenvectors} differ from the ones given in \cite[\S3.5]{BLLV} by fixed constants that are independent of $\lambda$ and $\mu$. Therefore, the stated equality follows from the same proof as the last assertion of Theorem~3.6.5 in op. cit.}

Let us put $\fL=\cL^{(\alpha,\alpha)}_{f,g,m}\left(\BF^{(\alpha,\beta)}_{f,g,m}\right)$. Let $v$ be a prime of $\QQ(m)$ above $p$ and set $F:=\QQ(m)_v$. Since
the map $\cL_{f,g,F}^{(\alpha,\alpha)}$ sends an element in $\HIw(F(\mu_{p^\infty}),T_{f,g})$ to  $\cH_{\ord_p(\mu_g)}(\Gamma_\cyc)\otimes F$, it follows that  $\fL\in\cH_{k_g+1}(\Gamma_\cyc)\otimes_{\Zp}\QQ(m)_p$. It therefore remains to show that $\fL$ is divisible by $\log_{p,k_g+1}$. To do so, it suffices to show that $\fL$ vanishes at all characters of the form $\chi^j\theta$, where $0\le j\le k_g$ and $\theta$ is a finite Dirichlet character on $\Gamma$. Let $F$ be as above. Note that the natural image of $\BF^{(\alpha,\mu')}_{f,g,m}$ in $H^1(F(\mu_{p^n}),T(-j)\otimes_\cO L)$ falls within $H^1_{\mathrm g}(F(\mu_{p^n}),T(-j)\otimes_\cO L)$ for all $n\ge 0$ and $0\le j\le k_g$ (by \cite{KLZ2} Proposition~3.3.3) and that $H^1_{\mathrm f}=H^1_{\mathrm g}$ for the representation $T(-j)\otimes_\cO L$. The interpolative property of the Perrin-Riou map then shows that $\fL$ does vanish at the aforementioned set of characters.
\end{proof}

We henceforth work under the following additional hypothesis on $\Lambda_\f$:
\begin{itemize}
     \item[\mylabel{reg}{{\bf (Reg)}}] $\Lambda_\f$ is a regular ring.
\end{itemize}

\begin{remark}
\label{remark_regular_is_light_assumption}
Note that if one is content to work over a sufficiently small open disc in the weight space (rather than the entire weight space itself), this condition on the coefficients can be ensured by replacing $\LL_\f$ with the restriction of the universal Hecke algebra to this open disc. 
\end{remark}

Under \ref{reg}, we have the following $\Lambda_\f$-adic version of Lemma~\ref{lem:geom}.
\begin{lemma}\label{lem:geom-Hida}Suppose that \ref{reg} holds. Let $m\in\cN$.
\item[i)]For  $\mu\in\{\alpha,\beta\}$, we have
\[
  \cL_{\f,g,\mu,m}(\BF_{\f,g,\mu,m})=0.
  \]
\item[ii)] We have
  \[
\cL_{\f,g,\alpha,m}(\BF_{\f,g,\beta,m})=-\cL_{\f,g,\beta,m}(\BF_{\f,g,\alpha,m})\in \log_{p,k_g+1}\Lambda_\f(\Gamma_\cyc)\otimes_{\Zp}\QQ(m)_p\,.
\]
\end{lemma}
\begin{proof}
Let $X_{\rm cr}^{(k_g)}\subset \mathrm{Spec}(\Lambda_\f)$ denote the set of prime ideals of the form $P_\kappa=\ker(\kappa)$, where $\kappa$ is a crystalline (necessarily classical) specialization of $\LL_f$ such that the weight of the eigenform $\f_{\kappa}$ is at least $k_g+2$. Define $\fI$ to be the set of all finite intersections of elements of  $X_{\rm cr}^{(k_g)}$. Since  $X_{\rm cr}^{(k_g)}$ is Zariski dense in $\mathrm{Spec}(\Lambda_\f)$, hypothesis \ref{reg} tells us that
\[
\Lambda_\f=\varprojlim_{I\in\fI} \Lambda_\f/I.
\]

Let $v$ be a fixed prime of $\QQ(m)$ above $p$ and write $F=\QQ(m)_v$. For $I\in\fI$, let 
\[
\cF^-R_{\f_I}^*=\cF^-R_{\f}\otimes_\cO\Lambda_\f/I, \qquad T_{\f,g}^{-,\emptyset}=\cF^-R_{\f_I}^*\otimes R_g^*
\]
and write  $$ \cL_{F,T_{\f_I,g}^{-,\emptyset}}:\HIw(F(\mu_{p^\infty}),T_{\f_I,g}^{-,\emptyset})\lra \Dcris(T_{\f_I,g}^{-,\emptyset})\otimes \cH(\Gamma_\cyc)\otimes F$$ for the Perrin-Riou map of $\cF^-R_{\f_I}^*$ over $F$. By Remark~\ref{rk:explicit} and the fact that $\vp$ acts invertibly on $\Lambda_\f(\alpha_\f^{-1})$, we have
\begin{align*}
\vp^*\NN(F,T_{\f,g}^{-,\emptyset})^{\psi=0}&=\Lambda_\f(\alpha_\f^{-1})\widehat\otimes_\cO\ \vp^*\NN(R_g^*)^{\psi=0}\otimes_{\Zp}\cO_F\\
&=\left(\varprojlim_I \Lambda_\f(\alpha_\f^{-1})/I\right)\widehat\otimes_\cO\ \vp^*\NN(R_g^*)^{\psi=0}\otimes_{\Zp}\cO_F.
\end{align*}  
Therefore, 
we may realize the Perrin-Riou map $\cL_{F,T_{\f,g}^{-,\emptyset}}$ in the proof of Theorem~\ref{thm:PRHida}  as
\begin{align*}
\varprojlim_I  \cL_{F,T_{\f_I,g}^{-,\emptyset}}:\ &\HIw(F(\mu_{p^\infty}),T_{\f,g}^{-,\emptyset})=\varprojlim_I\HIw(F(\mu_{p^\infty}),T_{\f_I,g}^{-,\emptyset})\rightarrow \\
    &\left(\varprojlim_I\Dcris(T_{\f_I,g}^{-,\emptyset})\right)\otimes \cH(\Gamma_\cyc)\otimes F=\DD(\Qp,T_{\f,g}^{-,\emptyset})\otimes \cH(\Gamma_\cyc)\otimes F
\end{align*}

For $\mu\in\{\alpha,\beta\}$, we write $\cL_{F,T_{\f_I,g}^{-,\emptyset}}^{(\alpha,\mu)}$ for the map defined by $ \cL_{F,T_{\f_I,g}^{-,\emptyset}}$ paired with the $\vp$-eigenvector  $(H_\f\cdot\eta_\f\mod I)\otimes v_{g,\mu}$. Then, on taking inverse limits again, we have the morphism
\[
\left(\varprojlim_I\cL_{F,T_{\f_I,g}^{-,\emptyset}}^{(\alpha,\mu)}\right):\HIw(F(\mu_{p^\infty}),T_{\f,g}^{-,\emptyset})\lra \left(\varprojlim_I \Lambda_\f/I\right)\widehat\otimes_\cO\ \cH_{\ord(\mu_g)}(\Gamma_\cyc)\otimes_{\Zp}F.
\]
This allows us to rewrite $\cL_{F,\f,g,\mu}$ as 
$\left(\varprojlim_I\cL_{F,T_{\f_I,g}^{-,\emptyset}}^{(\alpha,\mu)}\right)\circ {\Pr}_{\f,g}$, where $\Pr_{\f,g}$ denotes the natural projection
\[
\HIw(\Qp(\mu_{p^\infty}),T_{\f,g})\lra\HIw(\Qp(\mu_{p^\infty}),T_{\f,g}^{-,\emptyset}).
\]
Therefore, it is enough to show that
\begin{itemize}
    \item[i)]For  $\mu\in\{\alpha,\beta\}$, we have
\[
  \cL_{F,T_{\f_I,g}^{-,\emptyset}}^{(\alpha,\mu)}(\BF_{\f_I,g,\mu,v})=0.
  \]
\item[ii)] We have
  \[
\cL_{F,T_{\f_I,g}^{-,\emptyset}}^{(\alpha,\alpha)}(\BF_{\f_I,g,\beta,v})=-\cL_{F,T_{\f_I,g}^{-,\emptyset}}^{(\alpha,\beta)}(\BF_{\f_I,g,\alpha,v})\in \left(\Lambda_\f/I\right)\otimes\log_{p,k_g+1}\Lambda_\cO(\Gamma_\cyc)\otimes_{\Zp}F,
\] 
\end{itemize}
where $\BF_{\f_I,g,\mu,v}$ denotes the image of $\BF_{\f,g,\mu,m}$ in $H^1(F,T_{\f_I,g}^{-,\emptyset}\otimes \cH_{\ord(\mu_g)}(\Gamma_\cyc)^\iota)$.

Suppose that $I=\bigcap_i P_{\kappa_i} $. For $\tau,\mu\in\{\alpha,\beta\}$, we have the commutative diagram
\[
\resizebox{\displaywidth}{!}{
\xymatrix{
\HIw(F(\mu_{p^\infty}),T_{\f_I,g}^{-,\emptyset}\otimes \cH_{\ord(\tau_g)}(\Gamma_\cyc)^\iota)\ar[r]^{\cL_{F,T_{\f_I,g}^{-,\emptyset}}^{(\alpha,\mu)}}\ar@{^{(}->}[d]&\left(\Lambda_\f/I\right)\otimes_{\cO}\cH_{\ord(\tau_g\mu_g)}(\Gamma_\cyc)\otimes_{\Zp}F\ar@{^{(}->}[d]\\
\bigoplus_i\HIw(F(\mu_{p^\infty}),\cF^-R_{\f_{\kappa_i}}^*\otimes R_g^*\otimes \cH_{\ord(\tau_g)}(\Gamma_\cyc)^\iota)\ar[r]^{\ \ \ \ \ \ \ \ \ \ \ \oplus \cL_{F,\f_{\kappa_i},g}^{(\alpha,\mu)} \ \ \ \ \ \ }&\left(\bigoplus_{i}\Lambda_\f/P_{\kappa_i}\right)\otimes_{\cO}\cH_{\ord(\tau_g\mu_g)}(\Gamma_\cyc)\otimes_{\Zp}F,
}}
\]
where the vertical arrows are injective and $\cL_{F,\f_{\kappa_i},g}^{(\alpha,\mu)}$ are defined in a similar manner as $\cL_{F,f,g}^{(\alpha,\mu)}$. Lemma~\ref{lem:geom} says that (i) and (ii) hold for all $i$. The result follows. 
\end{proof}

In particular, we see that Lemma~\ref{lem:geom}(ii) holds without the assumption that $k_f\ge k_g$, {since we may first directly work with $\f$ and specialize to $f_\alpha$}.
\begin{defn}\label{defn:geopadicL}
 \item[i)]We define the geometric Rankin--Selberg $p$-adic $L$-function  $L_p^{\rm geo}(f_\alpha,g)\in \LL_\cO(\Gamma_\cyc)\otimes_\cO L$ as the unique element satisfying\index{$p$-adic $L$-functions! $L_p^{\rm geo}(f_\alpha,g)$ (Loeffler--Zerbes ``geometric'' $p$-adic $L$-function)}
 \begin{equation} \label{eq:RSfdom}  
\cL^{(\alpha,\mu)}_{f,g}\left(\BF^{(\alpha,\mu')}_{f,g,1}\right)=\frac{\log_{p,k_g+1}}{\mu_g'-\mu_g}\cdot L_p^{\rm geo}(f_\alpha,g)
 \end{equation}
 for $\mu,\mu'\in\{\alpha,\beta\}$ with $\mu\ne\mu'$.
\item[ii)] We define the $\Lambda_\f$-adic geometric Rankin--Selberg $p$-adic $L$-function 
$$L_p^{\rm geo}(\f,g) \in \Lambda_\f(\Gamma_\cyc)\otimes_{\cO}L$$ 
as the unique element satisfying\index{$p$-adic $L$-functions! $L_{p}^{\rm geo}(\f,g)$ (Loeffler--Zerbes ``geometric'' $p$-adic $L$-function)}
 \begin{equation}
      \label{eq:RSfdomHida}  \cL_{\f,g,\mu}(\BF_{\f,\mu',1})=\frac{\log_{p,k_g+1}}{\mu'_g-\mu_g}\cdot L_{p}^{\rm geo}(\f,g),
      \end{equation}
where $\mu$ and $\mu'$ are as in (i).
\end{defn}
We refer the reader to the explicit reciprocity law of Loeffler--Zerbes in \cite[Theorem~9.3.2]{LZ1} for interpolation properties of the $p$-adic $L$-function $L_p^{\rm geo}(\f,g)$.

\begin{proposition}\label{prop:values-of-col-BF}
\item[i)]Suppose that  \ref{item_FLg} holds. For $m\in\cN$ and $\bullet\in\{\#,\flat\}$, we have
\[
\col_{\f,g,\bullet,m}(\BF_{\f,g,\bullet,m})=0.
\]
Furthermore,
\begin{align*}
\col_{\f,g,\#,m}(\BF_{\f,g,\flat,m})&=-\col_{\f,g,\flat,m}(\BF_{\f,g,\#,m})\\
&=\frac{\alpha_g\beta_g}{(\alpha_g-\beta_g)\det M_g}\cL_{\f,g,\alpha,m}(\BF_{\f,g,\beta,m}).
\end{align*}
When $m=1$, 
\[
\col_{\f,g,\#}(\BF_{\f,g,\flat,1})=-\col_{\f,g,\flat}(\BF_{\f,g,\#,1})=D_{g}\delta_{k_g+1}L_{p}^{\rm geo}(\f,g),
\]
where $D_{g}$ is a unit in $\Lambda_\cO(\Gamma_1)$.\index{Coleman maps! $D_g$}\index{Coleman maps! $\delta_{m}$}
\item[ii)]If \ref{item_SSf} holds, the analogous results where one replaces $\#$ and $\flat$ by $+$ and $-$ hold true.
\end{proposition}
\begin{proof}
We only prove part (i) since part (ii) can be proved in a similar manner. Consider the anti-symmetric matrix
\begin{align}\label{eq:antisym}
\begin{aligned}
   & \begin{pmatrix}
\cL_{\f,g,\alpha,m}(\BF_{\f,g,\alpha,m})&\cL_{\f,g,\alpha,m}(\BF_{\f,g,\beta,m})\\\\
\cL_{\f,g,\beta,m}(\BF_{\f,g,\alpha,m})&\cL_{\f,g,\beta,m}(\BF_{\f,g,\beta,m})
\end{pmatrix}
\\
&\qquad\qquad\qquad\qquad\qquad=\begin{pmatrix}
0&\cL_{\f,g,\alpha,m}(\BF_{\f,g,\beta,m})\\\\
-\cL_{\f,g,\alpha,m}(\BF_{\f,g,\beta,m})&0
\end{pmatrix}
\end{aligned}
\end{align}
as given in Lemma~\ref{lem:geom-Hida}.
Recall from Proposition~\ref{prop:semiordPR} that 
\[
\begin{pmatrix}
\cL_{\f,g,\alpha,m}\\ \\ \cL_{\f,g,\beta,m}
\end{pmatrix}
=Q_g^{-1}M_g\begin{pmatrix}
\col_{\f,g,\#,m}\\ \\ \col_{\f,g,\flat,m}
\end{pmatrix}.
\]
Thus, together with \eqref{eq:BFdecomp}, we may rewrite the left-hand side of \eqref{eq:antisym} as 
\[
Q_g^{-1}M_g\begin{pmatrix}
\col_{\f,g,\#,m}(\BF_{\f,g,\#,m})&\col_{\f,g,\#,m}(\BF_{\f,g,\flat,m})\\\\
\col_{\f,g,\flat,m}(\BF_{\f,g,\#,m})&\col_{\f,g,\flat,m}(\BF_{\f,g,\flat,m})
\end{pmatrix}(Q_g^{-1}M_g)^t.
\]
This in turn implies
\[
\col_{\f,g,\#,m}(\BF_{\f,g,\#,m})=\col_{\f,g,\flat,m}(\BF_{\f,g,\flat,m})=0
\]
and
\[
\col_{\f,g,\#,m}(\BF_{\f,g,\flat,m})=-\col_{\f,g,\flat,m}(\BF_{\f,g,\#,m})=\frac{1}{\det\left(Q_g^{-1}M_g\right)}\cL_{\f,g,\alpha,m}(\BF_{\f,g,\beta,m})
\]
as required.

The assertion on the case $m=1$ follows from \eqref{eq:RSfdom} and the fact that $\det(M_g)$ is, up to a unit of $\Lambda_\cO(\Gamma_1)$, given by $\log_{p,k_g+1}/p^{k_g+1}\delta_{k_g+1}$; see \cite[Lemma~2.7]{BFSuper} and \cite[Corollary~3.2]{leiloefflerzerbes11}.
\end{proof}

\begin{remark}\label{rk:special}
The exact same proof gives the analogous statements for the Coleman maps in Proposition~\ref{prop:semiordPR} and the bounded Beilinson--Flach elements from Definition~\ref{defn:boundedBF}. Namely, for  $\bullet\in\{\#,\flat\}$ and $m\in\cN$,
\begin{equation}
\label{eqn_BF_locally_restricted}
  \col_{f,g,m}^{(\omega,\bullet)}(\BF^{(\omega,\bullet)}_{f,g,m})=0.
\end{equation}

Furthermore,
\[
\col_{f,g,m}^{(\omega,\#)}(\BF_{f,g,m}^{(\omega,\flat)})=-\col_{f,g,m}^{(\omega,\flat)}(\BF_{f,g,m}^{(\omega,\#)})=\frac{\alpha_g\beta_g}{(\alpha_g-\beta_g)\det M_g}\cL_{f,g,m}^{(\alpha,\alpha)}(\BF_{f,g,m}^{(\alpha,\beta)}).
\]
\end{remark}

We end this section by defining the following bounded $p$-adic $L$-functions.

\begin{defn}\label{defn:padicL}
\item[i)]Suppose {\ref{item_FL} holds} and set $\fC=(\fF_1,\fF_2)$ be an ordered pair, where $$\fF_1\in\{(\omega,\#),(\omega,\flat),(\eta,\#),(\eta,\flat)\},$$ $$\fF_2\in\{(\omega,\#),(\omega,\flat)\}$$
with $\fF_1\ne\fF_2$. We define the signed $($geometric$)$ $p$-adic $L$-function\index{$p$-adic $L$-functions! $L_{\fC}^{\rm geo}(f,g)$ (Doubly-signed $p$-adic $L$-function)}
$$L_{\fC}^{\rm geo}(f,g):=\col_{f,g}^{\fF_1}(\BF_{f,g,1}^{\fF_2})\in\varpi^{-s(g)} \Lambda_\cO(\Gamma_\cyc).$$
\item[ii)]If {\ref{item_SS}}  holds, we define similar objects on replacing $\#$ and $\flat$ by the symbols $+$ and $-$.
\end{defn}
 
\begin{remark}
\label{remark_signed_geometric_padicL}
\item[i)]The fact that $L_\fC^{\rm geo}(f,g)\in \varpi^{-s(g)}\Lambda_\cO(\Gamma_\cyc)$ follows from the integrality of the Coleman maps and Theorem~\ref{thm_BF_factorization_for_families}.
\item[ii)]Suppose that either \ref{item_FLg} or \ref{item_SSf} (instead of \ref{item_FL} and \ref{item_SS})  holds and that \ref{reg} holds. We may still define the objects in Definition~\ref{defn:padicL} as long as $\fF_1\in \{(\omega,\bullet),(\omega,\circ)\}$ using the appropriate specializations of the $\Lambda_\f$-adic Coleman maps and Beilinson--Flach classes, where $(\bullet,\circ)$ denotes $(\#,\flat)$ or $(+,-)$. This results in only two choices of $\fC$ (with $\fF_1=(\omega,\circ)$ and $\fF_2=(\omega,\bullet)$ or vice versa). It follows from \eqref{eq:RSfdom} and Remark~\ref{rk:special} that 
$$L_{\fC}^{\rm geo}(f,g)=\pm D_g\delta_{k_g+1}L_p^{\rm geo}(f_\alpha,g).$$ 
We may also define the $\Lambda_\f$-adic $L$-function given by
$$L_{\fC}^{\rm geo}(\f,g)=\pm D_g\delta_{k_g+1}L_p^{\rm geo}(\f,g)\in \varpi^{-s(g)}\Lambda_\f(\Gamma_\cyc),$$\index{$p$-adic $L$-functions! $L_{\fC}^{\rm geo}(\f,g)$ (Signed $p$-adic $L$-function)}
where $\fF_1=(\omega,\circ)$ and $\fF_2=(\omega,\bullet)$ or vice versa.%:=\col_{\f,g,\circ}(\BF_{f,g,\bullet}
\end{remark}

%%%%%%%%%%%
%%%%%%%%%%%

%-----------------------------------------------------------------------
% End of chap1.tex
%-----------------------------------------------------------------------

%% file: InertOrdMainChapter4.tex
\chapter{Selmer groups and main conjectures}
We are now ready to define the various Selmer groups arising from the Coleman maps we defined in \S\ref{S:Coleman} and relate them to the $p$-adic $L$-functions we defined in \S\ref{subsec_padicL}.
\section{Definitions of Selmer groups}
\label{subsec_def_Selmer_classical}
Recall that $T_{\f,g}=R_{\f}^*\otimes R_g^*$ and for a crystalline specialization $f$ of the Hida family $\f$ (corresponding to the ring homomorphism $\kappa_f:\LL_\f\to \cO$), we have set $T=R_f^*\otimes R_g^*$. We shall define discrete Selmer groups for $T_{f,g}^\vee(1)$ and $T_{\f,g}^\vee(1)$ determined by our signed Coleman maps. Throughout this section, suppose either {\ref{item_FLg}   or \ref{item_SSf}  holds and write $\{\bullet,\circ\}$ for $\{\#,\flat\}$ or $\{+,-\}$ accordingly. We also assume that the hypothesis \ref{reg} holds.}

{
\begin{defn}\label{defn:signedSel}
\item[i)]Let $\fC_\omega:=((\omega,\bullet),(\omega,\circ))$\index{Classical and signed Selmer groups! $\fC_\omega$}. The discrete Selmer group $\Sel_{\fC_\omega}(T_{f,g}^\vee(1)/\QQ(\mu_{p^\infty}))$\index{Classical and signed Selmer groups! $\Sel_{\fC_\omega}(T_{f,g}^\vee(1)/\QQ(\mu_{p^\infty}))$} is given by the kernel of the restriction map
\[
 H^1(\QQ(\mu_{p^\infty}),T_{f,g}^\vee(1))\lra \prod_{v|p}\frac{H^1(\QQ(\mu_{p^\infty})_v,T_{f,g}^\vee(1))}{H^1_{\fC_\omega}(\QQ(\mu_{p^\infty})_v,T_{f,g}^\vee(1))} \times \prod_{v\nmid p}{H^1(\QQ(\mu_{p^\infty})_v,T_{f,g}^\vee(1))},
  \]
  where $v$ runs through all primes of $\QQ(\mu_{p^\infty})$, and for $v \mid p$ the local condition $H^1_{\fC_\omega}(\QQ(\mu_{p^\infty})_v,T_{f,g}^\vee(1))$ is the orthogonal complement of 
  $$\ker\left(\col_{f,g}^{(\omega,\bullet)}\right)\,\cap\,\ker\left(\col_{f,g}^{(\omega,\circ)}\right)$$ 
  under the local Tate pairing. 
\item[ii)]Suppose that either \ref{item_FL} or \ref{item_SS} holds. We then define the dual Selmer group $\Sel_\fC(T_{f,g}^\vee(1)/\QQ(\mu_{p^\infty}))$\index{Classical and signed Selmer groups! $\Sel_\fC(T_{f,g}^\vee(1)/\QQ(\mu_{p^\infty}))$} for $\fC=(\fF_1,\fF_2)$\index{Classical and signed Selmer groups! $\fC$}, where $\fF_i\in\{(\omega,\bullet),(\omega,\circ),(\eta,\bullet),(\eta,\circ)\}$ with $\fF_1\ne\fF_2$, with the local condition at $p$ given by the orthogonal complement of $\ker\left(\col_{f,g}^{\fF_1}\right)\cap\ker\left(\col_{f,g}^{\fF_2}\right)$.
\end{defn}}

Definition~\ref{defn:signedSel}(ii) gives rise to a total of six Selmer groups, which is in line with \cite[Definition~6.1.2]{BLLV}. We have the following $\Lambda_\f$-adic version of the Selmer group defined in Definition~\ref{defn:signedSel}(i).
\begin{defn}
The discrete Selmer group $\Sel_{\fC_\omega}(T_{\f,g}^\vee(1)/\QQ(\mu_{p^\infty}))$ is given by the kernel of the restriction map\index{Classical and signed Selmer groups!  $\Sel_{\fC_\omega}(T_{\f,g}^\vee(1)/\QQ(\mu_{p^\infty}))$}
\[
 H^1(\QQ(\mu_{p^\infty}),T_{\f,g}^\vee(1))\lra \prod_{v|p}\frac{H^1(\QQ(\mu_{p^\infty})_v,T_{\f,g}^\vee(1))}{H^1_{\fC_\omega}(\QQ(\mu_{p^\infty})_v,T_{\f,g}^\vee(1))} \times \prod_{v\nmid p}{H^1(\QQ(\mu_{p^\infty})_v,T_{\f,g}^\vee(1))},
  \]
  where $v$ runs through all primes of $\QQ(\mu_{p^\infty})$,  and for $v \mid p$ the local condition $H^1_{\fC_\omega}(\QQ(\mu_{p^\infty})_v,T_{\f,g}^\vee(1))$ is the orthogonal complement of $\ker\left(\col_{\f,g,\bullet}\right)\cap\ker\left(\col_{\f,g,\circ}\right)$ under the local Tate pairing. 
\end{defn}
\begin{remark}
It follows from the last assertion of Theorem~\ref{thm_decompPRLamdaf} and (the analogue of the plus/minus version)  that we have a canonical isomorphism  
$$\Sel_{\fC_\omega}(T_{\f,g}^\vee(1)/\QQ(\mu_{p^\infty}))^\vee\otimes_{\kappa_f}\cO\stackrel{\sim}{\lra}\Sel_{\fC_\omega}(T_{f,g}^\vee(1)/\QQ(\mu_{p^\infty}))^\vee.$$ As a matter of fact, we will see in Corollary~\ref{cor_somesignedSelmergroupsareGreenberg} below that these Selmer groups agree with the classical $($strict$)$ Greenberg Selmer groups. Therefore, the canonical isomorphism above can be seen as an instance of classical control theorems for Greenberg Selmer groups.
\end{remark}

\begin{lemma}
\label{lemma_somesignedSelmergroupsareGreenberg}
The intersection $\ker\left(\col_{\f,g,\circ}\right)\cap\ker\left(\col_{\f,g,\bullet}\right)$ is precisely
\[
{\rm im}\left(\HIw(\Qp(\mu_{p^\infty}),\cF^+R_\f^*\otimes R_g^*)\hookrightarrow \HIw(\Qp(\mu_{p^\infty}),T_{\f,g})\right).
\]
Likewise, for any crystalline specialization $f$ of the Hida family $\f$, the intersection  $  \ker\left(\col_{f,g}^{(\omega,\circ)}\right)\cap\ker\left(\col_{f,g}^{(\omega,\bullet)}\right)$ equals
\[
  {\rm im}\left(\HIw(\Qp(\mu_{p^\infty}),\cF^+R_f^*\otimes R_g^*)\hookrightarrow \HIw(\Qp(\mu_{p^\infty}),T_{f,g})\right).\]
\end{lemma}
\begin{proof}
Let $\cL_{T_{\f,g}^{-,\emptyset}}$ be the Perrin-Riou map on $\HIw\left(\Qp(\mu_{p^\infty}),T_{\f,g}^{-,\emptyset}\right)$ defined in the proof of Theorem~\ref{thm:PRHida} for $F=\Qp$. It is injective since 
\[
\ker\left(\cL_{T_{\f,g}^{-,\emptyset}}\right)\subset \NN(\Qp,T_{\f,g}^{-,\emptyset})^{\vp=1}\subset\DD(\Qp,T_{\f,g}^{-,\emptyset})^{\vp=1}=0,
\]
where the last equality is a consequence of the fact that the representation $T_{\f,g}^{-,\emptyset}$ does not admit the trivial representation as a sub-representation.

By Theorem~\ref{thm_decompPRLamdaf}, $\bz\in \ker\left(\col_{\f,g,\circ}\right)\cap\ker\left(\col_{\f,g,\bullet}\right)$ if and only if $\bz$ lies inside {$\ker\left(\cL_{\f,g,\alpha}\right)\cap\ker\left(\cL_{\f,g,\beta}\right)$}. This is equivalent to the condition that $\cL_{T_{\f,g}^{-,\emptyset}}\circ\Pr_{\f,g}(\bz)=0$, where  $\Pr_{\f,g}$ denotes the natural projection
\[
\HIw(\Qp(\mu_{p^\infty}),T_{\f,g})\rightarrow\HIw(\Qp(\mu_{p^\infty}),T_{\f,g}^{-,\emptyset}),
\]
since $\cL_{\f,g,\alpha}$ and $\cL_{\f,g,\beta}$ are defined by projecting $\cL_{T_{\f,g}^{-,\emptyset}}\circ \Pr$ to the two chosen $\vp$-eigenvectors in $\DD(\Qp,T_{\f,g}^{-,\emptyset})$. Thus, the injectivity of  $\cL_{T_{\f,g}^{-,\emptyset}}$ implies that
\[
\ker\left(\col_{\f,g,\circ}\right)\cap\ker\left(\col_{\f,g,\bullet}\right)=\ker({\Pr}_{\f,g}).
\]
The first assertion of the lemma now follows. The proof of the second assertion is similar.
 \end{proof}

 We recall the definition of the classical Greenberg Selmer groups.
\begin{defn}.\index{Classical and signed Selmer groups! $\Sel(T_{f,g}^\vee(1)/\QQ(\mu_{p^\infty}))$ (Greenberg Selmer group for families)}
The classical Greenberg Selmer group for $T_{\f,g}^\vee(1) $ over $\Qp(\mu_{p^\infty})$, denoted by $\Sel_{\mathrm{Gr}}(T_{\f,g}^\vee(1)/\QQ(\mu_{p^\infty}))$, is defined as the kernel of
\[
 H^1(\QQ(\mu_{p^\infty}),T_{\f,g}^\vee(1))\lra \prod_{v|p}\frac{H^1(\QQ(\mu_{p^\infty})_v,T_{\f,g}^\vee(1))}{H^1_{\cF_{\rm Gr}}(\QQ(\mu_{p^\infty})_v,T_{\f,g}^\vee(1))} \times \prod_{v\nmid p}{H^1(\QQ(\mu_{p^\infty})_v,T_{\f,g}^\vee(1))}\, ,
\]
where  $H^1_{\cF_{\rm Gr}}(\QQ(\mu_{p^\infty})_v,T_{\f,g}^\vee(1))$  is the orthogonal complement of 
\[
{\rm im}\left(\HIw(\Qp(\mu_{p^\infty}),\cF^+R_\f^*\otimes R_g^*)\hookrightarrow \HIw(\Qp(\mu_{p^\infty}),T_{\f,g})\right)
\]
under the local Tate duality. We also define the classical Greenberg Selmer group  $\Sel_{\rm Gr}(T_{f,g}^\vee(1)/\QQ(\mu_{p^\infty}))$ similarly.
\end{defn} 
  
 We deduce from Lemma~\ref{lemma_somesignedSelmergroupsareGreenberg} the following corollary.
\begin{corollary}
\label{cor_somesignedSelmergroupsareGreenberg}
The Selmer group {$\Sel_{\fC_\omega}(T_{\f,g}^\vee(1)/\QQ(\mu_{p^\infty}))$  coincides} with the classical $($strict$)$ Greenberg Selmer group $\Sel(T_{\f,g}^\vee(1)/\QQ(\mu_{p^\infty}))$\index{Classical and signed Selmer groups! $\Sel_{\mathrm{Gr}}(T_{\f,g}^\vee(1)/\QQ(\mu_{p^\infty}))$ (Greenberg Selmer group for families)} given as the kernel of
\[
 H^1(\QQ(\mu_{p^\infty}),T_{\f,g}^\vee(1))\lra \prod_{v|p}\frac{H^1(\QQ(\mu_{p^\infty})_v,T_{\f,g}^\vee(1))}{H^1_{\cF_{\rm Gr}}(\QQ(\mu_{p^\infty})_v,T_{\f,g}^\vee(1))} \times \prod_{v\nmid p}{H^1(\QQ(\mu_{p^\infty})_v,T_{\f,g}^\vee(1))}\, ,
\]
Similarly, $\Sel_{\fC_\omega}(T_{f,g}^\vee(1)/\QQ(\mu_{p^\infty}))=\Sel_{\rm Gr}(T_{f,g}^\vee(1)/\QQ(\mu_{p^\infty}))$
\end{corollary}

\section{Classical and signed Iwasawa main conjectures}
\label{subsec_MC_signed_classical}
We now formulate Iwasawa main conjectures relating the $p$-adic $L$-functions defined in Definition~\ref{defn:padicL} to the Selmer groups in Definition~\ref{defn:signedSel}.
\begin{conj}\label{conj:signedIMC}
Suppose that either \ref{item_FL} or \ref{item_SS} holds. For $\fC=(\fF_1,\fF_2)$, where $$\fF_1\in\{(\omega,\bullet),(\omega,\circ),(\eta,\bullet),(\eta,\circ)\},\quad\fF_2\in\{(\omega,\bullet),(\omega,\circ)\}$$ with $\fF_1\ne\fF_2$ and $\theta$ a character of $\Delta$, the $\theta$-isotypic component  $e_\theta\Sel_\fC(T_{f,g}^\vee(1)/\QQ(\mu_{p^\infty}))$ is a cotorsion $\Lambda_\cO(\Gamma_1)$-module. Furthermore,
\[
\Char_{\Lambda_\cO(\Gamma_1)}e_\theta\Sel_\fC(T_{f,g}^\vee(1)/\QQ(\mu_{p^\infty}))^\vee =(e_\theta\varpi^{s(g)} L_{\fC}^{\rm geo}(f,g)/I_\fC),
\]
where $I_\fC=\det(\image(\col_{\fF_1}\oplus\col_{\fF_2}))$.\index{Classical and signed Selmer groups! $I_\fC$}
\end{conj}

\begin{remark}\label{rk:imagePR}
\item[i)] When $\fF_1\in\{(\eta,\bullet),(\eta,\circ)\}$, the  $p$-adic $L$-functions $L_{\fC}^{\rm geo}(f,g)$ does not have a simple description  (as far as we are aware).

\item[ii)] It follows from \cite[Appendix]{BLLV} that $I_\fC$ is a product of powers of $\Tw^{-i}X$, $i=0,1,\ldots, k_f+k_g+1$.

\item[iii)] Consider  the special case when $\fC=\fC_\omega$. Since $\{\omega_{f^\star}\otimes \omega_{g^\star},\omega_{f^\star}\otimes \vp(\omega_{g^\star})\}$ is an $\cO$-basis of $\Dcris(\cF^-R_f^*\otimes R_g^*)$, it follows from \cite[\S3.4]{perrinriou94} that 
\begin{align*}
\det \left(\image \left(\cL_{f,g}^{(\alpha,\alpha)}\oplus \cL_{f,g}^{(\alpha,\beta)}\right)\right)&=\frac{1}{\det Q_g}\prod_{i=0}^{k_g}\left(\frac{\log_p}{\log_p\chi(\Gamma_\cyc)}-i\right)\Lambda_\cO(\Gamma_\cyc)\\
&=\frac{\log_{p,k_g+1}}{p^{k_g+1}\det Q_g}\Lambda_\cO(\Gamma_\cyc).
\end{align*}
Hence, we deduce from the last assertions of Propositions~\ref{prop:semiordPR} and \ref{prop:PRpm} that 
$$I_\fC=\frac{\log_{p,k_g+1}}{p^{k_g+1}\det M_g}\Lambda_\cO(\Gamma_\cyc)=\delta_{k_g+1}\Lambda_\cO(\Gamma_\cyc).$$

In view of Proposition~\ref{prop:values-of-col-BF} (which tells us that $D_g\in \Lambda_\cO(\Gamma_1)^\times$) and Remark~\ref{remark_signed_geometric_padicL}(i), Conjecture~\ref{conj:signedIMC} can be rephrased as the conjectural identity
\begin{equation}
    \label{eqn_mainconj_reformulated}
   \Char_{\Lambda_\cO(\Gamma_1)}e_\theta\Sel_{\rm Gr}(T_{f,g}^\vee(1)/\QQ(\mu_{p^\infty}))\stackrel{?}{=}\left(e_\theta \varpi^{s(g)}L_p^{\rm geo}(f,g)\right)
\end{equation}
for every character $\theta$ of $\Delta$.
\end{remark}
We finish this chapter by stating a two-variable  ($\Lambda_\f(\Gamma_1)$-adic) Iwasawa main conjecture.
\begin{defn}
 Suppose $M$ is a finitely generated torsion $\LL_\f(\Gamma_1)$-module. We define its  characteristic ideal on setting
\[
\Char_{\LL_\f(\Gamma_1)}(M):=\prod_{P} P^{{\rm length}_{\LL_\f(\Gamma_1)_P}(M_P)},
\]
 where the product runs over  all height-one primes of $\LL_\f(\Gamma_1)$.
\end{defn}

\begin{conj}\label{conj:bigIMC}
Suppose that either \ref{item_FLg} or \ref{item_SSf} holds and that \ref{reg} holds. For any character $\theta$ of $\Delta$, the $\Lambda_\f(\Gamma_1)$-module $e_\theta \Sel_{\fC_\omega}(T_{\f,g}^\vee(1)/\QQ(\mu_{p^\infty}))^\vee$ is torsion. Furthermore, 
$$\Char_{\Lambda_\f(\Gamma_1)}e_\theta\Sel_{\fC_\omega}(T_{\f,g}^\vee(1)/\QQ(\mu_{p^\infty}))^\vee =(e_\theta \varpi^{s(g)}L_p^{\rm geo}(\f,g)/I_{\fC_\omega}\,.$$
\end{conj}

\begin{remark}
By Corollary~\ref{cor_somesignedSelmergroupsareGreenberg} and Remarks~\ref{remark_signed_geometric_padicL}(ii) and~\ref{rk:imagePR}(iii), Conjecture~\ref{conj:bigIMC} is equivalent to the assertion that 
$$\Char_{\Lambda_\f(\Gamma_1)}e_\theta\Sel_{\mathrm{Gr}}(T_{\f,g}^\vee(1)/\QQ(\mu_{p^\infty}))^\vee =(e_\theta \varpi^{s(g)}L_{p}^{\rm geo}(\f,g))\,.$$
\end{remark}

%%%%%%%%%%%
%%%%%%%%%%%

%-----------------------------------------------------------------------
% End of chap1.tex
%-----------------------------------------------------------------------

%% file: InertOrdMainChapter5.tex
\chapter{Applications towards main conjectures}

We now employ the bounded Beilinson--Flach elements we have constructed in \S\ref{subsec_review_BF_elements} to show that one inclusion in the Main Conjectures \ref{conj:signedIMC} and \ref{conj:bigIMC} holds true (under mild hypotheses concerning the images of the Galois representations under consideration). In \S\ref{subsec_IMCimagquadinert}, we discuss implications of these results towards  main conjectures for Rankin--Selberg convolutions $\f_{/K}\otimes \psi$ where $K$ is an imaginary quadratic field in which $p$ remains inert and $\psi$ is an algebraic Hecke character of $K$.

\section{Cyclotomic main conjectures for $f\otimes g$}
\label{subsec_fotimesg}
Let $f\in S_{k_f+2}(\Gamma_1(N_f))$ and $g\in S_{k_g+2}(\Gamma_1(N_g))$ be primitive eigenforms as in \S\ref{subsec_setting}, where $N_f$ and $N_g$ are coprime and $k_f>k_g\geq 0$. We recall also that we are working under the assumption that $f$ is $p$-ordinary and $g$ is non-$p$-ordinary (relative to the embedding $\iota_p$).

Recall the Galois representation $T_{f,g}=R_f^*\otimes R_g^*$, which is a free $\cO$-module of rank $4$, where $\cO$ is the ring of integers of a finite extension $L$ of  $\QQ_p$ containing the images $\iota_p(K_f)$ and $\iota_p(K_g)$ of the Hecke fields of $f$ and $g$, as well as the roots of the Hecke polynomials of $f$ and $g$ at $p$.

Suppose throughout \S\ref{subsec_fotimesg} that the residual representations $\overline{\rho}_f$ associated to $f$ is absolutely irreducible.  Assume in addition that at least one of the following conditions holds true:
\begin{itemize}
    \item[\mylabel{item_Irr1}{$(\mathbf{Irr}_1)$}] $p>k_g+1$\,.
        \item[\mylabel{item_Irr2}{$(\mathbf{Irr}_2)$}]  $a_p(g)=0$, $k_g\in [p+1,2p-2]$\,.
    \item[\mylabel{item_Irr3}{$(\mathbf{Irr}_3)$}]  $a_p(g)=0$, $p+1\nmid k_g+1$  and $k_g\geq 2p-1$.
\end{itemize}

\begin{remark}
\label{remark_HnA_trivial}
We note that any one of the hypotheses \ref{item_Irr1}--\ref{item_Irr3} guarantees that $\overline{\rho}_g{\vert_{G_{\QQ_p}}}$ is absolutely irreducible; c.f. \cite{FontaineEdixhoven92} in the situation of \ref{item_Irr1} and \cite[Th\'eor\`eme 3.2.1]{Berger2010} in the remaining cases. Since $\overline{\rho}_f{\vert_{G_{\QQ_p}}}$ (and therefore, also $\overline{\rho}_f{\vert_{G_{\QQ_p}}}\otimes\omega$) is reducible, we conclude that neither $\overline{\rho}_f{\vert_{G_{\QQ_p}}}$ nor $\overline{\rho}_f{\vert_{G_{\QQ_p}}}\otimes\omega^{-1}$ is isomorphic to $\overline{\rho}_g^\vee{\vert_{G_{\QQ_p}}}$ granted the truth of either one of the hypotheses \ref{item_Irr1}--\ref{item_Irr3}. This in turn ensures the validity of the following  ``non-anomality'' condition: $$H^0(\QQ_p,\overline{T}_{f,g})=0=H^2(\QQ_p,\overline{T}_{f,g})\,.$$
\end{remark}
 
 We consider the following big image condition.\index{Big image conditions! $(\tau_{f\otimes g})$}
\begin{itemize}
     \item[\mylabel{item_BI_fg}{$(\tau_{f\otimes g})$}] There exists $\tau\in \Gal(\overline{\QQ}/\QQ(\mu_{p^\infty}))$ such that $T_{f,g}/(\tau-1)T_{f,g}$ is a free $\cO$-module of rank $1$.
\end{itemize}

 Our main result towards the validity of cyclotomic main conjectures for the Rankin--Selberg convolution $f\otimes g$ is the following:
\begin{theorem}
\label{thm_cyclo_main_conj_fotimesg}
Suppose that the residual representation $\overline{\rho}_f$ is absolutely irreducible. Assume also that one of the hypotheses \ref{item_Irr1}--\ref{item_Irr3} as well as the condition \ref{item_BI_fg} hold true. Then for any character $\theta$ of $\Delta$ we have the following containment in the Iwasawa main conjecture \eqref{eqn_mainconj_reformulated}:
$$e_{\theta}\,\varpi^{s(g)}L_p^{\rm geo}(f,g)\,\in \,\Char_{\Lambda_\cO(\Gamma_1)}e_{\theta}\left(\Sel_{\rm Gr}(T_{f,g}^\vee(1)/\QQ(\mu_{p^\infty}))^\vee\right) \,,$$
where the integer $s(g)$ is given as in the statement of Theorem~\ref{thm_BF_factorization_for_families}. 
\end{theorem}

\begin{proof}
Given the locally restricted Euler system $\{\varpi^{s(g)}\BF_{f,g,m}^{\alpha,\ast}\}_{m\in\cN}$ of signed Beilinson--Flach elements (c.f. \S\ref{subsec_review_BF_elements}), where 
\begin{itemize}
\item $\alpha=\alpha_f$ is the root of the Hecke polynomial of $f$ at $p$ such that $\iota_p(\alpha)$ is a $p$-adic unit,
    \item $\ast\in \{\#,\flat\}$ if \ref{item_Irr1} holds and $\ast\in \{+,-\}$ if \ref{item_Irr2} or \ref{item_Irr3} holds,
\end{itemize} 
the proof of this theorem is identical to the proof of \cite[Theorem 6.2.4]{BLLV}, in view of Corollary~\ref{cor_somesignedSelmergroupsareGreenberg} (which identifies the signed Selmer group with the Greenberg Selmer group) and Remark~\ref{remark_signed_geometric_padicL} (which compares the signed $p$-adic $L$-functions to the geometric Rankin--Selberg $p$-adic $L$-function). 

We recall that this argument builds on the locally restricted Euler system machinery developed in \cite[Appendix A]{kbbleiPLMS} (more specifically, the conclusion of Theorem A.14 in op. cit.; see also  \cite{kbbstarkiwasawa,kbbesrankr,kbbCMabvar,kbbRSsupersingularEllipticCurves} for earlier incarnations and applications of this machinery in a variety of contexts). We note that the adjective ``locally restricted'' refers to the $p$-local property \eqref{eqn_BF_locally_restricted} of the classes $\{\BF_{f,g,m}^{\alpha,\ast}\}_m$, which asserts that 
$$\loc_p\left(\BF_{f,g,m}^{\alpha,\ast}\right)\in\ker\left( \col_{f,g,m}^{(\alpha,\ast)}\right).$$   
\end{proof}

\begin{remark}
\item[i)]By the explicit reciprocity law of Loeffler--Zerbes \cite[Theorem~9.3.2]{LZ1}, the $p$-adic $L$-function $L_p^{\rm geo}(f,g)$ evaluated at $\chi^j$ equals (up to explicit fudge-factors) the complex $L$-value $L(f,g,1+j)$ whenever $k_g< j\le k_f$.  If $j>\dfrac{k_f+k_g+1}{2}$, then $1+j$ falls within the range of absolute convergence for $L(f,g,s)$.  Hence, we have $L(f,g,1+k_f)\ne 0$ when $k_f-k_g \geq 3$. It is also easy to see in that case that the explicit fudge-factors are non-vanishing as well. In particular, $L_p^{\rm geo}(f,g)\ne0$.
\item[ii)]We have proved in Theorem~\ref{thm_cyclo_main_conj_fotimesg} one inclusion of Conjecture~\ref{conj:signedIMC} with the choice of $\fC=\fC_\omega$. The same proof can be carried over to the other choices of $\fC$ in Conjecture~\ref{conj:signedIMC} without any difficulties.
\end{remark}

\section{Cyclotomic main conjectures for $\f\otimes g$}
\label{subsec_ffotimesg}
Suppose that the newform $g\in S_{k_g+2}(\Gamma_1(N_g))$ is as in \S\ref{subsec_fotimesg}. We let $\f$ denote the primitive Hida family of tame level $N_f$ (where we still assume that $N_f$ is coprime to $N_g$), whose basic properties were outlined in \S\S\ref{subsec_setting} and \ref{subsec_twovarPRandColeman}.

Recall that $\LL_{\f}$ stands for the branch of Hida's universal (ordinary) Hecke algebra. Without loss of generality, we may (and will) assume that $\LL_{\f}$ contains $\cO$ as a subring. We assume that $\LL_\f$ is a regular ring (c.f., Remark~\ref{remark_regular_is_light_assumption}). 

Recall the Galois representation $T_{\f,g}=R_\f^*\otimes R_g^*$, which is a free $\LL_\f$-module of rank $4$. We consider the following big image condition on the representation $T_{\f,g}$. \index{Big image conditions! $(\tau_{\f\otimes g})$}
\begin{itemize}
     \item[\mylabel{item_BI_ffg}{$(\tau_{\f\otimes g})$}] There exists $\tau\in \Gal(\overline{\QQ}/\QQ(\mu_{p^\infty}))$ such that $T_{\f,g}/(\tau-1)T_{\f,g}$ is a free $\LL_\f$-module of rank $1$.
\end{itemize}

Our main result towards the validity of cyclotomic main conjectures for the family $\f\otimes g$ of Rankin--Selberg convolutions is the following:

\begin{theorem}
\label{thm_cyclo_main_conj_ffotimesg}
Suppose that the residual representation $\overline{\rho}_\f$ associated to $\f$ is absolutely irreducible. Assume also that one of the hypotheses \ref{item_Irr1}--\ref{item_Irr3} as well as the condition \ref{item_BI_ffg} hold true. For any character $\theta$ of $\Delta$, we have the following containment in Conjecture~\ref{conj:bigIMC} up to $\mu$-invariants: 
$$ e_{\theta}\, L_p^{\rm geo}(\f,g)\,\in \,\Char_{\Lambda_\cO(\Gamma_1)}e_{\theta}\left(\Sel_{\rm Gr}(T_{\f,g}^\vee(1)/\QQ(\mu_{p^\infty}))^\vee\right)\otimes_\cO L\,.$$
\end{theorem}

\begin{proof}
Let us fix $\theta$ as in the statement of our theorem and assume $e_{\theta}L_p^{\rm geo}(\f,g)\neq 0$ without loss of generality, since there is nothing to prove otherwise. Put $\fL:=\varpi^{\mu_1}e_{\theta}L_p^{\rm geo}(\f,g)$ where $\mu_1$ is the unique natural number with $\fL\in \LL_\f(\Gamma_1)\setminus \varpi\LL_\f(\Gamma_1)$. Let us also choose an element $S\in \LL_\f(\Gamma_1)\setminus\varpi\LL_\f(\Gamma_1)$ such that $$\varpi^{\mu_2}\LL_\f(\Gamma_1)S=\Char_{\LL_\f(\Gamma_1)}e_{\theta}\left(\Sel_{\rm Gr}(T_{\f,g}^\vee(1)/\QQ(\mu_{p^\infty}))^\vee\right)$$
for some suitable choice of a natural number $\mu_2$. We contend to prove that $S \mid \mathfrak{L}$, using the conclusion of Theorem~\ref{thm_cyclo_main_conj_fotimesg} for varying $f$ and the divisibility criterion established in Appendix~\ref{Appendix_Regular_Rings_Divisibility}. This will in turn prove that 
\begin{equation}
    \label{eqn_thm_cyclo_main_conj_fotimesg}
    e_{\theta}\, \varpi^{\mu_2-\mu_1} L_p^{\rm geo}(\f,g)\,\in \,\Char_{\Lambda_\cO(\Gamma_1)}e_{\theta}\left(\Sel_{\rm Gr}(T_{\f,g}^\vee(1)/\QQ(\mu_{p^\infty}))^\vee\right)\,.
\end{equation}
We will see along the way $\mu_2\leq \mu_1+s(g)$. This combined with \eqref{eqn_thm_cyclo_main_conj_fotimesg} will verify the required containment. 

Let us choose any sequence $\{\cP_{i}=\ker(\kappa_i)\}_{i=1}^\infty\subset {\rm Spec}(\LL_\f)$ of distinct height-one primes of $\Lambda_\f$, where each $\kappa_i: \LL_\f\to \cO$ is an $\cO$-valued crystalline specialization of weight $\geq 0$ (in the sense that the overconvergent eigenform $\f(\kappa_i)$ is a classical eigenform which is $p$-old and has weight $\geq 2$). For each positive integer $i$, let us put $\fL_i:=\kappa_i(\fL)\in \Lambda_\cO(\Gamma_1)$ and similarly define $S_i\in \Lambda_\cO(\Gamma_1)$.

Let us fix an index $i$. We will explain that 
$$\varpi^{\mu_2}S_i \hbox{\,\,\,divides\,\,\,} \Char_{\Lambda_\cO(\Gamma_1)}e_{\theta}\left(\Sel_{\rm Gr}(T_{\f(\kappa_i),g}^\vee(1)/\QQ(\mu_{p^\infty}))^\vee\right).$$ 
Since we have assumed that $\LL_\f$ is regular, the height-one prime $\cP_i=P_i\LL_\f$ is principal. It follows from \cite[Lemma 3.5.3]{mr02} that
\begin{equation}
\label{eqn_2023_04_11_1030}
\Sel_{\rm Gr}(T_{\f,g}^\vee(1)/\QQ(\mu_{p^\infty}))[\cP_i]\simeq \Sel_{\rm Gr}(T_{\f,g}^\vee[\cP_i](1)/\QQ(\mu_{p^\infty}))\,.
\end{equation}
Moreover, passing to Pontryagin duals in the exact sequence
$$0\lra T_{\f,g} \xrightarrow{\times P_i} T_{\f,g} \lra T_{\f,g}/\cP_i T_{\f,g}\lra 0$$
and recalling that $T_{\f,g}/\cP_i T_{\f,g}=T_{\f(\kappa_i),g}$ by definition, we deduce that the sequence
$$0\lra T_{\f(\kappa_i),g}^\vee \lra T_{\f,g}^\vee \xrightarrow{\times P_i} T_{\f,g}^\vee\lra 0$$
is exact. In particular, we have a natural isomorphism 
\begin{equation}
\label{eqn_2023_04_11_1037}
    T_{\f(\kappa_i),g}^\vee \stackrel{\sim}{\lra} T_{\f,g}^\vee[\cP_i]\,.
\end{equation}
Combining \eqref{eqn_2023_04_11_1030} and \eqref{eqn_2023_04_11_1037}, we deduce that 
\begin{equation}
\label{eqn_2023_04_11_1039}
\Sel_{\rm Gr}(T_{\f,g}^\vee(1)/\QQ(\mu_{p^\infty}))[\cP_i]\simeq \Sel_{\rm Gr}(T_{\f(\kappa_i),g}^\vee(1)/\QQ(\mu_{p^\infty}))\,.
\end{equation}

Only in this paragraph, let us set $M:=\Sel_{\rm Gr}(T_{\f,g}^\vee(1)/\QQ(\mu_{p^\infty}))$ to ease our notation. Passing to Pontryagin duals in the tautological exact sequence
$$0\lra M[\cP_i]\lra M \xrightarrow{\times P_i} P_iM \lra 0\,$$
and noting that $(P_iM)^\vee=P_iM^\vee$ (by the definition of the $\LL_\f$-action on the Pontryagin dual of $M$), yield the short exact sequence
$$0\lra \cP_i M^\vee \lra M^\vee \lra M[\cP_i]^\vee \lra 0\,,$$
and therefore an isomorphism
\begin{equation}
    \label{eqn_2023_04_11_1049}
M^\vee/\cP_i M^\vee \xrightarrow{\,\,\sim\,\,} M[\cP_i]^\vee \,.
\end{equation}
We can now combine \eqref{eqn_2023_04_11_1039} and \eqref{eqn_2023_04_11_1049} to deduce that
$$\Sel_{\rm Gr}(T_{\f,g}^\vee(1)/\QQ(\mu_{p^\infty}))^\vee\Big{/}\cP_{i}\,\Sel_{\rm Gr}(T_{\f,g}^\vee(1)/\QQ(\mu_{p^\infty}))^\vee\xrightarrow{\,\sim\,} \Sel_{\rm Gr}(T_{\f(\kappa_i),g}^\vee(1)/\QQ(\mu_{p^\infty}))^\vee\,.$$
Hence,
\begin{align*}
\Char_{\Lambda_\cO(\Gamma_1)}e_{\theta}&\left(\Sel_{\rm Gr}(T_{\f,g}^\vee(1)/\QQ(\mu_{p^\infty}))^\vee[\cP_{i}]\right)\, \cdot\, (\varpi^{\mu_2}S_i)\\
&\qquad\qquad\qquad\qquad\qquad= \Char_{\Lambda_\cO(\Gamma_1)}e_{\theta}\left(\Sel_{\rm Gr}(T_{\f(\kappa_i),g}^\vee(1)/\QQ(\mu_{p^\infty}))^\vee\right)\,,
\end{align*}
in particular, also that 
\begin{equation}
\label{eqn_reduction_Si_step_1}
    \varpi^{\mu_2}S_i\,\mid\, \Char_{\LL_\cO(\Gamma_1)}e_{\theta}\left(\Sel_{\rm Gr}(T_{\f(\kappa_i),g}^\vee(1)/\QQ(\mu_{p^\infty}))^\vee\right)\,.
\end{equation}
Observe also that 
\begin{equation}
    \label{eqn_reduction_Li_step_1}
   \varpi^{\mu_1}\fL_i=e_\theta L_p^{\rm geo}(\f(\kappa_i),g) 
\end{equation}
 by definition.
 
 Under the assumptions of our theorem, the hypotheses of Theorem~\ref{thm_cyclo_main_conj_fotimesg} are valid with the choice $f=\f(\kappa_i)$, for any positive integer $i$. It follows from Theorem~\ref{thm_cyclo_main_conj_fotimesg} (applied with $f=\f(\kappa_i)$)
 $$\Char_{\Lambda_\cO(\Gamma_1)}e_{\theta}\left(\Sel_{\rm Gr}(T_{\f(\kappa_i),g}^\vee(1)/\QQ(\mu_{p^\infty}))^\vee\right)\,\mid\, e_\theta\,\varpi^{s(g)} L_p^{\rm geo}(\f(\kappa_i),g)\,.$$
 This, combined with \eqref{eqn_reduction_Si_step_1} and \eqref{eqn_reduction_Li_step_1}, allows us to conclude that $\varpi^{\mu_2} S_i \, \mid\, \varpi^{s(g)+\mu_1}  \fL_i $. Since $\mu$-invariants of $\fL_i$ and $S_i$ are both zero (by definition), it follows that $\mu_2\leq s(g)+\mu_1$, and also that
 \begin{equation}
     \label{eqn_reduction_SiLi_step_2}
     S_i\,\mid\,\fL_i\,.
 \end{equation}
 We may now use Proposition~\ref{prop_appendix_regular_divisibility} (with the choices $R=\LL_\f(\Gamma_1)$, $R_0:=\Zp[[1+p\ZZ_p]]$ the weight space, $\iota_0:R_0\to \LL_\f\to R$ the structure map and $F=S$, $G=\fL$), we conclude that $S\mid \fL$, as required.
\end{proof}

\section[Main conjectures at inert primes]{Main conjectures over an imaginary quadratic field where $p$ is inert}
\label{subsec_IMCimagquadinert}
In this subsection, we shall discuss the consequences of Theorem~\ref{thm_cyclo_main_conj_fotimesg} and Theorem~\ref{thm_cyclo_main_conj_ffotimesg} to the Iwasawa theory for ${\rm GL}_2\times {{\rm GL}_1}_{/K}$, where $K$ is an imaginary quadratic field where $p$ is \emph{inert}. Fix such an imaginary quadratic number field $K$ and let $\Gamma_K=\Gal(K_\infty/K)$ denote the Galois group of the unique $\ZZ_p^2$-extension of $K$. 

To that end, we shall pick an eigenform $g$ which has CM by $K$ and has weight $k_g+2\geq 2$. As such, $g$ arises as the theta-series of an $A_0$-type Hecke character $\psi$ of $K$. As a start, in \S\ref{subsubsec_cyclo_main_inert}, we shall recast (in Theorem~\ref{thm_cyclo_main_inert_f} and Theorem~\ref{thm_cyclo_main_inert_ff}) our Theorems~\ref{thm_cyclo_main_conj_fotimesg} and \ref{thm_cyclo_main_conj_ffotimesg} above as results towards the cyclotomic main conjectures for the Rankin--Selberg products $f_{/K}\otimes\psi$ and $\f_{/K}\otimes \psi$ of the base change of a (family of) non-CM form(s) to $K$ and the Hecke character $\psi$. 

In \S\ref{subsubsec_3var_inert_ord}, we explain how to interpolate the conclusions of Theorem~\ref{thm_cyclo_main_inert_ff} under a very natural assumption (which we verify granted the existence of a rank-$2$ Euler system in the present setting in Section~\ref{subsubsec_ESrank2}) to a divisibility statement in the Iwasawa main conjectures over $\LL_\f(\Gamma_K)$; see Theorem~\ref{thm_3var_main_inert_ff} for our main result in this direction. Utilizing the descent formalism of \cite[\S5.3.1]{BL_SplitOrd2020} (which relies crucially on the work of Nekov\'a\v{r}~\cite{nekovar06}), we shall obtain divisibilities in both ``definite'' and ``indefinite'' Iwasawa main conjectures over $\LL_\f(\Gamma_\ac)$, where $\Gamma_\ac:=\Gal(K_\ac/K)$ is the Galois group of the anticyclotomic $\ZZ_p$-extension $K_\ac/K$.

We briefly comment why we use two versions of Selmer groups (namely, \`a la Greenberg versus Nekov\'a\v{r}) in our work. The Euler system machinery, which is what we employ as the first step of our argument, is set up in terms of Greenberg Selmer groups. We translate these statements to those in terms of Selmer complexes for the key reason that Selmer complexes are often perfect objects in the derived category. In favourable scenarios, this key property allows one to identify the characteristic ideals of the relevant Selmer groups with determinants, allowing one to avoid an analysis of pseudo-null submodules when employing descent arguments.

\subsection{The setting}
\label{subsubsec_setting_fKpsi} 
Let $D_K$ denote the discriminant of the imaginary quadratic number field $K$ which we have fixed above, and let $\cO_K$ denote its ring of integers. We let\index{Hecke characters of $K$! $\psi$}
$$\psi:\mathbb{A}_K^\times/K^\times \lra \mathbb{C}^\times$$ denote an $A_0$-type Hecke character with infinity type $(0,k_g+1)$ (where $k_g$ is a natural number) and conductor $\ff$,  where $\ff$ is coprime to $p$. We let $\widehat{\psi}$\index{Hecke characters of $K$! $\widehat{\psi}$ (Galois character attached to $\psi$)} denote the Galois character associated (via the geometrically normalized Artin map of global class field theory) to the $p$-adic avatar of $\psi$. We say that $\psi$ is a crystalline Hecke character\index{Hecke characters of $K$! Crystalline Hecke character} to mean that $\widehat{\psi}$ is crystalline at $p$. In \S\ref{subsubsec_cyclo_main_inert}, we will work with a fixed choice of crystalline $\psi$, whereas in \S\S\ref{subsubsec_3var_inert_ord}--\ref{subsubsec_anticyclo_3}, we will allow $\psi$ vary among crystalline Hecke characters (but keeping $\ff$ fixed) to prove our main results.

We put\index{Hecke characters of $K$! $\theta(\psi)$ ($\theta$-series of $\psi$)}
$$\theta(\psi):=\sum_{(\mathfrak{a},\ff)=1}\psi(\mathfrak{a})q^{\bf{N}\mathfrak{a}}\in S_{k_g+2}(\Gamma_1(N_g))$$ 
where $N_g=|D_K|\,\bf{N}\ff$. In what follows, we shall replace $g$ in \S\ref{subsec_fotimesg} with $\theta(\psi)$. Note that since we have assumed that $p$ is inert in $K/\QQ$, it follows that $a_p(\theta(\psi)))=0$, i.e., the eigenform $\theta(\psi)$ is indeed $p$-non-ordinary.

\begin{defn}
\label{define_CM_nebentype}
Let us define the Dirichlet character $\varepsilon_\psi$ on setting\index{Hecke characters of $K$! $\varepsilon_\psi$ (Nebentype)}
$$\varepsilon_\psi:=\epsilon_K\psi_{\vert_{\mathbb{A}_{\QQ}^\times}}\mathbf{N}^{-k_g-1}$$
where $\mathbf{N}$ is the norm character \index{Hecke characters of $K$! $\mathbf{N}$ (Norm character)} on $\mathbb{A}_\QQ^\times$ and $\epsilon_K$\index{$K$: imaginary quadratic field! $\epsilon_K$ (quadratic character associated to $K/\QQ$)}  is the quadratic character associated to $K/\QQ$. Then $\varepsilon_\psi$ is the nebentype character of the cuspidal eigen-newform $\theta(\psi)$. Let $\varepsilon_f$ denote the nebentype of the eigenform $f$ (as well as the Hida family $\f$ we have fixed). 
\end{defn}

We let $f\in S_{k_f+2}(\Gamma_1(N_f),\varepsilon_f)$ be a primitive eigenforms as in \S\ref{subsec_setting}, where $N_f$ and $N_g$ are coprime and $k_f>k_g$. We recall also that we are working under the assumption that $f$ is $p$-ordinary; we denote by $f^\alpha$ its $p$-ordinary specialization. Thanks to our assumptions that $\ff$, $D_K$ and $N_f$ are pairwise coprime, observe that $\varepsilon_\psi\varepsilon_f$ can never be the trivial character.

We now let $\cO$ be the ring of integers of a finite extension ${\rm Frac}(\cO)$ of $\QQ_p$ containing the images $\iota_p(K_f)$ and $\iota_p(\psi(\mathbb{A}_K^\times))$ of the Hecke fields, as well as the roots of the Hecke polynomials of $f$ and $\theta(\psi)$ at $p$.

As before, we let $\LL_{\f}$ denote the branch of Hida's universal (ordinary) Hecke algebra which admits $f^\alpha$ as a specialization. We again assume (without loss of generality) that $\LL_{\f}$ contains $\cO$ as a subring and that $\LL_\f$ is a regular ring (c.f., Remark~\ref{remark_regular_is_light_assumption}).

We assume throughout \S\ref{subsec_IMCimagquadinert} that the residual representation $\overline{\rho}_\f=\overline{\rho}_f$ is irreducible. The Galois representation $T_{f,\theta(\psi)}:=R_f^*\otimes R_{\theta(\psi)}^*$, which is a free $\cO$-module of rank $4$, can be (and will be) identified with ${\rm Ind}_{K/\QQ}T_{f,\psi}$, where $T_{f,\psi}:=R_f^*\otimes\widehat{\psi}^{-1}$\index{Galois representations! $T_{f,\psi}$}. We set $\TT_{f,\psi}^{\rm cyc}:=T_{f,\psi}\otimes \LL_\cO(\Gamma_1)^{\iota}$,\index{Galois representations! $\TT_{f,\psi}^?$} where $\LL_\cO(\Gamma_1)^{\iota}$ is the $\LL_\cO(\Gamma_1)$-module of rank one on which $G_K$ acts via the character $G_K\twoheadrightarrow \Gamma_1\xrightarrow{\gamma\mapsto\gamma^{-1}}\Gamma_1\hookrightarrow \LL_\cO(\Gamma_1)^\times$, and where $G_K$ acts on this tensor product diagonally. We similarly define $\TT_{\f,\psi}^{\rm cyc}$. We analogously define the ``big'' Galois representations $\TT_{f,\psi}^{\rm ac}:=T_{f,\psi}\otimes \LL_\cO(\Gamma_\ac)^{\iota}$, $\TT_{f,\psi}^K:=T_{f,\psi}\otimes \LL_\cO(\Gamma_K)^{\iota}$ as well as $\TT_{\f,\psi}^{\rm ac}$ and $\TT_{\f,\psi}^K$.\index{Galois representations! $\TT_{\f,\psi}^?$}

We next define the Selmer complexes associated to these $G_K$-representations. We will choose to work with these Selmer complexes (rather than their classical counterparts, the Greenberg Selmer groups) due to the utility of Nekov\'a\v{r}'s base change and descent formalism (which we shall crucially rely on in \S\S\ref{subsubsec_3var_inert_ord}--\ref{subsubsec_anticyclo_3}).

\begin{defn}
\label{defn_Selmer_complex_inert_ord}
Let $\Sigma$ denote the set of places of $K$ which divide $p\ff N_f|D_K|\infty$. In what follows, $?$ stands for any one of the symbols $\{\ac,\cyc,K\}$. The complex \index{Selmer complexes! $\widetilde{{\bf R}\Gamma}_{\rm f}(G_{K,\Sigma},\TT_{f,\psi}^{?};\Delta_{\Gr})$ (Greenberg Selmer complex)}
$$\widetilde{{\bf R}\Gamma}_{\rm f}(G_{K,\Sigma},\TT_{f,\psi}^{?};\Delta_{\Gr})\in D_{\rm ft} (_{\LL_\cO(\Gamma_?)}{\rm Mod})$$
is the Greenberg Selmer complex given by the local conditions which are unramified for all primes in $\Sigma$ that are coprime to $p$, and which is given by the strict Greenberg conditions with the choice
$$j_{p}^+:\,F^+\TT_{f,\psi}^{?}:=F^+R_f^*\otimes \widehat{\psi}^{-1}\otimes \LL_\cO(\Gamma_?)^\iota \longrightarrow \TT_{f,\psi}^{?}$$
at the unique prime of $K$ above $p$ (which we shall abusively denote by $p$ as well). We shall denote its cohomology by $\widetilde{H}^\bullet_{\rm f}(G_{K,\Sigma},\TT_{f,\psi}^{?};\Delta_{\Gr})$.\index{Selmer complexes! $\widetilde{H}^\bullet_{\rm f}(G_{K,\Sigma},\TT_{f,\psi}^{?};\Delta_{\Gr})$ (Extended Greenberg Selmer groups)}

We similarly define the complex $\widetilde{{\bf R}\Gamma}_{\rm f}(G_{K,\Sigma},\TT_{\f,\psi}^{?};\Delta_{\Gr})$ as well as its cohomology groups $\widetilde{H}^\bullet_{\rm f}(G_{K,\Sigma},\TT_{\f,\psi}^{?};\Delta_{\Gr})$.  \index{Selmer complexes! $\widetilde{{\bf R}\Gamma}_{\rm f}(G_{K,\Sigma},\TT_{\f,\psi}^{?};\Delta_{\Gr})$ (Greenberg Selmer complex)}\index{Selmer complexes! $\widetilde{H}^\bullet_{\rm f}(G_{K,\Sigma},\TT_{\f,\psi}^{?};\Delta_{\Gr})$ (Extended Greenberg Selmer groups)}
\end{defn}

Let us denote the set of places of $\QQ$ that lie below the set $\Sigma$ also with $\Sigma$. Attached to the $G_{\QQ,\Sigma}$-representation $T_{f,\theta(\psi)}$, we may similarly define a Selmer complex
$\widetilde{{\bf R}\Gamma}_{\rm f}(G_{\QQ,\Sigma},T_{f,\theta(\psi)};\Delta_{\Gr})$ given by the local conditions which are unramified for all primes in $\Sigma$ that are coprime to $p$, and which is given by the strict Greenberg condition with the choice
$$j_{p}^+:\,F^+T_{f,\theta(\psi)}:=F^+R_f^*\otimes R_{\theta(\psi)}^* \longrightarrow T_{f,\theta(\psi)}$$
at $p$. Shapiro's lemma then induces an isomorphism
$$\mathscr{S}:\widetilde{{\bf R}\Gamma}_{\rm f}(G_{\QQ,\Sigma},T_{f,\theta(\psi)};\Delta_{\Gr})\stackrel{\sim}{\lra} \widetilde{{\bf R}\Gamma}_{\rm f}(G_{K,\Sigma},T_{f,\psi};\Delta_{\Gr})\,.$$

\begin{proposition}
\label{prop_NekvseGre}
Let us denote by $\mathds{1}$ the trivial character of $\Delta$. We have,
$$e_{\mathds{1}}\,\Char_{\LL_\cO(\Gamma_\cyc)}\left(\Sel_{\rm Gr}(T_{f,\theta(\psi)}^\vee(1)/\QQ(\mu_{p^\infty}))^\vee\right)=\Char_{\LL_\cO(\Gamma_1)}\left(\widetilde{H}^2_{\rm f}(G_{K,\Sigma},\TT_{f,\psi}^{\cyc};\Delta_{\Gr})^{\iota}\right)\,.$$
\end{proposition}

\begin{proof}
By control theorems and the isomorphism $\mathscr{S}$ induced by Shapiro's lemma, it suffices to prove that 
$$|\Sel_{\rm Gr}(T_{f,\theta(\psi)}^\vee(1+\eta)/\QQ(\mu_{p^\infty}))^\vee|=|\widetilde{H}^2_{\rm f}(G_{K,\Sigma},T_{f,\psi\eta}^{\cyc};\Delta_{\Gr})|$$
for infinitely many characters $\eta$ of $\Gamma_1$. This fact was proved in \cite{BL_SplitOrd2020}, Lemma 5.7 and Lemma 5.8. Note that the set up in op. cit. a priori requires $p$ be split in $K/\QQ$, but this assumption plays no role in the proofs of Lemma 5.7 and Lemma 5.8. It is also worth noting we are using the fact that $\eta\circ\iota=\eta^{-1}$ (and this is the reason for the use of the twisted $\Gamma_1$-action in the definitions of $\TT_{f,\chi}^\cyc$). 
\end{proof}

\begin{defn}
\label{defn_geo_for_trivial_char}\index{$p$-adic $L$-functions!  $L_p^{\rm RS}(f_{/K}\otimes\psi)$}
We set 
$$L_p^{\rm RS}(f_{/K}\otimes\psi):=e_{\mathds{1}}L_p^{\rm geo}(f,\theta(\psi))\in \LL_\cO(\Gamma_1)$$ 
and similarly define $L_p^{\rm RS}(\f_{/K}\otimes\psi)\in  \LL_\f(\Gamma_1)$. \index{$p$-adic $L$-functions!  $L_p^{\rm RS}(\f_{/K}\otimes\psi)$}
\end{defn}
\subsection{Cyclotomic main conjectures in the inert case}
\label{subsubsec_cyclo_main_inert}
We consider the following big image condition: \index{Big image conditions! {{\bf (Full)}}}

\begin{enumerate}[align=parleft, labelsep=0.2cm,]
\item[\mylabel{item_fullness_main_body}{{\bf{(Full)}}}]\,\, ${\rm SL}_2(\FF_p)\subset  \overline{\rho}_\f(G_{\QQ(\mu_{p^\infty})})$\,.
\end{enumerate}
In Theorem~\ref{appendix_big_images_subsec_2_thm_3}, we explain that the condition~\ref{item_fullness_main_body} together with the assumption that $p\geq 7$ is sufficient to ensure the validity of the hypothesis \ref{item_BI_fg} when $g=\theta(\psi)$. 

 The following divisibility statement in the cyclotomic main conjecture for the Rankin--Selberg product $f_{/K}\otimes \psi$ is a reformulation of Theorem~\ref{thm_cyclo_main_conj_fotimesg}, in view of Proposition~\ref{prop_NekvseGre} and Definition~\ref{defn_geo_for_trivial_char}. Note also that the hypothesis that $k_g\neq p-1$ and $p+1\nmid k_g+1$ guarantees that one of \ref{item_Irr1}--\ref{item_Irr3} holds true.

\begin{theorem}
\label{thm_cyclo_main_inert_f}
Suppose that $\overline{\rho}_f$ is absolutely irreducible as well as that $k_g\neq p-1$ and $p+1\nmid k_g+1$. Assume also that $p\geq 7$ and \ref{item_fullness_main_body} holds true. Then,
$$\varpi^{s(g)}L_p^{\rm RS}(f_{/K}\otimes\psi)\,\in \,\Char_{\LL_\cO(\Gamma_1)}\left(\widetilde{H}^2_{\rm f}(G_{K,\Sigma},\TT_{f,\psi}^{\cyc};\Delta_{\Gr})^\iota\right) \,.$$
\end{theorem}

Using Theorem~\ref{thm_cyclo_main_inert_f} as $f_\alpha$ varies in the Hida family $\f$, we have the following divisibility statement in the cyclotomic main conjecture for the family $\f_{/K}\otimes \psi$:

\begin{theorem}
\label{thm_cyclo_main_inert_ff}
Suppose that $\overline{\rho}_\f$ is absolutely irreducible, as well as that $k_g\neq p-1$ and $p+1\nmid k_g+1$. Assume also that $p\geq 7$ and \ref{item_fullness_main_body} holds true. Then,
$$ \varpi^{s(\psi)} L_p^{\rm RS}(\f_{/K}\otimes\psi)\,\in \,\Char_{\LL_\cO(\Gamma_1)}\left(\widetilde{H}^2_{\rm f}(G_{K,\Sigma},\TT_{\f,\psi}^{\cyc};\Delta_{\Gr})^\iota\right)$$
where $s(\psi)=s(\theta(\psi))$ is given as in Theorem~\ref{thm_BF_factorization_for_families}.
\end{theorem}

\begin{proof}
The proof of this theorem is formally identical to the proof of Theorem~\ref{thm_cyclo_main_conj_ffotimesg} and we shall only provide a brief sketch. 

As in the proof of Theorem~\ref{thm_cyclo_main_conj_ffotimesg}, one relies on the divisibility criterion Proposition~\ref{prop_appendix_regular_divisibility} and the control theorem for Nekov\'a\v{r}'s extended Selmer groups (c.f., \cite{nekovar06}, Corollary 8.10.2) to reduce to the validity of Theorem~\ref{thm_cyclo_main_inert_f} for infinitely many crystalline specializations $\f$. Thanks to our running assumptions, this holds true. 
\end{proof}

\begin{remark}
\label{rem_s_g_uniform_bound_required}
We will use Proposition~\ref{prop_appendix_regular_divisibility} to patch the conclusions of Theorem~\ref{thm_cyclo_main_inert_ff} as $\psi$ varies among crystalline Hecke characters (with conductor dividing $\ff p^\infty$) to a statement towards a 3-variable main conjecture over $\LL_\f(\Gamma_K)$, assuming in addition that the integers $s(\psi)$ that appear in the statement of Theorem~\ref{thm_cyclo_main_inert_ff} are uniformly bounded as $\psi$ varies (see \S\ref{subsubsec_3var_inert_ord} where we employ this idea; see also \S\S\ref{subsubsec_anticyclo_2}--\ref{subsubsec_anticyclo_3} for applications in the anticyclotomic main conjectures in the inert case, both in the definite and the indefinite setting). 

We recall that by Corollary~\ref{cor_s_g_is_uniformly_bounded}, the exponents $s(\psi)$ are uniformly bounded as $\psi$ varies, granted the existence of a rank-$2$ Euler system that the Perrin-Riou philosophy predicts. 
\end{remark}

\begin{remark}
\label{rem_s_g_uniform_bound_required_but_not_enough}
In contrast to the discussion in Remark~\ref{rem_s_g_uniform_bound_required}, the conclusions of Theorem~\ref{thm_cyclo_main_inert_f} as $\psi$ varies among crystalline Hecke characters (with fixed conductor dividing $\ff p^\infty$) \emph{cannot} be patched to a statement towards a 2-variable main conjecture over $\LL_\cO(\Gamma_K)$, since there are only finitely many $\psi$ verifying the conditions of Theorem~\ref{thm_cyclo_main_inert_f}, as this theorem requires $k_g<k_f$. This is primarily the reason for our emphasis on the signed-splitting procedure for Beilinson--Flach elements for families (c.f. \S\ref{subsec_review_BF_elements}) and on Theorem~\ref{thm_cyclo_main_inert_ff}.
\end{remark}

\subsection{Results on the $3$-variable main conjectures over an imaginary quadratic field where $p$ is inert}
\label{subsubsec_3var_inert_ord}
Let us fix a ray class character $\chi$ of $K$\index{Hecke characters of $K$! $\chi$ (branch character)} with conductor dividing $\ff p^\infty$ (where we assume that $\ff$ is as before) and order coprime to $p$. We assume that $\widehat{\chi}_{\vert_{G_{\QQ_{p^2}}}}\neq \widehat{\chi^c}_{\vert_{G_{\QQ_{p^2}}}}$; in particular, the conductor of $\chi$ is necessarily divisible by $p$. Note that $\QQ_{p^2}$ stands for the completion of $K$ at its unique prime above $p$. Let us denote by $$\Psi:G_K\twoheadrightarrow\Gamma_K\xrightarrow{\gamma\mapsto \gamma^{-1}}\Gamma_K \hookrightarrow \LL_\cO(\Gamma_K)^\times$$ 
$$\Psi_{1}:G_K\twoheadrightarrow\Gamma_1\xrightarrow{\gamma\mapsto \gamma^{-1}}\Gamma_1 \hookrightarrow \LL_\cO(\Gamma_1)^\times$$ 
the tautological characters\index{$K$: imaginary quadratic field! Universal characters $\Psi_?$}. Put $\bbchi:=\widehat{\chi}\Psi$ and $r_{\bbchi}:={\rm Ind}_{K/\QQ}\bbchi$.\index{$K$: imaginary quadratic field! Universal characters $\bbchi_?$}
\index{$K$: imaginary quadratic field!  $r_{\bbchi}$}
\begin{defn}
\label{defn_univ_sp_map}
Let us fix an $\cO$-valued character $\rho$ of $\Gamma_K$. Denote by $x_{\rho}:\LL_\cO(\Gamma_K)\to \LL_\cO(\Gamma_1)$ the unique  homomorphism of $\cO$-algebras which induces the isomorphism
$$\Psi\otimes_{x_{\rho}}\LL_\cO(\Gamma_1)\stackrel{\sim}{\lra} \rho^{-1}\otimes\Psi_1\,.$$
In explicit terms, $x_\rho(\gamma)=\rho(\gamma)\Psi_1(\gamma)$ for every $\gamma\in \Gamma_K$. We let $X_\rho:=\ker(x_\rho)\subset \LL_\cO(\Gamma_K)$ denote the corresponding height-one prime ideal.

%\item[ii)] We let ${\rm Tw}_\rho:\LL_{\cO}(\Gamma_K) \to \LL_{\cO}(\Gamma_K)$ be the $\cO$-linear twisting map given by $\gamma\mapsto \rho(\gamma)\gamma$ on group like elements $\gamma\in \Gamma_K\subset\LL_{\cO}(\Gamma_K)$. Let us also set $$\cO_\rho:=\LL_\cO(\Gamma_K)/(\gamma_1-\rho(\gamma_1),\gamma_2-\rho(\gamma_2)),$$ which is a free $\cO$-module of rank one on which $\Gamma_K$ acts via $\rho$.
\end{defn} 
Suppose $\psi$ is a crystalline Hecke character as before, with the additional requirement that $\widehat{\chi\psi}$ factors through $\Gamma_K$. We will repeatedly make use of the following observation:
\begin{equation}
    \label{eqn_specialize_bbchi_at_chipsi}
    \bbchi\,\otimes_{x_{\widehat{\chi\psi}}}\LL_{\cO}(\Gamma_1)=\widehat{\chi}\Psi\,\otimes_{x_{\widehat{\chi\psi}}}\LL_{\cO}(\Gamma_1)\stackrel{\sim}{\lra} \widehat{\chi}\otimes\widehat{\chi\psi}^{-1}\otimes \Psi_1 = \widehat{\psi}^{-1}\otimes \Psi_1\,.
\end{equation}

Given $\rho$ as in Definition~\ref{defn_univ_sp_map}, we shall also denote the induced homomorphism $\LL_\f(\Gamma_K)\to \LL_\f(\Gamma_1)$ by the same symbol $x_{\rho}$.
\begin{remark}
\label{remark_univ_sp_map}
Since $X_\rho$ is a height-one prime of the regular ring $\LL_{\cO}(\Gamma_K)$, it is principal. In this remark, we will describe a certain generator of $X_\rho$, which will be useful in what follows.

We first describe the kernel of the natural continuous surjection of $\cO$-algebras
\begin{equation}
    \label{eqn_x_0_map}
     x_0:\LL_{\cO}(\Gamma_K)\lra \LL_{\cO}(\Gamma_1),
\end{equation}
which is also a height-one (therefore principal) prime ideal of $\LL_{\cO}(\Gamma_K)$. Let ${\rm ver}_{\rm ac}: \Gamma_\ac\to\Gamma_K$ be the verschiebung map, given by $\gamma\mapsto \widehat{\gamma}^{c-1}$, where $\widehat{\gamma}\in \Gamma_K$ is any lift of $\gamma\in \Gamma_\ac$ and $c$ is the generator of $\Gal(K/\QQ)$. Let us put $\gamma_0={\rm ver}_{\rm ac}(\gamma_{\rm ac})$, where $\gamma_{\rm ac}$ is a topological generator of $\Gamma_\ac$. Then $x_0(\gamma_0)=1$ and hence $(\gamma_0-1)\subset \ker(x_0)$. Since both $\ker(x_0)$ and $(\gamma_0-1)$ are height one primes, it follows that $\ker(x_0)=(\gamma_0-1)$. Abusing the language, we shall also put $x_0:=\gamma_0-1$.

Fix now an $\cO$-valued character $\rho$ of $\Gamma_K$ as in Definition~\ref{remark_univ_sp_map}. We will consider the topological sub-algebra $\cO[[\gamma_0-1]]=\cO[[x_0]]\subset \LL_{\cO}(\Gamma_K)$, with its prime ideal generated by $f_\rho:=\gamma_0-\rho(\gamma_0)\subset \cO[[x_0]]$. We will also regard $\rho$ as a ring homomorphism $\cO[[x_0]] \to \cO$, given by $\gamma_0\mapsto \rho(\gamma_0)$. Via the containment $\cO[[x_0]]\subset \LL_{\cO}(\Gamma_K)$, we may and will treat $f_\rho$ as an element of $\LL_{\cO}(\Gamma_K)$. 

We next check that $X_\rho\subset \LL_{\cO}(\Gamma_K)$ is generated by $f_\rho$. To see that, we first observe that 
$$\rho\Psi_1(\gamma_0-\rho(\gamma_0))=\rho\Psi_1(\gamma_0)-\rho(\gamma_0)=\rho(\gamma_0)\Psi(\gamma_0)-\rho(\gamma_0)=0\,,$$
which means that $f_\rho=\gamma_0-\rho(\gamma_0) \in \ker(x_\rho)= X_\rho$. Since both $(f_\rho)$ and $X_\rho$ are height-one primes of $\LL_{\cO}(\Gamma_K)$ and $(f_\rho)\subset X_\rho$, it follows that $(f_\rho)=X_\rho$, as required.
\end{remark}
We consider the following uniform boundedness condition on the variation of the possible denominators $s(\psi)$ as the Hecke character $\psi$ varies:
\begin{enumerate}[align=parleft, labelsep=0.4cm,] \index{Beilinson--Flach elements! ${\mathbf{(Bdd_{s(\psi)})}}$}
\item[\mylabel{item_bounded_spsi}{{{$\mathbf{(Bdd_{s(\psi)})}$}}}] \,\,\,\,\,\,\,\, There exists a sequence of crystalline Hecke characters $\psi_i$ such that the Galois characters $\widehat{\chi\psi}_i$ factor through $\Gamma_K$, and the infinity type of $\psi_i$ equals  $(0,k_i+1)$ where $k_i \in \mathbb{N}$ and $k_i\to \infty$ as $i\to \infty$ for which the collection of integers $\{s(\psi_i)\}$ is bounded independently of $i$.
\end{enumerate}
We recall that by Corollary~\ref{cor_s_g_is_uniformly_bounded}, the condition \ref{item_bounded_spsi} holds true granted the existence of a rank-$2$ Euler system that the Perrin-Riou philosophy predicts. 
\begin{theorem}
\label{thm_3var_main_inert_ff}
Suppose $\chi$ is a ray class character as above and assume that $\overline{\rho}_\f$ is absolutely irreducible.
\item[i)] There exists an element $L_p^{\rm RS}(\f_{/K}\otimes\bbchi)\in \LL_\f(\Gamma_K)$ \index{$p$-adic $L$-functions! $L_p^{\rm RS}(\f_{/K}\otimes\bbchi)$ (``semi-universal'' Rankin--Selberg $p$-adic $L$-function)} with the following interpolation property: For any crystalline Hecke character $\psi$ as in \S\ref{subsubsec_setting_fKpsi} and such that $\widehat{\chi\psi}$ factors through $\Gamma_K$, we have
$$x_{\widehat{\chi\psi}}\left(L_p^{\rm RS}(\f_{/K}\otimes\bbchi)\right)=L_p^{\rm RS}(\f_{/K}\otimes\psi)\,,$$
where the specialization map $x_{\widehat{\chi\psi}}$ is the one described in Definition~\ref{defn_univ_sp_map}.
\item[ii)] Assume that $p\geq 7$ as well as that the conditions \ref{item_fullness_main_body} and \ref{item_bounded_spsi} hold true. We then have the following containment in the Iwasawa main conjecture for the family $\f_{/K}\otimes\bbchi$ of Rankin--Selberg products:
\begin{equation}
    \label{eqn_3_var_main_conj}
    L_p^{\rm RS}(\f_{/K}\otimes\bbchi)\,\in \,\Char_{\LL_\f(\Gamma_K)}\left(\widetilde{H}^2_{\rm f}(G_{K,\Sigma},\TT_{\f,\chi}^{K};\Delta_{\Gr})^\iota\right)\otimes_{\ZZ_p}\QQ_p\,.
\end{equation}
\end{theorem}
In particular, the containment \eqref{eqn_3_var_main_conj} holds assuming only $p\geq 7$ and the validity of \ref{item_fullness_main_body}, if Conjecture~\ref{conj_ESrank2} (on the existence of rank-$2$ Euler systems) holds true.

\begin{proof}[Proof of Theorem~\ref{thm_3var_main_inert_ff}]
\item[i)] The existence of $L_p^{\rm RS}(\f_{/K}\otimes\bbchi)$ is an almost direct consequence of the recent work \cite{Loeffler2020universalpadic}, extending the results of op. cit. slightly to construct a $p$-adic $L$-function
$$L_p^{\rm RS}(\f_{/K}\otimes\rho^{\rm univ}_{r_{\overline{\chi}}})\in \LL_\f\,\widehat\otimes\, R^{\rm univ}(r_{\overline{\chi}})$$\index{$p$-adic $L$-functions! $L_p^{\rm RS}(\f_{/K}\otimes\rho^{\rm univ}_{r_{\overline{\chi}}})$}
where $r_{\overline{\chi}}:={\rm Ind}_{K/\QQ}\,\overline{\chi}^{-1}$ and   $\rho^{\rm univ}_{r_{\overline{\chi}}}$ is the minimally ramified universal deformation representation, and where $R^{\rm univ}(r_{\overline{\chi}})$ is the minimally ramified universal deformation ring of $r_{\overline{\chi}}$. This construction is carried out in detail (following Loeffler's work very closely) in Appendix~\ref{appendix_sec_padicRankinSelbergHida}. The key point is that our running assumptions on $\chi$ imply that 
\begin{itemize}
    \item The representations $r_{\overline{\chi}}\,{\vert_{G_{\QQ(\mu_p)}}}$ and $r_{\overline{\chi}}\,{\vert_{G_{\QQ_p}}}$ are both absolutely irreducible (since we assumed that $\chi\neq \chi^c$ and that the order of $\chi$ is prime to $p$);
    \item the lift $r_{\bbchi}={\rm Ind}_{K/\QQ}\,\bbchi$ is minimally ramified (since we assumed that the order of $\chi$ is prime to $p$).
\end{itemize}
The second property induces (by the universality of $\rho^{\rm univ}_{r_{\overline{\chi}}}$) a continuous ring homomorphism $\phi_{\bbchi}: R^{\rm univ}(r_{\overline{\chi}})\to \LL_{\cO}(\Gamma_K)$ and the $p$-adic $L$-function $L_p^{\rm RS}(\f_{/K}\otimes\bbchi)$ is defined as the image of $L_p^{\rm RS}(\f_{/K}\otimes\rho^{\rm univ}_{r_{\overline{\chi}}})$ under the map ${\rm id}\otimes \phi_{\bbchi}: \LL_\f\widehat{\otimes}\,R^{\rm univ}(r_{\overline{\chi}})\to \LL_f(\Gamma_K)$.
\item[ii)] The proof of this portion is very similar to the proof of Theorem~\ref{thm_cyclo_main_conj_ffotimesg}. Let us put $\fL:=\varpi^{\mu_1}L_p^{\rm RS}(\f_{/K}\otimes\bbchi)$ where $\mu_1$ is the unique natural number with $\mathfrak{L}\in \LL_\f(\Gamma_K)\setminus p\LL_\f(\Gamma_K)$. Let us also choose an element 
$S\in \LL_\f(\Gamma_K)\setminus p\LL_\f(\Gamma_K)$ such that $$\varpi^{\mu_2}\LL_\f(\Gamma_K)S=\Char_{\LL_\f(\Gamma_K)}\left(\widetilde{H}^2_{\rm f}(G_{K,\Sigma},\TT_{\f,\chi}^{K};\Delta_{\Gr})^\iota\right)$$
for the suitable choice of a (uniquely determined) natural number $\mu_2$. Let us put $s:=\sup_{i=1}^\infty\{s(\psi_i)\}$, where $\psi_i$ are the Hecke characters given as in the statement of our theorem. We will prove that $S \mid \fL$, using the conclusion of Theorem~\ref{thm_cyclo_main_inert_ff} applied with $\psi=\psi_i$ and the divisibility criterion established in Appendix~\ref{Appendix_Regular_Rings_Divisibility}. This will in turn prove the validity of the  containment in the statement of our theorem.

As in Remark~\ref{remark_univ_sp_map} (whose notation we shall adopt in the remainder of this proof), we will consider the topological subring $\cO[[x_0]]\subset \LL_{\cO}(\Gamma_K)$, together with its sequence of prime ideals generated by $x_0-\rho_i(x_0)\subset \cO[[x_0]]$, where we have put $\rho_i:=\widehat{\chi\psi}_i$ (likewise $f_i=f_{\rho_i}:=x_0-\rho_i(x_0)$, $x_i=x_{\rho_i}:\LL_{\cO}(\Gamma_K)\to \LL_{\cO}(\Gamma_1)$ and $X_i=X_{\rho_i} \subset\LL_{\cO}(\Gamma_K$)) to ease notation. As before, we shall also denote by $x_i$ the ring map 
\begin{equation}
    \label{eqn_Fiasspecializationmap1}
    \LL_\f(\Gamma_K)=\LL_\f\widehat{\otimes}\LL_{\cO}(\Gamma_K)\xrightarrow{{\rm id}\otimes x_i}\LL_\f(\Gamma_1)
    \end{equation}
as well as its kernel (which equals $f_i\LL_\f(\Gamma_K)=X_i\LL_\f(\Gamma_K)$) also by $X_i$.
Let us write $\fL_i:=x_i(\fL)\in \LL_{\f}(\Gamma_1)$ and similarly $S_i:=x_i(S)$. Recall that the morphisms \eqref{eqn_Fiasspecializationmap1} induce the morphisms of $G_K$-representations
\begin{equation}
    \label{eqn_Fiasspecializationmap2}
x_i:\,T_{\f,\chi}^K\lra \TT_{\f,\psi_i}^{\cyc}\,,
\end{equation}

Let us fix an index $i$. It follows from the control theorem for Selmer complexes \cite[Corollary 8.10.2]{nekovar06} that 
$$\widetilde{H}^2_{\rm f}(G_{K,\Sigma},\TT_{\f,\chi}^\cyc;\Delta_{\Gr})^\iota\Big{/}X_{i}\,\widetilde{H}^2_{\rm f}(G_{K,\Sigma},\TT_{\f,\chi}^\cyc;\Delta_{\Gr})^\iota\stackrel{\sim}{\lra} \widetilde{H}^2_{\rm f}(G_{K,\Sigma},\TT_{\f,\psi_i}^\cyc;\Delta_{\Gr})^\iota\,.$$
We therefore infer that
\begin{align*}
   \Char_{\LL_\cO(\Gamma_1)}\left(\widetilde{H}^2_{\rm f}(G_{K,\Sigma},\TT_{\f,\chi}^\cyc;\Delta_{\Gr})^\iota[X_{i}]\right) \cdot \,&(\varpi^{\mu_2}S_i)\\
   &= \Char_{\LL_\cO(\Gamma_1)}\left(\widetilde{H}^2_{\rm f}(G_{K,\Sigma},\TT_{\f,\psi_i}^\cyc;\Delta_{\Gr})^\iota\right)\,,
\end{align*}
in particular, also that 
\begin{equation}
\label{eqn_pmuextendedSelmerdescent}
    \varpi^{\mu_2}S_i\,\mid\, \Char_{\LL_\cO(\Gamma_1)}\left(\widetilde{H}^2_{\rm f}(G_{K,\Sigma},\TT_{\f,\psi_i}^\cyc;\Delta_{\Gr})^\iota\right)
\end{equation}
as we contended to prove. Moreover, we have
\begin{equation}
    \label{eqn_pmuLdescent}
   \varpi^{\mu_1}\fL_i= L_p^{\rm RS}(\f_{/K}\otimes\psi_i) 
\end{equation}
 thanks to the interpolative property of the universal (geometric) $p$-adic $L$-function $L_p^{\rm RS}(\f_{/K}\otimes\bbchi)$.
 
 Under the assumptions of our theorem, the hypotheses of Theorem~\ref{thm_cyclo_main_inert_ff} are valid with the choice $\psi=\psi_i$, for any positive integer $i$. It follows from Theorem~\ref{thm_cyclo_main_inert_ff} (applied with $\psi=\psi_i$)
 $$\Char_{\LL_\cO(\Gamma_1)}\left(\widetilde{H}^2_{\rm f}(G_{K,\Sigma},\TT_{\f,\psi_i}^\cyc;\Delta_{\Gr})^\iota\right)\,\mid\, \,\varpi^{s} L_p^{\rm RS}(\f_{/K}\otimes\psi_i)\,.$$
 This, combined with \eqref{eqn_reduction_Si_step_1} and \eqref{eqn_reduction_Li_step_1}, allows us to conclude that $\varpi^{\mu_2} S_i \, \mid\, \varpi^{s+\mu_1}  \fL_i $. Since $\mu$-invariants of $\fL_i$ and $S_i$ are both zero (by definition), it follows that $\mu_2\leq s+\mu_1$, and also that
$$
     S_i\,\mid\,\fL_i\,.
$$
We may now use Proposition~\ref{prop_appendix_regular_divisibility} (with the choices $R=\LL_\f(\Gamma_K)$ and $R_0:=\ZZ_p[[\gamma_0-1]]$ given as in Remark~\ref{remark_univ_sp_map} and $F=S$, $G=\fL$), we conclude that $S\mid \fL$, as required.
\end{proof}

\begin{corollary}
\label{cor_torsion_H2_generic}
In the situation of Theorem~\ref{thm_3var_main_inert_ff}(ii), the $\LL_\f(\Gamma_K)$-module $\widetilde{H}^2_{\rm f}(G_{K,\Sigma},\TT_{\f,\chi}^{K};\Delta_{\Gr})$ is torsion and $\widetilde{H}^1_{\rm f}(G_{K,\Sigma},\TT_{\f,\chi}^{K};\Delta_{\Gr})=0$.
\end{corollary}

\begin{proof}
By the global Euler--Poincar\'e characteristic formula, the $\LL_\f(\Gamma_K)$-module $\widetilde{H}^2_{\rm f}(G_{K,\Sigma},\TT_{\f,\chi}^{K};\Delta_{\Gr})$ is torsion if and only if $\widetilde{H}^1_{\rm f}(G_{K,\Sigma},\TT_{\f,\chi}^{K};\Delta_{\Gr})$ is. Under our running assumptions (which guarantee that the residual representation $\overline{\TT}_{\f,\chi}^{K}$ is irreducible), the latter condition is equivalent to the vanishing of $\widetilde{H}^1_{\rm f}(G_{K,\Sigma},\TT_{\f,\chi}^{K};\Delta_{\Gr})$. It therefore suffices to verify the first assertion.

In view of Theorem~\ref{thm_3var_main_inert_ff}, it suffices to show that $L_p^{\rm RS}(\f_{/K}\otimes\bbchi)\neq 0$. Thanks to its interpolation property, one reduces to checking that the complex $L$-function $L(\f(\kappa)\otimes g,1+j)$ does not vanish for at least one choice of a crystalline specialization $\kappa$, a crystalline Hecke character $\psi$ of infinity type $(0,k_g+1)$ ($g=\theta(\psi)$) such that $\widehat{\chi\psi}$ factors through $\Gamma_K$, and an integer $j$ with $k_g+1\leq j \leq \kappa$. This obviously can be arranged, e.g. making sure that $1+j$ falls within the range of absolute convergence.
\end{proof}

\subsection{Anticyclotomic main conjectures in the inert case: general set up}
\label{subsubsec_anticyclo_1}
In this section, we will descend our Theorem~\ref{thm_3var_main_inert_ff} to the anticyclotomic tower, using the formalism in \cite[\S5.3.1]{BL_SplitOrd2020} (which is, essentially,  due to Nekov\'a\v{r}). In particular, we continue to work in the setting of \S\ref{subsubsec_3var_inert_ord} and retain the notation therein.

Our treatment will naturally break into two threads: The first will concern the definite case, where we generically have $\epsilon(\f(\kappa)_{/K},\kappa/2)\neq -1$ for the global root number at the central critical points of the crystalline specializations $\f(\kappa)$. The second will be a treatment of the indefinite case, where we have $\epsilon(\f(\kappa)_{/K},\kappa/2)=-1$. We shall assume throughout \S\ref{subsubsec_anticyclo_1} that the nebentype character $\varepsilon_f$ is trivial. At the expense of ink and space, one could also the more general scenario where one assumes only that $\varepsilon_f$ admits a square-root. We will also assume that $\chi$ is anticyclotomic, in the sense that $\chi^c=\chi^{-1}$.

The action of $\Gal(K/\QQ)$ on $\Gamma_K$ gives a natural decomposition 
\begin{equation}
\label{eqn_GammaK_eiegnfactorization}
    \Gamma_K=\Gamma_K^-\times \Gamma_K^+=\Gamma_{\ac} \times \Gamma_1\,.
\end{equation}
that the action of $\Gal(K/\QQ)$ on $\Gamma_K$ (by conjugation) determines, where $\Gamma_K^\pm$ is the $\pm1$-eigenspace for this action\index{$K$: imaginary quadratic field! $\Gamma_K^\pm$}. We shall treat both $\LL_{\cO}(\Gamma_\ac)$ and $\LL_{\cO}(\Gamma_1)$ both as a subring and a quotient ring of $\LL_{\cO}(\Gamma_K)$ via the identification \eqref{eqn_GammaK_eiegnfactorization}.  Along these lines, we shall also identify $\LL_\f(\Gamma_K)$ with  $\LL_\f(\Gamma_\ac)[[\Gamma_\cyc]]$.

We put $\bbchi_\ac:=\bbchi \otimes_{\LL_{\cO}(\Gamma_K)}\LL_{\cO}(\Gamma_\ac)$ and $\bbchi_\ac^\iota:=\widehat{\chi}^{c}\Psi^\iota \otimes_{\LL_{\cO}(\Gamma_K)}\LL_{\cO}(\Gamma_\ac)$.\index{$K$: imaginary quadratic field! Universal characters $\bbchi_?$}

We define the $p$-adic $L$-function $L_p^{\rm RS}(\f_{/K}\otimes\bbchi_\ac) \in \LL_{\cO}(\Gamma_\ac)$\index{$p$-adic $L$-functions! $L_p^{\rm RS}(\f_{/K}\otimes\bbchi_\ac)$} as the image of $L_p^{\rm RS}(\f_{/K}\otimes\bbchi)$ under the obvious canonical projection. Let us fix a topological generator $\gamma_+$ (resp., $\gamma_-$) of $\Gamma_1$ (resp., of $\Gamma_\ac$) such that $\{\gamma_-\times {\rm id}_{\Gamma_1},{\rm id}_{\Gamma_\ac}\times \gamma_+\}$ topologically generates $\Gamma_{\ac} \times \Gamma_1=\Gamma_K$.

\begin{defn}
\label{defn_twff_map}
We recall the universal weight character $\bbkappa: G_\QQ\to \LL_\f^\times$ from Definition~\ref{defn_Hecke_algebra_intro}. We denote its square-root by $\bbkappa^{\frac{1}{2}}$. In what follows, $?\in \{K,\ac\}$.
\item[i)] We let ${\rm Tw}_{\f}: \LL_\f(\Gamma_?) \to \LL_\f(\Gamma_?)$\index{Completed group rings! ${\rm Tw}_{\f}$} denote the $\LL_\f$-linear morphism induced by $\gamma\mapsto \bbkappa^{-\frac{1}{2}}(\gamma)\gamma$ for each $\gamma\in\Gamma_?$. We set $L_p^\dagger(\f_{/K}\otimes\bbchi):={\rm Tw}_{\f}(L_p^{\rm RS}(\f_{/K}\otimes\bbchi))$ and similarly define $L_p^\dagger(\f_{/K}\otimes\bbchi_\ac)$.\index{$p$-adic $L$-functions! $L_p^{\dagger}(\f_{/K}\otimes\bbchi_\ac)$ (central critical Rankin--Selberg $p$-adic $L$-function)} \item[ii)] We put 
$\TT_{\f,\bbchi}^\dagger:=\TT_{\f,\chi}^K(-\bbkappa^{\frac{1}{2}})$ and set $\TT_{\f,\bbchi_\ac}^\dagger:=\TT_{\f,\bbchi}^\dagger\otimes_{\LL_{\cO}(\Gamma_K)}\LL_{\cO}(\Gamma_\ac)$. Observe that $\TT_{\f,\bbchi_\ac}^\dagger=\TT_{\f,\chi}^\ac(-\bbkappa^{\frac{1}{2}})$. We shall call $\TT_{\f,\bbchi_\ac}^\dagger$ the central critical twist of $\TT_{\f,\bbchi}^\ac$\,. We set $\TT_{\f,\bbchi_\ac}^{\dagger,\iota}:=\TT_{\f,\bbchi_\ac}^\dagger\otimes_{\LL_{\cO}(\Gamma_\ac)}\LL_{\cO}(\Gamma_\ac)^\iota$. Via the decomposition~\ref{eqn_GammaK_eiegnfactorization}, we will identify $\TT_{\f,\bbchi}^{\dagger}$ with $\TT_{\f,\bbchi_\ac}^{\dagger}\,\widehat{\otimes}_{\LL_{\cO}(\Gamma_\ac)}\LL_{\cO}(\Gamma_\cyc)\,$.
\end{defn}

Thanks to our running hypothesis that $\varepsilon_f=\mathds{1}$, Poincar\'e duality induces a perfect pairing\footnote{In the more general scenario when $\varepsilon_f=\eta_f^2$ for some Dirichlet character $\eta_f$, the same conclusion is valid if one defines $\TT_{\f,\bbchi_\ac}^\dagger:=\TT_{\f,\bbchi_\ac}^K(-\bbkappa^{\frac{1}{2}}\eta_f^{-1})$.}
\begin{equation}
\label{eqn_twisted_conjugate_self_duality}
    \TT_{\f,\bbchi_\ac}^{\dagger}\otimes_{\LL_{\cO}(\Gamma_\ac)} \TT_{\f,\bbchi_\ac^\iota}^{\dagger,\iota}\lra \LL_{\cO}(\Gamma_\ac)(1)\,,
\end{equation}
namely, the  Galois representation $\TT_{\f,\bbchi_\ac}^{\dagger}$ is conjugate self-dual. 

We recall the Selmer complexes 
$\widetilde{{\bf R}\Gamma}_{\rm f}(G_{K,\Sigma},\TT_{\f,\chi}^{?};\Delta_{\Gr})$ $(?={\rm ac}, K)$ that we have introduced in Definition~\ref{defn_Selmer_complex_inert_ord}. In an identical manner, we also have the Selmer complexes\index{Selmer complexes! $\widetilde{{\bf R}\Gamma}_{\rm f}(G_{K,\Sigma},\TT_{\f,\bbchi_?^\cdot}^{\dagger,\cdot};\Delta_{\Gr})$} 
$$\widetilde{{\bf R}\Gamma}_{\rm f}(G_{K,\Sigma},\TT_{\f,\bbchi}^{\dagger};\Delta_{\Gr})\in D_{\rm ft} (_{\LL_\f(\Gamma_K)}{\rm Mod})$$
$$\widetilde{{\bf R}\Gamma}_{\rm f}(G_{K,\Sigma},\TT_{\f,\bbchi_\ac}^{\dagger};\Delta_{\Gr})\,\,,\,\,\widetilde{{\bf R}\Gamma}_{\rm f}(G_{K,\Sigma},\TT_{\f,\bbchi_\ac^\iota}^{\dagger,\iota};\Delta_{\Gr})\in D_{\rm ft} (_{\LL_\f(\Gamma_\ac)}{\rm Mod})$$ \index{Selmer complexes! $\widetilde{\bf H}^\bullet_{\rm f}(G_{K,\Sigma},\TT_{\f,\bbchi_?^\cdot}^{\dagger,\cdot};\Delta_{\Gr})$}
whose cohomology groups we denote by $\widetilde{H}^\bullet_{\rm f}(G_{K,\Sigma},\TT_{\f,\bbchi}^{\dagger};\Delta_{\Gr})$, $\widetilde{H}^\bullet_{\rm f}(G_{K,\Sigma},\TT_{\f,\bbchi_\ac}^{\dagger};\Delta_{\Gr})$ and $\widetilde{H}^\bullet_{\rm f}(G_{K,\Sigma},\TT_{\f,\bbchi_\ac^\iota}^{\dagger,\iota};\Delta_{\Gr})$, respectively. Twisting formalism (c.f., \cite{rubin00}, Lemma 6.1.2) shows that
\begin{equation}
    \label{eqn_twist_char_extended_Selmer}
    {\rm Tw}_\f\left(\Char\left(\widetilde{H}^\bullet_{\rm f}(G_{K,\Sigma},\TT_{\f,\chi}^{K};\Delta_{\Gr})\right) \right)=\Char\left(\widetilde{H}^\bullet_{\rm f}(G_{K,\Sigma},\TT_{\f,\bbchi}^{\dagger};\Delta_{\Gr})\right)\,.
\end{equation}

\begin{defn}
\label{defn_height_regulators_Nek}
We let\index{Selmer complexes! $\mathfrak{h}_{\f,\bbchi_\ac}^{\rm Nek}$ ($p$-adic height pairing)}
$$\mathfrak{h}_{\f,\bbchi_\ac}^{\rm Nek}: \widetilde{H}^1_{\rm f}(G_{K,\Sigma},\TT_{\f,\bbchi_\ac}^{\dagger};\Delta_{\Gr})\otimes \widetilde{H}^1_{\rm f}(G_{K,\Sigma},\TT_{\f,\bbchi_\ac^\iota}^{\dagger,\iota};\Delta_{\Gr})\lra \LL_\f(\Gamma_\ac)$$
denote the $\LL_\f(\Gamma_\ac)$-adic (cyclotomic) height pairing given as \cite[\S11.1.4]{nekovar06}, with $\Gamma=\Gamma^\cyc$. We define the $\LL_\f(\Gamma_\ac)$-adic regulator ${\rm Reg}_{\f,\bbchi_\ac}$ by setting\index{Selmer complexes! ${\rm Reg}_{\f,\bbchi_\ac}$ ($\LL_\f(\Gamma_\ac)$-adic regulator)}  
\\
\resizebox{\textwidth}{!} {
${\rm Reg}_{\f,\bbchi_\ac}:=\Char_{\LL_\f(\Gamma_\ac)}\left({\rm coker}\left(\widetilde{H}^1_{\rm f}(G_{K,\Sigma},\TT_{\f,\bbchi_\ac}^{\dagger};\Delta_{\Gr})\xrightarrow{{\rm adj}(\mathfrak{h}_{\f,\bbchi_\ac}^{\rm Nek})}\widetilde{H}^1_{\rm f}(G_{K,\Sigma},\TT_{\f,\bbchi_\ac^\iota}^{\dagger,\iota};\Delta_{\Gr})\right)\right)$
}
where ${\rm adj}$ denotes adjunction. Note that ${\rm Reg}_{\f,\bbchi_\ac}$ is non-zero if and only if $\mathfrak{h}_{\f,\bbchi_\ac}^{\rm Nek}$ is non-degenerate.
\end{defn}

\subsection{Anticyclotomic main conjectures: definite/inert case}
\label{subsubsec_anticyclo_2} We retain the notation and hypotheses of \S\ref{subsubsec_anticyclo_1}. Let us write $N_f=N^+N^-$ where $N^+$ (resp., $N^-$) is divisible by only those primes which split (resp., remain inert) in $K/\QQ$. We assume in \S\ref{subsubsec_anticyclo_2} that $N^-$ is a square-free product of an odd number of primes (this is what we refer to as the definite case). 

\begin{theorem}
\label{thm_anticyc_main_inert_ff_definite}
Suppose $\chi$ is a ring class character such that $\widehat{\chi}_{\vert_{G_{\QQ_p}}}\neq \widehat{\chi^c}_{\vert_{G_{\QQ_p}}}$. Assume that the following conditions hold true: 
\begin{itemize}
    \item $p\geq 7$ and the condition \ref{item_fullness_main_body} holds true.
    \item \ref{item_bounded_spsi} is valid.
\end{itemize}
Then the $\LL_\f(\Gamma_\ac)$-module $\widetilde{H}^2_{\rm f}(G_{K,\Sigma},\TT_{\f,\bbchi_\ac}^{\dagger};\Delta_{\Gr})$ is torsion and the following containment in the anticyclotomic Iwasawa main conjecture for the family $\f_{/K}\otimes\bbchi_{\rm ac}$ holds:
\begin{equation}
    \label{eqn_anticyc_main_inert_ff_definite}
    L_p^{\dagger}(\f_{/K}\otimes\bbchi_\ac)\,\in \,\Char_{\LL_\f(\Gamma_\ac)}\left(\widetilde{H}^2_{\rm f}(G_{K,\Sigma},\TT_{\f,\bbchi_\ac}^{\dagger};\Delta_{\Gr})^\iota\right)\otimes_{\ZZ_p}\QQ_p\,.
\end{equation}
\end{theorem}
We recall that by Corollary~\ref{cor_s_g_is_uniformly_bounded}, the condition \ref{item_bounded_spsi} holds true granted the existence of a rank-$2$ Euler system that the Perrin-Riou philosophy predicts.

\begin{proof}[Proof of Theorem~\ref{thm_anticyc_main_inert_ff_definite}]
It follows from the fundamental base change property for Selmer complexes~\cite[Corollary 8.10.2]{nekovar06}, it follows that 
$$\widetilde{H}^2_{\rm f}(G_{K,\Sigma},\TT_{\f,\bbchi}^{\dagger};\Delta_{\Gr})\Big{/}(\gamma_\cyc-1)\widetilde{H}^2_{\rm f}(G_{K,\Sigma},\TT_{\f,\bbchi}^{\dagger};\Delta_{\Gr})\xrightarrow{\sim}\widetilde{H}^2_{\rm f}(G_{K,\Sigma},\TT_{\f,\bbchi_\ac}^{\dagger};\Delta_{\Gr})\,.$$
Thence, if we let $\pi_\ac: \LL_\f(\Gamma_K)\to \LL_{\f}(\Gamma_\ac)$ denote the canonical morphism, we conclude that
\begin{align*}
    \Char_{\LL_\f(\Gamma_\ac)}\left(\widetilde{H}^2_{\rm f}(G_{K,\Sigma},\TT_{\f,\bbchi}^{\dagger};\Delta_{\Gr})^\iota[\gamma_\cyc-1]\right)&\,\cdot\,\pi_\ac\,\Char_{\LL_\f(\Gamma_K)}\left(\widetilde{H}^2_{\rm f}(G_{K,\Sigma},\TT_{\f,\bbchi}^{\dagger};\Delta_{\Gr})^\iota \right)\\
    &=\Char_{\LL_\f(\Gamma_\ac)}\left( \widetilde{H}^2_{\rm f}(G_{K,\Sigma},\TT_{\f,\bbchi_\ac}^{\dagger};\Delta_{\Gr})^\iota\right)\,.
\end{align*}
We deduce using this together with Theorem~\ref{thm_3var_main_inert_ff}(ii) combined with \eqref{eqn_twist_char_extended_Selmer} and the definition of $L_p^{\dagger}(\f_{/K}\otimes\bbchi_\ac)$ that
\begin{align}
\label{eqn_anticyc_main_inert_ff_definite_proof_1}
\begin{aligned}
    L_p^{\dagger}(\f_{/K}\otimes\bbchi_\ac)\cdot \Char_{\LL_\f(\Gamma_\ac)}&\left(\widetilde{H}^2_{\rm f}(G_{K,\Sigma},\TT_{\f,\bbchi}^{\dagger};\Delta_{\Gr})^\iota[\gamma_\cyc-1]\right) \\
    &\subset \Char_{\LL_\f(\Gamma_\ac)}\left( \widetilde{H}^2_{\rm f}(G_{K,\Sigma},\TT_{\f,\bbchi_\ac}^{\dagger};\Delta_{\Gr})^\iota\right)\otimes\QQ_p\,.
\end{aligned}
\end{align}

Under our running hypothesis, \cite[Theorem C]{HungNonVanishing} shows that $L_p^{\dagger}(\f_{/K}\otimes\bbchi_\ac)\neq 0$. This also shows (again using Theorem~\ref{thm_3var_main_inert_ff}(ii) together with \eqref{eqn_twist_char_extended_Selmer} and the definition of $L_p^{\dagger}(\f_{/K}\otimes\bbchi_\ac)$) 
\begin{equation}
\label{eqn_anticyc_main_inert_ff_definite_proof_2}
    \gamma_\cyc-1\nmid \Char_{\LL_\f(\Gamma_K)}\left(\widetilde{H}^2_{\rm f}(G_{K,\Sigma},\TT_{\f,\bbchi}^{\dagger};\Delta_{\Gr})^\iota\right), 
\end{equation}
which in turn shows that the $\LL_\f(\Gamma_\ac)$-module $\widetilde{H}^2_{\rm f}(G_{K,\Sigma},\TT_{\f,\bbchi}^{\dagger};\Delta_{\Gr})^\iota[\gamma_\cyc-1]$ is torsion and therefore,
$$\Char_{\LL_\f(\Gamma_\ac)}\left(\widetilde{H}^2_{\rm f}(G_{K,\Sigma},\TT_{\f,\bbchi}^{\dagger};\Delta_{\Gr})^\iota[\gamma_\cyc-1]\right)\neq 0\,.$$
This fact combined with \eqref{eqn_anticyc_main_inert_ff_definite_proof_1} shows that $\Char_{\LL_\f(\Gamma_\ac)}\left( \widetilde{H}^2_{\rm f}(G_{K,\Sigma},\TT_{\f,\bbchi_\ac}^{\dagger};\Delta_{\Gr})^\iota\right)\neq 0$, thence also that the $\LL_\f(\Gamma_\ac)$-module  $\widetilde{H}^2_{\rm f}(G_{K,\Sigma},\TT_{\f,\bbchi_\ac}^{\dagger};\Delta_{\Gr})$ is torsion. This concludes the proof of our first assertion.

We now prove the containment \eqref{eqn_anticyc_main_inert_ff_definite}, which is an improved version of \eqref{eqn_anticyc_main_inert_ff_definite_proof_1}. The proof of \cite[Theorem 5.32(ii)]{BL_SplitOrd2020} (which builds on Proposition 5.24 in op. cit.,  which itself is a translation of the general results in \cite{nekovar06}, \S11.7.11) applies verbatim to show that 
\begin{equation}
    \label{eqn_anticyc_main_inert_ff_definite_proof_4}\pi_\ac\,\Char_{\LL_\f(\Gamma_K)}\left(\widetilde{H}^2_{\rm f}(G_{K,\Sigma},\TT_{\f,\bbchi}^{\dagger};\Delta_{\Gr})^\iota \right)
=\Char_{\LL_\f(\Gamma_\ac)}\left( \widetilde{H}^2_{\rm f}(G_{K,\Sigma},\TT_{\f,\bbchi_\ac}^{\dagger};\Delta_{\Gr})^\iota\right)\,.
\end{equation}
We note that our morphism $\pi_\ac$ coincides with $\partial_\cyc^*$ in op. cit. thanks to \eqref{eqn_anticyc_main_inert_ff_definite_proof_1}. We remark that in order to apply \cite[Proposition 5.24]{BL_SplitOrd2020}, it suffices to verify that both $\LL_\f(\Gamma_\ac)$-modules $\widetilde{H}^1_{\rm f}(G_{K,\Sigma},\TT_{\f,\bbchi_\ac}^{\dagger};\Delta_{\Gr})$ and $\widetilde{H}^2_{\rm f}(G_{K,\Sigma},\TT_{\f,\bbchi_\ac}^{\dagger};\Delta_{\Gr})$ are torsion. %(which then ensures the non-degeneracy of the height pairing $\mathfrak{h}_{\f,\bbchi_\ac}^{\rm Nek}$). 
We have checked the latter above and the fact that $\widetilde{H}^1_{\rm f}(G_{K,\Sigma},\TT_{\f,\bbchi_\ac}^{\dagger};\Delta_{\Gr})$ is torsion follows from this and the global Euler--Poincar\'e  characteristic formulae.

Theorem~\ref{thm_3var_main_inert_ff}(ii) together with \eqref{eqn_anticyc_main_inert_ff_definite_proof_4} combined with \eqref{eqn_twist_char_extended_Selmer} and the fact  that $\pi_\ac(L_p^{\dagger}(\f_{/K}\otimes\bbchi))=L_p^{\dagger}(\f_{/K}\otimes\bbchi_\ac)$ (which follows from definitions) conclude the proof of \eqref{eqn_anticyc_main_inert_ff_definite}.
\end{proof}

\subsection{Anticyclotomic main conjectures: indefinite/inert case}
\label{subsubsec_anticyclo_3} We retain the notation and hypotheses of \S\ref{subsubsec_anticyclo_1}. As in the previous subsection, let us write $N_f=N^+N^-$ but assume (in contrast with the previous subsection) that $N^-$ is a square-free product of even number of primes (the adjective ``indefinite'' is in reference to this condition). In this scenario, we have  
\begin{equation}
\label{eqn_epsilon_factors_1}
    \epsilon(\f(\kappa)_{/K}\otimes\psi,(\kappa+k_g+1)/2)=-1\qquad\qquad (\kappa>k_g)
\end{equation}
for all crystalline specialization $\f(\kappa)$ and crystalline Hecke characters $\psi$ (c.f. \S\ref{subsubsec_setting_fKpsi} to recall our conventions) with infinity type $(0,k_g+1)$ verifying $k_g<\kappa$. Here, $\epsilon(\f(\kappa)_{/K}\otimes\psi,s)$ is the global root number for the Rankin--Selberg $L$-function $L(\f(\kappa)\times\theta(\psi),s)=L(\f(\kappa)_{/K},\psi,s)$; c.f. \cite[\S15]{jacquet} (see also \cite[\S4.1]{bertolinidarmonprasanna13} for a detailed summary which we rely on in the present discussion). We further remark that $s=(\kappa+k_g+1)/2$ is the central critical point. In particular, the $L$-function $L(\f(\kappa)\times\theta(\psi),s)$  vanishes to odd order at its central critical point.

In the complementary scenario where $k_g>\kappa$, we have 
\begin{equation}
\label{eqn_epsilon_factors_2}
    \epsilon(\f(\kappa)_{/K}\otimes\psi,(\kappa+k_g+1)/2)=+1\qquad\qquad (\kappa<k_g)
\end{equation}
for all crystalline specialization $\f(\kappa)$ and crystalline Hecke characters $\psi$ whose infinity type $(0,k_g+1)$ verifies $k_g>\kappa$ (c.f. the discussion in \cite{bertolinidarmonprasanna13}, Page 1036).

One may recast \eqref{eqn_epsilon_factors_1} and \eqref{eqn_epsilon_factors_2} in terms of anticyclotomic characters as follows: For all crystalline specialization $\f(\kappa)$ and crystalline anticyclotomic Hecke characters $\psi$ with infinity type $(m-\kappa/2,\kappa/2-m)$ and conductor coprime to $N_fD_K$ we have
\begin{equation}
\label{eqn_epsilon_factors_3}
    (-1)^{\ord_{s=\frac{\kappa}{2}}L(\f(\kappa)_{/K},\psi,s)}=\epsilon(\f(\kappa)_{/K}\otimes\psi,\kappa/2)=
    \begin{cases}
    -1\qquad\qquad\qquad (0<m<\kappa)\\
    +1\qquad\qquad\qquad (m<0)
    \end{cases}
\end{equation}

\begin{defn}
\label{defn_derivatives}\index{Completed group rings! $\partial_\cyc^{r}$}
Given $A\in \LL_{\cO}(\Gamma_K)$, let us put $r(A):=\ord_{(\gamma_+-1)}A$. For any $r\leq r(A)$\index{Completed group rings! $r(A)$}, we define the $r$th derivative $\partial_\cyc^{r}\, A\in \LL_\f(\Gamma_\ac)$ of $A$ on setting $\partial_\cyc^{r}\,A:=\pi_\ac((\gamma_+-1)^{-r}A)$\,. 

If $I=(A)$ is a principal ideal, we set $r(I)=r(A)$ and define the ideal $\partial_\cyc^{r}\, I:=(\partial_\cyc^{r}\, A)\subset\LL_\f(\Gamma_\ac)$.
\end{defn}
We shall write $\partial_\cyc$ alone in place of $\partial_\cyc^1$.
\begin{defn}
We set\index{Selmer complexes! $r(\f,\bbchi_\ac)$ (generic algebraic rank)}
$$r(\f,\bbchi_\ac):={\rm rank}_{\LL_{\f}(\Gamma_\ac)}\, \widetilde{H}^1_{\rm f}(G_{K,\Sigma},\TT_{\f,\bbchi_\ac}^{\dagger};\Delta_{\Gr})$$  $$r(\f,\bbchi):=r\left(\widetilde{H}^2_{\rm f}(G_{K,\Sigma},\TT_{\f,\bbchi}^{\dagger};\Delta_{\Gr})\right)\,.$$ 
\end{defn}

\begin{proposition}
\label{prop_semisimplicity_Selmer}
We have ${\rm rank}_{\LL_{\f}(\Gamma_\ac)}\, \widetilde{H}^2_{\rm f}(G_{K,\Sigma},\TT_{\f,\bbchi_\ac}^{\dagger};\Delta_{\Gr})=r(\f,\bbchi_\ac)$ and
\begin{equation}
\label{eqn_prop_semisimplicity_Selmer}
    r(\f,\bbchi_\ac)\leq r(\f,\bbchi)\,. 
\end{equation}
with equality if and only if the height pairing $\mathfrak{h}_{\f,\bbchi_\ac}^{\rm Nek}$ is non-degenerate.
\end{proposition}
\begin{proof}
The first equality is a consequence of the Euler-Poincar\'e characteristic formula. The inequality \eqref{eqn_prop_semisimplicity_Selmer} follows from the first equality combined with the control theorem for Nekov\'a\v{r}'s extended Selmer groups (c.f., \cite{nekovar06}, Corollary 8.10.2). It remains to prove the asserted criterion for equality.  

It follows from  \cite[Proposition 11.7.6(vii)]{nekovar06} that 
\begin{align}
\label{eqn_nek_1176vii}
\begin{aligned}
    {\rm length}_{\LL_\f(\Gamma_K)_{(\gamma_+-1)}}&\left(\widetilde{H}^2_{\rm f}(G_{K,\Sigma},\TT_{\f,\bbchi}^\dagger;\Delta_{\rm Gr})_{(\gamma_+-1)}\right)\\
    &\qquad\qquad\geq {\rm length}_{\LL_\f(\Gamma_\ac)_{(0)}}\left(\widetilde{H}^1_{\rm f}(G_{K,\Sigma},\TT_{\f,\bbchi_\ac}^\dagger;\Delta_{\rm Gr})_{(0)}\right)
    \end{aligned}
\end{align}
with equality if and only if the height pairing $\mathfrak{h}_{\f,\bbchi_\ac}^{\rm Nek}$ is non-degenerate. Thence,
\begin{align*}
    r(\f,\bbchi)&={\rm length}_{\LL_\f(\Gamma_K)_{(\gamma_+-1)}}\left(\widetilde{H}^2_{\rm f}(G_{K,\Sigma},\TT_{\f,\bbchi}^\dagger;\Delta_{\rm Gr})_{(\gamma_+-1)}\right)\\
    &\stackrel{\eqref{eqn_nek_1176vii}}{\geq} {\rm length}_{\LL_\f(\Gamma_\ac)_{(0)}}\left(\widetilde{H}^1_{\rm f}(G_{K,\Sigma},\TT_{\f,\bbchi_\ac}^\dagger;\Delta_{\rm Gr})_{(0)}\right)\\
    &={\rm rank}_{\LL_\f(\Gamma_\ac)}\left(\widetilde{H}^1_{\rm f}(G_{K,\Sigma},\TT_{\f,\bbchi_\ac}^\dagger;\Delta_{\rm Gr})\right)=r(\f,\bbchi_\ac) 
\end{align*}
(where the first and the final two equalities follows from definitions) with equality if and only if the height pairing $\mathfrak{h}_{\f,\bbchi_\ac}^{\rm Nek}$ is non-degenerate.
\end{proof}

\begin{remark} One expects that the height pairing $\mathfrak{h}_{\f,\bbchi_\ac}^{\rm Nek}$ is always non-degenerate (equivalently, ${\rm Reg}_{\f,\bbchi_\ac}\neq 0$); see \cite{burungale2015,burungaledisegni2020} for progress in this direction (in a setting that unfortunately has no overlap with the scenario we have placed ourselves in). 
\end{remark}

Our main result in \S\ref{subsubsec_anticyclo_3} is Theorem~\ref{thm_anticyc_main_inert_ff_indefinite} below, which is a partial $\LL_\f(\Gamma_\ac)$-adic BSD formula.
\begin{theorem}
\label{thm_anticyc_main_inert_ff_indefinite}
Suppose $\chi$ is a ring class character such that $\widehat{\chi}_{\vert_{G_{\QQ_p}}}\neq \widehat{\chi^c}_{\vert_{G_{\QQ_p}}}$. Assume that the following conditions hold true. 
\begin{itemize}
    \item $\overline{\rho}_\f$ is absolutely irreducible.
     \item $p\geq 7$ and the condition \ref{item_fullness_main_body} holds true.
    \item \ref{item_bounded_spsi} is valid.
\end{itemize}
Then: 
\item[i)] $\ord_{(\gamma_+-1)}\, L_p^{\dagger}(\f_{/K}\otimes\bbchi)\geq 1$.

\item[ii)] The following containment \emph{(partial $\LL_\f(\Gamma_\ac)$-adic BSD formula for the family $\f_{/K}\otimes\bbchi_{\rm ac}$)} is valid:
\begin{equation}
    \label{eqn_anticyc_main_inert_ff_indefinite}
    \begin{aligned}
    \partial_\cyc^{r(\f,\bbchi_\ac)} L_p^{\dagger}(\f_{/K}\otimes\bbchi)\,\in \,{\rm Reg}_{\f,\bbchi_\ac}\cdot\Char_{\LL_\f(\Gamma_\ac)}\left(\widetilde{H}^2_{\rm f}(G_{K,\Sigma},\TT_{\f,\bbchi_\ac}^{\dagger};\Delta_{\Gr})_{\rm tor}^\iota\right)\otimes_{\ZZ_p}\QQ_p\,.
    \end{aligned}
\end{equation}
\end{theorem}

We recall that by Corollary~\ref{cor_s_g_is_uniformly_bounded}, the condition \ref{item_bounded_spsi} holds true granted the existence of a rank-$2$ Euler system that the Perrin-Riou philosophy predicts. 

\begin{proof}[Proof of Theorem~\ref{thm_anticyc_main_inert_ff_indefinite}]
\item[i)]  Using the interpolation properties of the $p$-adic $L$-function $L_p^{\rm RS}(\f_{/K}\otimes\bbchi)$ and the definition of $L_p^{\dagger}(\f_{/K}\otimes\bbchi_\ac)$ (c.f. Definition~\ref{defn_twff_map}(i)) and \eqref{eqn_epsilon_factors_3}, it follows that 
$$L_p^{\dagger}(\f_{/K}\otimes\bbchi_\ac)(\kappa,\widehat{\chi\psi})\,\dot{=}\,L(\f_{/K},\psi,\kappa/2)=0$$
(where $a\dot{=}b$ means $a=cb$ for some $c\in \mathbb{C}_p$) for all crystalline specializations $\f(\kappa)$ and  anticyclotomic (not necessarily crystalline) Hecke characters $\psi$ of infinity type $(m-{\kappa}/{2},{\kappa}/{2}-m)$ with $0<m<\kappa$\,.  This shows that $L_p^{\dagger}(\f_{/K}\otimes\bbchi_\ac)=0$, which is equivalent to the assertion in this part of our theorem.

 \item[ii)] If we combine Theorem~\ref{thm_3var_main_inert_ff}(ii) with \eqref{eqn_twist_char_extended_Selmer} and the definition of the twisted $p$-adic $L$-function $L_p^\dagger(\f\otimes\bbchi)$, we reduce to proving that 
\begin{align}
\label{eqn_anticyc_main_inert_ff_indefinite_proof_1}
\begin{aligned}
    \partial_\cyc^{r(\f,\bbchi_\ac)} \Char&_{\LL_\f(\Gamma_K)}\left(\widetilde{H}^2_{\rm f}(G_{K,\Sigma},\TT_{\f,\bbchi}^{\dagger};\Delta_{\Gr})^\iota\right)\,\\
    &\quad\subset \,{\rm Reg}_{\f,\bbchi_\ac}\cdot\Char_{\LL_\f(\Gamma_\ac)}\left(\widetilde{H}^2_{\rm f}(G_{K,\Sigma},\TT_{\f,\bbchi_\ac}^{\dagger};\Delta_{\Gr})_{\rm tor}^\iota\right)\otimes_{\ZZ_p}\QQ_p\,.
\end{aligned}
\end{align}

If $r(\f,\bbchi)\gneq r(\f,\bbchi_\ac)$, we have
$$\partial_\cyc^{r(\f,\bbchi_\ac)} \Char_{\LL_\f(\Gamma_K)}\left(\widetilde{H}^2_{\rm f}(G_{K,\Sigma},\TT_{\f,\bbchi}^{\dagger};\Delta_{\Gr})^\iota\right)=0={\rm Reg}_{\f,\bbchi_\ac}$$
where the second equality follows from Proposition~\ref{prop_semisimplicity_Selmer}. In other words, \eqref{eqn_anticyc_main_inert_ff_indefinite_proof_1} trivially holds true when $r(\f,\bbchi)\neq r(\f,\bbchi_\ac)$. 

We will therefore assume in the remainder of our proof that $r(\f,\bbchi)=r(\f,\bbchi_\ac)$. We contend that
\begin{align}
\label{eqn_anticyc_main_inert_ff_indefinite_proof_2}
\begin{aligned}
    \partial_\cyc^{r(\f,\bbchi)} \Char&_{\LL_\f(\Gamma_K)}\left(\widetilde{H}^2_{\rm f}(G_{K,\Sigma},\TT_{\f,\bbchi}^{\dagger};\Delta_{\Gr})^\iota\right)\,\\
    &\qquad\subset \,{\rm Reg}_{\f,\bbchi_\ac}\cdot\,\Char_{\LL_\f(\Gamma_\ac)}\left(\widetilde{H}^2_{\rm f}(G_{K,\Sigma},\TT_{\f,\bbchi_\ac}^{\dagger};\Delta_{\Gr})_{\rm tor}^\iota\right)\otimes_{\ZZ_p}\QQ_p\,.
    \end{aligned}
\end{align}
 We will verify \eqref{eqn_anticyc_main_inert_ff_indefinite_proof_2} using \cite[Proposition 5.24]{BL_SplitOrd2020}, which is a simplification of Nekov\'a\v{r}'s result \cite[\S 11.7.11]{nekovar06}. To apply Proposition 5.24 in \cite{BL_SplitOrd2020}, we need to check that the following properties hold true.
\begin{enumerate}
\item The height pairing $\mathfrak{h}_{\f,\bbchi_\ac}^{\rm Nek}$ is non-degenerate.
    \item The Tamagawa factors denoted by ${\rm Tam}_v(\TT_{\f,\bbchi_\ac}^{\dagger},P)$ in \cite{BL_SplitOrd2020} (c.f.Definition 5.21 in op. cit.) vanish for every height-one prime $P$ of $\LL_\f(\Gamma_\ac)$.
    \item  For $i=0,3$, we have $\widetilde{H}^i_{\rm f}(G_{K,\Sigma},\TT_{\f,\bbchi}^{\dagger};\Delta_{\Gr})=0=\widetilde{H}^i_{\rm f}(G_{K,\Sigma},\TT_{\f,\bbchi_\ac}^{\dagger};\Delta_{\Gr})$.
    \item Both $\LL_\f(\Gamma_K)$-modules $\widetilde{H}^1_{\rm f}(G_{K,\Sigma},\TT_{\f,\bbchi}^{\dagger};\Delta_{\Gr})$ and $\widetilde{H}^2_{\rm f}(G_{K,\Sigma},\TT_{\f,\bbchi}^{\dagger};\Delta_{\Gr})$ are torsion.
\end{enumerate}
The first property holds thanks to Proposition~\ref{prop_semisimplicity_Selmer} (since $r(\f,\bbchi)=r(\f,\bbchi_\ac)$). The second property follows from \cite[Corollary 8.9.7.4]{nekovar06} applied with $T=\TT_{\f,\bbchi}^\dagger$ and $\Gamma=\Gamma_\ac$. The third property is immediate thanks to our running assumptions, which guarantee that the residual representation $\overline{\TT}_{\f,\bbchi}^\dagger$ ($=\overline{\TT}_{\f,\bbchi_\ac}^\dagger$) is absolutely irreducible. The final property is Theorem~\ref{cor_torsion_H2_generic} combined with \eqref{eqn_twist_char_extended_Selmer}. This completes the proof.
\end{proof}

\begin{remark}
In the situation of Theorem~\ref{thm_anticyc_main_inert_ff_indefinite}, one expects that $${\rm ord}_{(\gamma_+-1)}\,L_p^{\dagger}(\f_{/K}\otimes\bbchi)\stackrel{?}{=}1\stackrel{?}{=}r(\f,\bbchi_\ac)\,.$$ 
In the setting where the prime $p$ splits in $K/\QQ$, the second expected equality follows from \cite[Theorem 3.15]{BLForum}. When the prime $p$ splits in $K/\QQ$, one could also utilize \cite[Theorem 3.30]{BLForum} to show that the first expected equality holds if and only if ${\rm Reg}_{\f,\bbchi_\ac}\neq 0$.
\end{remark}

\begin{corollary}
\label{cor_thm_anticyc_main_inert_ff_indefinite_improvement_1}
In the setting of Theorem~\ref{thm_anticyc_main_inert_ff_indefinite}, suppose in addition that $r(\f,\bbchi_\ac)\geq 1$. Then,
\begin{equation}
    \label{eqn_anticyc_main_inert_ff_indefinite_bis}\index{$p$-adic $L$-functions! $\partial_\cyc\, L_p^{\dagger}(\f_{/K}\otimes\bbchi)$} 
    \partial_\cyc L_p^{\dagger}(\f_{/K}\otimes\bbchi)\,\in \,{\rm Reg}_{\f,\bbchi_\ac}\cdot\Char_{\LL_\f(\Gamma_\ac)}\left(\widetilde{H}^2_{\rm f}(G_{K,\Sigma},\TT_{\f,\bbchi_\ac}^{\dagger};\Delta_{\Gr})_{\rm tor}^\iota\right)\otimes_{\ZZ_p}\QQ_p\,.
\end{equation}
\end{corollary}
\begin{proof}
Combining Proposition~\ref{prop_semisimplicity_Selmer} and Theorem~\ref{thm_3var_main_inert_ff}(ii) together with \eqref{eqn_twist_char_extended_Selmer} and the definition of the twisted $p$-adic $L$-function $L_p^\dagger(\f\otimes\bbchi)$, it follows that $\partial_\cyc L_p^{\dagger}(\f_{/K}\otimes\bbchi)=0$ if $r(\f,\bbchi_\ac)\gneq 1$; and hence, there is nothing to prove when $r(\f,\bbchi_\ac)\gneq 1$. We have therefore reduced to proving \eqref{eqn_anticyc_main_inert_ff_indefinite_bis} when $r(\f,\bbchi_\ac)= 1$, and this is precisely  \eqref{eqn_anticyc_main_inert_ff_indefinite} in this case.
\end{proof}
The containment \eqref{eqn_anticyc_main_inert_ff_indefinite_bis} is in line with Perrin-Riou's anticyclotomic main conjectures. Based on the recent work of Andreatta and Iovita~\cite{AndreattaIovitaBDP}, it seems that  the necessary tools to prove that $r(\f,\bbchi_\ac)\geq 1$ might be available. We discuss this point in the following remark. 

\begin{remark}
\label{remark_andreatta_iovita}
Suppose that we are in the setting of Theorem~\ref{thm_anticyc_main_inert_ff_indefinite}. If $r(\f,\bbchi_\ac)=0$, it follows from control theorems for Selmer complexes that there exists $n_\f \in \ZZ^+$ such that $\widetilde{H}^1_{\rm f}(G_{K,\Sigma},\TT_{\f,\psi}(-\kappa/2);\Delta_{\Gr})=0$ for all crystalline anticyclotomic Hecke characters $\psi$ of infinity type $(-a,a)$ with $a>n_\f$ and such that $\widehat{\chi\psi}$ factors through $\Gamma_K$. This implies (again by the control theorems for Selmer complexes) that for each $\psi$ as above, there exists another constant $M_{\f,\psi}$ such that $\widetilde{H}^1_{\rm f}(G_{K,\Sigma},T_{\f(\kappa),\psi}(-\kappa/2);\Delta_{\Gr})=0$ for all (classical) crystalline weights $\kappa>M_{\f,\psi}$.

In light of the recent work of Andreatta and Iovita \cite{AndreattaIovitaBDP} towards Bertolini--Darmon--Prasanna type formulae in the inert case, one expects\footnote{One could even prove this once the $p$-adic $L$-functions of \cite{AndreattaIovitaBDP} are interpolated as the eigenform $f$ varies in a Hida family.} that 
$${\rm rank}\,\widetilde{H}^1_{\rm f}(G_{K,\Sigma},T_{\f(\kappa),\psi}(-\kappa/2);\Delta_{\Gr})\geq 1\,,$$ 
for $\kappa\gg0$, which would contradict the discussion above and prove that $r(\f,\bbchi_\ac)\geq 1$. In other words, granted a slight extension of the work of Andreatta--Iovita (which we expect to be established in the near future), the containment \eqref{eqn_anticyc_main_inert_ff_indefinite_bis} in Perrin-Riou's anticyclotomic main conjecture is valid.
\end{remark}

\begin{remark}
\label{remark_Hsieh_improve_indefinite_ff}
Suppose that we are still in the setting of Theorem~\ref{thm_anticyc_main_inert_ff_indefinite}. We will explain how an improvement of Hsieh's generic non-vanishing result\footnote{This result asserts in our set up that given a crystalline anticyclotomic Hecke character $\phi$ of infinity type $(m-\kappa/2,\kappa/2-m)$ and for which $\widehat{\phi}$ factors through $\Gamma_K$, we have $L(\f(\kappa)_{/K},\eta\phi,\kappa/2)\neq 0$ for all but finitely many characters $\eta$ of $\Gamma_\ac$, under the additional requirement that all local root numbers of $\f(\kappa)_{/K}\otimes\phi$ are $+1$ at every place of $K$.} \cite[Theorem C]{hsiehnonvanishing} can be used to prove the opposite containment (to what we have discussed in Remark~\ref{remark_andreatta_iovita})  $r(\f,\bbchi_\ac)\leq 1$.  

Suppose that the following non-vanishing statement $($in the flavor of those proved in \cite{Greenberg1985}$)$ holds true: There exist a crystalline specialization $\f(\kappa)$ and infinitely many crystalline anticyclotomic Hecke characters $\psi$ of infinity type $(m-\kappa/2,\kappa/2-m)$ such that $L(\f(\kappa)_{/K},\psi,\kappa/2)\neq 0$ (necessarily with $m<0$). One expects that this is valid for any $\kappa$. We can then  prove using \cite[Theorem  8.1.4]{LZ1} that the Bloch--Kato Selmer group $H^1_{\FFF_{\rm BK}}(K,T_{\f(\kappa),\psi}(-\kappa/2))$ vanishes. Moreover, it is easy to see that 
$${\rm rank}\,\widetilde{H}^1_{\rm f}(G_{K,\Sigma},T_{\f(\kappa),\psi}(-\kappa/2);\Delta_{\Gr})\,\,-\,\,{\rm rank} H^1_{\FFF_{\rm BK}}(K,T_{\f(\kappa),\psi}(-\kappa/2)) \leq 1\,.$$
This in turn shows that ${\rm rank}\,\widetilde{H}^2_{\rm f}(G_{K,\Sigma},T_{\f(\kappa),\psi}(-\kappa/2);\Delta_{\Gr})\leq 1$ for all $\kappa$ and $\psi$ as above. By the control theorems for Selmer complexes, we conclude that $r(\f,\bbchi_\ac)\leq 1$.
\end{remark}

%%%%%%%%%%%
%%%%%%%%%%%

%-----------------------------------------------------------------------
% End of chap1.tex
%-----------------------------------------------------------------------

%% file: InertOrdMainAppendixA.tex
\chapter{A divisibility criterion in regular rings}
\label{Appendix_Regular_Rings_Divisibility}
Suppose $R$ is a complete local regular ring of dimension $n+2$ ($n\geq 1$) and with mixed characteristic $(0,p)$. Let $\mathfrak{m}$ denote its maximal ideal. Assume that $x_0,x_1, \cdots,x_n$ is a regular sequence in $R$ such that $p\not\in (x_0,x_1, \cdots,x_n)$. Let $(\varpi)\subset R$ denote the unique height-one prime of $R$ that contains $p$. The injective ring homomorphism
$$\iota_0:\ZZ_p[[x_0]] \lra R$$
endows $R$ with the structure of a flat $R_0:=\ZZ_p[[x_0]]$-module. 

Suppose that $F,G \in R$ are two non-zero elements with the following properties:
\begin{itemize}
    \item[a)] $\varpi\nmid FG$.
    \item[b)] $x_0\nmid F$.
    \item[c)] There exists a collection of irreducible elements $\{f_i\}_{i=1}^\infty\subset \ZZ_p[[x_0]]$ which are pairwise coprime, with the property that the image of $G$ under the natural map
    $$R \lra R/(f_i,F)=R_0/R_0f_i\bigotimes_{\iota_0: R_0\hookrightarrow R} R/(F)$$
    is zero. Here and also below, we write $x$ in place $\iota_0(x)$ for $x\in R_0$.
\end{itemize}
\begin{proposition}
\label{prop_appendix_regular_divisibility}
 If $F$ and $G$ are as above, then $F$ divides $G$ over $R$.
\end{proposition}

\begin{proof}
For each positive integer $n$, let us define $g_n:= f_1\cdots f_n$, so that we have $Rg_{n+1}\subsetneq Rg_n=: \mathfrak{b}_n$.

Observe that we have an injection
$$R_0/R_0g_{n} \hookrightarrow \prod_{i=1}^n R_0/R_0f_i$$
for each positive integer $n$. Since we assumed $x_0\nmid F$ and $\varpi\nmid F$, and since $R$ is a UFD, it follows that $R/(F)$ is flat as an $R_0$-module. Thence, we have an injection
\begin{align*}
R/(g_{n}, F)=R_0/R_0g_{n} \bigotimes_{\iota_0: R_0\hookrightarrow R} R/(F) \hookrightarrow \prod_{i=1}^n&\left( R_0/R_0f_i \bigotimes_{\iota_0: R_0\hookrightarrow R} R/(F)\right)\\
&\qquad\qquad\qquad=\prod_{i=1}^n R/(f_i,F)\,.
\end{align*}
This injection together with property (c) above shows that $G\in (g_{n}, F)$ for every positive integer $n$.

We will next prove that $\bigcap_n\, (g_{n}, F)=(F)$ and this will conclude the proof of our proposition. 

To that end, suppose $a\in \bigcap_n\, (g_{n}, F)$, so that for every positive integer $n$, there exists $r_n,s_n\in R$ with $a=r_ng_n+Fs_n$. We first prove that the sequence $s_n$ converges $\mathfrak{m}$-adically. Since the ring $R$ is complete with respect to the $\mathfrak{m}$-adic topology, it suffices to prove that $\{s_n\}$ is a Cauchy sequence. Suppose we are given any positive integer $M$. According to \cite[Lemma 7]{Chevalley1943}, there exists $N\in \ZZ^+$ with $g_n \in \mathfrak{m}^M$ for all $n\geq N$. This in turn means that 
$$(s_n-s_m)\cdot F=r_mg_m-r_ng_n \in  \mathfrak{m}^M$$ 
for every $n,m>N$. This concludes the proof that the sequence $\{s_n\}$ is Cauchy, and therefore convergent. Let us put $s:=\displaystyle\lim_{n\rightarrow\infty} s_n \in R$. The lemma of Chevalley also shows that $\displaystyle\lim_{n\rightarrow\infty} g_n=0$, and in turn also that 
$$a=\lim_{n\to \infty} a=\lim_{n\to \infty} r_ng_n+Fs_n=\lim_{n\to \infty} r_ng_n+\lim_{n\to \infty} Fs_n=Fs$$
concluding the proof that $a\in (F)$. This shows that $\bigcap_n\, (g_{n}, F) \subset (F)$, and in turn also $\bigcap_n\, (g_{n}, F)=(F)$, as required.
\end{proof}

%%%%%%%%%%%
%%%%%%%%%%%

%-----------------------------------------------------------------------
% End of chap1.tex
%-----------------------------------------------------------------------

%% file: InertOrdMainAppendixB.tex
\chapter[$p$-adic $L$-functions and universal deformations]{$p$-adic Rankin--Selberg $L$-functions and universal deformations}
\label{appendix_sec_padicRankinSelbergHida}
In this appendix, we will go over the work of Loeffler's recent work~\cite[\S4]{Loeffler2020universalpadic}, where he constructs $p$-adic Rankin--Selberg $L$-functions and extend it  slightly to cover the case of minimally ramified universal deformation representation.

We make crucial use of Loeffler's construction to study $h_{/K}\otimes \chi$ where $h_{/K}$ is the base change of a $p$-ordinary $p$-stabilized cuspidal eigenform $h$ to the imaginary quadratic field $K$ where $p$ is \emph{inert} and $\chi$ is a ray class character of $K$. 

To be more precise, let us put $K(p^\infty)=\cup_n K(p^n)$, where $K(p^n)$ is the ray class extension modulo $p^n$. Let us also set $H_{p^\infty}:=\Gal(K(p^\infty)/K)$. In this scenario, one is interested in the theta-series $\theta(\chi\psi)$, as $\psi$ ranges over Hecke characters whose $p$-adic avatars $\widehat{\psi}$ give rise to (via global class field theory) a character of $H_{p^\infty}$. These are CM forms, which are non-ordinary at $p$ (since $p$ is inert in $K/\QQ$). It is easy to see that the eigencurve cannot contain families of CM forms {of positive slope  (neither in its cuspidal components nor in the Eisenstein components)}; c.f.~\cite[Corollary 3.6]{CITJHC70}. This is due to the fact that slope of any non-ordinary CM form of weight $k$ is at least $(k-1)/2$, so the refinements of $\theta(\chi\psi)$ cannot be contained in a finite slope family.  
All this tells us that the $p$-adic $L$-function which can have any bearing to our situation cannot descend from a construction over the eigencurve. In other words, one should work instead over the universal deformation space.

Indeed, in contrast to \cite[Corollary 3.6]{CITJHC70}, one can interpolate the Galois representations $\left\{{\rm Ind}_{K/\QQ}\chi\psi\right\}_{\psi}$ to a $p$-adic family of Galois representations ${\rm Ind}_{K/\QQ}\chi\Psi$, where $\Psi$ is the tautological (universal) $p$-adic Hecke character, which is given as the compositum
$$ G_K \lra H_{p^\infty}\lra \LL_{\cO}(H_{p^\infty})^{\times}\,.$$
Loeffler's work~\cite{Loeffler2020universalpadic} gives rise to a $p$-adic $L$-function
\begin{equation}
\label{eqn_Appendix_LoefflerpadicLsingleform}
    L_p^{\rm RS}(f_{/K}\otimes\Psi) \in \frac{1}{H_f}\LL_{\cO}(H_{p^\infty})
\end{equation}
(where $H_f$ is the congruence number of $f$, in the sense of Hida). To be more precise, in order to obtain the $p$-adic $L$-function $L_p(f_{/K}\otimes\Psi)$ above, one in fact first constructs a $p$-adic $L$-function
\begin{equation}
\label{eqn_Appendix_LoefflerpadicLfamily}
L_p^{\rm RS}(\f_{/K}\otimes\Psi) \in \frac{1}{H_\f}\LL_{\f}\,\widehat{\otimes}\,\LL_{\cO}(H_{p^\infty}),
\end{equation}
(where $\f$ is the unique Hida family through $f$ and $H_\f$ is Hida's congruence ideal) and specializes $\f$ to $f$.

In this appendix, we review (and  slightly extend, along the lines of the suggestion in \cite[\S5.1]{Loeffler2020universalpadic}) Loeffler's construction in a setting just a tad more general than what is required to obtain the $p$-adic $L$-function in \eqref{eqn_Appendix_LoefflerpadicLfamily} (we still work under a simplifying ``minimality'' condition). All results herein are due to Loeffler (but mistakes are ours) and we claim no originality.

\section{The set up}
\label{subsec_appendix_setup}
We let $\mathbb{F}$ denote the residue field of $\cO$ (where $\cO$ is the ring of integers of a finite extension $L$ of $\QQ_p$, as in the main body of this article) and let $\overline{\rho}: G_\QQ \to \GL_2(\mathbb{F})$ be a two-dimensional, odd, irreducible Galois representation. It is known (thanks to Khare--Wintenberger) that $\overline{\rho}$ is modular. We assume in addition the following conditions. 
\begin{itemize}
 \item[\mylabel{item_deformrho1}{$(\overline{\rho}1)$}]    The restriction $\overline{\rho}{\vert_{G_{\QQ(\mu_p)}}}$ of $\overline{\rho}$ to $G_{\QQ(\mu_p)}$ is irreducible.
   \item[\mylabel{item_deformrho2}{$(\overline{\rho}2)$}] Either $\overline{\rho}{\vert_{G_{\QQ_p}}}$ is irreducible, or else if we have 
    $$\overline{\rho}^{s.s.}=\chi_{1,p}\oplus \chi_{2,p}$$
    for the semisimplification of $\overline{\rho}$, then $\chi_{1,p}/\chi_{2,p}$ is different from $\mathds{1}$ or $\omega^{\pm 1}$, where $\omega$ is the Teichm\"uller character giving the action of $G_{\QQ_p}$ on $\mu_p$. 
\end{itemize}
We let $N$ denote the Artin conductor of $\overline{\rho}$; this is a positive integer coprime to $p$.
\begin{remark}
In applications, we shall set $\overline{\rho}=\overline{\rho}_g$, where $g$ is the non-ordinary newform in the main body of the present article and the integer $N$ will be the level $N_g$ of $g$. In other words, whenever we refer to Appendix~\ref{appendix_sec_padicRankinSelbergHida}, we will assume that $g$ is minimally ramified at all primes dividing $N_g$, in the sense of \cite[\S3]{Diamond1997FLT}. It should be possible to relax this assumption, but we will content to work under this assumption since it covers our case of interest when $g$ arises as the $\theta$-series of a Hecke character of an imaginary quadratic field where $p$ is inert.
\end{remark}

\begin{example}
\label{app_example_residual_conditions}
Let $K$ be an imaginary quadratic field as above where $p$ remains \emph{inert} and let $\chi$ be a ray class character of $K$ as in 
\S\ref{subsubsec_3var_inert_ord}\footnote[3]{Except that in this appendix, we denote the Galois character that one associates to $\chi$ also by $\chi$, rather than $\widehat{\chi}$.}, of conductor $p\mathfrak{f}_\chi$. We shall always assume that both the order of $\chi$ and {$\mathfrak{f}_\chi$ are coprime to $p$}. %Self-remark: The assumption on the order is solely related to minimality: cond(\chi)=cond(\chi^c). If that can be relaxed, then this assumption can be relaxed as well.

For applications towards Iwasawa main conjectures for $f_{/K}\otimes \chi$ and $\f_{/K}\otimes \chi$ $($where $f$ and $\f$ are as in the main body of this article$)$, we shall take $\overline{\rho}:={\rm Ind}_{K/\QQ}\overline{\chi}$. As in \S\ref{subsubsec_3var_inert_ord}, we shall assume that 
\begin{itemize}
   \item[\mylabel{item_nonEismodp}{\textbf{\textup{(non-Eis)}}}] \hspace{50pt}${\chi}_{\vert_{G_{\QQ_{p^2}}}}\neq {\chi}^c_{\vert_{G_{\QQ_{p^2}}}}\,,$\hspace{10pt} where \hspace{3pt} $c$ is the generator of $\Gal(K/\QQ)$ \hspace{3pt} and \hspace{3pt} ${\chi}^c(\sigma)={\chi}(c\sigma c^{-1})$\,.
\end{itemize}
Let us also put  $N=|D_{K/\QQ}|\,N_{K/\QQ}\mathfrak{f}_\chi$, where $D_{K/\QQ}$ is the discriminant of $K/\QQ$.
\item[i)] We explain that $\overline{\rho}={\rm Ind}_{K/\QQ}\overline{\chi}$ verifies the hypothesis \ref{item_deformrho1}. First of all, since $p$ doesn't divide the order of $\chi$, it follows from the condition \ref{item_nonEismodp} that 
\begin{equation}
    \label{eqn_nonEis_stronger}
    \overline{\chi}\neq \overline{\chi}^c\,.
\end{equation}
Let us set $H_1:=G_K$, $H_2:=G_{\QQ(\mu_p)}$ and $G=G_\QQ$. Observe that both $H_1$ and $H_2$ are normal subgroups of $G$. Since $K$ and $\QQ(\mu_p)$ are linearly disjoint, it follows that $H_1H_2=G$ and $H_1\cap H_2 = G_{K(\mu_p)}$. Moreover, $c$ lifts canonically to a generator of $\Gal(\QQ(\mu_p)/K(\mu_p))=H_2/H_1\cap H_2$, which we still denote by $c$. This discussion is summarized in the following diagram.
$$\xymatrix{&\overline{\QQ}\ar@{-}[d]^{H_1\cap H_2}&\\
&K(\mu_p)\ar@{-}[rd]^{H_2/H_1\cap H_2=\langle c\rangle} \ar@{-}[ld]_{H_1/H_1\cap H_2}&\\
K\ar@{-}[rd]_{G/H_1=\langle c \rangle}&&\QQ(\mu_p)\ar@{-}[ld]^{G/H_2}\\
&\QQ&}$$
Then by Mackey theory, we have an isomorphism
$${\overline \rho}\vert_{H_2}={\rm Ind}_{H_1\cap H_2}^{H_2}\overline{\chi}$$
which is irreducible thanks to \eqref{eqn_nonEis_stronger} and because $H_2/H_1\cap H_2$ is generated by $c$. This concludes the proof that \ref{item_deformrho1} holds true.
\item[ii)] The condition \ref{item_deformrho2} holds true thanks to our running hypothesis \ref{item_nonEismodp} together with our assumption that the order of $\chi$ is prime to $p$.
\end{example}

\section{The minimally ramified universal deformation representation}
\label{subsec_app_minimal_Universal}
%The contents here will be relevant only for the second factor.
Let us fix a representation  $\overline{\rho}$ as in \S\ref{subsec_appendix_setup}. We let $S$ denote the set of places of $\QQ$ consisting of the archimedean place and all primes dividing $Np$ and set $G_{\QQ,S}:=\Gal(\QQ_S/\QQ)$, the Galois group of the maximal extension $\QQ_S$ of $\QQ$ unramifed outside $S$. In what follows, we shall consider $\overline{\rho}$ as a representation of $G_{\QQ,S}$.

Our summary in this subsection is essentially identical to the discussion in \cite[\S3]{Loeffler2020universalpadic} (where $N=1$), since we will only consider minimally ramified deformations.
\begin{defn}
\item[i)] We let $\cR(\overline{\rho})$ denote the minimally ramified universal deformation ring of $\olinerho$, parametrizing deformations of $\olinerho$ to $G_{\QQ,S}$-representations to complete local Noetherian $\cO$-algebras with residue field $\FF$ which are minimally ramified at primes dividing $N$ in the sense of \cite[\S3]{Diamond1997FLT}. We put 
$$\rho^{\rm univ}:G_{\QQ,S}\lra \GL(V(\olinerho)^{\rm univ})\cong \GL_2(\cR(\overline{\rho}))$$ 
to denote the minimally ramified universal deformation representation. We set $\mathfrak{X}(\olinerho):={\rm Spf}\,\cR(\overline{\rho})$ and call it the minimally ramified universal deformation space.
\item[ii)] If $g$ is a classical newform of level $Np^s$ for some non-negative integer $s$ such that $\overline{\rho}_g=\olinerho$, then $\rho_g$ determines a $\overline{\QQ}_p$-valued point of $\mathfrak{X}(\olinerho)$. Such points of the minimally ramified universal deformation space will be called classical.
\item[iii)] We say that a $\overline{\QQ}_p$-valued point of $\mathfrak{X}(\olinerho)$ is nearly-classical of weight $k$ if the said point corresponds to the Galois representation $\rho$ for which we have $\rho\cong \rho_g\otimes\chi_\cyc^t$ for some $($uniquely determined$)$ classical newform $g$ of weight $k$ as in (ii) and $($a uniquely determined$)$  integer $t$. 
\end{defn}

\begin{proposition}
\label{appendix_prop_3_5}
For any integer $k\geq 2$, the nearly-classical points of weight $k$ are Zariski-dense in the minimally ramified universal deformation space.
\end{proposition}
\begin{proof}
The proof of Proposition 3.5 in \cite{Loeffler2020universalpadic} applies verbatim (except that one needs to work with the spaces $S_k(\Gamma_1(Np^s))$ in the present set up) to deduce the required statement as a consequence of a fundamental result due to B\"ockle~\cite{Bockle2001} and Emerton~\cite[Theorem 1.2.3]{Emerton2011LocalGlobal}.
\end{proof}

\begin{defn}
\label{defn_appendix_Def3_7}
Let us write $\det\olinerho=\epsilon_{\olinerho}^{(N)}\epsilon_{\olinerho}^{(p)}$, where $\epsilon_{\olinerho}^{(N)}$ $($resp., $\epsilon_{\olinerho}^{(p)}$$)$ is a character of conductor $N$ $($resp., of conductor $p$$)$. We let $\epsilon:= G_{\QQ,S}\to W(\FF)^\times \hookrightarrow \cO^\times$ denote the Teichm\"uller lift of $\epsilon_{\olinerho}^{(N)}$ to  the Witt vectors $W(\FF)$ of $\FF$.
%So we are really looking at minimal lifts here too.
\item[i)] We let ${\bf k}:\ZZ_p^\times \to \cR(\olinerho)^\times$ denote the character which is characterized by the property that $\det \rho^{\rm univ}=\epsilon\chi_\cyc^{{\bf k}-1}\,.$
\item[ii)] We let $\{t_n\}_{p\nmid n}\subset \cR(\olinerho)$ be the sequence determined via the formal identity
$$\sum_{p\nmid n}t_n n^{-s}=\prod_{\ell\neq p} \det\left(1-\ell^{-s}{\rm Frob}_\ell^{-1}\vert (V(\olinerho)^{\rm univ})^{I_\ell}\right)^{-1}\,.$$ %Minimality seems to be crucial at this stage: In general, it is difficult to see how $(V(\olinerho)^{\rm univ})^{I_\ell}$ interpolates $V_g^{I_\ell}$, as $g$ varies in the deformation space. Even in that case, things are not crystal clear to me. 
We set 
$$\cG_{\olinerho}^{[p]}:=\sum_{p\nmid n}t_n q^n \in \cR(\olinerho)[[q]]\,.$$
\end{defn}

%I left out Theorem 3.8 here. I think it is OK for any nearly-classical specialization, and that's pretty much the point.

\section{Hida families}
\label{subsec_appendix_Hida_lambda}
Let $f\in S_{k_f+2}(\Gamma_1(N_f),\varepsilon_f)$ be a cuspidal newform which admits a $p$-ordinary $p$-stabilization $f_\alpha \in  S_{k_f+2}(\Gamma_1(N_f)\cap\Gamma_0(p),\varepsilon_f^{(p)})$, as in the main body of our article (where $\varepsilon_f^{(p)}$ is the non-primitive Dirichlet character modulo $N_fp$ associated to $\varepsilon_f$). 

Also as before, we let $\f\in \mathcal{S}^{\rm ord}(N_f,\LL)$ denote the unique primitive $\LL$-adic ordinary eigenform of tame level $N_f$ which admits $f_\alpha$ as a specialization, where $\LL=\ZZ_p[[\ZZ_p^\times]]$ is the weight space. {Let $\LL_\f$ denote the corresponding branch of Hida's universal $p$-ordinary Hecke algebra $\mathbf{h}^{\rm ord}(N_f)$ of tame level $N_f$. Mimicking Hida's construction in~\cite[(7.5)]{hida88}, we define 
\begin{equation}
    \label{eqn_appendix_HidasLambda_0}
    \lambda_\f: \mathcal{S}^{\rm ord}(N_f,\LL)\otimes_{\Lambda} \LL_\f \xrightarrow{<\mathbf{1}_{\f},\,\,>} {\rm Frac}(\LL_\f)\,,
\end{equation}
where $\mathbf{1}_{\f}\in  \mathbf{h}^{\rm ord}(N_f)\otimes_{\LL}{\rm Frac}(\LL_\f)$ is the idempotent that Hida in op. cit. denotes by $\mathbf{1}_{{\rm Frac}(\LL_\f)}$, which corresponds to the choice of the primitive form $\f$, and where the pairing $<\,,\,>$ is the one given in \cite[(7.3)]{hida88}. Note that the image of $\lambda_\f$ a priori lands in ${\rm Frac}(\LL_\f)$, since  $\mathbf{1}_{\f}\in  \mathbf{h}^{\rm ord}(N_f)\otimes_{\LL}{\rm Frac}(\LL_\f)$.  %We may need to use here the Gorenstein property for the rings that appear here, for Hida's construction yields a pairing into $\LL_\f^*=\Hom(\LL_\f,\cO)$.

Let us choose a non-zero annihilator $H_\f\in \LL_\f$ of the congruence module $\mathcal{C}_0(\lambda_\f;\LL_\f)$ that Hida has associated to $\f$ in \cite[\S4]{hida88}. When $\LL_\f$ is Gorenstein, then the annihilator ${\rm Ann}_{\LL_\f}\left(\mathcal{C}_0(\lambda_\f;\LL_\f)\right)$ of the congruence module $\mathcal{C}_0(\lambda_\f;\LL_\f)$ is principal (c.f. Theorem 4.4 (4.6b) in op. cit.) and in this case, we choose $H_\f$ as a generator of ${\rm Ann}_{\LL_\f}\left(\mathcal{C}_0(\lambda_\f;\LL_\f)\right)$. We recall that in the main body of our article, we assume that $\LL_\f$ is regular (see also Remark~\ref{remark_regular_is_light_assumption} for possible variations on this condition). Therefore, with applications in the present work in mind, one may prefer to assume that ${\rm Ann}_{\LL_\f}\left(\mathcal{C}_0(\lambda_\f;\LL_\f)\right)$ is indeed principle (although we will not do so in Appendix~\ref{appendix_sec_padicRankinSelbergHida}).

Hida explains that, by the very definition of the congruence ideal in \cite[(4.3)]{hida88}, it follows that $H_\f\cdot \mathbf{1}_{\f} \in \mathbf{h}^{\rm ord}(N_f)\otimes_{\LL}\LL_\f$, so that the map \eqref{eqn_appendix_HidasLambda_0} factors as 
\begin{equation}
    \label{eqn_appendix_HidasLambda}
    \lambda_\f: \mathcal{S}^{\rm ord}(N_f,\LL)\otimes_{\Lambda} \LL_\f \xrightarrow{<\mathbf{1}_{\f},\,\,>} \frac{1}{H_\f}\LL_\f\,\,\subset\, {\rm Frac}(\LL_\f)\,,
\end{equation}
}

 For any integer $M$ divisible by $N_f$ and coprime to $p$, we will write $\lambda_\f^{(M)}$ for the compositum
\begin{equation}
    \label{eqn_appendix_HidasLambdaM}
    \lambda_\f^{(M)}:\,\,\mathcal{S}^{\rm ord}(M,\LL)\otimes_{\Lambda} \LL_\f\xrightarrow{{\rm tr}^M_{N_f}}\mathcal{S}^{\rm ord}(N_f,\LL)\otimes_{\Lambda} \LL_\f \xrightarrow{\lambda_\f} \frac{1}{H_\f}\LL_\f
\end{equation}
where ${\rm tr}^{M}_{N_f}$ is the trace map.

\section{$p$-adic Rankin--Selberg $L$-function}
Let $\olinerho:G_{\QQ,S}\to \GL_2(\FF)$ be as in \S\ref{subsec_appendix_setup} and recall that $N$ denotes its Artin conductor and $\rho^{\rm univ}$ its minimal universal deformation. We also fix $f$ and $\f$ as in \S\ref{subsec_appendix_Hida_lambda}, with the additional assumption that $\gcd(N_f,N_g)=1$. Following \cite[\S4]{Loeffler2020universalpadic}, we will define a $p$-adic Rankin--Selberg $L$-function 
$$L_p(\rho_\f\otimes \rho^{\rm univ}) \in \frac{1}{H_{\f}}\LL_{\f}\,\widehat{\otimes}\,\cR(\overline{\rho}).$$ 
Note that thanks to this simplification together with the minimality condition on the deformation problem we consider on $\overline{\rho}$, the bad primes require no additional treatment.

Let us put $M=NN_f$. We enlarge $\cO$ so that it contains $M$th roots of unity. We set $\zeta_M:=\iota(e^{2\pi/M})$ where we recall that $\iota:\mathbb{C}\to \mathbb{C}_p$ is the isomorphism we have fixed at the start of this article.

\begin{defn}
\label{def_appendix_universalPadicLfunction}
Consider the universal $p$-depleted Eisenstein series 
$$E^{[p]}_{\bf{k}}:=\sum_{\substack{p\nmid n\\ n\geq 1}}\left( \sum d^{\mathbf{k}-1}\left(\zeta_M^d+(-1)^{\mathbf{k}}\zeta_M^{-d}\right)\right)q^n\in R(\overline{\rho})[[q]]\,.$$
We define the $p$-adic Rankin--Selberg $L$-function $L_p(\rho_\f\otimes \rho^{\rm univ})$
on setting \index{$p$-adic $L$-functions! $L_p(\rho_\f\otimes \rho^{\rm univ})$}
$$L_p(\rho_\f\otimes \rho^{\rm univ}):=\lambda_\f^{(M)}\left( e^{\rm ord}\left(\cG^{[p]}_{\olinerho} \cdot E^{[p]}_{\bf{k}}\right)\right)\,.$$
\end{defn}

The $p$-adic $L$-function $L_p(\rho_\f\otimes \rho^{\rm univ})$ interpolates the critical values of Rankin--Selberg $L$-series. We formulate this as a statement which relates $L_p(\rho_\f\otimes \rho^{\rm univ})$ to the ``geometric'' $p$-adic $L$-functions of Loeffler and Zerbes, which we have recalled in \S\ref{subsec_review_BF_elements}. For a more ``honest'' version of the interpolation property in a particular case of interest involving the special values of the Rankin--Selberg $L$-series, see Theorem~\ref{thm_appendix_interpolation_inertCMpadicLfunction}.

\begin{theorem}\index{$p$-adic $L$-functions! $L_{p}^{\rm geo}(\f,g)$ (Loeffler--Zerbes ``geometric'' $p$-adic $L$-function)}
\label{thm_appendix_interpolation}
Suppose $g \in \mathfrak{X}(\olinerho)(L)$ be an $L$-valued classical point. Then
$$L_p(\rho_\f\otimes \rho^{\rm univ})(g)=L_p^{\rm geo}(\f,g)\,.$$
\end{theorem}

\begin{proof}
This is immediate thanks to the interpolation properties of both sides, which one obtains via the Rankin--Selberg integral formulae, c.f. \cite[Proposition 2.10]{loeffler18}.  The theorem is proved in a manner similar to \cite[Theorem 6.3]{loeffler18}; see also \cite[Theorem 4.5]{Loeffler2020universalpadic}. 
\end{proof}

\subsection{$p$-adic $L$-functions over an imaginary quadratic field where $p$ is inert}
\label{app_subsubsec_padicLinert}
In this paragraph, we specialize to the setting of Example~\ref{app_example_residual_conditions}. To that end, 
we let $K$ be an imaginary quadratic field as above where $p$ remains \emph{inert} and let $\chi$ be a ray class character of $K$ of conductor $\mathfrak{f}_\chi$. Let us put $N=|D_{K/\QQ}|\,N_{K/\QQ}\mathfrak{f}_\chi$, where $D_{K/\QQ}$ is the discriminant of $K/\QQ$.

We assume that 
\begin{itemize}
    \item the order of $\chi$ and $\mathfrak{f}_\chi$ are coprime to $p$;
    \item The non-Eisenstein hypothesis \ref{item_nonEismodp} holds true.
\end{itemize}

 We will consider the representation $\overline{\rho}:={\rm Ind}_{K/\QQ}\overline{\chi}$. Recall that we have checked in Example~\ref{app_example_residual_conditions} that $\overline{\rho}$ verifies the hypotheses \ref{item_deformrho1} and \ref{item_deformrho2} so that the general theory in \S\ref{subsec_app_minimal_Universal} applies and gives rise to the $p$-adic $L$-function 
 $$L_p(\rho_\f\otimes \rho^{\rm univ})\in \frac{1}{H_\f}\LL_\f\,\widehat{\otimes}\,\cR(\overline{\rho}).$$
 
Let $K_\infty\subset K(p^\infty)$ denote the maximal $\ZZ_p$-power extension of $K$ and put $\Gamma_K:=\Gal(K_\infty/\QQ)$. Consider the tautological character
$$\Psi: G_K\twoheadrightarrow \Gamma_K \hookrightarrow \LL_{\cO}(\Gamma_K)^\times\,.$$
Observe that 
$$\rho_\Psi := {\rm Ind}_{K/\QQ}\,\chi\Psi:\, G_\QQ\to \GL_2(\LL_{\cO}(\Gamma_K))$$ 
is a representation of $G_\QQ$ with coefficients in the complete local Noetherian $\cO$-algebra $\LL_{\cO}(\Gamma_K)$. Moreover, $\rho_\Psi$ is lifts $\overline{\rho}$ and it is evidently minimally ramified. By the universality of $\rho^{\rm univ}$, there exists a unique homomorphism of local rings
\begin{equation}
    \label{eqn_app_universalproperty_pi}
    \pi_\Psi: \cR(\olinerho)\lra  \LL_{\cO}(\Gamma_K) \hspace{1cm} \hbox{ with } \hspace{0.3cm} \rho_\Psi=\rho^{\rm univ}\otimes_{\pi_\Psi}\LL_{\cO}(\Gamma_K)\,.
\end{equation}

\begin{defn}
\label{def_app_rankinselbergPadicInertCM}\index{$p$-adic $L$-functions! $L_p^{\rm RS}(\f_{/K}\otimes\bbchi)$ (``semi-universal'' Rankin--Selberg $p$-adic $L$-function)}
We set 
$$\displaystyle{L_p^{\rm RS}(\f_{/K}\otimes\chi\Psi):=\pi_\Phi\left(L_p(\rho_\f\otimes \rho^{\rm univ})\right)\in \frac{1}{H_\f}\LL_\f\,\widehat{\otimes}\,\LL_{\cO}(\Gamma_K)\,.}$$
\end{defn}

We will conclude this appendix recording an interpolative property that characterizes the $p$-adic $L$-function $L_p^{\rm RS}(\f_{/K}\otimes\chi\Psi)$. Before that, we define the relevant collection of Hecke characters for the said interpolation formula.
\begin{defn}
\label{defn_critical_range_for_Hecke_chars}
Given a Hecke character $\psi$ of $K$, let us write $\infty(\psi)\in \ZZ^2$ for the infinity-type of $\psi$. For any integer $k\geq -1$, we put 
$$\Sigma^{(1)}_{\rm crit}(k):=\{\textup{Hecke characters } \psi \textup{ of } K: \infty(\psi)=(\ell_1,\ell_2) \textup{ with } 0\leq \ell_1,\ell_2\leq k\}\,.$$
\end{defn}

\begin{theorem}
\label{thm_appendix_interpolation_inertCMpadicLfunction}
Let $\kappa\geq 0$  be an integer such that $\kappa\equiv k_f \mod (p-1)$ and let $\f(\kappa)$ denote the unique specialization of $\f$ of weight $\kappa+2$ which is $p$-old. We put $\f^\circ(\kappa)\in S_{\kappa+2}(\Gamma_1(N_f))$ for the newform whose $p$-ordinary $p$-stabilization is $\f(\kappa)$. We write $\alpha(\kappa)$ for the $U_p$-eigenvalue on $\f(\kappa)$ and put $\beta(\kappa):=p^{\kappa+1}\varepsilon_f(p)/\alpha(\kappa)$. Let $\psi\in \Sigma^{(1)}_{\rm crit}(\kappa)$ be a Hecke character whose associated $p$-adic Galois character $\widehat{\psi}$ factors through $\Gamma_K$ and is crystalline at $p$. Then,
\begin{align*}
    L_p^{\rm RS}(\f_{/K}\otimes\chi\Psi)(\f(\kappa),\psi)&=\frac{\left(1-\alpha(\kappa)^{-2}\chi\psi(p)\right)\left(1-p^{-2}\beta(\kappa)^2\chi^{-1}\psi^{-1}(p)\right)}{\left(1-\beta(\kappa)/p\alpha(\kappa)\right)\left(1-\beta(\kappa)/\alpha(\kappa)\right)}\\
    &\qquad\qquad\qquad\quad\times\, \frac{i^{\kappa+1-\ell_1+\ell_2}M^{1+\ell_1+\ell_2-\kappa}\ell_1! \ell_2!}{2^{\ell_1+\ell_2+\kappa}\pi^{\ell_1+\ell_2}}\\
    &\qquad\qquad\qquad\qquad\qquad
    \times \, \frac{L(\f(\kappa)_{/K},\chi^{-1}\psi^{-1},1)}{8\pi^2\langle \f^\circ(\kappa),\f^\circ(\kappa)\rangle_{N_f}}\,.
\end{align*}
\end{theorem}

\begin{proof}
This follows on translating the interpolation formula for $L_p(\rho_\f\otimes \rho^{\rm univ})$ that we alluded to above in the proof of Theorem~\ref{thm_appendix_interpolation}.
\end{proof}

\begin{remark}
One may give a direct construction of $ L_p^{\rm RS}(\f_{/K}\otimes\chi\Psi)$, without going through the minimally ramified universal deformation ${\rho}^{\rm univ}$. We briefly outline this alternative construction in this remark. 

Let us write $K(\mathfrak{f}_\chi p^{\infty})=\cup_n K(\mathfrak{f}_\chi p^{n})$ denote the ray class extension of $K$ modulo $\mathfrak{f}_\chi p^{\infty}$. We let $H_{\mathfrak{f}_\chi p^{\infty}}:=\varprojlim_n H_{\mathfrak{f}_\chi p^{n}}$ denote the ray class group modulo $\mathfrak{f}_\chi p^{\infty}$, which we identify with $\Gal(K(\mathfrak{f}_\chi p^{\infty})/K)$ via the geometrically normalized Artin reciprocity map $\mathfrak{A}$. Let us write 
$$H_{\mathfrak{f}_\chi p^{\infty}}=\Delta\times H_{\mathfrak{f}_\chi p^{\infty}}^{(p)}$$
where $\Delta$ is a finite group and $H_{\mathfrak{f}_\chi p^{\infty}}^{(p)}\cong \ZZ_p^2$, so that $\mathfrak{A}(H_{\mathfrak{f}_\chi p^{\infty}}^{(p)})=\Gamma_K$. Given $\mathfrak{b}\in H_{\mathfrak{f}_\chi p^{\infty}}$, we write $[\mathfrak{b}]\in \LL_{\cO}(\Gamma_K)^\times$ for the image of $\mathfrak{b}$ under the compositum
$$H_{\mathfrak{f}_\chi p^{\infty}}\lra H_{\mathfrak{f}_\chi p^{\infty}}^{(p)} \xrightarrow{\mathfrak{A}} \Gamma_K \hookrightarrow \LL_{\cO}(\Gamma_K)^\times\,.$$

We define the $p$-depleted universal $\theta$-series for the branch character $\chi$ on setting
$$\cG^{[p]}_{\chi}:=\sum_{\substack{\mathfrak{b}<\cO_K\\ (\mathfrak{b},p)=1}}\chi(\mathfrak{b})[\mathfrak{b}]q^{N\mathfrak{b}}\in \LL_{\cO}(\Gamma_K)[[q]]\,.$$
If one uses $\cG^{[p]}_{\chi}$ in Definition~\ref{def_appendix_universalPadicLfunction} in place of $\cG^{[p]}_{\olinerho}$, one obtains the $p$-adic $L$-function $L_p^{\rm RS}(\f_{/K}\otimes\chi\Psi)$.
\end{remark}

%-----------------------------------------------------------------------
% End of AppB.tex
%-----------------------------------------------------------------------

%% file: InertOrdMainAppendixC.tex
\chapter[Images of Galois representations]{Images of Galois representations attached to Rankin--Selberg convolutions} 
\label{appendix_big_images} 
\index{Big image conditions! $(\tau_{f\otimes g})$} \index{Big image conditions! $(\tau_{\f\otimes g})$} In this appendix, we study the hypotheses \ref{item_BI_fg} and \ref{item_BI_ffg} when $g=\theta(\psi)$ is a crystalline CM form, and provide sufficient conditions for their validity. It will be clear to reader that our arguments here draw greatly\footnote{We would like to express our gratitude to Jackie Lang for carefully reading through this appendix and pointing out several inaccuracies, which helped us improve our exposition drastically. In particular, we thank her for bringing the papers of Klingenberg~\cite{Klingenberg1960, Klingenberg1961} to our attention and explaining relevant portions of the results in her joint work \cite{ContiLangMedvedovsky} with Conti and Medvedovsky (which lead to a significant simplification of some of our statements and proofs).} from the earlier works \cite{Fischmann2002, Hida2015, Lang2016, LoefflerGlasgow, ContiLangMedvedovsky} on the subject. 

Throughout this appendix, we assume that $p\geq 7$. 
\section{Group theory}
Let $R$ be a complete local ring with finite residue field and residue characteristic $p$. Let $\mathfrak{m}$ denote the maximal ideal of $R$ and $\FF:=R/\mathfrak{{m}}$ its residue field. 
\begin{lemma}
\label{appendix_big_images_lemma_1}
There is a unique maximal normal subgroup $N_0$ of ${\rm SL}_2(R)$, and the quotient group ${\rm SL}_2(R)/N_0$ is isomorphic to ${\rm PSL}_2(\mathbb{F})$. In particular, the quotient ${\rm SL}_2(R)/N_0$ is non-solvable.
\end{lemma}

\begin{proof}
The argument we record here was suggested to us by Jackie Lang. 

 For an ideal $J\triangleleft R$ we let ${\rm GC}(J)$ denote the set of all the matrices in $\GL_2(R)$ which reduce to a scalar matrix modulo $J$. We also let ${\rm SC}(J)$ denote the matrices of determinant $1$ that reduce to the identity matrix modulo $J$.  
 
 It follows from \cite[Theorem 3(ii)]{Klingenberg1961} that if $N$ is a proper normal subgroup of ${\rm SL}_2(R)$, then there is an ideal $J(N)\triangleleft R$ such that 
 $${\rm SC}(J(N))\subset N\subset {\rm GC}(J(N)) \cap {\rm SL}_2(R)\,.$$  
Note that since $N$ is not equal to ${\rm SL}_2(R)$ and $N$ contains ${\rm SC}(J(N))$, it follows that the ideal $J(N)\triangleleft R$ must be a proper ideal and hence, it is contained in the maximal ideal $\frak{m}$ of the local ring $R$.  $N$ is therefore contained in the subgroup 
$${\rm GC}(\frak m) \cap {\rm SL}_2(R)\supset {\rm GC}(J(N)) \cap {\rm SL}_2(R).$$  
of ${\rm SL}_2(R)$. Note that 
$${\rm GC}(\frak m) \cap {\rm SL}_2(R)=\ker\left({\rm SL}_2(R) \lra {\rm PSL}_2(\FF)\right)=:N_0\,.$$  
We have therefore proved that every proper normal subgroup of ${\rm SL}_2(R)$ is contained in $N_0$, as required.
\end{proof}

\begin{lemma}
\label{appendix_big_images_lemma_2}
Let $G_1$ and $G_2$ be two groups and $H<G_1\times G_2$ such that for each $i=1,2$, the natural projection map $\pr_i:H\to G_i$ is surjective. Suppose in addition that $G_2$ is solvable and $G_1$ admits a unique proper maximal normal subgroup $N$ with $G_1/N$ is non-solvable. Then $H=G_1\times G_2$. 
\end{lemma}

\begin{proof}
This is an easy application of Goursat's lemma, which we recall for the convenience of the reader. Let $H<G_1\times G_2$ be a subgroup such that the natural projection map $\pr_i:H\to G_i$ is surjective ($i=1,2$). Set $N_2:=\ker(\pr_1)$ and $N_1:=\ker(\pr_2)$. Let us identify $N_2$ (respectively, $N_1$) as a normal subgroup of $G_2$  (respectively, of $G_1$). Then the image of $H$ in $G_1/N_1\times G_2/N_2$ is the graph of an isomorphism $G_1/N_1\cong G_2/N_2$. It therefore suffices to prove that $N_1=G_1$ (equivalently, $N_2=G_2$) for $N_i$ as above.

Suppose on the contrary that $N_1\triangleleft G_1$ is a proper normal subgroup. Then by definition, $N_1<N$ and $G_1/N_1$ admits $G_1/N$ as a quotient. In particular, $G_1/N_1$ is non-solvable. But since $G_1/N_1\cong G_2/N_2$ and $G_2/N_2$ is solvable (since $G_2$ is). This contradiction concludes our proof that $N_1$ cannot be a proper subgroup, and in turn also the proof our lemma.
\end{proof}

\section{Applications to families on $\GL_2\times {\rm Res}_{K/\QQ}\GL_1$}
\label{appendix_big_images_subsec_2}
Let $\f$ be (a non-CM branch of) a Hida family as in the main body of this article (c.f. Definition~\ref{defn_Hecke_algebra_intro}). We also fix a crystalline Hecke character $\psi$ of our fixed imaginary quadratic field $K$, given as in \S\ref{subsubsec_setting_fKpsi}. We will freely use our notation concerning $\f$  and $\psi$ from the main text. 

In particular, recall the local ring $\LL_\f$, which is a module-finite flat $\Lambda_\cO(1+p\ZZ_p)$-algebra, which we assume contains the ring of integers $\cO$ of a sufficiently\footnote{We do not make this precise, but it at least contains the image of $\widehat{\psi}$.} large extension of $\QQ_p$. Let us denote by $\texttt{k}$ the residue field of $\LL_\f$. Recall also the theta-series $g:=\theta(\psi)\in S_{k_g+2}(\Gamma_1(N_g),\varepsilon_\psi)$, whose associated Galois representation $\rho_{\theta(\psi)}$ giving the action on $R_{\theta(\psi)}^*$ can be explicitly described as follows:
\begin{equation}
\label{eqn_explicit_induced_rep}
    \sigma\stackrel{\rho_{\theta(\psi)}}{\longmapsto}
    \begin{cases}
    \left(\begin{array}{cc}
        \widehat{\psi}^{-1}(\sigma) & 0 \\
        0 &  \widehat{\psi}^{-1}(c\sigma c^{-1})
    \end{array}\right)& \hbox{ if } \sigma\in G_K\,\\\\
     \left(\begin{array}{cc}
        0 & \widehat{\psi}^{-1}(\sigma c) \\
        \widehat{\psi}^{-1}(c \sigma) &  0
    \end{array}\right)& \hbox{ if } \sigma \in G_\QQ\setminus G_K\,.
    \end{cases}
\end{equation}

Throughout \S\ref{appendix_big_images_subsec_2} we shall assume that $\rho_\f$ is residually full, in the sense that  the condition \index{Big image conditions! {{\bf (Full)}}}
\begin{enumerate}[align=parleft, labelsep=0.2cm,]
\item[\mylabel{item_fullness}{{\bf{(Full)}}}] \, ${\rm SL}_2(\FF_p)\subset  \overline{\rho}_\f(G_{\QQ(\mu_{p^\infty})})$
\end{enumerate}
holds true. The condition \ref{item_fullness} is often (but not always) satisfied.

Let us put $\mathscr{G}_1:=\ker\left(\det\,\circ \,{\rho}_{\f}\right)$. Note that $\mathscr{G}_1$ is a subgroup of $G_{\QQ(\mu_{p^\infty})}$ with finite index. If we have $\varepsilon_f=\mathds{1}$ for the tame nebentype of the family $\f$, then  $\mathscr{G}_1=G_{\QQ(\mu_{p^\infty})}$. In general, if we let $K_1:=\QQ(\varepsilon_f)$  denote the abelian (in fact, cyclic) extension cut out by $\varepsilon_f$, then $\mathscr{G}_1=G_{K_1(\mu_{p^\infty})}$. Note also that $\QQ(\varepsilon_f) \cap K=\QQ$, thanks to our assumption that $(N_f,D_K)=1$. 

\begin{lemma}
\label{lemma_fullness_+}
If \ref{item_fullness} holds, then in fact 
$${\rm SL}_2(\FF_p)\subset  \overline{\rho}_\f(\mathscr{G}_1)\,.$$
\end{lemma}

\begin{proof}
Let us put $\mathscr{G}_0:=G_{\QQ(\mu_{p^\infty})}$ to ease notation. Note that $\mathscr{G}_0/\mathscr{G}_1$ is a finite abelian group. In particular, the groups 
$${\rm SL}_2(\FF_p)\big{/}\overline{\rho}_\f(\mathscr{G}_1)\cap {\rm SL}_2(\FF_p)\,\, <\,\overline{\rho}_\f(\mathscr{G}_0)\big{/}\overline{\rho}_\f(\mathscr{G}_1)$$
are abelian as well. Moreover, since $\overline{\rho}_\f(\mathscr{G}_1)\cap {\rm SL}_2(\FF_p)$ is a normal subgroup of $ {\rm SL}_2(\FF_p)$ and $p\geq 7$, it follows that either $\overline{\rho}_\f(\mathscr{G}_1)\cap {\rm SL}_2(\FF_p)<\{\pm 1\}$ or else $\overline{\rho}_\f(\mathscr{G}_1)$ contains ${\rm SL}_2(\FF_p)$. In the former scenario, it would follow (using the fact that the quotient group ${\rm SL}_2(\FF_p)\big{/}\overline{\rho}_\f(\mathscr{G}_1)\cap {\rm SL}_2(\FF_p)$ is abelian) that ${\rm PSL}_2(\FF_p)$ is abelian, which is absurd. We therefore conclude that the  latter scenario holds, which is the assertion of our lemma.
\end{proof}

Let us put $\varepsilon_\psi=\eta_\psi\epsilon_K$ where $\eta_\psi$ is a Dirichlet character of conductor coprime to $pD_K$. We define $K_2:=\QQ(\eta_\psi)$ as the field cut out by $\eta_\psi$ and set $K^\prime:=K_1K_2$. We put $\mathscr{G}:=G_{K^\prime(\mu_{p^\infty})}$. It is a normal subgroup of $\mathscr{G}_1$ with finite index and moreover, the quotient $\mathscr{G}_1/\mathscr{G}$ is abelian (in fact cyclic, but we won't need that).

Note that $K^\prime$ is still linearly disjoint from $K(\mu_{p^\infty})$ (since the conductors of $\eta_\psi$ and $\varepsilon_f$ are both coprime to $pD_K$). In particular, $\mathscr{G}\not\subset G_{K(\mu_{p^\infty})}$. We shall use this observation crucially in the proof of Theorem~\ref{appendix_big_images_subsec_2_thm_3}.

\begin{lemma}
\label{lemma_fullness_++}
If \ref{item_fullness} holds, then in fact 
$${\rm SL}_2(\FF_p)\subset  \overline{\rho}_\f(\mathscr{G})\,.$$
\end{lemma}

\begin{proof}
The proof of this lemma is identical to that of Lemma~\ref{lemma_fullness_+}: One simply replaces $\mathscr{G}_0$ with $\mathscr{G}_1$ and $\mathscr{G}_1$ with $\mathscr{G}$ everywhere in the proof.
\end{proof}

\begin{proposition}
\label{appendix_big_images_subsec_2_prop_1}
Let $f$ be a classical specialization of $\f$ and let $\rho_f$ denote the $G_\QQ$-representation giving rise to the action on $R_f^*$. We then have, 
$$\rho_f\times \rho_{\theta(\psi)}(\mathscr{G})=\rho_f(\mathscr{G})\times \rho_{\theta(\psi)}(\mathscr{G})\,,$$
$$\rho_\f\times \rho_{\theta(\psi)}(\mathscr{G})=\rho_\f(\mathscr{G})\times \rho_{\theta(\psi)}(\mathscr{G})\,.$$
\end{proposition}

\begin{proof}
We will verify the first assertion using Lemma~\ref{appendix_big_images_lemma_2} with the choices $H:=\rho_f\times \rho_{\theta(\psi)}(\mathscr{G})$, $G_1:=\rho_f(\mathscr{G})$, and $G_2:=\rho_{\theta(\psi)}(\mathscr{G})$. Once we check that the conditions of Lemma~\ref{appendix_big_images_lemma_2} are satisfied with these choices, our proposition follows at once from this lemma. 

Indeed, $H$ is clearly a subgroup of $G_1\times G_2$ and the natural projections ${\pr}_i: H\to G_i$ ($i=1,2$) are surjective. Moreover, since $\rho_{\theta(\psi)}(G_{K})$ is abelian, it follows that $G_2$ is solvable. %Under our running hypotheses that \ref{item_fullness} holds and $p\geq 7$, it follows that the regularity hypothesis of \cite{ContiLangMedvedovsky} is verified (owing to the existence of  $$\left(\begin{array}{cc}    x &  0\\    0 & x^{-1}\end{array}\right)\in {\rm SL}_2(\FF_p)$$ with $x^2\neq \pm 1$). 
It follows from \cite[Proposition 6.7]{ContiLangMedvedovsky} that $G_1={\rm SL_2}(A)$ for some finite $\ZZ_p$-algebra $A$ (in particular, $A$ is a complete local ring). We can now invoke Lemma~\ref{appendix_big_images_lemma_1} to see that $G_1$ too verifies the conditions required in Lemma~\ref{appendix_big_images_lemma_2}, and thereby concludes the proof of our first assertion.

The proof of the second part proceeds in an identical manner, using the fact that the local ring $A_0=A_0(\rho)$ given as in \cite[\S1]{ContiLangMedvedovsky} with $A=\LL_\f$ and $\rho=\rho_\f$ is a complete local ring with finite residue field\footnote{The following argument was supplied to us by Jackie Lang (and it is in part due to Anna Medvedovsky). Lacking a specific reference, we record the proof of this claim here. We thank Jackie Lang for explaining this argument to us and allowing us to reproduce it here. 

Recall that $A_0(\rho)$ is a closed subring of $\LL_\f$, which is local and pro-$p$.  Any closed subring of a pro-$p$ ring is also pro-$p$, and hence complete (in fact more generally, any closed subgroup of a profinite group is also profinite). This shows that $A_0(\rho)$ is pro-$p$ (and therefore, complete). We next explain why $A_0(\rho)$ is local. Let us put $\frak m_0 := A_0(\rho) \cap \frak m$.  Note then that $k_0 := A_0(\rho)/\frak m_0$ injects into $k$ and is therefore a field. This in turn tells us that $\frak m_0$ is a maximal of $A_0(\rho)$ and it follows that $A_0(\rho)$ is a $W(k_0)$-algebra.  To complete the proof that $A_0(\rho)$ is local, we need to verify that every element in $A_0(\rho) \setminus \frak m_0$ is invertible in $A_0(\rho)$. Indeed, if $x\in \frak m_0$, then $1/(1+x)$ is a power series in $x$, hence also in $A_0(\rho)$ since $A_0(\rho)$ is closed. Thence, $1 + \frak m_0 \subset A_0(\rho)^\times$.  But every element $a \in A_0(\rho)\setminus \frak m_0$ factors as $a=um$ where $u\in W(k_0)$ (the Teichm\"uller representative) and $m\in 1+\frak m_0$.  This concludes the proof that  $A_0(\rho) \setminus \frak m_0 \subset A_0(\rho)^\times$.}. 
\end{proof}

Recall that $\varepsilon_\psi$ is the nebentype character for $\theta(\psi)$. It is a Dirichlet character of conductor dividing $|D_K|\mathbf{N}\ff$; see Definition~\ref{define_CM_nebentype} for its explicit description.

\begin{theorem}
\label{appendix_big_images_subsec_2_thm_3} Let $\f$ be a non-CM Hida family with trivial nebentype and suppose $\psi$ is a crystalline Hecke character as above. Suppose that $p\geq 7$ and the condition $\ref{item_fullness}$ holds true. %We assume that the following condition holds true: If $\varepsilon_\psi\neq\epsilon_K$, then there exists an integer $u$ such that $\epsilon_K(u)=1$ and $v_p(\varepsilon_\psi(u)-1)=0$.
\item[i)] Let $f_\alpha$ be a crystalline specialization of $\f$, which arises as the $p$-stabilization of a cuspidal eigen-newform $f$. Then the condition \ref{item_BI_fg} is satisfied with $g=\theta(\psi)$.
\item[ii)] The condition \ref{item_BI_ffg} holds true with $g=\theta(\psi)$. %In particular,  the condition \ref{item_BI_ffpsivaries} is also satisfied.
\end{theorem}

\begin{proof}
Since the proofs of both parts are very similar, we shall only provide details for the proof of (i). Let us denote by $\frak{t}$ the natural map
$$\GL_2(\cO)\times \GL_2(\cO) \cong \GL(R_f^*)\times \GL(R_{\theta(\psi)}^*)\stackrel{\mathfrak{t}}{\lra} \GL(T_{f,\theta(\psi)})\cong \GL_4(\cO)\,.$$

We shall prove that there exist $M_1\in \rho_f(\mathscr{G})\subset \GL(R_f^*)$ and $M_2\in \rho_{\theta(\psi)}(\mathscr{G})\subset \GL(R_{\theta(\psi)}^*)$ such that the minimal polynomial of $\frak{t}(M_1,M_2)$ equals $(X-1)^2(X+1)^2$. The proof of our theorem then follows from Proposition~\ref{appendix_big_images_subsec_2_prop_1}.

As we have explained in the proof of Proposition~\ref{appendix_big_images_subsec_2_prop_1}, the group $\rho_f(\mathscr{G})$ contains $M_1:=\left(\begin{array}{cc}
    1 &1  \\
    0 & 1
\end{array}\right)$. Moreover, it follows from the explicit description \eqref{eqn_explicit_induced_rep} of $\rho_{\theta(\psi)}$ that the group $\rho_{\theta(\psi)}(\mathscr{G})$ contains an element $M_2$ of the form $\left(\begin{array}{cc}
    0 & x  \\
    x^{-1} & 0
\end{array}\right)$
for some $x\in \cO^\times$. Indeed, for any $\sigma\in \mathscr{G}$, we have 
$$\det\rho_{\theta(\psi)}(\sigma)=\eta_f(\sigma)\epsilon_K(\sigma)=\epsilon_K(\sigma),$$ 
where the second equality follows from the definition of the group $\mathscr{G}$. In particular, if $\sigma\in \mathscr{G}\setminus G_{K(\mu_{p^\infty})}$, then $\det\rho_{\theta(\psi)}(\sigma)=-1$. (Recall our remark in the paragraph preceding the statement of Lemma~\ref{lemma_fullness_++} that $\mathscr{G}\setminus G_{K(\mu_{p^\infty})}$ is non-empty.) Given the description of the representation $\rho_{\theta(\psi)}$ in \eqref{eqn_explicit_induced_rep}, it follows that $\rho_{\theta(\psi)}(\sigma)$ has necessarily the required form for any $\sigma\in \mathscr{G}\setminus G_{K(\mu_{p^\infty})}$.

Note that $M_2$ is conjugate to $\left(\begin{array}{cc}
    1 & 0  \\
    0 & -1
\end{array}\right)$ and in turn, $\frak{t}(M_1,M_2)$ is conjugate to the upper triangular matrix 
$$\left(\begin{array}{cccc}
    1 & 1 & 0 & 0   \\
    0 & 1 & 0 & 0 \\
0 & 0 & -1 & -1\\
0 & 0 & 0 & -1
\end{array}\right)\,.$$
The proof of our theorem is now complete.
\end{proof}

%\begin{remark}
%\label{remark_conclusion_appendix_C}
%In view of Theorem~\ref{appendix_big_images_subsec_2_thm_3}(i) and Remark~\ref{rem_hypotau_central_char_ff}, we can replace the assumption \ref{item_BI_fpsi} in Theorem~\ref{thm_cyclo_main_inert_f}, the assumption \ref{item_BI_ffpsi} in Theorem~\ref{thm_cyclo_main_inert_ff},  the assumption \ref{item_BI_ffpsivaries} in Theorems~\ref{thm_3var_main_inert_ff}, \ref{thm_anticyc_main_inert_ff_definite} and \ref{thm_anticyc_main_inert_ff_indefinite} with the following two hypothesis: \index{Big image conditions! $(\tau_{f_{/K}\otimes \psi})$} \index{Big image conditions! $(\tau_{\f_{/K}\otimes \psi})$}
   % \begin{itemize}
      %  \item[\mylabel{item_BI_1}{{\bf BI)}}] $p\geq 7$ and the condition $\ref{item_fullness}$ holds true.
    %    \item[\mylabel{item_BI_2}{{\bf BI.2)}}] If $\varepsilon_\psi\neq\epsilon_K$, then there exists an integer $u$ such that $\epsilon_K(u)=1$ and $v_p(\varepsilon_\psi(u)-1)=0$.
    %\end{itemize}
 %\index{Big image conditions! $(\tau_{\f_{/K}\otimes \bbchi})$}
%\end{remark}
%%%%%%%%%%%
%%%%%%%%%%%

%-----------------------------------------------------------------------
% End of appC.tex
%-----------------------------------------------------------------------